\documentclass[11pt,letterpaper]{amsart}
\usepackage[foot]{amsaddr}

\usepackage{mathtools}
\usepackage{amsmath}
\usepackage[T1]{fontenc}
\usepackage[latin9]{inputenc}
\usepackage{geometry}
\usepackage{wrapfig}
\usepackage{multicol}
\usepackage{graphicx}
\usepackage{soul}
\usepackage{xcolor}
\usepackage{amssymb}
\usepackage{placeins}
\usepackage{bbm}
\usepackage{cite}
\usepackage{caption}
\usepackage{enumerate}
\usepackage{afterpage}
\usepackage{enumitem}
\usepackage{bmpsize}
\usepackage{hyperref}
\usepackage{tabu}
\usepackage{enumitem}
\numberwithin{equation}{section}
\usepackage{stmaryrd}
\usepackage{tikz}
\usetikzlibrary{matrix,graphs,arrows,positioning,calc,decorations.markings,shapes.symbols}
\usepackage{ifthen}
\usetikzlibrary{math}
\usetikzlibrary{matrix,graphs,arrows,positioning,calc,decorations.markings,shapes.symbols}

\newcommand{\im}{\mathrm{i}}
 \DeclareMathOperator{\Ai}{Ai}

\newcommand{\R}{\mathbb R}
\newcommand{\N}{\mathbb N}

\newtheorem{theorem}{Theorem}[section]
\newtheorem{proposition}[theorem]{Proposition}
\newtheorem{lemma}[theorem]{Lemma}

\theoremstyle{definition}
\newtheorem{definition}[theorem]{Definition}
\theoremstyle{remark}
\newtheorem{remark}[theorem]{Remark}

\DeclareMathOperator{\diag}{\mathrm{diag}}

\newcommand{\Seq}{\Lambda}
\newcommand{\weyl}{W^{\circ}}

\title{The periodic Airy line ensembles}
\author{Evgeni Dimitrov} 
\author{Samir Sarieddine}
\author{Zongrui Yang} 
\author{Zhengye Zhou}

\date{\today}

\begin{document}

\begin{abstract}
In this paper we construct a one-parameter family of periodic Airy line ensembles $\mathcal{A}^{\mathrm{per},R}$ for $R>0$ as weak scaling limits of Borodin's periodic Schur processes. We establish several properties of $\mathcal{A}^{\mathrm{per},R}$, including a determinantal point process structure with an explicit correlation kernel. We also show that, after a parabolic shift and normalization, $\mathcal{A}^{\mathrm{per},R}$ satisfies, on its domain $[-R/2,R/2]$, the same Brownian Gibbs property enjoyed by the usual Airy line ensemble.

A main technical contribution of the paper is a general tightness criterion for determinantal line ensembles, which extends the argument used by Johansson in \cite{Kurt03} to prove a functional limit theorem for the discrete polynuclear growth model.
\end{abstract}

\maketitle
\tableofcontents

%----------------------------------------------------------------------
%
% Section 1
%
%----------------------------------------------------------------------
\section{Introduction and main results}

%----------------------------------------------------------------------
%
% Section 1.1
%
%----------------------------------------------------------------------
\subsection{Preface}\label{Section1.1}

The {\em Airy line ensemble} $\mathcal A=\{\mathcal A_i\}_{i\geq1}$ is a collection of real-valued random continuous functions, all defined on $\mathbb{R}$ and strictly ordered in the sense that $\mathcal{A}_{i}(t) > \mathcal{A}_{i+1}(t)$ for every $i \geq 1$ and $t \in \mathbb{R}$. It appears as the edge scaling limit of Wigner matrices \cite{Sod15}, lozenge tilings \cite{AH21}, non-intersecting Brownian bridges (also called {\em Brownian watermelons}) \cite{CorHamA}, as well as various integrable models of non-intersecting random walkers and last passage percolation (LPP) \cite{DNV23}. The list of models becomes quite large if one also includes those converging to various projections of $\mathcal{A}$ such as $(\mathcal{A}_1(t): t \in \mathbb{R})$ (known as the {\em Airy process}), $(\mathcal{A}_i(t_0): i \geq 1)$ for a fixed $t_0 \in \mathbb{R}$ (known as the {\em Airy point process}), and $\mathcal{A}_{1}(t_0)$ for a fixed $t_0$ (known as the {\em Tracy--Widom GUE distribution}). Because of its emergence as a universal scaling limit and its role in the construction of the {\em directed landscape} in \cite{DOV22}, the Airy line ensemble has become a central object in the {\em Kardar-Parisi-Zhang (KPZ)} universality class \cite{CU2}.

Building on the earlier works \cite{BBCS18,BCD25}, \cite{DDY26,DSY26,DY25} recently introduced a one-parameter family of {\em half-space} analogues of the Airy line ensemble and showed that these arise as scaling limits of the half-space (or symmetrized) geometric LPP model. The construction of these ensembles is part of a broader research program aimed at defining key objects in the KPZ universality class in different geometries and under different boundary conditions. Recent progress in the half-space setting includes the derivation of the transition probabilities for the half-space KPZ fixed point in \cite{Zha24} and the construction of the half-space directed landscape in \cite{DZ26}.

Another natural setting, which has received considerable attention, is the KPZ universality class on a {\em finite} interval with {\em periodic} boundary conditions. Recent progress in this direction includes the derivation of the transition probabilities for the periodic KPZ fixed point in \cite{BLL26} and the construction of the periodic directed landscape in \cite{ACS26}. Motivated in part by these developments, we introduce in the present paper a one-parameter family of line ensembles, indexed by $R \in (0, \infty)$, that are analogues of $\mathcal{A}$ in the periodic setting. We denote the $R$-indexed member of this family by $\mathcal{A}^{\mathrm{per},R}=\{\mathcal{A}^{\mathrm{per},R}_i\}_{i\geq1}$, and refer to it as the {\em periodic Airy line ensemble of period $R$}.

We construct the periodic Airy line ensemble $\mathcal{A}^{\mathrm{per},R}$ as a weak scaling limit of the {\em trace--seam--volume} cylindric plane partition model from \cite{BO22}, which is a special case of the {\em periodic Schur processes} from \cite{Bor07}. In addition to establishing this convergence in Theorem~\ref{Thm.Convergence}, we prove several properties of $\mathcal{A}^{\mathrm{per},R}$, including the following.
\begin{enumerate}[leftmargin=*,labelsep=0.5em]
\item[1.] For each $i \geq 1$, $\mathcal{A}^{\mathrm{per},R}_i$ is a random continuous function on $[-R/2,R/2]$ and almost surely satisfies the periodicity condition $\mathcal{A}^{\mathrm{per},R}_i(-R/2) = \mathcal{A}^{\mathrm{per},R}_i(R/2)$.
\item[2.] Almost surely $\mathcal{A}^{\mathrm{per},R}_i(t) > \mathcal{A}^{\mathrm{per},R}_{i+1}(t)$ for all $i \geq 1$ and $t \in [-R/2,R/2]$.
\item[3.] The point process formed by the points $(t_j,\mathcal{A}^{\mathrm{per},R}_i(t_j))$ for any finitely many $t_1, \dots, t_m$ and all $i\geq1$ is determinantal with an explicit correlation kernel $K^{\mathrm{per},R}(s,x;t,y)$, see Theorem~\ref{Thm.DeterminantalStructure}.
\item[4.] After a parabolic shift and normalization, $\mathcal{A}^{\mathrm{per},R}$ satisfies the {\em Brownian Gibbs property} of \cite{CorHamA}, see Theorem~\ref{Thm.GibbsProperty}.
\item[5.] The sequence of line ensembles $\mathcal{A}^{\mathrm{per},R_N}$ converges weakly to $\mathcal{A}$ as $R_N \to \infty$, see Theorem~\ref{Thm.ConvergenceToAiryLE}.
\end{enumerate}

\medskip

Let us briefly explain why we believe $\mathcal{A}^{\mathrm{per},R}$ is the natural periodic analogue of the Airy line ensemble $\mathcal{A}$. First, upon setting $R=b$ and $u=t-s$, the kernel $K^{\mathrm{per},R}(s,x;t,y)$ becomes the finite-temperature extended Airy kernel $\mathcal{K}^{\mathrm{edge}}_b(x,y,u)$ introduced in \cite[(99)]{le2017periodic}. The latter is a natural one-parameter deformation of the extended Airy kernel, which characterizes the finite-dimensional distributions of $\mathcal{A}$, see \cite[Definition~1.3 and Theorem~1.4]{CS14}. We mention that the extended Airy kernel is frequently attributed to \cite{Spohn}, where it arises in the context of the polynuclear growth model, although it appeared earlier in \cite{Mac94,FNH99}. The finite-temperature extended Airy kernel arises in the scaling limit of non-intersecting periodic Ornstein--Uhlenbeck bridges in \cite{le2017periodic} and of cylindric Plancherel processes in \cite{BB19}. In view of the kernel convergence established in these works, we expect both models to converge, at the level of their full line ensembles, to $\mathcal{A}^{\mathrm{per},R}$. Thus, already from the determinantal perspective, $\mathcal{A}^{\mathrm{per},R}$ appears as a natural universal limit for periodic models in the KPZ class.

A second indication comes from property~4 above. In \cite{AH23}, it was shown that, under a suitable condition on the global behavior of the top curve, the parabolically shifted Airy line ensemble is the unique (up to an independent affine shift) line ensemble that satisfies the Brownian Gibbs property. It is therefore natural to view the Brownian Gibbs structure of $\mathcal{A}^{\mathrm{per},R}$ as further evidence that it is the correct periodic counterpart of $\mathcal{A}$. This interpretation is also strongly supported by the weak convergence in property~5 above.

In the full-space setting, there are well-established connections between the Airy line ensemble of \cite{CorHamA}, the KPZ fixed point of \cite{MQR21}, and the directed landscape of \cite{DOV22}. It is therefore natural to ask whether analogous connections hold in the periodic setting between $\mathcal{A}^{\mathrm{per},R}$, the periodic KPZ fixed point of \cite{BLL26}, and the periodic directed landscape of \cite{ACS26}. Despite our efforts, we have not been able to establish such connections. Understanding whether $\mathcal{A}^{\mathrm{per},R}$ can be related directly to either of these periodic objects is an interesting open problem that we hope to pursue in future work.

%----------------------------------------------------------------------
%
% Section 1.2
%
%----------------------------------------------------------------------
\subsection{Main results}\label{Section1.2}

In this section, we summarize the main results of the paper. We begin by introducing the discrete model we consider. To this end, we first recall a few key concepts.

A {\em partition} is a decreasing sequence of non-negative integers $\lambda=(\lambda_1\geq\lambda_2\geq\cdots)$ with finitely many nonzero terms, and the {\em weight} of $\lambda$ is $|\lambda|=\sum_{i=1}^{\infty}\lambda_i$. We denote by $\emptyset=(0,0,\ldots)$ the {\em empty partition}. Given two partitions $\lambda$ and $\mu$, we write $\mu\preceq\lambda$, or equivalently $\lambda\succeq\mu$, and say that they {\em interlace} if $\lambda_1\geq\mu_1\geq\lambda_2\geq\mu_2\geq\cdots$. 

The state space of our model depends on an integer $M\in\mathbb{N}$ and consists of periodic sequences of interlacing partitions of the form
\begin{equation}\label{Eq.DefLambda}
\Seq=
\left\{(\lambda^{(-M)},\ldots,\lambda^{(M)}): 
\lambda^{(-M)}
\preceq\cdots\preceq \lambda^{(-1)} \preceq
\lambda^{(0)}
\succeq\lambda^{(1)}
\succeq\cdots\succeq
\lambda^{(M)}
=\lambda^{(-M)}
\right\}.
\end{equation}

The elements of the state space $\Seq$ can naturally be interpreted as {\em cylindric plane partitions}, as we explain next. For each $(\lambda^{(-M)},\ldots,\lambda^{(M)})\in\Seq$, we associate an integer array $\pi=(\pi_{i,j}:i,j\geq1)$ as follows. For every $r\geq1$, set
$$
\pi_{r-k,r}=\lambda^{(k)}_r,\qquad -M\leq k\leq0,
\qquad\text{and}\qquad
\pi_{r,r+k}=\lambda^{(k)}_r,\qquad 0\leq k\leq M,
$$
and extend the array to all $i,j\geq1$ using the periodicity relation
$$
\pi_{i+M,j}=\pi_{i,j+M}.
$$
One readily verifies that this defines a bijection between $\Seq$ and the set $\Pi$ of arrays $\pi=(\pi_{i,j}:i,j\geq1)$ satisfying the following conditions:
\begin{enumerate}
\item $\pi_{i,j}\in\mathbb Z_{\geq 0}$ for all $i,j\geq1$;
\item $\pi_{i,j}\geq\pi_{i,j+1}$ and $\pi_{i,j}\geq\pi_{i+1,j}$ for all $i,j\geq1$;
\item $\pi_{i+M,j}=\pi_{i,j+M}$ for all $i,j\geq1$;
\item $\pi_{i,j}=0$ whenever $i+j$ is sufficiently large.
\end{enumerate}
The elements of $\Pi$ are special cases of cylindric plane partitions, which were first introduced by Gessel and Krattenthaler in \cite{GK97}. In the notation of \cite[Section~5]{Bor07}, they correspond to cylindric partitions with $d=l=M$ and $\mu=\emptyset$.

For $\pi\in\Pi$, we define its {\em diagonal sums} $\diag_k(\pi)$ and {\em norm} $|\pi|$ via
$$
\diag_k(\pi):=\sum_{i=1}^{\infty}\pi_{i-k,i}, \quad k \leq 0, \quad \diag_k(\pi):=\sum_{i=1}^{\infty}\pi_{i,i+k}, \quad k \geq 0, \quad |\pi|:=\sum_{k = -M}^{M-1}\diag_k(\pi).
$$
Under the preceding bijection, these quantities satisfy
$$
\diag_k(\pi)=|\lambda^{(k)}|, \quad -M \leq k \leq M, \quad |\pi| = \sum_{k=-M}^{M-1}|\lambda^{(k)}|.
$$

The probability distribution we consider on $\Seq$ depends on two parameters $b,q\in(0,1)$, as well as the parameter $M \geq 1$. It is defined by
\begin{equation}\label{Eq.S1ProbDistr}
\mathbb{P}\left(\lambda^{(-M)},\ldots,\lambda^{(M)}\right) = \frac{1}{Z_{M,b,q}} \cdot b^{2|\lambda^{(0)}| - 2|\lambda^{(M)}|} \cdot q^{-M|\lambda^{(M)}|} \cdot q^{\sum_{k=-M}^{M-1}|\lambda^{(k)}|},
\end{equation}
where $Z_{M,b,q}$ is the normalization constant. As we explain in Section~\ref{Section3}, see (\ref{Eq.FormulaForZ}) and (\ref{Eq.Specialization}), the normalization constant is given by
$$
Z_{M,b,q} =\frac{1}{(q^M;q^M)_\infty} \cdot \prod_{i,j=1}^{M} \frac{1}{(b^2q^{i+j-1};q^M)_\infty},
$$
where $(a;q)_{\infty} = \prod_{k \geq 0} (1-aq^k)$ is the $q$-Pochhammer symbol. The corresponding probability distribution on $\Pi$ is given by
$$ \mathbb{P}(\pi) = \frac{1}{Z_{M,b,q}} \cdot b^{2\diag_0(\pi) - 2 \diag_M(\pi)} \cdot q^{-M \diag_M(\pi)} \cdot q^{|\pi|}. $$

The following definition introduces the scaling of the parameters that we use and constructs a natural sequence of random continuous curves encoding the cylindric plane partition. A sequence of random continuous functions $\mathcal{L}=\{\mathcal{L}_i\}_{i\geq1}$ on a common interval is known as an {\em $\mathbb{N}$-indexed line ensemble}. The precise definition of a line ensemble is given in Section~\ref{Section2.2}.
\begin{definition}\label{Def.S1Scaling}
Fix $b\in(0,1)$ and $R\in(0,\infty)$, and set
$$c_1=-2\log(1-b), \qquad c_2=b^{1/3}(1-b)^{-2/3}. $$
For each $N\in \mathbb{N}$, define
$$ q_N=e^{-1/N}, \qquad M_N=\left\lceil c_2^{-1}RN^{2/3}\right\rceil.$$
Let $(\lambda^{(-M_N)},\ldots,\lambda^{(M_N)})$
be distributed according to \eqref{Eq.S1ProbDistr} with $q=q_N$ and $M=M_N$, and denote its law by $\mathbb{P}_N$. For $-M_N\leq k\leq M_N$, define
$t_{k,N}=Rk/(2M_N)$. For each $i\geq1$, we define the scaled curve $\mathcal{L}_i^N\in C([-R/2,R/2])$ to be the unique function that is linear on every interval $[t_{k,N},t_{k+1,N}]$, where $-M_N\leq k<M_N$, and satisfies
\begin{equation}\label{Eq.S1ScaledCurves} 
\mathcal{L}_i^N(t_{k,N}) = c_2^{-1}N^{-1/3} \left( \lambda_i^{(k)}-c_1N +\frac{b|k|}{1-b}-\frac{bk^2}{4(1-b)^2N} \right)
\end{equation}
for every $-M_N\leq k\leq M_N$. In this way $\mathcal{L}^N = \{\mathcal{L}_i^N\}_{i \geq 1}$ defines a random element of $C(\mathbb{N} \times [-R/2,R/2])$, that is, an $\mathbb{N}$-indexed line ensemble on $[-R/2,R/2]$ in the sense of Definition~\ref{Def.LineEnsembles}.
\end{definition}

With the above notation in place, we are ready to state the main result we establish about the measures in (\ref{Eq.S1ProbDistr}).
\begin{theorem}\label{Thm.Convergence}
Assume the same notation as in Definition~\ref{Def.S1Scaling}. As $N\to\infty$, the line ensembles $\mathcal{L}^N$ converge weakly to an $\mathbb{N}$-indexed line ensemble on $[-R/2,R/2]$, denoted by $\mathcal{A}^{\mathrm{per},R}= \{\mathcal{A}^{\mathrm{per},R}_i\}_{i\geq1}$. Moreover, $\mathcal{A}^{\mathrm{per},R}$ is ordered, in the sense that almost surely
\begin{equation}\label{Eq.S1OrderedPerAiry}
\mathcal{A}^{\mathrm{per},R}_i(t) \geq \mathcal{A}^{\mathrm{per},R}_{i+1}(t) \mbox{ for all } i \geq 1 \mbox{ and } t \in [-R/2,R/2],
\end{equation}
and is periodic, in the sense that almost surely
\begin{equation}\label{Eq.S1PeriodicPerAiry}
\mathcal{A}^{\mathrm{per},R}_i(-R/2) = \mathcal{A}^{\mathrm{per},R}_i(R/2) \mbox{ for all } i \geq 1.
\end{equation}
We call $\mathcal{A}^{\mathrm{per},R}$ the {\em periodic Airy line ensemble} of period $R$.
\end{theorem}

\begin{remark}\label{Rem.Convergence}
In view of (\ref{Eq.S1PeriodicPerAiry}), we may equivalently view $\mathcal{A}^{\mathrm{per},R}$ as a line ensemble on the circle of circumference $R$, obtained by identifying the endpoints of $[-R/2,R/2]$.
\end{remark}

In the remainder of this section, we state a sequence of theorems describing the main properties of the periodic Airy line ensemble $\mathcal{A}^{\mathrm{per},R}$ constructed in Theorem~\ref{Thm.Convergence}. We begin with its determinantal structure, which characterizes its finite-dimensional distributions. For the precise definition of a determinantal point process and related background, we refer the reader to Section~\ref{Section2.1}. To formulate the result, we first introduce the relevant reference measures and correlation kernel.

\begin{definition}\label{Def.S1Measures}
For a set $\mathsf{S}\subset\mathbb{R}$, we let $\mu_{\mathsf{S}}$ denote the counting measure on $\mathbb{R}$ defined by
$$
\mu_{\mathsf{S}}(A)=|A\cap\mathsf{S}|.
$$
We also let $\mathrm{Leb}$ denote Lebesgue measure on $\mathbb{R}$ and write
$\mu_{\mathsf{S}}\times\mathrm{Leb}$ for the corresponding product measure on $\mathbb{R}^2$.
\end{definition}

\begin{definition}\label{Def.S1PeriodicAiryKernel}
Fix $R\in(0,\infty)$. For $s,t\in(-R/2,R/2]$ and $x,y\in\mathbb{R}$, we define the kernel
\begin{equation}\label{Eq.PeriodicAiryKernel}
K^{\mathrm{per},R}(s,x;t,y) =
\begin{cases}
\displaystyle\int_{-\infty}^{\infty}\frac{e^{-(t-s)v}} {1+e^{-Rv}}
\Ai(x+v)\Ai(y+v)\,dv,
&\quad \text{if }s\leq t,\\[3ex] \displaystyle -\int_{-\infty}^{\infty}
\frac{e^{-(t-s)v}}{1+e^{Rv}} \Ai(x+v)\Ai(y+v)\,dv, &\quad \text{if }s>t,
\end{cases}
\end{equation}
where $\Ai$ denotes the Airy function.
\end{definition}

We may now state the determinantal characterization of the periodic Airy line ensemble.

\begin{theorem}\label{Thm.DeterminantalStructure}
Let $\mathcal{A}^{\mathrm{per},R}$ be as in Theorem~\ref{Thm.Convergence}. Fix a finite set
$$\mathsf{S}=\{s_1,\ldots,s_m\}\subset(-R/2,R/2],\qquad s_1<\cdots<s_m,$$
and define the random measure on $\mathbb{R}^2$ by setting, for every Borel set $A\subseteq\mathbb{R}^2$,
\begin{equation}\label{Eq.PeriodicAiryPointProcess}
M^{\mathsf{S};R}(A) = \sum_{i=1}^{\infty}\sum_{j=1}^{m} {\bf 1} \left\{ \left(s_j,\mathcal{A}^{\mathrm{per},R}_i(s_j)\right)\in A \right\}.
\end{equation}
Then $M^{\mathsf{S};R}$ is a determinantal point process on $\mathbb{R}^2$ with reference measure $\mu_{\mathsf{S}}\times\mathrm{Leb}$, as in Definition~\ref{Def.S1Measures}, and correlation kernel $K^{\mathrm{per},R}$ as in Definition~\ref{Def.S1PeriodicAiryKernel}.
\end{theorem}

\begin{remark}\label{Rem.DeterminantalUniqueness}
From \cite[Proposition~2.13(3)]{ED24a}, \cite[Corollary~2.20]{ED24a}, and \cite[Lemma~3.1]{DM21}, there is at most one ordered line ensemble satisfying the conditions of Theorem~\ref{Thm.DeterminantalStructure}. Consequently, one may define $\mathcal{A}^{\mathrm{per},R}$ as the unique ordered line ensemble satisfying these conditions, whose existence is guaranteed by Theorems~\ref{Thm.Convergence} and~\ref{Thm.DeterminantalStructure}.
\end{remark}

The determinantal characterization in Theorem~\ref{Thm.DeterminantalStructure} also reveals a fundamental symmetry of the periodic Airy line ensemble. Namely, its periodic extension is translation invariant, as stated in the following theorem.
\begin{theorem}\label{Thm.TranslationInvariance}
Let $\mathcal{A}^{\mathrm{per},R}$ be as in Theorem~\ref{Thm.Convergence}, and let
$\widehat{\mathcal{A}}^{\mathrm{per},R} = \{\widehat{\mathcal{A}}^{\mathrm{per},R}_i\}_{i\geq1}$
denote its continuous $R$-periodic extension to $\mathbb{R}$. Then, for every $\theta \in\mathbb{R}$,
$$ \left\{ \widehat{\mathcal{A}}^{\mathrm{per},R}_i(t+\theta ):
i\geq1,\ t\in\mathbb{R}
\right\} \overset{d}{=} \left\{ \widehat{\mathcal{A}}^{\mathrm{per},R}_i(t): i\geq1,\ t\in\mathbb{R} \right\}.
$$
\end{theorem}

The next result shows that the periodic Airy line ensemble has a local description in terms of avoiding Brownian bridges.
\begin{theorem}\label{Thm.GibbsProperty}
Let $\mathcal{A}^{\mathrm{per},R}$ be as in Theorem~\ref{Thm.Convergence}, and define the $\mathbb{N}$-indexed line ensemble
$\mathcal{L}^{\mathrm{per},R}
=
\{\mathcal{L}^{\mathrm{per},R}_i\}_{i\geq1}$ on $[-R/2,R/2]$
through
\begin{equation}\label{Eq.ParabolicPeriodicAiry}
\mathcal{L}^{\mathrm{per},R}_i(t)
= 2^{-1/2}\left(\mathcal{A}^{\mathrm{per},R}_i(t)-t^2\right),
\qquad
i\geq1,\quad t\in[-R/2,R/2].
\end{equation}
Then $\mathcal{L}^{\mathrm{per},R}$ satisfies the periodic Brownian Gibbs property from Definition~\ref{Def.PeriodicAreaTiltedBGP}. In particular, the line ensemble $\mathcal{A}^{\mathrm{per},R}$ is non-intersecting, i.e., $\mathcal{A}^{\mathrm{per},R}_i(t) > \mathcal{A}^{\mathrm{per},R}_j(t)$ for all $t\in[-R/2,R/2]$ and $1\leq i<j$ almost surely.
\end{theorem}
\begin{remark}\label{Rem.GibbsProperty}
As shown in Lemma~\ref{Lem.PeriodicBGPCharacterization}, the statement that $\mathcal{L}^{\mathrm{per},R}$ satisfies the periodic Brownian Gibbs property is equivalent to the following two statements: (1) $\mathcal{L}^{\mathrm{per},R}_i(-R/2)=\mathcal{L}^{\mathrm{per},R}_i(R/2)$ for every $i\geq1$ almost surely; and (2) $\mathcal{L}^{\mathrm{per},R}$ satisfies the usual Brownian Gibbs property introduced by Corwin and Hammond in \cite{CorHamA}.
\end{remark}

The next result describes the behavior of the periodic Airy line ensemble as its period tends to infinity. In this limit, the periodic constraint becomes negligible and one recovers the full-space Airy line ensemble.

\begin{theorem}\label{Thm.ConvergenceToAiryLE}
For each $R>0$, let $\mathcal{A}^{\mathrm{per},R}$ be the periodic Airy line ensemble from Theorem~\ref{Thm.Convergence}, and let
$\widehat{\mathcal{A}}^{\mathrm{per},R}$
denote its continuous $R$-periodic extension to $\mathbb{R}$, as in Theorem~\ref{Thm.TranslationInvariance}. Then, $\widehat{\mathcal{A}}^{\mathrm{per},R} \Rightarrow \mathcal{A}$ as $R\to\infty$, where $\mathcal{A}$ is the full-space Airy line ensemble
$\mathcal{A}=\{\mathcal{A}_i\}_{i\geq1}$ from \cite{CorHamA}.
\end{theorem}
\begin{remark}
The weak convergence in Theorem~\ref{Thm.ConvergenceToAiryLE} is in $C(\mathbb{N}\times\mathbb{R})$; see Section~\ref{Section2.2}.
\end{remark}

We conclude with the opposite, short-period regime. As $R\to0+$, the suitably centered and scaled curve values at a fixed time converge, in the finite-dimensional sense, to the decreasingly ordered atoms of a Poisson point process with intensity $e^{-x}\,dx$.
\begin{theorem}\label{Thm.ConvergenceToGumbel}
For each $R>0$, let $\mathcal{A}^{\mathrm{per},R}$ be the periodic Airy line ensemble from Theorem~\ref{Thm.Convergence}. Let $ \mathcal{G}_1>\mathcal{G}_2>\cdots$ denote the decreasing enumeration of the atoms of a Poisson point process on $\mathbb{R}$ with intensity $e^{-x}\,dx$. Then, as $R\to0+$,
\begin{equation}\label{Eq.ConvToGumbelK}
\left(
R\mathcal{A}^{\mathrm{per},R}_i(0)
+\frac{1}{2}\log(4\pi R^3)
\right)_{i=1}^{k}
\Rightarrow
\left(\mathcal{G}_i\right)_{i=1}^{k}
\end{equation}
for every $k\in\mathbb{N}$. In particular,
\begin{equation}\label{Eq.ConvToGumbel}
\lim_{R\to0+} \mathbb{P}\left(R\mathcal{A}^{\mathrm{per},R}_1(0) +\frac{1}{2}\log(4\pi R^3) \leq s \right) = e^{-e^{-s}}, \qquad s\in\mathbb{R}.
\end{equation}
\end{theorem}
\begin{remark}
For each fixed $k\in\mathbb{N}$, the convergence in Theorem~\ref{Thm.ConvergenceToGumbel} is weak convergence of random vectors in $\mathbb{R}^k$.
\end{remark}

%----------------------------------------------------------------------
%
% Section 1.3
%
%----------------------------------------------------------------------

\subsection{Proof ideas and paper outline}\label{Section1.3}

The weak convergence in Theorem~\ref{Thm.Convergence} is established by showing that the sequence of line ensembles $\{\mathcal{L}^N\}_{N\geq1}$ from Definition~\ref{Def.S1Scaling}
\begin{enumerate}
\item[(1)] converges in the finite-dimensional sense, and
\item[(2)] is tight in $C(\mathbb{N} \times [-R/2,R/2])$. 
\end{enumerate}
The starting point for the analysis of the measures in (\ref{Eq.S1ProbDistr}) is that they are special cases of the periodic Schur processes introduced in \cite{Bor07}. This observation was made in \cite{BO22}, where the measures in (\ref{Eq.S1ProbDistr}) were introduced and studied under the name ``trace--seam--volume'' distributions.

The periodic Schur processes are recalled in Section~\ref{Section3.1}, where the identification with our models is spelled out explicitly. The same section explains that if one performs a random shift 
$$\lambda_i^{(k)} \mapsto \lambda_i^{(k)} - i + S,$$ 
where $(\lambda^{(-M)}, \dots, \lambda^{(M)})$ is distributed according to (\ref{Eq.S1ProbDistr}) and $S$ is an independent discrete Gaussian random variable, then the elements of $\{(k, \lambda_i^{(k)}- i + S): M\geq k > -M, i \geq 1\}$ form a determinantal point process with an explicit correlation kernel. This determinantal structure is the key to establishing the two points above. In particular, it allows us to study the ensembles $\widehat{\mathcal{L}}^N$, constructed in the same way as $\mathcal{L}^N$ but using the shifted variables above. Under the scaling we are considering, the shift between $\mathcal{L}^N$ and $\widehat{\mathcal{L}}^N$ disappears and so it suffices to establish the two points above for $\widehat{\mathcal{L}}^N$.

The correlation kernel for the point process associated with $\widehat{\mathcal{L}}^N$ is derived in Proposition~\ref{Prop.ScaledKernel} by specializing and changing variables in Borodin's correlation kernel formula recalled in Proposition~\ref{Prop.SchurDPP}. The kernel has a double-contour integral representation suitable for asymptotic analysis using the method of steepest descent, and its limit is derived in Proposition~\ref{Prop.S4KernelLimit}. The asymptotic analysis of the kernel then feeds into the general framework for convergence of determinantal point processes developed in \cite{ED24a}, which ultimately allows us to establish the finite-dimensional convergence of $\widehat{\mathcal{L}}^N$.

\medskip 
The main challenge is therefore establishing tightness for $\widehat{\mathcal{L}}^N$. In previous related works, including \cite{DNV23} and \cite{D24b}, tightness of line ensembles arising from Schur processes was established by recognizing a discrete Gibbsian line ensemble structure. This structure identifies the curves of the discrete ensemble as trajectories of avoiding or interlacing random walk bridges, which can then be compared with regular Brownian bridges. In such a comparison, one leverages the well-understood structure of Brownian paths to argue that discrete paths have a well-behaved modulus of continuity along subsequences, ultimately establishing tightness.

Our model does admit a Gibbsian line ensemble structure, which we use to establish the Brownian Gibbs property of the periodic Airy line ensemble in Theorem~\ref{Thm.GibbsProperty}. However, as we explain next, this structure is not well suited for proving tightness by the usual techniques for Gibbsian line ensembles. The essential difficulty is that the curves of our discrete ensembles do not simply behave as locally interacting random walks; their law is further reweighted by a factor depending on the area enclosed by the paths, giving rise to what is sometimes called an {\em area tilt}. This area tilt can be traced to the fact that the periodic Schur process corresponding to our model is not homogeneous. Rather, the parameters of the specializations used to define the process form a geometric sequence in $q$. 

This area tilt causes a problem in the usual argument showing that the minimum of each curve $\widehat{\mathcal{L}}^N_k$ is not too low on a given interval $[\alpha, \beta]$, which roughly proceeds as follows:
\begin{itemize}
\item $\widehat{\mathcal{L}}^N_k(\alpha)$ and $\widehat{\mathcal{L}}^N_k(\beta)$ are not too low (which follows from finite-dimensional convergence);
\item pushing $\widehat{\mathcal{L}}^N_{k + i}$ for $i \geq 1$ down to $-\infty$ causes $\widehat{\mathcal{L}}^N_k$ to go down by stochastic monotonicity;
\item on $[\alpha, \beta]$, the resulting paths $\{\widehat{\mathcal{L}}^N_i\}_{i = 1}^k$ (after removing all others) behave like $k$ avoiding Brownian bridges with endpoints that are not too low, and their minimum is therefore unlikely to be too low;
\item as the finite-path ensemble is stochastically lower than the original one, the minimum of the original $\widehat{\mathcal{L}}^N_k$ is also not too low.
\end{itemize}
This argument completely breaks down in our case, since the moment we push the curve $\widehat{\mathcal{L}}^N_{k+1}$ to $-\infty$, the curve $\widehat{\mathcal{L}}^N_k$ goes down with it due to the exponential penalty of the area tilt.

It is worth mentioning that area-tilted Gibbsian line ensembles have recently been studied, see \cite{HKS25,Ser23} and the references therein, in situations where exact formulas are not available. However, in all of these works, the paths under consideration are assumed to be non-negative. Consequently, zero plays the role of a uniform barrier that prevents low minima from ever occurring in these ensembles, allowing one to avoid the aforementioned problem.

\medskip 

In view of the issues outlined above, we decided to approach the tightness problem differently, avoiding Gibbsian line ensembles altogether and relying instead entirely on the determinantal point process structure of our model. To this end, we develop a general tightness criterion based on fourth-moment estimates for a large family of determinantal point processes. This idea goes back to Johansson's work on the polynuclear growth model \cite{Kurt03}, where a fourth-moment argument is used to prove a functional limit theorem for the top curve of the associated line ensemble. In particular, he showed tightness of the top curve and convergence to the Airy process. We also mention that the fourth-moment method was used in \cite{Spohn} to show that the Airy process has a continuous modification.

The tightness criterion we develop is described in Sections~\ref{Section2.3} and~\ref{Section2.4}, and we view it as one of the main technical contributions of the paper. It applies to line ensembles obtained by interpolating ordered particles of space-time determinantal point processes on increasingly fine time meshes. Roughly speaking, it requires that the two-time correlation kernels satisfy certain decomposition and contraction identities, together with integrability and quantitative bounds on the corresponding difference kernels that yield a fourth-moment estimate for linear statistics. Combined with suitable one-time upper- and lower-tail bounds for the particle locations, these estimates imply tightness of the entire line ensemble in $C(\mathbb{N}\times[\alpha,\beta])$. We refer the interested reader to Section~\ref{Section2.5} for a detailed discussion of the assumptions underlying the criterion and a comparison with Johansson's argument in \cite{Kurt03}, including the main differences between the two approaches and a point in the latter argument that appears to require additional justification.

\medskip

In Sections~\ref{Section5.1} and~\ref{Section5.2}, we verify that our model satisfies the tightness criterion of Section~\ref{Section2.4}. In Section~\ref{Section5.3}, we combine this with the kernel asymptotics from Proposition~\ref{Prop.S4KernelLimit} to establish the weak convergence in Theorem~\ref{Thm.Convergence}. Theorem~\ref{Thm.DeterminantalStructure} is proved concurrently in the same section. 

In Section~\ref{Section6}, we turn to the Brownian Gibbs property. We first introduce Brownian Gibbsian line ensembles with area tilts and establish a characterization of the corresponding periodic Brownian Gibbs property. We then identify a discrete resampling structure for the random cylindric plane partition model in terms of certain geometric random walk loops and prove that, under our scaling, these loops converge to Brownian bridges. Combining this convergence with the discrete Gibbs property allows us to pass the resampling identity to the limit and prove Theorem~\ref{Thm.GibbsProperty}. In Section~\ref{Section7}, we establish the remaining properties of the periodic Airy line ensemble stated in Section~\ref{Section1.2}. We prove translation invariance using the determinantal structure of $\mathcal{A}^{\mathrm{per},R}$ and a gauge transformation of its correlation kernel, show that $\mathcal{A}^{\mathrm{per},R}$ converges to the usual Airy line ensemble as $R\to\infty$ by combining kernel convergence, one-time tail estimates, and the Brownian Gibbs property, and finally analyze the regime $R\to0$, where the appropriately rescaled one-time point process converges to a Poisson point process, yielding Theorem~\ref{Thm.ConvergenceToGumbel}.

%----------------------------------------------------------------------
%
% Section 1.4
%
%----------------------------------------------------------------------
\subsection{Related works and future directions}\label{Section1.4}

As mentioned earlier, the distribution (\ref{Eq.S1ProbDistr}) is the trace--seam--volume cylindric plane partition model studied in \cite{BO22}. In the notation of Theorem~\ref{Thm.Convergence}, \cite[Theorem~1(ii)]{BO22} proves that the cdf of $\mathcal{L}_1^N(0)$ converges to a Fredholm determinant involving the kernel $K^{\mathrm{per},R}$ from (\ref{Eq.PeriodicAiryKernel}) with $s = t$. Thus, their result establishes one-point convergence for the top curve at a single time. Theorem~\ref{Thm.Convergence} substantially strengthens this by proving weak convergence of the entire discrete line ensemble to $\mathcal{A}^{\mathrm{per},R}$ in $C(\mathbb{N}\times[-R/2,R/2])$.

As mentioned in Section~\ref{Section1.1}, upon setting $R=b$ and $u=t-s$, the kernel $K^{\mathrm{per},R}(s,x;t,y)$ from (\ref{Eq.PeriodicAiryKernel}) becomes the finite-temperature extended Airy kernel $\mathcal{K}^{\mathrm{edge}}_b(x,y,u)$ introduced in \cite[(99)]{le2017periodic} in the context of non-intersecting Ornstein--Uhlenbeck bridges, which also appears in the edge-scaling limit of the stationary cylindric Plancherel process; see \cite[Proposition~6.9]{BB19}. In particular, Theorem~\ref{Thm.DeterminantalStructure} shows that the top curve of $\mathcal{A}^{\mathrm{per},R}$ has the finite-dimensional distributions of the periodic Airy process from \cite{le2017periodic}, while the full collection of curves provides a continuous line-ensemble lift of this process.

When $s = t$, the kernel $K^{\mathrm{per},R}(s,x;s,y)$ in \eqref{Eq.PeriodicAiryKernel} has an earlier history. Johansson introduced this kernel in his study of a grand canonical version of the Moshe--Neuberger--Shapiro (MNS) random matrix model (see \cite[(1.14)]{Joh07} with $\alpha = R$) and proved that the resulting determinantal point process interpolates between the Poisson point process with intensity $e^{-x}\,dx$ and the Airy point process; see \cite[Proposition~1.1 and Theorem~1.2]{Joh07}. The corresponding edge crossover for the canonical MNS model was subsequently obtained in \cite{LW20}. The same kernel was also identified at the edge of a gas of non-interacting fermions in a harmonic trap at finite temperature in \cite{DLMS15,DLMS16}, where its connection with the one-point distribution of the KPZ equation at finite time was also observed.

As explained in Section~\ref{Section1.3}, the tightness criterion developed in Sections~\ref{Section2.3} and~\ref{Section2.4} is inspired by Johansson's fourth-moment argument for the discrete polynuclear growth model in \cite{Kurt03}. Our formulation isolates the structural and analytic properties of the correlation kernel that are needed for the argument and yields tightness of the entire line ensemble rather than only of the top curve. We refer to Section~\ref{Section2.5} for a detailed comparison with \cite{Kurt03}.

The finite-dimensional convergence arguments used throughout the paper rely on the general framework for determinantal point processes and ordered particle configurations developed in \cite{ED24a}. In particular, the passage from convergence of correlation kernels to convergence of the associated point processes, and then to convergence of the ordered particles, is based on \cite[Section 2]{ED24a}. We use this framework in the proofs of Theorems~\ref{Thm.Convergence} and~\ref{Thm.DeterminantalStructure}, in the large-period convergence of Theorem~\ref{Thm.ConvergenceToAiryLE}, and in the short-period limit of Theorem~\ref{Thm.ConvergenceToGumbel}. We mention that a Pfaffian analogue of this framework was subsequently developed in \cite[Section~5]{DY25} and used there in the construction of half-space Airy line ensembles.

Finally, the short-period limit in Theorem~\ref{Thm.ConvergenceToGumbel} is closely related to the high-temperature limit studied by Johansson. For the one-time finite-temperature Airy kernel, \cite[Theorem~1.2]{Joh07} proves the corresponding kernel limit to the Poisson kernel with intensity $e^{-x}\,dx$, while \cite[Theorem~1.3]{Joh07} deduces convergence of the rightmost particle to the Gumbel distribution. In the present setting, we prove convergence of the entire suitably centered and scaled ordered configuration
$\{\mathcal{A}^{\mathrm{per},R}_i(0)\}_{i\geq1}$ to the decreasingly ordered atoms of the Poisson point process with intensity $e^{-x}\,dx$. The Gumbel convergence of the top curve then follows as an immediate consequence.

\medskip

We conclude by discussing two possible directions for future research. The first concerns applications of the tightness criterion developed in Sections~\ref{Section2.3} and~\ref{Section2.4}. As explained in Section~\ref{Section2.5}, the structural assumptions underlying our fourth-moment estimates arise naturally for determinantal processes whose correlation kernels have an Eynard--Mehta-type representation in terms of transition kernels. We therefore expect the criterion to apply to a broad class of models, provided the required analytic estimates and one-time tail bounds can be established. Of particular interest are determinantal models for which no suitable Gibbsian line-ensemble structure is apparent, or for which the Gibbs resampling property involves area tilts. Natural candidates include Schur processes \cite{OR03}, $q^{\mathrm{vol}}$-weighted lozenge tilings \cite{Pet14}, and dimer models on rail yard graphs \cite{BBCCR17}. It would be interesting to investigate whether our approach can be used to establish tightness and construct continuous limiting line ensembles for such models without relying on Gibbs resampling arguments.

A second direction is to investigate the relationship between the periodic Airy line ensemble constructed in this paper and the periodic KPZ fixed point and directed landscape, recently constructed in \cite{ACS26,BLL26}. The connections between these objects are well understood in the full-space setting, and establishing their periodic counterparts would be an important step toward a broader understanding of periodic KPZ universality. We have made some attempts to identify such connections, but have not yet been successful, and hope to pursue this question in future work.

%----------------------------------------------------------------------
%
% Section 1.5
%
%----------------------------------------------------------------------
\subsection{Use of AI}
We used ChatGPT, developed by OpenAI, primarily for proofreading and improving the exposition of the paper. The only substantive use beyond this concerns the tightness criterion developed in Sections~\ref{Section2.3} and~\ref{Section2.4}. Our original proof was considerably longer and followed Johansson's argument in \cite{Kurt03} more closely. ChatGPT suggested a change of coordinates under which several error terms in the relevant moment expansions vanish, substantially simplifying the argument and shortening the proof. The authors then completely rewrote the proof and checked all details. They take full responsibility for the mathematical content of the paper.

\subsection*{Acknowledgments}
E.D. was partially supported by Simons Foundation International through Simons Award TSM-00014004. The authors gratefully acknowledge the hospitality and support of the American Institute of Mathematics during the workshops ``Algebraic and combinatorial structures in exactly solvable models'' and ``Gibbsian line ensembles,'' where part of this work was completed.
Z.Y. thanks Ivan Corwin for helpful discussions and was partially supported by the Miller Institute for Basic Research in Science at the University of California, Berkeley, Ivan Corwin's NSF grants DMS-1811143, DMS-2246576, Simons Foundation Grant 929852, and the Fernholz Foundation's ``Summer Minerva Fellows'' program.

%----------------------------------------------------------------------
%
% Section 2
%
%----------------------------------------------------------------------
\section{Determinantal point processes}\label{Section2}
In this section, we introduce the frameworks of determinantal point processes and line ensembles and prove two auxiliary results that will be used later in the paper. In Section~\ref{Section2.1}, we recall the basic definitions and notation for determinantal point processes, and in Section~\ref{Section2.2}, we introduce line ensembles. In Section~\ref{Section2.3}, we establish a fourth moment bound for differences of linear statistics of determinantal point processes restricted to two vertical slices. In Section~\ref{Section2.4}, we use this bound to derive a general tightness criterion for line ensembles obtained by linearly interpolating the particle locations of determinantal point processes on increasingly fine time meshes. Finally, in Section~\ref{Section2.5}, we compare these results with the functional tightness argument of Johansson~\cite{Kurt03} and discuss the role of the assumptions introduced in Sections~\ref{Section2.3} and \ref{Section2.4}.

%----------------------------------------------------------------------
%
% Section 2.1
%
%----------------------------------------------------------------------
\subsection{Definitions and notation}\label{Section2.1} In this section, we introduce the core notation, definitions, and results about determinantal point processes that we require in the rest of the paper. Our exposition follows \cite[Section 2]{ED24a}, to which we refer for additional details.

Throughout this section we assume that $E$ is a closed subset of
$\mathbb{R}^k$, with the inherited topology and Borel $\sigma$-algebra $\mathcal E$, for some fixed $k\in\mathbb{N}$. We also fix a
probability space $(\Omega,\mathcal{F},\mathbb{P})$.

We say that $M$ is a \emph{random measure} on $(E,\mathcal{E})$ if
$M:\Omega\times\mathcal{E}\rightarrow[0,\infty]$ is such that
$M_\omega=M(\omega,\cdot)$ is a measure on $\mathcal{E}$ for each
$\omega\in\Omega$ and $M(\cdot,A)=M(A)$ is an extended random variable
for each $A\in\mathcal{E}$. We also define the \emph{mean} of $M$ to be the measure $\mu$ on $(E,\mathcal{E})$ given by
\[
\mu(A)=\mathbb{E}[M(A)],\qquad A\in\mathcal{E}.
\]

We say that $M$ is \emph{locally finite} or \emph{locally bounded} if,
for each bounded Borel set $B\in\mathcal{E}$, we have that $M(B)$ is a
real random variable. In particular, if $M$ is locally finite, then
$M_\omega\in\mathcal{M}_E$, the space of locally bounded measures on
$E$ equipped with the vague topology and its Borel $\sigma$-algebra
$\mathcal B(\mathcal M_E)$. By \cite[Chapter 4]{kallenberg2017random},
$\mathcal M_E$ is Polish, so a locally finite random measure may be
viewed as a random element of
$(\mathcal M_E,\mathcal B(\mathcal M_E))$ in the sense of
\cite[Section 3]{Billing}.

We say that a random element $M$ in
$(\mathcal M_E,\mathcal B(\mathcal M_E))$ is a
\emph{point process} if
\[
M(\omega,B)\in\mathbb{Z}_{\geq 0}
\]
for each $\omega\in\Omega$ and each bounded Borel set
$B\in\mathcal{E}$.

Suppose that $(\bar E,\bar{\mathcal E})$ is a measurable space such that
$E\subseteq\bar E$ and $\mathcal E\subseteq\bar{\mathcal E}$. We say
that a sequence of random elements $\{X_n\}_{n\geq 1}$ in
$(\bar E,\bar{\mathcal E})$, all defined on
$(\Omega,\mathcal{F},\mathbb{P})$, \emph{forms} the random measure $M$
if
\begin{equation}\label{FormX}
M(\omega,A)
=
\sum_{n\geq 1}{\bf 1}\{X_n(\omega)\in A\}
\end{equation}
for each $\omega\in\Omega$ and $A\in\mathcal E$. Unless otherwise
specified, we assume that
\[
\bar E=E\cup\{\partial\},
\qquad
\bar{\mathcal E}=\sigma(\mathcal E,\{\partial\}),
\]
where $\partial$ is an extra point added to $E$. The addition of $\partial$ allows us to represent random measures having a finite (possibly random) number of atoms using an infinite sequence of random elements.

\begin{lemma}\label{Lem.Formed}
Suppose that $M$ is a point process on $(E,\mathcal E)$. There exists a
sequence of random elements $\{X_n\}_{n\geq 1}$ in
$(\bar E,\bar{\mathcal E})$ that satisfies \eqref{FormX}.
\end{lemma}
\begin{proof} This is \cite[Lemma 2.1]{ED24a}.
\end{proof}

Suppose that $M$ is a point process on $(E,\mathcal E)$ and let
$\{X_n\}_{n\geq 1}$ be as in Lemma~\ref{Lem.Formed}. With this data, for any
$n\in\mathbb N$ we define a new point process $M_n$ on
$(E^n,\mathcal E^{\otimes n})$ by
\begin{equation}\label{Eq.RM1}
M_n(\omega,A)
=
\sum_{\substack{i_1,\dots,i_n=1\\
|\{i_1,\dots,i_n\}|=n}}^\infty
{\bf 1}
\bigl\{
(X_{i_1}(\omega),\dots,X_{i_n}(\omega))\in A
\bigr\},
\end{equation}
for $\omega\in\Omega$ and $A\in\mathcal E^{\otimes n}$. We denote the
mean of $M_n$ by $\mu_n$, which is a measure on $(E^n,\mathcal E^{\otimes n})$.

\begin{definition}\label{CorrelationFunctions}
Suppose that $\lambda$ is a locally finite measure on
$(E,\mathcal E)$, called a \emph{reference measure}. Suppose that $M$
is a point process on $(E,\mathcal E)$ and let $\mu_n$ be as above. If $\mu_n$ is absolutely continuous with respect to $\lambda^n$, the
product measure of $n$ copies of $\lambda$ on $E^n$, then we call its
density $\rho_n$ the \emph{$n$-th correlation function of $M$} with
respect to the reference measure $\lambda$. Specifically,
$\rho_n:E^n\rightarrow[0,\infty]$ is
$\mathcal E^{\otimes n}$-measurable and satisfies
\begin{equation}\label{RN1}
\mu_n(A)
=
\int_A
\rho_n(x_1,\dots,x_n)\,\lambda^n(dx)
\end{equation}
for all $A\in\mathcal E^{\otimes n}$.
\end{definition}

\begin{definition}\label{Def.DPP}
Suppose that $M$ is a point process on $(E,\mathcal E)$ and $\lambda$
is a locally finite measure on $(E,\mathcal E)$. Suppose further that
$K:E\times E\rightarrow\mathbb C$ is a locally bounded measurable
function, called the \emph{correlation kernel}. We say that $M$ is a
\emph{determinantal point process} with correlation kernel $K$ and
reference measure $\lambda$ if the $n$-th correlation function
$\rho_n$ as in Definition~\ref{CorrelationFunctions} exists for each
$n\geq 1$ and
\begin{equation}\label{DP1}
\rho_n(x_1,\dots,x_n)
=
\det\left[K(x_i,x_j)\right]_{i,j=1}^n.
\end{equation}
\end{definition}
\begin{remark}\label{Rem.KernelSupport}
In some applications below, we define $K(x,y)$ only for $x,y \in F$ as opposed to $x,y \in E$, where $\mathrm{supp}(\lambda) \subseteq F$. As explained in \cite[Remark~2.12]{ED24a}, the values of $K(x,y)$ for which $x$ or $y$ lies outside the support of $\lambda$ are immaterial. More specifically, if $M$ is as in Definition \ref{Def.DPP} and $\widetilde{K}: E\times E\rightarrow\mathbb C$ is another locally bounded measurable function such that $\widetilde{K}(x,y) = K(x,y)$ for all $x,y \in \mathrm{supp}(\lambda)$, then $M$ is also a determinantal point process with reference measure $\lambda$ and correlation kernel $\widetilde{K}$, as follows from \cite[(2.13) and Lemma 2.8]{ED24a}. For the sake of concreteness, whenever we define $K(x,y)$ only for $x,y \in F$ below, we set $K(x,y)=0$ whenever $\{x,y\} \not \subseteq F$.
\end{remark}

We end this section by introducing some additional notation and making a few observations. If $\nu$ is a measure on $(E,\mathcal{E})$ and $f:E\to[0,\infty)$ is measurable, we write
\begin{equation}\label{Eq.S2Pair}
\langle \nu,f\rangle:=\int_E f(x)\,\nu(dx).
\end{equation}
In addition, if $f:E\to\mathbb{R}$ is measurable and dominated by a $\nu$-integrable nonnegative function, we use the same notation for the resulting well-defined integral.

If $M_n$ and $\mu_n$ are as above and $f:E^n\to[0,\infty)$ is measurable, then, by approximating $f$ from below by simple functions and applying the monotone convergence theorem, $\langle M_n,f\rangle$ is an extended real-valued random variable, and
\begin{equation}\label{Eq.S2PairMean}
\mathbb{E}\!\left[\langle M_n,f\rangle\right]=\langle\mu_n,f\rangle.
\end{equation}
If $f:E^n\to\mathbb{R}$ is measurable and dominated by a $\mu_n$-integrable nonnegative function, then $\langle M_n,f\rangle$ is almost surely finite, integrable, and satisfies \eqref{Eq.S2PairMean}. In particular, if $M$ is a determinantal point process with reference measure $\lambda$ and correlation kernel $K$, then
\begin{equation}\label{Eq.S2PairMeanDPP}
\mathbb{E}\!\left[\langle M_n,f\rangle\right]
=
\int_{E^n}
\det\!\left[K(x_i,x_j)\right]_{i,j=1}^n
f(x)\,\lambda^n(dx).
\end{equation}

%----------------------------------------------------------------------
%
% Section 2.2
%
%----------------------------------------------------------------------
\subsection{Line ensembles}\label{Section2.2}
In this section, we give a brief introduction to line ensembles. Our notation and presentation largely follow \cite[Section~2]{DEA21}, and we refer the interested reader to this paper for additional details.

Given two integers $p \leq q$, we let $\llbracket p, q \rrbracket$ denote the set $\{p, p+1, \dots, q\}$. Given an interval $\Lambda \subseteq \mathbb{R}$ we endow it with the subspace topology of the usual topology on $\mathbb{R}$. We let $(C(\Lambda), \mathcal{C})$ denote the space of continuous functions $f: \Lambda \rightarrow \mathbb{R}$ with the topology of uniform convergence over compact sets, see \cite[Chapter 7, Section 46]{Munkres}, and Borel $\sigma$-algebra $\mathcal{C}$. Given a set $\Sigma \subseteq \mathbb{Z}$ we endow it with the discrete topology and denote by $\Sigma \times \Lambda$ the set of all pairs $(i,x)$ with $i \in \Sigma$ and $x \in \Lambda$ with the product topology. We also denote by $\left(C (\Sigma \times \Lambda), \mathcal{C}_{\Sigma}\right)$ the space of continuous functions on $\Sigma \times \Lambda$ with the topology of uniform convergence over compact sets and Borel $\sigma$-algebra $\mathcal{C}_{\Sigma}$. By
\cite[Lemma~2.2]{DEA21}, this topology is induced by a complete and separable metric. In particular, $C(\Sigma\times\Lambda)$ is a Polish space.

The following defines the notion of a line ensemble.
\begin{definition}\label{Def.LineEnsembles}
Let $\Sigma \subseteq \mathbb{Z}$ and $\Lambda \subseteq \mathbb{R}$ be an interval. A {\em $\Sigma$-indexed line ensemble $\mathcal{L}$} is a random variable defined on a probability space $(\Omega, \mathcal{F}, \mathbb{P})$ that takes values in $\left(C (\Sigma \times \Lambda), \mathcal{C}_{\Sigma}\right)$. Intuitively, $\mathcal{L}$ is a collection of random continuous curves (sometimes referred to as {\em lines}), indexed by $\Sigma$,  each of which maps $\Lambda$ to $\mathbb{R}$. We will often slightly abuse notation and write $\mathcal{L}: \Sigma \times \Lambda \rightarrow \mathbb{R}$, even though it is not $\mathcal{L}$ which is such a function, but $\mathcal{L}(\omega)$ for every $\omega \in \Omega$. For $i \in \Sigma$ we write $\mathcal{L}_i(\omega) = (\mathcal{L}(\omega))(i, \cdot)$ for the curve of index $i$ and note that the latter is a map $\mathcal{L}_i: \Omega \rightarrow C(\Lambda)$, which is $\mathcal{F}/\mathcal{C}$-measurable. 
\end{definition}

As mentioned above, $C(\Sigma \times \Lambda)$ is a Polish space, and so a line ensemble $\mathcal{L}$ is just a random element of $C(\Sigma \times \Lambda)$ in the sense of \cite[Section 3]{Billing}.

We say that a $\Sigma$-indexed line ensemble $\mathcal{L}$ on $\Lambda$ is {\em ordered} if, almost surely, $\mathcal{L}_i(t)\geq\mathcal{L}_j(t)$ for all $i,j\in\Sigma$ with $i<j$ and all $t\in\Lambda$. We say that $\mathcal{L}$ is {\em non-intersecting} if, almost surely, $\mathcal{L}_i(t)>\mathcal{L}_j(t)$ for all $i,j\in\Sigma$ with $i<j$ and all $t\in\Lambda$. 

For $\Sigma'\subseteq\Sigma$ and a subinterval $\Lambda'\subseteq\Lambda$, we write $\mathcal{L}|_{\Sigma'\times \Lambda'}$ for the restriction of $\mathcal{L}$ to $\Sigma'\times \Lambda'$. For $i\in\Sigma$, we similarly write $\mathcal{L}_i|_{\Lambda'}$ for the restriction of $\mathcal{L}_i$ to $\Lambda'$. Notice that $\mathcal{L}|_{\Sigma'\times \Lambda'}$ is a random element of $C(\Sigma'\times \Lambda')$, while $\mathcal{L}_i|_{\Lambda'}$ is a random element of $C(\Lambda')$.

%----------------------------------------------------------------------
%
% Section 2.3
%
%----------------------------------------------------------------------
\subsection{A fourth moment bound}\label{Section2.3} We continue with the notation from Section \ref{Section2.1}. The goal of this section is to derive a fourth moment bound for determinantal point processes on $\mathbb{R}^2$; see Lemma \ref{Lem.FourthMomentBound} below. We first introduce the assumptions under which this bound will be proved. Throughout the section, the point processes under consideration are supported on two vertical slices of $\mathbb{R}^2$, that is, on a set of the form $\{u,v\}\times\mathbb{R}$. As explained in Section \ref{Section2.4}, more general determinantal point processes naturally give rise to point processes of this form by restriction. Consequently, Lemma \ref{Lem.FourthMomentBound} can be applied to a broader class of determinantal point processes on $\mathbb{R}^2$.

Fix $u,v\in\mathbb{R}$ with $u<v$, and let $\lambda$ be a locally
finite measure on $\mathbb{R}$. Suppose that $M$ is a determinantal point process on $\mathbb{R}^2$ with
reference measure $\mu_{\{u,v\}}\times\lambda$ and correlation
kernel $K$, where $\mu_{\{u,v\}}$ is the counting measure from Definition \ref{Def.S1Measures}. For $a\in\{u,v\}$, define the random measure $M^a$ on
$\mathbb R$ by
\begin{equation}\label{Eq.FourthMomentTimeSlices}
M^a(A)=M(\{a\}\times A),
\qquad A\subseteq\mathbb{R}\text{ Borel},
\end{equation}
and write
\begin{equation}\label{Eq.FourthMomentKernelNotation}
K_{aa'}(x,y)=K(a,x;a',y),
\qquad a,a'\in\{u,v\},\quad x,y\in\mathbb{R}.
\end{equation}
Finally, fix a bounded Borel-measurable $f: \mathbb{R} \rightarrow \mathbb{R}$. We make the following assumptions.

\medskip
\noindent
{\bf \raggedleft Assumption A.1.}
\textit{Kernel decomposition and contraction identities.}

There exist locally bounded measurable functions $\phi_{uv}:\mathbb{R}^2\to\mathbb{C}$
and
\begin{equation}\label{Eq.FourthMomentModifiedKernels}
\widetilde K_{aa'}:\mathbb{R}^2\to\mathbb{C},
\qquad a,a'\in\{u,v\},
\end{equation}
such that
\begin{equation}\label{Eq.FourthMomentKernelDecomposition}
K_{aa'}(x,y) = \widetilde K_{aa'}(x,y) - {\bf 1}\{a=v,\ a'=u\}\phi_{uv}(y,x).
\end{equation}
Moreover,
\begin{equation}\label{Eq.FourthMomentContractionUU}
\widetilde K_{uu}(x,y)
=
\int_{\mathbb{R}}
\widetilde K_{uv}(x,z)\phi_{uv}(y,z)\,\lambda(dz),
\end{equation}
\begin{equation}\label{Eq.FourthMomentContractionVV}
\widetilde K_{vv}(x,y)
=
\int_{\mathbb{R}}
\phi_{uv}(z,x)\widetilde K_{uv}(z,y)\,\lambda(dz),
\end{equation}
\begin{equation}\label{Eq.FourthMomentContractionVU}
\widetilde K_{vu}(x,y)=\int_{\mathbb{R}^2}\phi_{uv}(z,x)\widetilde K_{uv}(z,w)\phi_{uv}(y,w)\,\lambda(dz)\lambda(dw),
\end{equation}
for every $x,y\in\mathbb{R}$, and all the integrals in
\eqref{Eq.FourthMomentContractionUU}--\eqref{Eq.FourthMomentContractionVU}
are absolutely convergent.

\medskip
\noindent
{\bf \raggedleft Assumption A.2.}
\textit{Absolute integrability.}

For every $1\leq m\leq4$ and all positive integers
$k_1,\ldots,k_m$ satisfying $k_1+\cdots+k_m=4$, we have
\begin{equation}\label{Eq.FourthMomentAbsoluteIntegrability}
\begin{split}
&\sum_{\sigma\in\mathfrak{S}_m}
\int_{\mathbb{R}^{2m}}
\prod_{i=1}^m
\Bigl[
|\phi_{uv}(z_i,x_i)|
\bigl(|f(z_i)|+|f(x_i)|\bigr)^{k_i}
|\widetilde K_{uv}(z_i,x_{\sigma(i)})|
\Bigr]
\prod_{i=1}^m\lambda(dx_i)\lambda(dz_i)
<\infty,
\end{split}
\end{equation}
where $\mathfrak{S}_m$ denotes the set of permutations of
$\{1,\ldots,m\}$.

\medskip
\noindent
{\bf \raggedleft Assumption A.3.}
\textit{Bounds on the difference kernels.}

For $1\leq k\leq4$, define
\begin{equation}\label{Eq.FourthMomentDifferenceKernels}
L_k(x,y)
=
\int_{\mathbb{R}}
\phi_{uv}(z,x)
\bigl(f(z)-f(x)\bigr)^k
\widetilde K_{uv}(z,y)\,\lambda(dz).
\end{equation}
There exist nonnegative measurable functions $U,V:\mathbb{R}\to[0,\infty)$
and finite nonnegative constants $B_1,\ldots,B_4,S$ such that
\begin{equation}\label{Eq.FourthMomentEnvelopeIntegrability}
\int_{\mathbb{R}}U(x)V(x)\,\lambda(dx)\leq S
\end{equation}
and
\begin{equation}\label{Eq.FourthMomentDifferenceKernelBounds}
|L_k(x,y)|
\leq
B_k(v-u)^{k/2}U(x)V(y)
\end{equation}
for every $x,y\in\mathbb{R}$ and $1\leq k\leq4$.

\medskip
With the above notation in place, we can state the main result of this section.
\begin{lemma}\label{Lem.FourthMomentBound}
Under Assumptions A.1--A.3 above, the following statements hold. 
\begin{itemize}
\item[(a)] We have $\mathbb{E}\left[\langle M^u,|f|\rangle^4\right] < \infty$ and $\mathbb{E}\left[\langle M^v,|f|\rangle^4\right] < \infty$.
In particular, $\langle M^u, f\rangle$ and $\langle M^v,f\rangle$ are well-defined random variables with finite fourth moments.
\item[(b)] We have
\begin{equation}\label{Eq.FourthMomentBound}
\begin{split}
\mathbb{E}\left[\left(
\langle M^u,f\rangle-\langle M^v,f\rangle
\right)^4\right]
\leq{}&(v-u)^2
\Bigl[ B_4S+\bigl(8B_1B_3+6B_2^2\bigr)S^2+36B_1^2B_2S^3+24B_1^4S^4
\Bigr].
\end{split}
\end{equation}
\end{itemize}
\end{lemma}
\begin{proof}
For $r > 0$, put $f_r=f{\bf 1}_{[-r,r]}$.  For either
$(g_1,g_2)=(f_r,-f_r)$ or $(g_1,g_2)=(|f_r|,|f_r|)$, set
\begin{equation}\label{Eq.DefWQ}
 W(g_1,g_2)=\langle M^u,g_1\rangle+\langle M^v,g_2\rangle,
 \hspace{2mm}
 Q_k^{g_1,g_2}(x,y)=\int_{\mathbb R}\phi_{uv}(z,x)
 (g_1(z)+g_2(x))^k\widetilde K_{uv}(z,y)\,\lambda(dz).
\end{equation}
Later in the proof, once the corresponding quantities are shown to be well-defined, we use the same notation with $f_r$ replaced by $f$.

In addition, put $E = \mathbb{R}^2$ and $E_0=\{u,v\}\times\mathbb R$, and let $\{X_n\}_{n \geq 1}$ be a sequence of random elements of
$(\bar E,\bar{\mathcal E})$ that forms $M$ (its existence is ensured by Lemma \ref{Lem.Formed}). We extend any function $g: E \rightarrow \mathbb{C}$ to $\bar{E}$ by setting $g(\partial) = 0$. We also observe from (\ref{Eq.S2PairMean}) that since the reference measure $\nu = \mu_{\{u,v\}}\times\lambda$ is supported on $E_0$, the same is true almost surely for $M$. In particular, if we set $g(a,x)={\bf 1}\{a = u\} g_1(x) + {\bf 1}\{a = v\} g_2(x)$, then 
\begin{equation}\label{Eq.DecomposeMg}
\langle M, g\rangle = \langle M^u, g_1 \rangle + \langle M^v, g_2 \rangle.
\end{equation}

For clarity, we divide the remainder of the proof into four steps. In the first step, all integrals (or sums) are finite because the truncation $ {\bf 1}_{[-r,r]}$ above restricts the integration variables to a compact set, while the relevant kernels are locally bounded and the reference measure is locally finite. When we use the contraction identities in Step 3, Assumptions~A.1 and A.2 guarantee the absolute convergence of the resulting integrals. Consequently, all applications of Fubini's and Tonelli's theorems throughout the proof are justified, and we use them without further comment.

\medskip

{\bf \raggedleft Step 1.} In this step, we express $\mathbb{E}\left[W(g_1,g_2)^4\right]$ in terms of $K$, $\lambda$ and $g$, see (\ref{Eq.FourthMomentExpandedKernels}), using the formulas for determinantal point processes from Section \ref{Section2.1}.

For a partition $\mathcal P=\{\mathcal B_1,\ldots,\mathcal B_\ell\}$ of $\{1,2,3,4\}$, whose blocks are ordered by their least elements, let $\ell = |\mathcal P|$ and set
$$G_{\mathcal P}(p_1,\ldots,p_\ell) = \prod_{i=1}^\ell g(p_i)^{|\mathcal B_i|}.$$
Using (\ref{Eq.DecomposeMg}), rearranging terms, and applying (\ref{Eq.RM1}), we get
\begin{equation}\label{Eq.FourthMomentFactorialExpansion}
\begin{split}
&W(g_1,g_2)^4 = \langle M, g\rangle^4 = \sum_{i_1, i_2, i_3, i_4 \geq 1}g(X_{i_1})g(X_{i_2})g(X_{i_3})g(X_{i_4}) \\
&= \sum_{\mathcal P}\sum_{\substack{i_1, \dots, i_{\ell} \geq 1 \\ |\{i_1, \dots, i_{\ell}\}| = \ell }}G_{\mathcal P}(X_{i_1},\dots, X_{i_\ell}) = \sum_{\mathcal P}\langle M_{\ell},G_{\mathcal P}\rangle,
\end{split}
\end{equation}
where the sum is over partitions $\mathcal{P}$ of $\{1,2,3,4\}$.

Taking expectations on both sides of \eqref{Eq.FourthMomentFactorialExpansion} and applying \eqref{Eq.S2PairMeanDPP} yields
\begin{equation*}
\mathbb{E}\left[W(g_1,g_2)^4\right] =\sum_{\mathcal P}\int_{E_0^{\ell}} G_{\mathcal P}(p_1,\ldots,p_{\ell})
\det[K(p_i,p_j)]_{i,j=1}^{\ell}
\,\nu^{\ell}(dp).
\end{equation*}
Writing $p_i=(a_i,y_i)$, expanding the product measure $\nu^\ell$
over $a_i\in\{u,v\}$, and using the Leibniz formula for the
determinant gives
\begin{equation}\label{Eq.FourthMomentPartitionExpansion}
 \mathbb{E}\left[W(g_1,g_2)^4\right]
 =\sum_{\mathcal P}\sum_{a_1,\ldots,a_\ell\in\{u,v\}}
 \sum_{\pi\in\mathfrak S_\ell}\!\operatorname{sgn}(\pi)
 \int_{\mathbb R^\ell}\prod_{i=1}^\ell
 g(a_i,y_i)^{|\mathcal B_i|}
 K_{a_i a_{\pi(i)}}(y_i,y_{\pi(i)})\prod_{i = 1}^{\ell}\lambda(dy_i).
\end{equation}
We mention that in deriving the above expression we used (\ref{Eq.FourthMomentKernelNotation}).

For fixed $\vec{a} = (a_1,\ldots,a_\ell)\in\{u,v\}^\ell$ and
$\pi\in\mathfrak S_\ell$, set
$$ D(\vec{a},\pi) = \{i\in\{1,\ldots,\ell\}:a_i=v,\ a_{\pi(i)}=u\}. $$
Substituting the following identities
\begin{equation*}
K_{a_i a_{\pi(i)}}(y_i,y_{\pi(i)}) = \widetilde K_{a_i a_{\pi(i)}}(y_i,y_{\pi(i)}) - {\bf 1}\{a_i=v,\ a_{\pi(i)}=u\}\phi_{uv}(y_{\pi(i)},y_i),
\end{equation*}
which follow from (\ref{Eq.FourthMomentKernelDecomposition}), into (\ref{Eq.FourthMomentPartitionExpansion}), we obtain
\begin{equation}\label{Eq.FourthMomentExpandedKernels}
\begin{split}
\mathbb E\left[W(g_1,g_2)^4\right]
={}&
\sum_{\mathcal P}
\sum_{a_1,\ldots,a_\ell\in\{u,v\}}
\sum_{\pi\in\mathfrak S_\ell}
\sum_{I\subseteq D(\vec{a},\pi)}
\operatorname{sgn}(\pi)(-1)^{|I|} \\
&\times
\int_{\mathbb R^\ell}
\prod_{i=1}^\ell g(a_i,y_i)^{|\mathcal B_i|}
\prod_{i\in I}\phi_{uv}(y_{\pi(i)},y_i)
\prod_{i\notin I}
\widetilde K_{a_i a_{\pi(i)}}(y_i,y_{\pi(i)})
\prod_{i=1}^\ell\lambda(dy_i).
\end{split}
\end{equation}

{\bf \raggedleft Step 2.} The goal of this and the next step is to establish the following identity
\begin{equation}\label{Eq.FourthMomentContractedExpansion}
\begin{split}
\mathbb E\left[W(g_1,g_2)^4\right]
={}&
\sum_{\mathcal{C}}
\sum_{J_1\subseteq C_1,\ldots,J_m\subseteq C_m}
\sum_{\sigma\in\mathfrak S_m}
\operatorname{sgn}(\sigma) \\
&\times
\int_{\mathbb R^{2m}}
\prod_{i=1}^m
\phi_{uv}(z_i,x_i)
\widetilde K_{uv}(z_i,x_{\sigma(i)})
g_2(x_i)^{|J_i|}
g_1(z_i)^{|C_i|-|J_i|}
\prod_{i=1}^m\lambda(dx_i)\lambda(dz_i),
\end{split}
\end{equation}
where the sum is over partitions $\mathcal C=\{C_1,\ldots,C_m\}$ of $\{1,2,3,4\}$, ordered by their least elements. 

We prove \eqref{Eq.FourthMomentContractedExpansion} by constructing a bijection between the data indexing the summands on the right-hand sides of \eqref{Eq.FourthMomentExpandedKernels} and \eqref{Eq.FourthMomentContractedExpansion}, and then showing that the summands corresponding under this bijection are equal. In the remainder of this step, we describe the bijective correspondence.

\medskip

It is convenient to reinterpret the indexing data in terms of directed graphs, as we explain next. First, we regard the blocks $\mathcal B_1,\ldots,\mathcal B_\ell$ as the vertices of a directed labeled graph. For every $1\leq i\leq\ell$, we assign the label $a_i$ to the vertex $\mathcal B_i$, and we draw a directed arrow
$$\mathcal B_i[a_i]\longrightarrow \mathcal B_{\pi(i)}[a_{\pi(i)}].$$
Since $\pi$ is a permutation, each vertex in this graph has exactly one incoming and one outgoing arrow; equivalently, the graph is $(1,1)$-regular. Consequently, the data 
$$
\mathcal P=\{\mathcal B_1,\ldots,\mathcal B_\ell\},
\qquad (a_1,\ldots,a_\ell)\in\{u,v\}^{\ell},
\qquad \pi\in\mathfrak S_\ell,
\qquad I\subseteq D(\vec{a},\pi)
$$
correspond to a directed labeled $(1,1)$-regular graph, as well as a distinguished collection of arrows which start from a $v$-labeled vertex and end at a $u$-labeled vertex (namely those indexed by $I$).

Similarly, we regard the blocks $C_1,\ldots, C_m$ as the vertices of an unlabeled directed graph, and for every $1\leq i\leq m$, we draw a directed arrow
$$C_i\longrightarrow C_{\sigma(i)}.$$
Since $\sigma$ is a permutation, the data 
$$
\mathcal C=\{C_1,\ldots,C_m\},
\qquad J_1 \subseteq C_1, \dots, J_m \subseteq C_m,
\qquad \sigma\in\mathfrak S_m,
$$
correspond to a directed $(1,1)$-regular graph, as well as the collection $\{J_i\}_{i = 1}^m$ of subsets $J_i \subseteq C_i$.

The first labeled graph description can be converted to the second as follows. For each $i \in I$, we contract the arrow connecting $\mathcal B_i$ to $\mathcal{B}_{\pi(i)}$ and replace the merged vertices by the vertex $\mathcal B_i \cup \mathcal{B}_{\pi(i)}$. In this way we obtain a new $(1,1)$-regular graph whose vertices are $C_1, \dots, C_m$, the blocks of a new partition $\mathcal{C}$, ordered by their least elements. We construct the subsets $J_i \subseteq C_i$ as follows.
\begin{itemize}
\item If $C_i = \mathcal B_j \cup \mathcal{B}_{\pi(j)}$ for some $j \in I$, we set $J_i = \mathcal{B}_j$. 
\item If $C_i = \mathcal{B}_j$ for some $j\in \{1, \dots, \ell\}$ and $a_j = v$, we set $J_i = C_i$.
\item If $C_i = \mathcal{B}_j$ for some $j\in \{1, \dots, \ell\}$ and $a_j = u$, we set $J_i = \emptyset$.
\end{itemize}

For the reverse direction, we proceed as follows. If $\emptyset\neq J_i\neq C_i$, we split the vertex $C_i$ into two labeled vertices, joined by the directed arrow
$$J_i[v]\longrightarrow (C_i\setminus J_i)[u].$$
The incoming arrow to $C_i$ is redirected to $J_i[v]$, while the outgoing arrow from $C_i$ is redirected to start at $(C_i\setminus J_i)[u]$. If $J_i=C_i$ or $J_i=\emptyset$, we leave $C_i$ unchanged and assign it the label $v$ or $u$, respectively. Performing these operations for all $1\leq i\leq m$, we obtain a new labeled $(1,1)$-regular directed graph whose vertices form the blocks $\mathcal B_1,\ldots,\mathcal B_\ell$ of a partition $\mathcal P$, ordered by their least elements. The set $I$ consists of those indices $j \in \{1, \dots, \ell\}$ such
that $\mathcal B_j=J_i$ for some $i\in\{1,\dots,m\}$ with
$\emptyset\neq J_i\neq C_i$. By construction, every arrow indexed by $I$ goes from a $v$-labeled vertex to a $u$-labeled vertex, so $I\subseteq D(\vec{a},\pi)$.

One readily verifies that the transformations described in the last two paragraphs are inverses of each other, which establishes the bijective correspondence of the indexing data on the right-hand sides of \eqref{Eq.FourthMomentExpandedKernels} and \eqref{Eq.FourthMomentContractedExpansion}.

\medskip

{\bf \raggedleft Step 3.} In this step, we show that terms in (\ref{Eq.FourthMomentExpandedKernels}) and (\ref{Eq.FourthMomentContractedExpansion}) with corresponding indexing data under the bijection from Step 2 are equal, thereby establishing (\ref{Eq.FourthMomentContractedExpansion}).

Fix corresponding data
$$
(\mathcal P,(a_j)_{j=1}^{\ell},\pi,I)
\qquad\text{and}\qquad
(\mathcal C,(J_i)_{i=1}^m,\sigma).
$$
Since precisely the arrows indexed by $I$ are contracted, we have $m=\ell-|I|$. Moreover, contracting an arrow does not change the number of directed cycles. Thus, if $c$ is the common number of cycles of $\pi$ and $\sigma$, then
\begin{equation*}
\operatorname{sgn}(\pi)(-1)^{|I|} =(-1)^{\ell-c+|I|}=(-1)^{m-c}=\operatorname{sgn}(\sigma).
\end{equation*}
In view of the last identity, what remains is to show
\begin{equation}\label{Eq.FourthMomentRED1}
\begin{split}
&\int_{\mathbb R^\ell}
\prod_{j=1}^\ell g(a_j,y_j)^{|\mathcal B_j|}
\prod_{j\in I}\phi_{uv}(y_{\pi(j)},y_j)
\prod_{j\notin I}
\widetilde K_{a_j a_{\pi(j)}}(y_j,y_{\pi(j)})
\prod_{j=1}^\ell\lambda(dy_j) \\
= &\int_{\mathbb R^{2m}}
\prod_{i=1}^m
\phi_{uv}(z_i,x_i)
\widetilde K_{uv}(z_i,x_{\sigma(i)})
g_2(x_i)^{|J_i|}
g_1(z_i)^{|C_i|-|J_i|}
\prod_{i=1}^m\lambda(dx_i)\lambda(dz_i).
\end{split}
\end{equation}

For each $j\notin I$, we use
(\ref{Eq.FourthMomentContractionUU})--
(\ref{Eq.FourthMomentContractionVU}) to write
\begin{equation}\label{Eq.FourthMomentEdgeExpansions}
\begin{array}{rcl}
\widetilde K_{uv}(y_j,y_{\pi(j)})
&=&
\widetilde K_{uv}(y_j,y_{\pi(j)}),
\\[2mm]
\widetilde K_{uu}(y_j,y_{\pi(j)})
&=&
\displaystyle
\int_{\mathbb R}
\widetilde K_{uv}(y_j,\xi_j)
\phi_{uv}(y_{\pi(j)},\xi_j)\,\lambda(d\xi_j),
\\[4mm]
\widetilde K_{vv}(y_j,y_{\pi(j)})
&=&
\displaystyle
\int_{\mathbb R}
\phi_{uv}(\zeta_j,y_j)
\widetilde K_{uv}(\zeta_j,y_{\pi(j)})\,\lambda(d\zeta_j),
\\[4mm]
\widetilde K_{vu}(y_j,y_{\pi(j)})
&=&
\displaystyle
\int_{\mathbb R^2}
\phi_{uv}(\zeta_j,y_j)
\widetilde K_{uv}(\zeta_j,\xi_j)
\phi_{uv}(y_{\pi(j)},\xi_j)
\,\lambda(d\zeta_j)\lambda(d\xi_j).
\end{array}
\end{equation}

Define $I_v=\{j\notin I:a_j=v\}$, $I_u=\{j\notin I:a_{\pi(j)}=u\}$. Thus, the variable $\zeta_j$ is introduced precisely for
$j\in I_v$, while $\xi_j$ is introduced precisely for
$j\in I_u$. For $j\notin I$, set
$$
r_j=
\begin{cases}
y_j,&a_j=u,\\
\zeta_j,&a_j=v,
\end{cases}
\qquad
s_j=
\begin{cases}
y_{\pi(j)},&a_{\pi(j)}=v,\\
\xi_j,&a_{\pi(j)}=u.
\end{cases}
$$
Substituting (\ref{Eq.FourthMomentEdgeExpansions}) into the left-hand side of (\ref{Eq.FourthMomentRED1}) and relabeling variables gives
\begin{equation}\label{Eq.FourthMomentExpandedAuxiliary}
\begin{split}
&\mbox{[LHS of (\ref{Eq.FourthMomentRED1})]} = \int_{\mathbb{R}^{2m}}
\prod_{j=1}^{\ell}
g(a_j,y_j)^{|\mathcal B_j|}
\prod_{j\in I}
\phi_{uv}(y_{\pi(j)},y_j) \prod_{j\in I_v}
\phi_{uv}(\zeta_j,y_j)\\
&\qquad\times \prod_{j\in I_u}
\phi_{uv}(y_{\pi(j)},\xi_j)
\prod_{j\notin I}
\widetilde K_{uv}(r_j,s_j)
\prod_{j=1}^{\ell}\lambda(dy_j)
\prod_{j\in I_v}\lambda(d\zeta_j)
\prod_{j\in I_u}\lambda(d\xi_j),
\end{split}
\end{equation}
where we used that the dimension of the integral is 
\begin{equation*}
\begin{split}
\ell + |I_v| + |I_{u}| = &\ell + |\{j \in \{ 1, \dots, \ell\}: a_j = v \}| - |I| + |\{j \in \{ 1, \dots, \ell\}: a_{\pi(j)} = u\}| - |I| \\
= & 2\ell - 2|I| = 2m.
\end{split}
\end{equation*}
The last equality used our earlier identity $m = \ell - |I|$.

We relabel the variables using the bijection between indexing data. For each $1\leq i\leq m$, define
\begin{equation*}
(x_i,z_i)
=
\begin{cases}
\bigl(y_j,y_{\pi(j)}\bigr) = \bigl(s_{\pi^{-1}(j)},r_{\pi(j)}\bigr),
&
\mbox{if }C_i=\mathcal B_j\cup\mathcal B_{\pi(j)}
\text{ for some }j\in I,
\\[1mm]
\bigl(y_j,\zeta_j\bigr) = \bigl(s_{\pi^{-1}(j)},r_j\bigr),
&
\mbox{if }C_i=\mathcal B_j\text{ and }a_j=v,
\\[1mm]
\bigl(\xi_{\pi^{-1}(j)},y_j\bigr) = \bigl(s_{\pi^{-1}(j)},r_j\bigr),
&
\mbox{if }C_i=\mathcal B_j\text{ and }a_j=u.
\end{cases}
\end{equation*}
Under this relabeling, and recalling that $g(a,x)={\bf 1}\{a = u\} g_1(x) + {\bf 1}\{a = v\} g_2(x)$, we get
\begin{equation}\label{Eq.FourthMomentRelabeledFactors}
\begin{split}
&
\prod_{j\in I}\phi_{uv}(y_{\pi(j)},y_j)
\prod_{j\in I_v}\phi_{uv}(\zeta_j,y_j)
\prod_{j\in I_u}\phi_{uv}(y_{\pi(j)},\xi_j)
=
\prod_{i=1}^m\phi_{uv}(z_i,x_i),
\\
&\prod_{j\notin I}\widetilde K_{uv}(r_j,s_j)
=
\prod_{i=1}^m
\widetilde K_{uv}(z_i,x_{\sigma(i)}), \quad \prod_{j=1}^{\ell}g(a_j,y_j)^{|\mathcal B_j|}
= \prod_{i=1}^m
g_2(x_i)^{|J_i|}
g_1(z_i)^{|C_i|-|J_i|}, \\
&\prod_{j=1}^{\ell}\lambda(dy_j)
\prod_{j\in I_v}\lambda(d\zeta_j)
\prod_{j\in I_u}\lambda(d\xi_j)
= \prod_{i=1}^m\lambda(dx_i)\lambda(dz_i).
\end{split}
\end{equation}
Substituting (\ref{Eq.FourthMomentRelabeledFactors}) into (\ref{Eq.FourthMomentExpandedAuxiliary}), we obtain (\ref{Eq.FourthMomentRED1}).

\medskip

{\bf \raggedleft Step 4.} In this step, we use
(\ref{Eq.FourthMomentContractedExpansion}) together with
Assumptions~A.2 and A.3 to complete the proof of the lemma.

Summing (\ref{Eq.FourthMomentContractedExpansion}) over the subsets $J_i$ and using the following trivial identity
$$
\sum_{J_i\subseteq C_i}
g_2(x_i)^{|J_i|}
g_1(z_i)^{|C_i|-|J_i|}
=
\bigl(g_1(z_i)+g_2(x_i)\bigr)^{|C_i|},
$$
gives
\begin{equation*}
\mathbb E\left[W(g_1,g_2)^4\right]
=\sum_{\mathcal{C}}
\hspace{-1mm}\sum_{\sigma\in\mathfrak S_m}
\hspace{-1mm}\operatorname{sgn}(\sigma)\hspace{-1mm}
\int_{\mathbb R^{2m}}
\prod_{i=1}^m
\phi_{uv}(z_i,x_i)
\bigl(g_1(z_i)+g_2(x_i)\bigr)^{|C_i|}
\widetilde K_{uv}(z_i,x_{\sigma(i)})
\prod_{i=1}^m\lambda(dx_i)\lambda(dz_i).
\end{equation*}
Integrating in $z_1,\ldots,z_m$ and using the definition of $Q_k^{g_1,g_2}$ from (\ref{Eq.DefWQ}) gives
$$
\mathbb E\left[W(g_1,g_2)^4\right]
=\sum_{\mathcal{C}}
\sum_{\sigma\in\mathfrak S_m}
\operatorname{sgn}(\sigma)
\int_{\mathbb R^m}
\prod_{i=1}^m
Q_{|C_i|}^{g_1,g_2}(x_i,x_{\sigma(i)}) \prod_{i=1}^m\lambda(dx_i).$$
Finally, by the Leibniz formula for the determinant, this becomes
\begin{equation}\label{Eq.FourthMomentQExpansion}
\mathbb E\left[W(g_1,g_2)^4\right]
=\sum_{\mathcal{C}}
\int_{\mathbb R^m}
\det\left[
Q_{|C_i|}^{g_1,g_2}(x_i,x_j)
\right]_{i,j=1}^m
\prod_{i=1}^m\lambda(dx_i).
\end{equation}

We next replace the sum over partitions $\mathcal C$ by a sum
over their block sizes. The integral in
\eqref{Eq.FourthMomentQExpansion} is unchanged if the blocks
$C_1,\ldots,C_m$ are simultaneously relabeled. Hence we may first order the blocks arbitrarily and then divide by $m!$. For fixed
positive integers $k_1,\ldots,k_m$ satisfying
$$
k_1+\cdots+k_m=4,
$$
the number of ordered decompositions of $\{1,2,3,4\}$ into blocks
of sizes $k_1,\ldots,k_m$ is
$$
\frac{4!}{k_1!\cdots k_m!}.
$$
Consequently,
\begin{equation}\label{Eq.PairingIdentity}
\mathbb E\left[W(g_1,g_2)^4\right]
=
\sum_{m=1}^4\frac{1}{m!}
\sum_{\substack{k_1,\ldots,k_m\geq1\\
k_1+\cdots+k_m=4}}
\frac{4!}{k_1!\cdots k_m!}
\int_{\mathbb R^m}
\det\left[
Q_{k_i}^{g_1,g_2}(x_i,x_j)
\right]_{i,j=1}^m
\prod_{i=1}^m\lambda(dx_i).
\end{equation}

For $g_1=g_2=|f_r|$, we may expand the determinants in
\eqref{Eq.PairingIdentity}, substitute the definition of $Q_k^{g_1,g_2}$ from (\ref{Eq.DefWQ}), apply the inequality $|f_r(z)|+|f_r(x)|\le |f(z)|+|f(x)|$, and use \eqref{Eq.FourthMomentAbsoluteIntegrability} to conclude that for some constant $C \in (0, \infty)$ and all $r > 0 $
\begin{equation}\label{Eq.UniformIntegralBoundFR}
\sum_{m=1}^4\frac{1}{m!}
\sum_{\substack{k_1,\ldots,k_m\geq1\\
k_1+\cdots+k_m=4}}
\frac{4!}{k_1!\cdots k_m!}
\int_{\mathbb R^m}
\left|\det\left[ Q_{k_i}^{g_1,g_2}(x_i,x_j)\right]_{i,j=1}^m
\right|\prod_{i=1}^m\lambda(dx_i) \leq C.
\end{equation}
Using the definition of $W(g_1,g_2)$ from (\ref{Eq.DefWQ}), applying the monotone convergence theorem to the left-hand side of (\ref{Eq.PairingIdentity}) with $g_1=g_2=|f_r|$, and using (\ref{Eq.UniformIntegralBoundFR}) gives
\begin{equation}\label{Eq.UnsignedFourthMomentFinite}
 \mathbb E\left[(\langle M^u,|f|\rangle+\langle M^v,|f|\rangle)^4 \right] = \mathbb E\left[W(|f|,|f|)^4\right] = \lim_{r \rightarrow \infty} \mathbb E\left[W(|f_r|,|f_r|)^4\right] \leq C,
\end{equation}
which verifies part (a) of the lemma.

We now substitute $(g_1,g_2)=(f_r,-f_r)$ in
\eqref{Eq.PairingIdentity} and send $r\to\infty$, obtaining 
\begin{equation}\label{Eq.FinalDifferenceV1}
\mathbb E\left[W(f,-f)^4\right]
=
\sum_{m=1}^4\frac{1}{m!}
\sum_{\substack{k_1,\ldots,k_m\geq1\\
k_1+\cdots+k_m=4}}
\frac{4!}{k_1!\cdots k_m!}
\int_{\mathbb R^m}
\det\left[
Q_{k_i}^{f,-f}(x_i,x_j)
\right]_{i,j=1}^m
\prod_{i=1}^m\lambda(dx_i).
\end{equation}
Let us briefly justify the limits on both sides by dominated convergence. On the left,
$$
|W(f_r,-f_r)|^4
\leq
(\langle M^u,|f_r|\rangle+\langle M^v,|f_r|\rangle)^4
\leq W(|f|,|f|)^4,
$$
and the latter is integrable by \eqref{Eq.UnsignedFourthMomentFinite}, while $W(f_r,-f_r)\to W(f,-f)$ almost surely.

For the right-hand side, fix $m$, $k_1,\ldots,k_m$, and
$\sigma\in\mathfrak S_m$. After expanding the determinant and
substituting the definition of $Q_k^{f_r,-f_r}$ from (\ref{Eq.DefWQ}), the corresponding
integrand on $\mathbb R^{2m}$ is
$$
\prod_{i=1}^m
\phi_{uv}(z_i,x_i)
\bigl(f_r(z_i)-f_r(x_i)\bigr)^{k_i}
\widetilde K_{uv}(z_i,x_{\sigma(i)}).
$$
Since $|f_r(z_i)-f_r(x_i)|\leq |f(z_i)|+|f(x_i)|$, its absolute value is bounded by
$$
\prod_{i=1}^m
|\phi_{uv}(z_i,x_i)|
\bigl(|f(z_i)|+|f(x_i)|\bigr)^{k_i}
|\widetilde K_{uv}(z_i,x_{\sigma(i)})|,
$$
which is integrable by \eqref{Eq.FourthMomentAbsoluteIntegrability}.

By \eqref{Eq.FourthMomentDifferenceKernels} and \eqref{Eq.DefWQ}, we have $Q_k^{f,-f}(x,y)=L_k(x,y)$ for $1 \leq k \leq 4$. Combining the latter with \eqref{Eq.FourthMomentEnvelopeIntegrability} and \eqref{Eq.FourthMomentDifferenceKernelBounds}, we conclude that for any $m\in\{1,\ldots,4\}$ and positive integers $k_1,\ldots,k_m$
satisfying $k_1+\cdots+k_m=4$, we have
\begin{align*}
&\left|
\int_{\mathbb R^m}
\det\left[Q^{f,-f}_{k_i}(x_i,x_j)\right]_{i,j=1}^m
\prod_{i=1}^m\lambda(dx_i)
\right|  \leq
\sum_{\sigma\in\mathfrak S_m}
\int_{\mathbb R^m}
\prod_{i=1}^m
\left|L_{k_i}(x_i,x_{\sigma(i)})\right|
\prod_{i=1}^m\lambda(dx_i)\\
&\qquad\leq
(v-u)^{(k_1+\cdots+k_m)/2}
\prod_{i=1}^m B_{k_i}
\sum_{\sigma\in\mathfrak S_m}
\int_{\mathbb R^m}
\prod_{i=1}^m U(x_i)V(x_{\sigma(i)})
\prod_{i=1}^m\lambda(dx_i) \leq m!S^m(v-u)^2\prod_{i=1}^m B_{k_i}.
\end{align*}
Substituting the latter bound into (\ref{Eq.FinalDifferenceV1}) and using that $W(f,-f) = \langle M^u,f\rangle-\langle M^v,f\rangle$ by (\ref{Eq.DefWQ}), we obtain
\begin{equation*}
\mathbb E\left[(\langle M^u,f\rangle-\langle M^v,f\rangle)^4\right]
\leq (v-u)^2 \sum_{m=1}^4
\sum_{\substack{k_1,\ldots,k_m\geq1\\
k_1+\cdots+k_m=4}} \frac{4!}{k_1!\cdots k_m!} S^m\prod_{i=1}^m B_{k_i}.
\end{equation*}
Evaluating the finite sum on the right-hand side gives \eqref{Eq.FourthMomentBound}.
\end{proof}

%----------------------------------------------------------------------
%
% Section 2.4
%
%----------------------------------------------------------------------
\subsection{A tightness criterion}\label{Section2.4} We continue with the notation from Sections \ref{Section2.1} and \ref{Section2.2}. In this section, we consider a sequence of random arrays $X^N$ that give rise, on the one hand, to a sequence of point processes $M^N$ and, on the other hand, to a sequence of line ensembles $\mathcal{L}^N$. Our goal is to give general conditions on $X^N$ that ensure that the corresponding sequence of line ensembles $\mathcal{L}^N$ is tight. The precise result is stated in Lemma \ref{Lem.DPPTightnessCriterion} below. Our main assumption is that each $M^N$ is a determinantal point process on $\mathbb{R}^2$ whose restrictions to pairs of vertical slices satisfy the assumptions of Section \ref{Section2.3}. We now describe the setup precisely.

\medskip
Fix a compact interval $[\alpha,\beta]\subset\mathbb{R}$ with $\alpha<\beta$, and let $\{d_N\}_{N\geq 1}$ be a sequence of positive integers such that $d_N\to\infty$. For each $N$, define the uniform time mesh
$$
\mathcal{T}_N=\{t_0^N,\ldots,t_{d_N}^N\}, \qquad t_j^N=\alpha+\frac{j}{d_N}(\beta-\alpha), \qquad 0\leq j\leq d_N.
$$
Let $\mu_{\mathcal{T}_N}$ denote the counting measure from Definition~\ref{Def.S1Measures}.

We fix a probability space $(\Omega, \mathcal{F}, \mathbb{P})$. For each $N$, let
$$
X^N=\left(X_i^{j,N}:i\geq 1,\ 0\leq j\leq d_N\right)
$$
be an array of real-valued random variables such that, almost surely,
\begin{equation}\label{Eq.S24Decreasing}
X_1^{j,N}>X_2^{j,N}>\cdots, \qquad \lim_{i\to\infty}X_i^{j,N}=-\infty, \qquad 0\leq j\leq d_N.
\end{equation}
Define the point process $M^N$ on $\mathbb{R}^2$ by
\begin{equation}\label{Eq.S24RandomMeasure}
M^N(A)=\sum_{j=0}^{d_N}\sum_{i\geq 1}{\bf 1}\bigl\{(t_j^N,X_i^{j,N})\in A\bigr\}, \qquad A\subseteq\mathbb{R}^2\text{ Borel}.
\end{equation}

For each $N$, define the $\mathbb{N}$-indexed line ensemble $\mathcal{L}^N=\{\mathcal{L}_i^N\}_{i\geq 1}$ on $[\alpha,\beta]$ by linearly interpolating the points $\{(t_j^N,X_i^{j,N})\}_{j=0}^{d_N}$. More precisely, for $i\geq 1$, $0\leq j\leq d_N-1$, and $t\in[t_j^N,t_{j+1}^N]$, set
$$
\mathcal{L}_i^N(t)=\frac{t_{j+1}^N-t}{t_{j+1}^N-t_j^N}X_i^{j,N}+\frac{t-t_j^N}{t_{j+1}^N-t_j^N}X_i^{j+1,N}.
$$
We mention that $\mathcal{L}^N$ is indeed a line ensemble, since it is obtained by composing the measurable map $\omega\mapsto X^N(\omega)$ from $(\Omega,\mathcal{F})$ to $\mathbb{R}^{\infty}$, see \cite[Example 1.2]{Billing}, with a continuous map from $\mathbb{R}^{\infty}$ to $C(\mathbb{N}\times[\alpha,\beta])$.

We assume that $\{M^N\}_{N\geq 1}$ is a sequence of determinantal point processes on $\mathbb{R}^2$ with correlation kernels $K^N(s,x;t,y)$ and reference measures $\mu_{\mathcal{T}_N}\times\lambda_N$, where each $\lambda_N$ is a locally finite Borel measure on $\mathbb{R}$. If $u=t_a^N$, $v=t_b^N\in\mathcal{T}_N$, and $u<v$, we define the restriction of $M^N$ to $\{u,v\}\times\mathbb{R}$ by
$$
M^N\big|_{\{u,v\}\times\mathbb{R}}(A)
=
\sum_{i\geq 1}\left({\bf 1}\bigl\{(t_a^N,X_i^{a,N})\in A\bigr\}+{\bf 1}\bigl\{(t_b^N,X_i^{b,N})\in A\bigr\}\right),
\qquad A\subseteq\mathbb{R}^2\text{ Borel}.
$$
It follows from Definition \ref{Def.DPP} and \cite[(2.13) and Lemma 2.8]{ED24a} that $M^N\big|_{\{u,v\}\times\mathbb{R}}$ is a determinantal point process on $\mathbb{R}^2$ with correlation kernel $K^N(s,x;t,y)$ and reference measure $\mu_{\{u,v\}}\times\lambda_N$.

We write $\operatorname{Lip}_b(\mathbb{R})$ for the space of bounded Lipschitz functions from $\mathbb{R}$ to $\mathbb{R}$, and set
\begin{equation}\label{Eq.BoundedLipschitzLeftVanishing}
\operatorname{Lip}_{b,-}(\mathbb{R})
=
\left\{f\in\operatorname{Lip}_b(\mathbb{R}):f|_{(-\infty,a]}=0\text{ for some }a\in\mathbb{R}\right\}.
\end{equation}
In words, $\operatorname{Lip}_{b,-}(\mathbb{R})$ is the space of bounded Lipschitz functions that vanish on a neighborhood of $-\infty$. For $f\in\operatorname{Lip}_{b,-}(\mathbb{R})$ and $t=t_j^N\in\mathcal{T}_N$, define the linear statistic
\begin{equation}\label{Eq.LinearStatistic}
H_N(f,t)=\left\langle M^{N,t},f\right\rangle=\sum_{i\geq 1}f(X_i^{j,N}), \mbox{ where }
M^{N,t}(A)=\sum_{i\geq 1}{\bf 1}\{X_i^{j,N}\in A\},
\hspace{2mm}
A\subseteq\mathbb{R}\text{ Borel}.
\end{equation}

We next describe the assumptions imposed on the above setup.

\medskip
\noindent
{\bf \raggedleft Assumption B.1.}
\textit{Two-time fourth-moment hypotheses.}

There exists $N_0\in\mathbb{N}$ such that the following holds. For every $N\geq N_0$ and every $u,v\in\mathcal{T}_N$ with $u<v$, there exist measurable functions $\phi_{uv}^N:\mathbb{R}^2\to \mathbb{C}$ and $\widetilde K_{aa'}^N:\mathbb{R}^2\to\mathbb{C}$, $a,a'\in\{u,v\}$, such that, upon making the identifications
\begin{equation}\label{Eq.TightnessFourthMomentIdentification}
\begin{split}
M&=M^N\big|_{\{u,v\}\times\mathbb{R}},
\qquad \lambda=\lambda_N, \qquad \widetilde K_{aa'}(x,y)=\widetilde K_{aa'}^N(x,y),\\
K_{aa'}(x,y)&=K^N(a,x;a',y), \quad \phi_{uv}(z,x)=\phi_{uv}^N(z,x),
\end{split}
\end{equation}
Assumption A.1 holds, and Assumption A.2 holds for every $f\in\operatorname{Lip}_{b,-}(\mathbb{R})$.

\medskip
\noindent
{\bf \raggedleft Assumption B.2.}
\textit{Uniform bounds on the difference kernels.}

For every $f\in\operatorname{Lip}_{b,-}(\mathbb{R})$, there exist finite nonnegative constants
\begin{equation}\label{Eq.DifferenceKernelConstants}
B_1(f),\ldots,B_4(f),S(f)
\end{equation}
that do not depend on $N$, $u$, or $v$, such that the following holds. For every $N\geq N_0$ and every $u,v\in\mathcal{T}_N$ with $u<v$, there exist nonnegative measurable functions $U_{u,v}^{N,f}$ and $V_{u,v}^{N,f}$ for which Assumption A.3 holds after making the identifications \eqref{Eq.TightnessFourthMomentIdentification} together with the additional identifications
\begin{equation}\label{Eq.TightnessFourthMomentBoundsIdentification}
\begin{split}
L_k&=L_{k;u,v}^{N,f},\qquad
U=U_{u,v}^{N,f},\qquad V=V_{u,v}^{N,f},\\
B_k&=B_k(f),\qquad S=S(f),
\qquad 1\leq k\leq4.
\end{split}
\end{equation}
Here $L_{k;u,v}^{N,f}$ denotes the function defined by \eqref{Eq.FourthMomentDifferenceKernels} under the identifications \eqref{Eq.TightnessFourthMomentIdentification}.

For $f\in\operatorname{Lip}_{b,-}(\mathbb{R})$, define
\begin{equation}\label{Eq.FourthMomentConstant}
\begin{split}
\mathfrak{C}(f)
={}&
1+B_4(f)S(f)
+\bigl(8B_1(f)B_3(f)+6B_2(f)^2\bigr)S(f)^2\\
&+36B_1(f)^2B_2(f)S(f)^3
+24B_1(f)^4S(f)^4.
\end{split}
\end{equation}

\medskip
\noindent
{\bf \raggedleft Assumption B.3.}
\textit{One-time tail bounds compatible with the linear-statistic
estimates.}

Let
\begin{equation}\label{Eq.CutoffFunction}
\chi(x)=\min\{1,\max\{0,x\}\},
\end{equation}
and, for $\varsigma\geq0$, define
\begin{equation}\label{Eq.TranslatedCutoffFunctions}
g_\varsigma^+(x)=\chi(x-\varsigma),
\qquad
g_\varsigma^-(x)=\chi(x+\varsigma+1).
\end{equation}
Set
\begin{equation}\label{Eq.UpperTailProbability}
p_+(\varsigma) =\limsup_{N\to\infty} \max_{0\leq j\leq d_N} \mathbb{P}\bigl(X_1^{j,N}>\varsigma\bigr)
\end{equation}
and, for each $k\geq1$, set
\begin{equation}\label{Eq.LowerTailProbability}
p_{k,-}(\varsigma) = \limsup_{N\to\infty} \max_{0\leq j\leq d_N} \mathbb{P}\bigl(X_k^{j,N}<-\varsigma\bigr).
\end{equation}
We assume that
\begin{equation}\label{Eq.UpperTailCompatibility}
p_+(\varsigma)\mathfrak{C}(g_\varsigma^+) \rightarrow 0 \qquad\text{as }\varsigma\to\infty
\end{equation}
and, for every fixed $k\geq1$,
\begin{equation}\label{Eq.LowerTailCompatibility}
p_{k,-}(\varsigma)\mathfrak{C}(g_\varsigma^-) \rightarrow 0 \qquad\text{as }\varsigma\to\infty.
\end{equation}

\medskip
With the above notation in place, we can state the main result of this section.
\begin{lemma}\label{Lem.DPPTightnessCriterion}
Under Assumptions B.1--B.3 above, the sequence $\{\mathcal{L}^N\}_{N\geq 1}$ is tight in $C(\mathbb{N}\times[\alpha,\beta])$.
\end{lemma}
\begin{proof}
The proof we present is similar to that of \cite[Lemmas~5.2 and~5.3]{Kurt03}, which is itself based on ideas from \cite[Section~10]{Billing}.

By Assumptions B.1 and B.2, the conditions of Lemma~\ref{Lem.FourthMomentBound} are satisfied by $M^N\big|_{\{u,v\}\times\mathbb R}$ for every $N\geq N_0$, every $u,v\in\mathcal T_N$ with $u<v$, and every $f\in\operatorname{Lip}_{b,-}(\mathbb R)$. Consequently, (\ref{Eq.FourthMomentBound}), the definition of $H_N$ in (\ref{Eq.LinearStatistic}), and (\ref{Eq.FourthMomentConstant}) imply that, for every $f\in\operatorname{Lip}_{b,-}(\mathbb R)$, $N\geq N_0$, and $u,v\in\mathcal T_N$,
\begin{equation}\label{Eq.FourthMomentBoundTightnessProof}
\mathbb E\left[(H_N(f,v)-H_N(f,u))^4\right]
\leq \mathfrak C(f)(v-u)^2.
\end{equation}

Equation~(\ref{Eq.FourthMomentBoundTightnessProof}) is the only consequence of Assumptions B.1 and B.2 that we use, and hence the only place where Lemma~\ref{Lem.FourthMomentBound} enters the proof. The remainder of the argument relies only on (\ref{Eq.FourthMomentBoundTightnessProof}) and Assumption B.3 and is organized as follows. First, we state two claims that form the main ingredients of the proof. We then show how these claims imply the statement of the lemma. Finally, we prove each of the two claims. In what follows, we always assume that $N\geq N_0$ and $f\in\operatorname{Lip}_{b,-}(\mathbb R)$, without further mention.

\medskip
\noindent
{\bf Claim 1.}
For every fixed $k\geq1$,
\begin{equation}\label{Eq.MeshCompactContainmentTightnessProof}
\lim_{A\to\infty}\limsup_{N\to\infty}
\mathbb P\!\left(
\max_{0\leq j\leq d_N}|X_k^{j,N}|>A
\right)=0.
\end{equation}

\medskip
\noindent
{\bf Claim 2.}
Define
$$
w_{k,N}^{\rm disc}(a)
=
\max_{0\leq i,j\leq d_N}
\left\{
|X_k^{i,N}-X_k^{j,N}|:
|t_i^N-t_j^N|\leq a
\right\}.
$$
Then, for every fixed $k\geq1$ and $\eta>0$,
\begin{equation}\label{Eq.DiscreteModulusLimitTightnessProof}
\lim_{a\downarrow0}\limsup_{N\to\infty}
\mathbb P\!\left( w_{k,N}^{\rm disc}(a)>\eta \right)=0.
\end{equation}

We next show how Claims~1 and 2 imply the statement of the lemma. We apply the tightness criterion of \cite[Lemma~2.4]{DEA21} with
$$
\Sigma=\mathbb N,\qquad
\Lambda=[\alpha,\beta],\qquad
a_m=\alpha,\quad b_m=\beta\quad\text{for every }m\geq1.
$$

For $z\in C([\alpha,\beta])$, define its modulus of continuity by
$$
w(z,a)=
\sup_{\substack{s,t\in[\alpha,\beta]\\|s-t|\leq a}}
|z(s)-z(t)|.
$$
By the definition of $\mathcal L_k^N$ through linear interpolation and the definition of $w_{k,N}^{\rm disc}$,
\begin{equation}\label{Eq.LEXCorrespondence}
\|\mathcal L_k^N\|_\infty
=
\max_{0\leq j\leq d_N}|X_k^{j,N}|,
\qquad
w(\mathcal L_k^N,a)
\leq
w_{k,N}^{\rm disc}\left(a+2(\beta-\alpha)/d_N\right).
\end{equation}

By (\ref{Eq.MeshCompactContainmentTightnessProof}) and the first equality in (\ref{Eq.LEXCorrespondence}), condition {\rm (i)} of \cite[Lemma~2.4]{DEA21} holds with $a_0=\alpha$. Moreover, (\ref{Eq.DiscreteModulusLimitTightnessProof}), the second inequality in (\ref{Eq.LEXCorrespondence}), and the fact that $d_N \rightarrow \infty$ imply condition {\rm (ii)} of that lemma. This proves the desired tightness.

It remains to prove Claims~1 and 2.

\medskip
\noindent
{\bf Proof of Claim 1.} We begin by introducing some notation and deriving a key consequence of (\ref{Eq.FourthMomentBoundTightnessProof}), namely (\ref{Eq.MaximalEstimateTightnessProof}) below. We then establish (\ref{Eq.MeshCompactContainmentTightnessProof}) by proving upper- and lower-tail estimates for $X_k^{j,N}$ using (\ref{Eq.MaximalEstimateTightnessProof}) and Assumption B.3.

Put $D=\beta-\alpha$ and set $h_N=D/d_N$. Let $t_p^N,\ldots,t_q^N$ be consecutive mesh points and set $n=q-p$. Define, for $0\leq m\leq n$,
$$
S_m=H_N(f,t_{p+m}^N)-H_N(f,t_p^N).
$$
By (\ref{Eq.FourthMomentBoundTightnessProof}) and Markov's inequality, for $0\leq i\leq j\leq n$ and $a>0$,
$$
\mathbb P\bigl(|S_j-S_i|\geq a\bigr)
\leq a^{-4}\mathfrak C(f)
\bigl(t_{p+j}^N-t_{p+i}^N\bigr)^2.
$$
Thus, $\{S_i\}_{i=0}^n$ satisfies the conditions of \cite[Theorem~10.2]{Billing}, with the parameters denoted there by $\alpha$ and $\beta$ both equal to $1$, and with $u_l=\mathfrak C(f)^{1/2}h_N$ for $1\leq l\leq n$. Applying that theorem, we obtain a universal constant $C>0$ such that, for every $a>0$,
\begin{equation}\label{Eq.MaximalEstimateTightnessProof}
\mathbb P\!\left(
\max_{p\leq j\leq q}
|H_N(f,t_j^N)-H_N(f,t_p^N)|\geq a
\right)
\leq C\mathfrak C(f)a^{-4}(t_q^N-t_p^N)^2.
\end{equation}

We next fix $0<\ell<\min\{1,D\}$ and put $J=\lceil D/\ell\rceil$. For all sufficiently large $N$, depending on $\ell$, we may choose integers
$$
0=q_{0,N}<q_{1,N}<\cdots<q_{J,N}=d_N
$$
and define
$$
r_{s,N}=t_{q_{s,N}}^N \mbox{ for } 0\leq s \leq J, \mbox{ and }
\mathcal B_{s,N}
=\{q_{s,N},q_{s,N}+1,\ldots,q_{s+1,N}\} \mbox{ for } 0\leq s<J,
$$
so that for $0 \leq s < J$
\begin{equation}\label{Eq.BlockBoundsTightnessProof}
\frac{\ell}{4}\leq r_{s+1,N}-r_{s,N}\leq2\ell,
\qquad
J\leq\frac{D+1}{\ell}.
\end{equation}
For later use, we also define
\begin{equation}\label{Eq.TightnessProofDefDelta}
\Delta_{s,N}(f)
=
\max_{j\in\mathcal B_{s,N}}
\left|H_N(f,t_j^N)-H_N(f,r_{s,N})\right|,
\end{equation}
and note that by (\ref{Eq.MaximalEstimateTightnessProof}) and (\ref{Eq.BlockBoundsTightnessProof}), we have
\begin{equation}\label{Eq.BlockLinearStatisticBound}
\mathbb P\bigl(\Delta_{s,N}(f)\geq a\bigr)
\leq
4C\mathfrak C(f)a^{-4}\ell^2.
\end{equation}

We now turn to the proof of (\ref{Eq.MeshCompactContainmentTightnessProof}), beginning with the upper tail. By (\ref{Eq.LinearStatistic}), (\ref{Eq.CutoffFunction}), and (\ref{Eq.TranslatedCutoffFunctions}), the inequality $X_1^{j,N}\leq\varsigma$ implies $H_N(g_\varsigma^+,t_j^N)=0$, while $X_1^{j,N}\geq\varsigma+1$ implies $H_N(g_\varsigma^+,t_j^N)\geq1$. Together with the definition of $\Delta_{s,N}(g_\varsigma^+)$ from (\ref{Eq.TightnessProofDefDelta}), this yields
$$
\left\{
\max_{j\in\mathcal B_{s,N}}X_1^{j,N}\geq\varsigma+1
\right\}
\subseteq
\left\{X_1^{q_{s,N},N}>\varsigma\right\}
\cup
\left\{\Delta_{s,N}(g_\varsigma^+)\geq1\right\}.
$$
Taking a union over $0\leq s<J$ and using the preceding inclusion together with (\ref{Eq.UpperTailProbability}), (\ref{Eq.BlockBoundsTightnessProof}), and (\ref{Eq.BlockLinearStatisticBound}) with $a = 1$, we obtain
\begin{equation*}
\limsup_{N\to\infty}
\mathbb P\!\left(
\max_{0\leq j\leq d_N}X_1^{j,N}\geq\varsigma+1
\right)
\leq
(D+1)\left\{
\frac{p_+(\varsigma)}{\ell}
+
4C\mathfrak C(g_\varsigma^+)\ell
\right\}.
\end{equation*}

Fix $0<\varepsilon<\min\{1,D\}$ and set
$\ell=\varepsilon/\mathfrak C(g_\varsigma^+)$. By (\ref{Eq.FourthMomentConstant}), $\mathfrak C(g_\varsigma^+)\geq1$, and hence $0<\ell<\min\{1,D\}$. Substituting this choice of $\ell$ into the preceding inequality gives
$$
\limsup_{N\to\infty}
\mathbb P\!\left(
\max_{0\leq j\leq d_N}X_1^{j,N}\geq\varsigma+1
\right)
\leq
(D+1)\left\{
\frac{p_+(\varsigma)\mathfrak C(g_\varsigma^+)}{\varepsilon}
+
4C\varepsilon
\right\}.
$$
Letting first $\varsigma\to\infty$ and then $\varepsilon\downarrow0$ and using (\ref{Eq.UpperTailCompatibility}), we conclude that
\begin{equation}\label{Eq.UpperContainmentBoundTightnessProof}
\lim_{\varsigma\to\infty}\limsup_{N\to\infty}
\mathbb P\!\left(
\max_{0\leq j\leq d_N}X_1^{j,N}\geq\varsigma+1
\right)=0.
\end{equation}

We next control the lower tail. By (\ref{Eq.LinearStatistic}), (\ref{Eq.CutoffFunction}), and (\ref{Eq.TranslatedCutoffFunctions}), the inequality $X_k^{j,N}\geq-\varsigma$ implies $H_N(g_\varsigma^-,t_j^N)\geq k$, while $X_k^{j,N}\leq-\varsigma-1$ implies $H_N(g_\varsigma^-,t_j^N)\leq k-1$. Together with the definition of $\Delta_{s,N}(g_\varsigma^-)$ from (\ref{Eq.TightnessProofDefDelta}), this yields
$$
\left\{
\min_{j\in\mathcal B_{s,N}}X_k^{j,N}\leq-\varsigma-1
\right\}
\subseteq
\left\{X_k^{q_{s,N},N}<-\varsigma\right\}
\cup
\left\{\Delta_{s,N}(g_\varsigma^-)\geq1\right\}.
$$
Taking a union over $0\leq s<J$ and using the preceding inclusion together with (\ref{Eq.LowerTailProbability}), (\ref{Eq.BlockBoundsTightnessProof}), and (\ref{Eq.BlockLinearStatisticBound}), we obtain
\begin{equation*}
\limsup_{N\to\infty}
\mathbb P\!\left(
\min_{0\leq j\leq d_N}X_k^{j,N}\leq-\varsigma-1
\right)
\leq
(D+1)\left\{
\frac{p_{k,-}(\varsigma)}{\ell}
+
4C\mathfrak C(g_\varsigma^-)\ell
\right\}.
\end{equation*}

Fix $0<\varepsilon<\min\{1,D\}$ and set
$\ell=\varepsilon/\mathfrak C(g_\varsigma^-)$. By (\ref{Eq.FourthMomentConstant}), $\mathfrak C(g_\varsigma^-)\geq1$, and hence $0<\ell<\min\{1,D\}$. Substituting this choice of $\ell$ into the preceding inequality gives
$$
\limsup_{N\to\infty}
\mathbb P\!\left(
\min_{0\leq j\leq d_N}X_k^{j,N}\leq-\varsigma-1
\right)
\leq
(D+1)\left\{
\frac{p_{k,-}(\varsigma)\mathfrak C(g_\varsigma^-)}{\varepsilon}
+
4C\varepsilon
\right\}.
$$
Letting first $\varsigma\to\infty$ and then $\varepsilon\downarrow0$ and using (\ref{Eq.LowerTailCompatibility}), we conclude that
\begin{equation}\label{Eq.LowerContainmentBoundTightnessProof}
\lim_{\varsigma\to\infty}\limsup_{N\to\infty}
\mathbb P\!\left(
\min_{0\leq j\leq d_N}X_k^{j,N}\leq-\varsigma-1
\right)=0.
\end{equation}

Combining (\ref{Eq.UpperContainmentBoundTightnessProof}) with (\ref{Eq.LowerContainmentBoundTightnessProof}) and the fact that $X_1^{j,N} \geq X_k^{j,N}$ for all $k \geq 1$, we conclude (\ref{Eq.MeshCompactContainmentTightnessProof}).

\medskip
\noindent
{\bf Proof of Claim 2.} We begin by deriving a local oscillation estimate on the blocks $\mathcal B_{s,N}$ introduced in the proof of Claim~1. We then combine this estimate with (\ref{Eq.MeshCompactContainmentTightnessProof}) to establish (\ref{Eq.DiscreteModulusLimitTightnessProof}).

For $A,\rho>0$ and $n\in\mathbb Z$, define
$$
\mathcal U_{k,A,N} =
\left\{
\max_{0\leq j\leq d_N}|X_k^{j,N}|\leq A
\right\},
\qquad
\psi_{n,\rho}(x)
=
\chi\left(\frac{x-(n-1)\rho}{\rho}\right),
$$
where we recall that $\chi$ is as in (\ref{Eq.CutoffFunction}). We also introduce the finite set
$$
I_{A,\rho} = \left\{ n\in\mathbb Z: [n\rho,(n+1)\rho)\cap[-A,A]\neq\emptyset \right\}.
$$
For $0\leq s<J$, define the events
$$
\mathcal E_{s,N}^{k,A,\rho}
=
\left\{
\max_{j\in\mathcal B_{s,N}}
|X_k^{j,N}-X_k^{q_{s,N},N}|
\geq2\rho
\right\}
\cap\mathcal U_{k,A,N}.
$$
We claim that
\begin{equation}\label{Eq.BlockOscillationInclusion}
\mathcal E_{s,N}^{k,A,\rho}
\subseteq
\bigcup_{n\in I_{A,\rho}}
\left\{\Delta_{s,N}(\psi_{n,\rho})\geq1\right\},
\end{equation}
where we recall that $\Delta_{s,N}$ is as in (\ref{Eq.TightnessProofDefDelta}).

Indeed, suppose that $\mathcal E_{s,N}^{k,A,\rho}$ occurs. Then there exists $j\in\mathcal B_{s,N}$ such that
$$
|X_k^{j,N}-X_k^{q_{s,N},N}|\geq2\rho.
$$
Set $x=\max\{X_k^{j,N},X_k^{q_{s,N},N}\}$, $y=\min\{X_k^{j,N},X_k^{q_{s,N},N}\}$, and choose $n\in\mathbb Z$ so that $x\in[n\rho,(n+1)\rho)$. Since $\mathcal U_{k,A,N}$ occurs, we have $x\in[-A,A]$, and hence $n\in I_{A,\rho}$. Moreover,
$$
y\leq x-2\rho<(n-1)\rho.
$$
From the above inequalities, we conclude 
$$X_k^{j,N} \geq n\rho \mbox{ and }X_k^{q_{s,N},N} < (n-1)\rho \mbox{, or }X_k^{j,N} < (n-1)\rho \mbox{ and }X_k^{q_{s,N},N} \geq n\rho.$$
Combining the latter with the fact that particles are ordered, (\ref{Eq.LinearStatistic}), (\ref{Eq.CutoffFunction}), and the definition of $\psi_{n,\rho}$, we conclude $$H_N(\psi_{n,\rho},t_j^N) \geq k \mbox{ and }H_N(\psi_{n,\rho},r_{s,N}) \leq k -1 \mbox{, or } H_N(\psi_{n,\rho},t_j^N) \leq k-1 \mbox{ and }H_N(\psi_{n,\rho},r_{s,N}) \geq k.$$
Either way, using (\ref{Eq.TightnessProofDefDelta}) we see that $\Delta_{s,N}(\psi_{n,\rho}) \geq 1$, proving (\ref{Eq.BlockOscillationInclusion}).

Using (\ref{Eq.BlockLinearStatisticBound}) with $a = 1$ and (\ref{Eq.BlockOscillationInclusion}), followed by a union bound over $n\in I_{A,\rho}$, we obtain
\begin{equation}\label{Eq.LocalModulusBoundTightnessProof}
\mathbb P\bigl(\mathcal E_{s,N}^{k,A,\rho}\bigr)
\leq C_{A,\rho}\ell^2,
\mbox{ where }
C_{A,\rho}
=
4C\sum_{n\in I_{A,\rho}}\mathfrak C(\psi_{n,\rho})
<\infty.
\end{equation}

We now prove (\ref{Eq.DiscreteModulusLimitTightnessProof}). Fix $\eta,\varepsilon>0$ and set $\rho=\eta/6$. By (\ref{Eq.MeshCompactContainmentTightnessProof}), we may choose $A>0$ such that
$$
\limsup_{N\to\infty}
\mathbb P\bigl(\mathcal U_{k,A,N}^c\bigr)
<\varepsilon.
$$
Fix $0<\ell<\min\{1,D\}$ and use the blocks $\mathcal B_{s,N}$ constructed in the proof of Claim~1. From the last inequality, (\ref{Eq.BlockBoundsTightnessProof}), (\ref{Eq.LocalModulusBoundTightnessProof}), and a union bound over $0\leq s<J$, we conclude
$$
\limsup_{N\to\infty}
\mathbb P\left(
\mathcal U_{k,A,N}^c
\cup
\bigcup_{s=0}^{J-1}\mathcal E_{s,N}^{k,A,\rho}
\right)
\leq
\varepsilon+(D+1)C_{A,\rho}\ell.
$$

On the complement of the event in the preceding display, for every $0\leq s<J$ and every $j\in\mathcal B_{s,N}$,
\begin{equation}\label{Eq.NeedTriangle}
|X_k^{j,N}-X_k^{q_{s,N},N}|<2\rho.
\end{equation}
By the lower bound in (\ref{Eq.BlockBoundsTightnessProof}), two mesh points whose time coordinates are at distance at most $\ell/8$ must lie either in the same block or in two adjacent blocks. In the first case, the corresponding particle locations differ by less than $4\rho$ in view of (\ref{Eq.NeedTriangle}) and the triangle inequality. In the second case, using the common endpoint of the two adjacent blocks, (\ref{Eq.NeedTriangle}), and the triangle inequality, they differ by less than $6\rho=\eta$. Consequently,
$$
\limsup_{N\to\infty}
\mathbb P\left(
w_{k,N}^{\rm disc}(\ell/8)>\eta
\right)
\leq
\varepsilon+(D+1)C_{A,\rho}\ell.
$$
Letting first $\ell\downarrow0$ and then $\varepsilon\downarrow0$, we obtain (\ref{Eq.DiscreteModulusLimitTightnessProof}).
\end{proof}

%----------------------------------------------------------------------
%
% Section 2.5
%
%----------------------------------------------------------------------
\subsection{Comparison with Johansson's argument and the role of the assumptions}\label{Section2.5}

The proofs of Lemmas~\ref{Lem.FourthMomentBound} and \ref{Lem.DPPTightnessCriterion} are inspired by the arguments surrounding Lemmas~5.1--5.3 of Johansson~\cite{Kurt03}. In that work, Johansson proves a fourth moment estimate for differences of linear statistics in Lemma~5.1, derives from it a maximal inequality in Lemma~5.2, and then uses suitable cutoff functions to obtain a modulus of continuity estimate for the top curve in Lemma~5.3. The general structure of our argument is similar. There are, however, several differences that motivate the formulation adopted in Sections~\ref{Section2.3} and \ref{Section2.4}.

First, our goal is to prove tightness of the entire line ensemble, rather than only of its top curve. Thus, the lower-tail estimates in Section~\ref{Section2.4} are formulated for every fixed curve index $k$, and ultimately yield tightness in $C(\mathbb{N}\times[\alpha,\beta])$. On the other hand, Johansson works with the top curve and proves estimates that contain additional decay in the time variable; see, for example, \cite[(5.2)]{Kurt03}. This stronger control is useful for the subsequent study of the maximum of the top curve over $\mathbb{R}$ and of the location where this maximum is attained. Since our tightness argument takes place on a fixed compact interval, no analogous decay away from the origin is needed.

A second difference is that the argument of \cite{Kurt03} is closely intertwined with the particular form of the correlation kernel of the discrete polynuclear growth (PNG) model. Our purpose in Sections~\ref{Section2.3} and \ref{Section2.4} is instead to isolate a small collection of structural and analytic properties of the kernel that suffice for the argument. This makes the resulting fourth moment bound and tightness criterion applicable without repeating the model-specific cancellations and contour calculations appearing in \cite{Kurt03}.

\medskip

The most important of these structural requirements is Assumption~A.1, which isolates the algebraic identities responsible for the cancellations in the fourth moment expansion. The same structure is present in \cite{Kurt03}. Indeed, Johansson writes his correlation kernel in the form
\begin{equation}\label{Eq.JohanssonKernelDiscussion}
K_{n,M}(r,x;s,y)=\widetilde K_{n,M}(r,x;s,y)-\phi_{r,s}(x,y),
\end{equation}
see \cite[(1.22)]{Kurt03}, where
\begin{equation}\label{Eq.JohanssonModifiedKernelDiscussion}
\widetilde K_{n,M}(r,x;s,y)
=\sum_{i,j=1}^n\phi_{r,M}(x,x_i^M)(A^{-1})_{ij}\phi_{-M,s}(x_j^{-M},y);
\end{equation}
see \cite[(1.23)]{Kurt03}. The transition functions are defined by
\begin{equation}\label{Eq.JohanssonTransitionDefinitionDiscussion}
\phi_{r,s}=\phi_{r,r+1}*\cdots *\phi_{s-1,s},\qquad r<s,
\end{equation}
and are set to be identically zero when $r\geq s$; see the definition immediately preceding \cite[(1.21)]{Kurt03}. In particular, for $r<t<s$,
\begin{equation}\label{Eq.JohanssonConvolutionDiscussion}
\phi_{r,s}=\phi_{r,t}*\phi_{t,s}.
\end{equation}

There is a minor difference in conventions. In Section~5, Johansson abbreviates
$K_N(2u,x;2v,y)$ from \cite[(3.24)]{Kurt03} as $K_{uv}(x,y)$. Thus, for $u<v$, his
kernel decomposition takes the form
$$
K_{uv}(x,y) = \widetilde K_{uv}(x,y)-\phi_{uv}(x,y).
$$
In Assumption~A.1, by contrast, the corresponding correction term
$-\phi_{uv}(y,x)$ appears in $K_{vu}(x,y)$. Our convention is therefore obtained
from Johansson's by transposing the space-time kernel, namely by replacing $K(a,x;a',y)$ with $K(a',y;a,x)$. After this transposition, formulas \eqref{Eq.JohanssonModifiedKernelDiscussion} and \eqref{Eq.JohanssonConvolutionDiscussion} imply exactly the three contraction identities in Assumption~A.1.

The structure in \eqref{Eq.JohanssonKernelDiscussion}--\eqref{Eq.JohanssonConvolutionDiscussion} is not specific to the PNG model considered in \cite{Kurt03}. Johansson's general setup in Section~1.2 concerns measures built from products of determinants of transition functions, for which the associated space-time point process is determinantal with kernel of the form \eqref{Eq.JohanssonKernelDiscussion}--\eqref{Eq.JohanssonModifiedKernelDiscussion}. In the PNG model, such a measure arises from a system of non-intersecting paths through the Karlin--McGregor or Lindstr\"om--Gessel--Viennot formula; see the discussion following \cite[(1.32)]{Kurt03} and Proposition~3.13 therein. More generally, kernels of this form are characteristic of Eynard--Mehta-type constructions of space-time determinantal processes from transition kernels; see \cite[Section 1]{BR05}. Thus, Assumption~A.1 captures a structural feature naturally satisfied by a broad class of space-time determinantal processes, rather than one peculiar to the kernels appearing in \cite{Kurt03} or in the present paper.

For the fourth moment estimate, the relevance of this structure is that the identities \eqref{Eq.FourthMomentContractionUU}--\eqref{Eq.FourthMomentContractionVU} allow kernels involving different time slices to be contracted to copies of the single kernel $\widetilde K_{uv}$. These contractions are the main source of the cancellations in the proof of Lemma~\ref{Lem.FourthMomentBound}. Johansson's proof uses the same mechanism implicitly, but there the contractions are intertwined with the particular contour-integral representation of the PNG kernel; for example, the main contribution in \cite[(5.48)]{Kurt03} is explicitly identified as a ``contraction term'' and is used immediately afterward to cancel another contribution to the fourth moment expansion. Assumption~A.1 isolates this algebraic mechanism from the model-specific estimates that accompany it.

The remaining assumptions separate the analytic estimates from this algebraic mechanism. Assumption~A.2 provides the absolute integrability needed to justify the expansions, contractions, and changes in the order of integration. Assumption~A.3 contains the quantitative input: the difference kernels $L_k$ are required to gain a factor $(v-u)^{k/2}$. Since the total degree in the fourth moment expansion is four, these bounds produce the factor $(v-u)^2$ in \eqref{Eq.FourthMomentBound}. 

\medskip
We next turn to the assumptions used in the tightness argument. Just as Lemma~2.6 provides the fourth moment estimate that underlies Lemma~\ref{Lem.DPPTightnessCriterion}, Johansson's Lemma~5.1 provides the corresponding input for his proof of tightness. Assumptions~B.1 and B.2 ensure that the hypotheses of Lemma~2.6 hold uniformly over the determinantal processes and pairs of time slices appearing in Section~\ref{Section2.4}, while Assumption~B.3 supplies the additional one-time tail information needed to pass from estimates for linear statistics to tightness of the individual curves.

There is one point in the tightness argument of \cite{Kurt03} that deserves some care. Lemma~5.1 there gives the fourth moment estimate \cite[(5.2)]{Kurt03} with a constant $C(f,\alpha)$ depending on the test function $f$. In the proof of Claim~5.4, Johansson takes $f=g_M$, where $g_M(x)=g(x-M)$, with $g$ defined in \cite[(5.8)]{Kurt03} and essentially equal to our function $\chi$ from (\ref{Eq.CutoffFunction}). In the upper-tail argument $M\geq0$ is taken large, so that $g_M$ is translated toward $+\infty$. Immediately after \cite[(5.8)]{Kurt03}, Johansson states that, for $f=g_M$, the constant $C(f,\alpha)$ in \cite[(5.2)]{Kurt03} ``can be taken to be independent of $M$,'' explaining that ``only sup-norms of $f$ and its derivatives enter.'' This uniformity is then used, through the maximal estimate of Lemma~5.2, together with the one-point convergence of the top curve to $F_2$, to obtain the upper-tail part of Claim~5.4; see the remainder of the proof on \cite[p.~312]{Kurt03}. The corresponding lower-tail argument would require the same type of control when the cutoff is translated toward $-\infty$, which in the above notation corresponds to $M\leq0$.

As far as we can tell, however, the estimates given in the proof of Lemma~5.1 do not directly provide the corresponding uniformity of $C(g_M,\alpha)$ over $M\leq0$. One place where the location of the translate enters explicitly is \cite[(5.55)]{Kurt03}. There the summation variable is restricted by the location of the transition region of $f$, namely through the condition
$$
\xi - \tau^2 = \frac{y-cN}{dN^{1/3}}\in[K_1-1,K_2+1],
$$
and the summand contains an exponential factor depending on $\xi$. When $M\to+\infty$, this factor decays, in agreement with the uniformity needed for the upper-tail argument. When $M\to-\infty$, the same estimate instead produces a growing factor, so an additional argument appears to be needed to obtain a bound uniform in $M$. Thus, although Johansson remarks after \cite[(5.8)]{Kurt03} that the constant depends only on bounds for $f$ and its derivatives, we have not found in the proof a verification of the corresponding uniformity for negative translates. The lower-tail case is not treated separately: after completing the upper-tail estimate in Claim~5.4, Johansson states that ``the case $\max_{|j|\leq L}H_{N,0}(t_j)\leq-M$ is analogous''; see \cite[p.~312]{Kurt03}.

Assumption~B.3 is formulated so that no such uniformity is required. We allow the fourth moment constants $\mathfrak C(g_\varsigma^\pm)$ to depend on the location of the cutoff and require only that any growth of these constants be compensated by the corresponding one-time tail probabilities:
$$
p_+(\varsigma)\mathfrak C(g_\varsigma^+)\rightarrow 0,
\qquad
p_{k,-}(\varsigma)\mathfrak C(g_\varsigma^-)\rightarrow 0.
$$
In particular, we do not need a uniform bound on $\mathfrak C(g_\varsigma^-)$ as the cutoff is translated toward $-\infty$. Moreover, because we prove tightness of the full line ensemble rather than only of the top curve, the lower-tail condition is imposed separately for every fixed curve index $k$.

%----------------------------------------------------------------------
%
% Section 3
%
%----------------------------------------------------------------------
\section{Determinantal structure of the model}\label{Section3}
The goal of this section is to identify determinantal point processes associated with the probability measures introduced in Section~\ref{Section1.2}. In Section~\ref{Section3.1}, we recall the periodic Schur process of \cite{Bor07} and identify the specialization that yields the measures from Section~\ref{Section1.2}. There, we introduce the associated shift-mixed point process and, using Borodin's determinantal formula, derive a double-contour representation for its correlation kernel. In Section~\ref{Section3.2}, we introduce the scaled point processes corresponding to the setup of Definition~\ref{Def.S1Scaling} and derive a formula for their correlation kernel that is suitable for verifying the conditions in Section~\ref{Section2.4}. Throughout this section, we freely use the notation from Section~\ref{Section2.1}.

%----------------------------------------------------------------------
%
% Section 3.1
%
%----------------------------------------------------------------------
\subsection{The periodic Schur process}\label{Section3.1}

We continue to use the notation introduced at the beginning of Section~\ref{Section1.2}, including the notions of a partition and interlacing. We also recall from Section \ref{Section2.2} that $\llbracket p,q \rrbracket$ denotes the set $\{p, p+1, \dots, q\}$ for any two integers $p \leq q$.

The following definition introduces the periodic Schur process, which is the main object of interest in this section.

\begin{definition}\label{Def.PeriodicSchur}
Fix $M\in\mathbb{N}$ and $u\in(0,1)$, as well as $a_i,b_i\in(0,1)$ for $i\in\llbracket 1,M\rrbracket$. We define the \emph{periodic Schur process} to be the probability measure on $\Lambda$, as defined in \eqref{Eq.DefLambda}, given by
\begin{equation}\label{Eq.DefPeriodicSchur}
\mathbb{P}\left(\lambda^{(-M)},\ldots,\lambda^{(M)}\right) = \frac{1}{Z}u^{|\lambda^{(M)}|}
\prod_{i=1}^{M}s_{\lambda^{(-i+1)}/\lambda^{(-i)}}(a_i)
\prod_{i=1}^{M}s_{\lambda^{(i-1)}/\lambda^{(i)}}(b_i).
\end{equation}
Here, $s_{\lambda/\mu}(x) = {\bf 1}\{\lambda\succeq\mu\}x^{|\lambda|-|\mu|}$
is the skew Schur polynomial in a single variable, and
\begin{equation}\label{Eq.FormulaForZ}
Z= \frac{1}{(u;u)_\infty} \prod_{i,j=1}^{M}\frac{1}{(a_ib_j;u)_\infty}.
\end{equation}
\end{definition}

\begin{remark}\label{Rem.CorrespondenceBorodin}
The measures in \eqref{Eq.DefPeriodicSchur} are special cases of the periodic Schur processes introduced in \cite[Section~1]{Bor07}. More precisely, denote the sequence appearing in \cite[Section~1]{Bor07} by $\widetilde{\lambda}^{(0)},\ldots,\widetilde{\lambda}^{(N)}$. Then the measure in \eqref{Eq.DefPeriodicSchur} corresponds to the choices
$t=u$, $N=2M$, $\widetilde{\lambda}^{(j)}=\lambda^{(j-M)}$ for $j\in\llbracket 0,2M\rrbracket$, and the $\alpha$-specializations
$$
a[j] = \begin{cases}
\varnothing, & j\in\llbracket 1,M\rrbracket,\\
(b_{j-M}), & j\in\llbracket M+1,2M\rrbracket,
\end{cases}
\qquad
b[j]
=
\begin{cases}
(a_{M+1-j}), & j\in\llbracket 1,M\rrbracket,\\
\varnothing, & j\in\llbracket M+1,2M\rrbracket.
\end{cases}
$$

In particular, under these parameter choices, the formula for the normalization constant $Z$ in \eqref{Eq.FormulaForZ} follows from \cite[Proposition~1.1 and Remarks~1.2--1.3]{Bor07}.
\end{remark}

Using the formula for $s_{\lambda/\mu}(x)$, we readily observe that the measure in (\ref{Eq.S1ProbDistr}) is a special case of the one in (\ref{Eq.DefPeriodicSchur}), corresponding to setting
\begin{equation}\label{Eq.Specialization}
u=q^M,\quad a_i=b_i =bq^{i-1/2}\quad\mbox{for }i\in\llbracket1,M\rrbracket.
\end{equation}

For the remainder of this section, we use the notation of Definition~\ref{Def.PeriodicSchur} and introduce a family of shift-mixed point processes associated with \eqref{Eq.DefPeriodicSchur}. These point processes were introduced in \cite{Bor07}, where they were shown to be determinantal.

Let $(\lambda^{(-M)},\ldots,\lambda^{(M)})$ be distributed according to \eqref{Eq.DefPeriodicSchur}, and let $S$ be an independent integer-valued random variable with distribution
\begin{equation}\label{Eq.ShiftVar}
\mathbb{P}(S=s)=\frac{u^{s^2/2}}{\theta_3(1;u)},\qquad s\in\mathbb{Z}.
\end{equation}
Here $\theta_3$ denotes the third Jacobi theta function,
\begin{equation}\label{Eq.DefTheta}
\theta_3(z;t)=\sum_{m\in\mathbb{Z}}z^m t^{m^2/2}
=(t;t)_\infty (-zt^{1/2};t)_{\infty} (-t^{1/2}/z;t)_{\infty}.
\end{equation}
The second equality is a form of the Jacobi triple product; see \cite[Chapter~13, \S19, (16)]{EMOT}. In the notation of that reference, one sets $t=q^2$ and $z=e^{2\pi\mathrm{i}v}$. Both expressions in (\ref{Eq.DefTheta}) are defined for $z\in\mathbb{C}^{\times}$ and $t\in (0,1)$.

With this data, we define the point process $\mathfrak{S}$ on $\mathbb{R}^2$ by setting, for every Borel set $A\subseteq\mathbb{R}^2$,
\begin{equation}\label{Eq.PointProcessSchur}
\mathfrak{S}(A) =\sum_{i \geq 1} \sum_{j = 1}^{2M}{\bf 1} \left\{(j,\lambda^{(j-M)}_i-i+S) \in A \right\}.
\end{equation} 

Proposition \ref{Prop.SchurDPP} below shows that $\mathfrak{S}$ is a determinantal point process. To state it we require some additional notation that we summarize in the following definition.

\begin{definition}\label{Def.Funct}
Assume the notation of Definition~\ref{Def.PeriodicSchur}. For
$s\in\llbracket1,2M\rrbracket$, define
\begin{equation}\label{Eq.DefFInhom}
F_s(z)
=
\begin{cases}
\displaystyle
\frac{\prod_{k=1}^{M}(b_k/z;u)_\infty}
{\prod_{k=1}^{M-s}(ua_kz;u)_\infty
 \prod_{k=M-s+1}^{M}(a_kz;u)_\infty},
& 1\leq s\leq M,\\[3ex]
\displaystyle
\frac{\prod_{k=1}^{s-M}(ub_k/z;u)_\infty
 \prod_{k=s-M+1}^{M}(b_k/z;u)_\infty}
{\prod_{k=1}^{M}(a_kz;u)_\infty},
& M+1\leq s\leq2M.
\end{cases}
\end{equation}
Here and below, empty products are understood to equal $1$.

For $z,w\in\mathbb{C}^{\times}$, define the meromorphic function
\begin{equation}\label{Eq.DefKappa}
\kappa(z,w) = -\frac{(u;u)_\infty^3\theta_3(z/w;u)}
{\theta_3(-u^{-1/2}z/w;u)\theta_3(1;u)}.
\end{equation}
\end{definition}
\begin{remark}\label{Rem.SpecializedF}
Under the specialization (\ref{Eq.Specialization}), the functions
defined in (\ref{Eq.DefFInhom}) admit a simpler expression. Indeed, using $u=q^M$ and $a_j=b_j=bq^{j-1/2}$, together with
$$\prod_{j=1}^{M}(cq^{j-1};q^M)_\infty=(c;q)_\infty,$$
we obtain
\begin{equation}\label{Eq.SpecializedF}
F_s(z) = 
\begin{cases}
\displaystyle
\frac{(bq^{1/2}/z;q)_\infty}
{(bq^{M-s+1/2}z;q)_\infty},
& 1\leq s\leq M,\\
\displaystyle
\frac{(bq^{s-M+1/2}/z;q)_\infty}
{(bq^{1/2}z;q)_\infty},
& M+1\leq s\leq2M.
\end{cases}
\end{equation}
\end{remark}

With the above notation in place, we can state the following consequence of \cite[Theorem A]{Bor07}. 
\begin{proposition}\label{Prop.SchurDPP}
Assume the same notation as in Definitions~\ref{Def.PeriodicSchur} and~\ref{Def.Funct}, and let $\mathfrak{S}$ be as in \eqref{Eq.PointProcessSchur}. Then $\mathfrak{S}$ is a determinantal point process on $\mathbb{R}^2$ with reference measure given by the counting measure on $\llbracket 1, 2M \rrbracket \times \mathbb{Z}$ and correlation kernel given for $s,t \in \llbracket 1, 2M \rrbracket$ and $x,y \in \mathbb{Z}$ by
\begin{equation}\label{Eq.SchurKernel}
\begin{split}
K(s,x;t,y) = & -\frac{{\bf 1}\{s>t\}}{2\pi\mathrm{i}}
\oint_{C_{1}}dw\frac{F_s(w)}{F_t(w)}w^{y-x-1} \\
&+ \frac{1}{(2\pi\mathrm{i})^2} \oint_{C_{r_z}}dz\oint_{C_{r_w}}dw \frac{F_s(z)}{F_t(w)}\kappa(z,w)z^{-x-1}w^{y-1}.
\end{split}
\end{equation}
Here $C_r$ denotes the positively oriented circle of radius $r$, and $r_z,r_w$ are chosen in $(\alpha, \alpha^{-1})$ so that $1<r_z/r_w<\min(\alpha^{-1}, u^{-1})$, where $\alpha = \max_{1\leq i\leq M}\{a_i,b_i\}$.
\end{proposition}
\begin{proof}
Let $\mathfrak{S}'$ be the point process on $\mathbb{R}^2$ defined, for every Borel set $A\subseteq\mathbb{R}^2$, by
$$\mathfrak{S}'(A) = \sum_{i\geq1}\sum_{j=1}^{2M}
{\bf 1}\left\{\left(j,\lambda^{(j-M)}_i-i+\frac12+S\right)\in A \right\}.$$
Under the correspondence in Remark~\ref{Rem.CorrespondenceBorodin}, $\mathfrak{S}'$ is precisely the shift-mixed periodic Schur process considered in \cite{Bor07}, with Borodin's parameters $N=2M$, $t=u$, and $z = 1$, where $z$ is the shift-mixed parameter from \cite[Section 2]{Bor07}. Hence, by \cite[Theorem~A]{Bor07}, $\mathfrak{S}'$ is a determinantal point process with reference measure given by counting measure on $\llbracket1,2M\rrbracket\times \left(\mathbb{Z} + 1/2\right)$. By \cite[Remark~2.4]{Bor07} its correlation kernel is given, for $\sigma,\tau \in \llbracket 1, 2M \rrbracket$ and $X,Y \in \mathbb{Z} + 1/2$, by
\begin{equation}\label{Eq.BorKernel}
K'\left(\sigma,X;\tau,Y\right)
=-\frac{(u;u)_\infty^3}{\theta_3(1;u)(2\pi\mathrm{i})^2}
\oint_{C_{r_\zeta}} d\zeta\oint_{C_{r_\eta}} d\eta
\frac{F(\sigma,\zeta)}{F(\tau,\eta^{-1})}
\frac{\theta_3(\zeta\eta;u)}
{\theta_3(-\zeta\eta u^{-1/2};u)}
\frac{1}
{\zeta^{X+1/2}\eta^{Y+1/2}}.
\end{equation}
Here $r_\zeta, r_\eta \in (\alpha, \alpha^{-1})$ and $1 < r_{\zeta} r_{\eta} < \min(\alpha^{-1}, u^{-1})$ if $\sigma\leq\tau$, while $\max(\alpha, u)<r_{\zeta}r_{\eta}<1$ if $\sigma > \tau$. In addition, from the definition of $F(\sigma,\zeta)$ immediately preceding Theorem A in \cite{Bor07}, we have
\begin{equation}\label{Eq.BorF}
\begin{split}
F(\sigma,\zeta)
=
\exp\Bigg(
\sum_{n\geq1}
&\frac{\zeta^n}{1-u^n}\sum_{k=1}^{\sigma}b_n[k]
+ \frac{(u\zeta)^n}{1-u^n}\sum_{k=\sigma+1}^{2M}b_n[k] \\
&-\frac{(u/\zeta)^n}{1-u^n}\sum_{k=1}^{\sigma}a_n[k]
-\frac{(1/\zeta)^n}{1-u^n}\sum_{k=\sigma+1}^{2M}a_n[k]
\Bigg),
\end{split}
\end{equation}
where $a_n[k]=p_n(a[k])/n$ and $b_n[k]=p_n(b[k])/n$. Under the specialization in Remark~\ref{Rem.CorrespondenceBorodin},
$$a_n[k] =
\begin{cases}
0, & 1\leq k\leq M,\\
b_{k-M}^n/n, & M+1\leq k\leq2M,
\end{cases}
\qquad
b_n[k]
=
\begin{cases}
a_{M+1-k}^n/n, & 1\leq k\leq M,\\
0, & M+1\leq k\leq2M.
\end{cases}
$$
Substituting these expressions in (\ref{Eq.BorF}) and using the identity
$$
\exp\left(\sum_{n\geq1}\frac{c^n}{n(1-u^n)}\right) = \frac{1}{(c;u)_\infty}, $$
we obtain precisely the right-hand side of (\ref{Eq.DefFInhom}). Hence,
\begin{equation}\label{Eq.LabelMatch}
F(\sigma,\zeta)=F_\sigma(\zeta).
\end{equation}

Let $\phi(s,x) = (s, x-1/2)$ and note that $\mathfrak{S} = \mathfrak{S}'\phi^{-1}$. From \cite[Proposition~2.13(5)]{ED24a}, we conclude that $\mathfrak{S}$ is a determinantal point process with reference measure given by the counting measure on $\llbracket1,2M\rrbracket\times\mathbb{Z}$ and correlation kernel $K'\left(s,x+1/2;t,y+1/2\right)$. To finish the proof of the proposition, what remains is to check that  
\begin{equation}\label{Eq.SchurDPPRed1}
K(s,x;t,y) = K'\left(s,x+1/2;t,y+1/2\right).
\end{equation}

Set $\sigma=s$, $\tau=t$, $X=x+1/2$, $Y=y+1/2$ in (\ref{Eq.BorKernel}) and change variables $\zeta=z$, $\eta=w^{-1}$ to get
\begin{equation}\label{Eq.NewBorKer}
K'\left(s,x+\frac12;t,y+\frac12\right)
=\frac{1}{(2\pi\mathrm{i})^2} \oint_{C_{r_\zeta}}dz
\oint_{C_{r_\eta^{-1}}}dw \frac{F_s(z)}{F_t(w)}\kappa(z,w)z^{-x-1}w^{y-1},
\end{equation}
where $r_{\zeta}$ and $r_{\eta}$ satisfy the same conditions as before. When $s \leq t$, we can set $r_{\zeta} = r_z$ and $r_{\eta} = r_w^{-1}$, and so (\ref{Eq.SchurDPPRed1}) follows from (\ref{Eq.SchurKernel}) and (\ref{Eq.NewBorKer}).

Suppose now that $s>t$. Choose $\widetilde r_z\in(\alpha,\alpha^{-1})$, sufficiently close to $r_w$ from below, so that
$$
\max(\alpha,u)<\widetilde r_z/r_w<1.
$$
The contour conditions for $s>t$ allow us to choose $r_\zeta=\widetilde r_z$ and $r_\eta=r_w^{-1}$ in \eqref{Eq.NewBorKer}. 

We proceed to deform the $z$-contour from $C_{\widetilde r_z}$ to $C_{r_z}$. The functions $F_s(z)$ and $F_t(w)^{-1}$ are holomorphic for $\alpha<|z|,|w|<\alpha^{-1}$. Moreover, by \eqref{Eq.DefTheta},
$$
\theta_3(-u^{-1/2}z/w;u)=(u;u)_\infty(z/w;u)_\infty(uw/z;u)_\infty.
$$
Together with \eqref{Eq.DefKappa}, this shows that the possible poles of $\kappa(z,w)$, as a function of $z$, are $z=u^k w$, $k\in\mathbb{Z}$. In particular, there is a simple pole at $z=w$ because the corresponding zero of the denominator arises from
$$
(z/w;u)_\infty=(1-z/w)(uz/w;u)_\infty,
$$
while all other factors are nonzero at $z=w$. More precisely, we have
$$
\kappa(z,w)=\frac{w}{z-w}+O(1)\qquad\text{as }z\to w.
$$
Since $u<\widetilde r_z/r_w<1<r_z/r_w<u^{-1}$, the only pole crossed during the deformation is $z=w$. Its residue in the double-integral integrand is
$$
\frac{F_s(w)}{F_t(w)}w^{y-x-1}.
$$
The residue theorem therefore gives
\begin{equation*}
K'\left( s,x+\frac12;t,y+\frac12 \right)
= \frac{1}{(2\pi\mathrm{i})^2} \hspace{-0.5mm}
\oint_{C_{r_z}} \hspace{-2.5mm}dz\oint_{C_{r_w}} \hspace{-2.5mm}dw\,
\frac{F_s(z)}{F_t(w)} 
\kappa(z,w)z^{-x-1}w^{y-1} -\frac{1}{2\pi\mathrm{i}}
\oint_{C_{r_w}}\hspace{-2.5mm} dw\,
\frac{F_s(w)}{F_t(w)}w^{y-x-1}.
\end{equation*}
The right-hand side of the last equation agrees with (\ref{Eq.SchurKernel}) once we deform the second $C_{r_w}$ contour to $C_1$. Since no poles are crossed, Cauchy's theorem shows that the value of the integral is unchanged. We therefore obtain \eqref{Eq.SchurDPPRed1} when $s>t$, completing the proof of the proposition.
\end{proof}

%----------------------------------------------------------------------
%
% Section 3.2
%
%----------------------------------------------------------------------
\subsection{Scaled correlation kernel}\label{Section3.2}

The goal of this section is to introduce an $N$-indexed sequence of point processes that, after restriction to an appropriate time interval, fits into the framework of Section~\ref{Section2.4} and whose atoms are close to the values of the curves from Definition~\ref{Def.S1Scaling}. These point processes are obtained from the shift-mixed point processes of Section~\ref{Section3.1} by applying appropriate $N$-dependent shifts and rescalings, and will play a key role in the proofs of Theorems~\ref{Thm.Convergence} and~\ref{Thm.DeterminantalStructure}. In Proposition~\ref{Prop.ScaledKernel} we show that these point processes are determinantal and we derive a formula for their correlation kernels. For the remainder of this section, we use the notation of Definition~\ref{Def.S1Scaling}.

We introduce the following additional notation: 
\begin{equation}\label{Eq.S3Parameters}
\alpha_N=bq_N^{1/2}, \hspace{2mm} u_N = q_N^{M_N}, \hspace{2mm} h_N = c_2^{-1}N^{-1/3}, \hspace{2mm} A_{k,N} = \frac{b|k|}{1-b}-\frac{bk^2}{4(1-b)^2N},\hspace{2mm} m_{k,N}=\lfloor A_{k,N}\rfloor, 
\end{equation}
where $k \in \llbracket -M_N, M_N \rrbracket$. Let $(\lambda^{(-M_N)},\ldots,\lambda^{(M_N)})$ have law $\mathbb{P}_N$ and let $S_N$ be an independent integer-valued random variable whose distribution, as in (\ref{Eq.ShiftVar}), is given by
\begin{equation}\label{Eq.S3ShiftVar}
\mathbb{P}(S_N=s) = \frac{u_N^{s^2/2}}{\theta_3\left(1;u_N\right)}, \qquad s\in\mathbb{Z}.
\end{equation}
With this data, we define the random variables
\begin{equation}\label{Eq.S3ScaledParticles}
\widehat{X}_i^{k,N} = h_N\left(\lambda_i^{(k)}-i+S_N+m_{k,N}-c_1N\right),
\qquad i\geq1,\quad -M_N\leq k\leq M_N,
\end{equation}
and the point process $\widehat{M}^N$ on $\mathbb{R}^2$ by setting, for every Borel set $A\subseteq\mathbb{R}^2$,
\begin{equation}\label{Eq.S3ScaledPointProcess}
\widehat{M}^N(A) = \sum_{i\geq1}\sum_{k=-M_N+1}^{M_N}
{\bf 1}\left\{ \left(t_{k,N},\widehat{X}_i^{k,N}\right)\in A \right\}.
\end{equation}
We mention that the sum in (\ref{Eq.S3ScaledPointProcess}) over $k$ starts from $-M_N + 1$. This is intentional and allows us to apply Proposition~\ref{Prop.SchurDPP} directly in the proof of Proposition~\ref{Prop.ScaledKernel}.
\begin{remark}\label{Rem.PointProcessClose} From (\ref{Eq.S1ScaledCurves}) and (\ref{Eq.S3ScaledParticles}), we have
\begin{equation*}
\widehat{X}_i^{k,N} - \mathcal{L}_i^N(t_{k,N}) =
h_N\left(S_N-i+m_{k,N}-A_{k,N}\right).
\end{equation*}
Since $-1<m_{k,N}-A_{k,N}\leq0$, we conclude that, for every $i\geq1$,
\begin{equation}\label{Eq.S3ParticleComparisonBound}
\max_{-M_N\leq k\leq M_N}
\left|\widehat{X}_i^{k,N}-\mathcal{L}_i^N(t_{k,N})\right|\leq h_N\left(|S_N|+i+1\right).
\end{equation}
From \cite[Lemma~2.1]{BB19} with $t=1$ and $u=u_N$, we have that $S_N \sqrt{M_N/N} \Rightarrow Z$, where $Z$ is a standard normal variable. Since $M_{N}/N \sim c_2^{-1}R N^{-1/3}$ and $h_N = c_2^{-1} N^{-1/3}$, we conclude that 
\begin{equation}\label{Eq.S53SmallSN}
h_NS_N \overset{\mathbb{P}}{\rightarrow } 0.
\end{equation} 
It then follows from (\ref{Eq.S3ParticleComparisonBound}) that, for every fixed $i \geq 1$, the variables $\widehat{X}_i^{k,N}$ are close to $\mathcal{L}_i^N(t_{k,N})$ uniformly in $k \in \llbracket - M_N, M_N \rrbracket$. 
\end{remark}

Proposition~\ref{Prop.ScaledKernel} below shows that $\widehat{M}^N$ is a determinantal point process. To state it we require some additional notation that we summarize in the following definition.

\begin{definition}\label{Def.ScaledKernelNotation}
Assume the same notation as in Definition~\ref{Def.S1Scaling} and (\ref{Eq.S3Parameters}). Define
$$\widehat{\mathcal{T}}_N=\left\{t_{k,N}:k\in\llbracket-M_N+1,M_N\rrbracket\right\},\qquad\Lambda_N=h_N\left(\mathbb{Z}-c_1N\right),
$$
and, for $x\in\Lambda_N$, set
$$
\chi_N(x)=h_N^{-1}x+c_1N.
$$
Notice that $\chi_N(x)\in\mathbb{Z}$ for every $x\in\Lambda_N$.

For any $s=t_{k,N}\in\widehat{\mathcal{T}}_N$, define
\begin{equation}\label{Eq.S3GFunctions}
G_{s,N}(z)
=
\frac{F_{k+M_N}(z)}{F_{k+M_N}(1)}z^{m_{k,N}},
\end{equation}
where $F_{k+M_N}$ is as in Definition~\ref{Def.Funct}, under the specialization (\ref{Eq.Specialization}) with $q=q_N$ and $M=M_N$. Equivalently, by (\ref{Eq.SpecializedF}), we have
\begin{equation}\label{Eq.GPochSymbol}
G_{s,N}(z)
=
z^{m_{k,N}}
\begin{cases}
\displaystyle
\frac{(bq_N^{1/2}/z;q_N)_\infty}
     {(bq_N^{1/2};q_N)_\infty}
\frac{(bq_N^{-k+1/2};q_N)_\infty}
     {(bq_N^{-k+1/2}z;q_N)_\infty},
& k\in\llbracket-M_N+1,0\rrbracket,
\\[3ex]
\displaystyle
\frac{(bq_N^{k+1/2}/z;q_N)_\infty}
     {(bq_N^{k+1/2};q_N)_\infty}
\frac{(bq_N^{1/2};q_N)_\infty}
     {(bq_N^{1/2}z;q_N)_\infty},
& k\in\llbracket0,M_N\rrbracket.
\end{cases}
\end{equation}
We also define
\begin{equation}\label{Eq.S3KappaN}
\kappa_N(z,w)
=
-\frac{(u_N;u_N)_\infty^3\theta_3(z/w;u_N)}
{\theta_3(-u_N^{-1/2}z/w;u_N)\theta_3(1;u_N)}.
\end{equation}
\end{definition}

With the above notation in place, we can state the following
consequence of Proposition~\ref{Prop.SchurDPP}.
\begin{proposition}\label{Prop.ScaledKernel}
Assume the same notation as in Definitions~\ref{Def.S1Scaling}
and~\ref{Def.ScaledKernelNotation}, and let $\widehat{M}^N$ be as
in (\ref{Eq.S3ScaledPointProcess}). Then $\widehat{M}^N$ is a
determinantal point process on $\mathbb{R}^2$ with reference measure
given by the counting measure on
$\widehat{\mathcal{T}}_N\times\Lambda_N$ and correlation kernel
given, for $s,t\in\widehat{\mathcal{T}}_N$ and
$x,y\in\Lambda_N$, by
\begin{equation}\label{Eq.S3ScaledKernel}
\begin{split}
\widehat{K}^N(s,x;t,y)
={}&
-\frac{{\bf 1}\{s>t\}}{2\pi\mathrm{i}}
\oint_{C_{1}}dw\,
\frac{G_{s,N}(w)}{G_{t,N}(w)}
w^{\chi_N(y)-\chi_N(x)-1}
\\
&+
\frac{1}{(2\pi\mathrm{i})^2}
\oint_{C_{r_z}}dz
\oint_{C_{r_w}}dw\,
\frac{G_{s,N}(z)}{G_{t,N}(w)}
\kappa_N(z,w)
z^{-\chi_N(x)-1}w^{\chi_N(y)-1}.
\end{split}
\end{equation}
Here $C_r$ denotes the positively oriented circle of radius $r$,
and $r_z,r_w$ are chosen in $(\alpha_N,\alpha_N^{-1})$ so that
$1< r_z/r_w < \min\left(\alpha_N^{-1},u_N^{-1}\right)$.
\end{proposition}
\begin{proof}
Let $\mathfrak{S}^N$ denote the point process in
(\ref{Eq.PointProcessSchur}) under the specialization
(\ref{Eq.Specialization}) with $q=q_N$ and $M=M_N$, and denote its
correlation kernel in (\ref{Eq.SchurKernel}) by
$K_{\mathfrak{S}}^N$.

For $r\in\mathbb{R}$, set
$$
\mathfrak{m}_N(r)
=
\left\lfloor
\frac{b|r-M_N|}{1-b}
-
\frac{b(r-M_N)^2}{4(1-b)^2N}
\right\rfloor,
$$
and define the bijection $\phi_N:\mathbb{R}^2\to\mathbb{R}^2$ by
$$
\phi_N(r,a)
=
\left(
\frac{R(r-M_N)}{2M_N},
h_N\left(a+\mathfrak{m}_N(r)-c_1N\right)
\right).
$$
Notice that $\mathfrak{m}_N(k+M_N)=m_{k,N}$ and
$$
\phi_N^{-1}(t_{k,N},x)
=
\left(k+M_N,\chi_N(x)-m_{k,N}\right).
$$
It follows from the definitions that $\widehat{M}^N=\mathfrak{S}^N\phi_N^{-1}$. Moreover, $\phi_N$ maps $\llbracket 1,2M_N\rrbracket\times\mathbb{Z}$ bijectively onto $\widehat{\mathcal{T}}_N\times\Lambda_N$. Consequently, the push-forward of the counting measure on
$\llbracket 1,2M_N\rrbracket\times\mathbb{Z}$ under $\phi_N$ is the
counting measure on $\widehat{\mathcal{T}}_N\times\Lambda_N$.

By Proposition~\ref{Prop.SchurDPP} and
\cite[Proposition~2.13(5)]{ED24a}, $\widehat{M}^N$ is a
determinantal point process with reference measure given by the
counting measure on $\widehat{\mathcal{T}}_N\times\Lambda_N$ and
correlation kernel
$$
\overline{K}^N(s,x;t,y)
= K_{\mathfrak{S}}^N
\left( k+M_N,\chi_N(x)-m_{k,N};
\ell+M_N,\chi_N(y)-m_{\ell,N}
\right),
$$
where $s=t_{k,N}$ and $t=t_{\ell,N}$. Applying a gauge transformation, see \cite[Proposition~2.13(4)]{ED24a} with $f_N(s,x)=F_{k+M_N}(1)^{-1}$ for $s=t_{k,N}$, shows that $\widehat{M}^N$ also has correlation kernel
$$
\widehat{K}^N(s,x;t,y)
=
\frac{F_{\ell+M_N}(1)}{F_{k+M_N}(1)}
\overline{K}^N(s,x;t,y).
$$
Substituting (\ref{Eq.SchurKernel}) into the last expression and
using (\ref{Eq.S3GFunctions}) gives (\ref{Eq.S3ScaledKernel}).
\end{proof}

%----------------------------------------------------------------------
%
% Section 4
%
%----------------------------------------------------------------------
\section{Analysis of the scaled kernel}\label{Section4}

The goal of this section is to derive various estimates and asymptotic statements for the correlation kernels $\widehat{K}^N$ from Proposition~\ref{Prop.ScaledKernel}, which will be used in the proofs of Theorems~\ref{Thm.Convergence} and~\ref{Thm.DeterminantalStructure} in Section~\ref{Section5}. In Section~\ref{Section4.1}, we rewrite $\widehat{K}^N$ in an alternative form, denoted by $\widehat{K}^{\mathrm{ext},N}$, which is suitable for asymptotic analysis; see Lemma~\ref{Lem.S4KernelAgreement}. The definition of $\widehat{K}^{\mathrm{ext},N}$ involves a new function $\Psi_{s,N}$, and in Section~\ref{Section4.2} we establish several analytic estimates for $\Psi_{s,N}$ that will be used repeatedly in Sections~\ref{Section4.3}, \ref{Section4.4}, and~\ref{Section5}. In particular, in Section~\ref{Section4.3} we use these estimates to prove the convergence of the kernels $\widehat{K}^{\mathrm{ext},N}$, while in Section~\ref{Section4.4} we use them to analyze the asymptotic behavior of the point processes on $\mathbb{R}$ formed by $\{\widehat{X}_i^{k,N}\}_{i\geq1}$ from (\ref{Eq.S3ScaledParticles}).

%----------------------------------------------------------------------
%
% Section 4.1
%
%----------------------------------------------------------------------
\subsection{Extended scaled correlation kernel}\label{Section4.1}

In Section \ref{Section3.2}, we introduced an $N$-indexed sequence of determinantal point processes whose correlation kernel $\widehat{K}^N(s,x;t,y)$ is given in (\ref{Eq.S3ScaledKernel}). The representation in (\ref{Eq.S3ScaledKernel}) is suitable for verifying many of the assumptions from Section \ref{Section2.4}; however, it is not quite in a form that is convenient for asymptotic analysis. Consequently, our goal in this section is to derive an alternative representation, denoted by $\widehat{K}^{\mathrm{ext},N}(s,x;t,y)$ and given in (\ref{Eq.S4ExtendedKernel}), for which we can obtain precise asymptotics in subsequent sections. We mention that $\widehat{K}^{\mathrm{ext},N}(s,x;t,y)$ is defined for all $x,y\in\mathbb{R}$, unlike $\widehat{K}^N(s,x;t,y)$, which in (\ref{Eq.S3ScaledKernel}) is defined only for $x,y\in\Lambda_N$, and by convention satisfies $\widehat{K}^N(s,x;t,y) = 0$ when $\{x,y\} \not \subseteq \Lambda_N$; see Remark \ref{Rem.KernelSupport}. This extension is mostly for convenience and allows us to apply results from \cite{ED24a} directly in our setting. It is for this reason that we refer to $\widehat{K}^{\mathrm{ext},N}(s,x;t,y)$ as the {\em extended} scaled correlation kernel.

\medskip

In what follows we introduce notation that will be used to define $\widehat{K}^{\mathrm{ext},N}(s,x;t,y)$ and will appear throughout Sections~\ref{Section4} and~\ref{Section5}. In addition, for the reader's convenience, we recall pertinent notation from Definition~\ref{Def.S1Scaling} and Section~\ref{Section3.2}.

Recall from Definition~\ref{Def.S1Scaling} that $b\in(0,1)$ and $R\in(0,\infty)$ are fixed and 
\begin{equation}\label{Eq.DefS1Recalled}
c_1=-2\log(1-b), \hspace{2mm} c_2=b^{1/3}(1-b)^{-2/3}, q_N=e^{-1/N}, \hspace{2mm} M_N=\left\lceil c_2^{-1}RN^{2/3}\right\rceil, \hspace{2mm} t_{k,N}=\frac{Rk}{2M_N},
\end{equation}
for $k \in \llbracket -M_N, M_N \rrbracket$. In addition, from (\ref{Eq.S3Parameters}) and Definition~\ref{Def.ScaledKernelNotation}, we recall that
\begin{equation}\label{Eq.DefS3Recalled}
\alpha_N=bq_N^{1/2},\hspace{2mm} u_N = q_N^{M_N},\hspace{2mm} h_N = c_2^{-1}N^{-1/3}, \hspace{2mm} A_{k,N} = \frac{b|k|}{1-b}-\frac{bk^2}{4(1-b)^2N},\hspace{2mm} m_{k,N}=\lfloor A_{k,N}\rfloor, 
\end{equation}
and also $\widehat{\mathcal{T}}_N=\left\{t_{k,N}:k\in\llbracket-M_N+1,M_N\rrbracket\right\}$, $\Lambda_N=h_N\left(\mathbb{Z}-c_1N\right)$.

For $\zeta \in \mathbb{C}$ with $|\zeta| < 1$ and $q \in (0,1)$, we define 
\begin{equation}\label{Eq.LogPoch}
\psi_q(\zeta) = \sum_{j \geq 0} \log(1 - \zeta q^j),
\end{equation}
where we take the principal branch of the logarithm. Using that $\log(1-u) = -\sum_{m \geq 1} \frac{u^m}{m}$ for $|u| < 1$, we see that 
\begin{equation}\label{Eq.LogPochV2}
\psi_q(\zeta) = - \sum_{j \geq 0}\sum_{m \geq 1} \frac{\zeta^m q^{jm}}{m} = - \sum_{m \geq 1} \sum_{j \geq 0} \frac{\zeta^m q^{jm}}{m} = - \sum_{m \geq 1} \frac{\zeta^m}{m(1- q^m)}.
\end{equation}

If $Z \in \mathbb{C}$ satisfies $\alpha_N < \left|e^{h_NZ}\right| < \alpha_N^{-1}$ and $s = t_{k,N} \in \widehat{\mathcal{T}}_N$, we define
\begin{equation}\label{Eq.S4PsiDefinition}
\begin{split}
&\Psi_{s,N}(Z)= (m_{k,N} -c_1N) h_N Z \\
&+ \begin{cases} &\psi_{q_N}(\alpha_Ne^{-h_NZ}) - \psi_{q_N}(\alpha_N) - \psi_{q_N}(\alpha_Nq_N^{-k}e^{h_NZ}) + \psi_{q_N}(\alpha_N q_N^{-k}), \hspace{2mm} k\in\llbracket-M_N+1,0\rrbracket, \\
&\psi_{q_N}(\alpha_N q_N^k e^{-h_NZ}) - \psi_{q_N}(\alpha_N q_N^k) - \psi_{q_N}(\alpha_Ne^{h_NZ}) + \psi_{q_N}(\alpha_N ), \hspace{2mm} k\in\llbracket 0, M_N\rrbracket.
\end{cases}
\end{split}
\end{equation}
From (\ref{Eq.GPochSymbol}), (\ref{Eq.LogPoch}), and (\ref{Eq.S4PsiDefinition}), we conclude 
\begin{equation}\label{Eq.S41GInTermsOfPsi}
e^{\Psi_{s,N}(Z)} =  G_{s,N}(e^{h_N Z})e^{-c_1 N h_N Z}, \quad \alpha_N<|e^{h_NZ}|<\alpha_N^{-1}.  
\end{equation}
In addition, from (\ref{Eq.LogPochV2}), (\ref{Eq.S4PsiDefinition}), and the identities $q_N = e^{-1/N}$, $\alpha_N=bq_N^{1/2}$, we have
\begin{equation}\label{Eq.S41PsiSinh}
\begin{split}
&\Psi_{s,N}(Z)
=(m_{k,N}-c_1N)h_NZ+\sum_{m\geq1} \frac{b^m B_{m,k,N}(Z)}{2m\sinh(m/(2N))}, \mbox{ where }\\ 
 &B_{m,k,N}(Z) = \begin{cases} 
 e^{km/N}(e^{mh_NZ}-1)-e^{-mh_NZ}+1, \hspace{2mm} &k\in\llbracket-M_N+1,0\rrbracket,\\
e^{mh_NZ}-1- e^{-km/N}(e^{-mh_NZ}-1), \hspace{2mm} &k\in\llbracket 1, M_N\rrbracket.
 \end{cases}
\end{split}
\end{equation}

\begin{definition}\label{Def.Contours}
For $a\in\mathbb R$, let
\begin{equation}\label{Eq.S4VerticalContour}
 \Gamma_{a,N}=a+\mathrm{i}[-\pi/h_N,\pi/h_N]
\end{equation}
denote the vertical line segment oriented upward. For
$s,t\in (-R/2, R/2]$, define
\begin{equation}\label{Eq.S4ContourParameters}
 a(s,t)=-s+\frac{R+s-t}{4},
 \qquad
 d(s,t)=t+\frac{R+s-t}{4},
\end{equation}
and observe that
\begin{equation}\label{Eq.S4ContourRelations}
 a(s,t)+s=d(s,t)-t=\frac{R+s-t}{4}>0, \quad a(s,t), d(s,t) \in [-R,R],
\end{equation}
and
\begin{equation}\label{Eq.S4ContourSum}
 0<a(s,t)+d(s,t)=\frac{R+t-s}{2}<R.
\end{equation}

Let $\mathsf{N}_0$ be sufficiently large, depending on $b$ and $R$, so that for $N \geq \mathsf{N}_0$ and all $s,t \in \widehat{\mathcal{T}}_N$, we have 
\begin{equation}\label{Eq.Inequalities}
\alpha_N < e^{h_Na(s,t)}, e^{-h_Nd(s,t)} < \alpha_N^{-1}, \mbox{ and } 1<  e^{h_Na(s,t) + h_Nd(s,t)} < e^{h_N R} \leq u_N^{-1} < \alpha_N^{-1}.
\end{equation}
\end{definition}
\begin{remark} Let us briefly explain why the choice of $\mathsf{N}_0$ in Definition \ref{Def.Contours} is possible. Since $\widehat{\mathcal{T}}_N \subset (-R/2, R/2]$, the inequalities in (\ref{Eq.S4ContourRelations}) and (\ref{Eq.S4ContourSum}) both hold for $s,t \in \widehat{\mathcal{T}}_N$. From (\ref{Eq.DefS3Recalled}) we have $h_N = c_2^{-1}N^{-1/3} \rightarrow 0$ and $\alpha_N = be^{-1/(2N)} \rightarrow b \in (0,1)$, while (\ref{Eq.S4ContourRelations}) gives $a(s,t), d(s,t) \in [-R, R]$, verifying the first set of inequalities in (\ref{Eq.Inequalities}) for large $N$. The second set of inequalities is verified using (\ref{Eq.S4ContourSum}), the fact that $u_N^{-1} =  e^{\lceil c_2^{-1}R N^{2/3} \rceil/N} \geq e^{h_NR}$, and the limits $u_N \rightarrow 1$, $\alpha_N \rightarrow b$.
\end{remark}

\begin{lemma}\label{Lem.S4KernelAgreement} Assume the same notation as in Definition \ref{Def.Contours}, let $\Psi_{s,N}$ be as in (\ref{Eq.S4PsiDefinition}), and $\kappa_N$ as in (\ref{Eq.S3KappaN}). Then for all $N \geq \mathsf{N}_0$, $s,t \in \widehat{\mathcal{T}}_N$, and $x,y \in \mathbb{R}$, the following kernel is well-defined and finite:
\begin{equation}\label{Eq.S4ExtendedKernel}
\begin{split}
&\widehat{K}^{\mathrm{ext},N}(s,x;t,y)
=-\frac{h_N\mathbf 1\{s>t\}}{2\pi\mathrm{i}}
\int_{\Gamma_{0,N}}
e^{ \Psi_{s,N}(W)-\Psi_{t,N}(W)+(y-x)W}dW\\
&+\frac{h_N^2}{(2\pi\mathrm{i})^2}
\int_{\Gamma_{a(s,t),N}}dZ
\int_{\Gamma_{-d(s,t),N}}dW\,
e^{ \Psi_{s,N}(Z)-\Psi_{t,N}(W)-xZ+yW} \cdot \kappa_N(e^{h_NZ},e^{h_NW}).
\end{split}
\end{equation}
In addition, if $N \geq \mathsf{N}_0$, $x,y\in\Lambda_N$, $s,t \in \widehat{\mathcal{T}}_N$, and $\widehat{K}^N$ is as in (\ref{Eq.S3ScaledKernel}), then
\begin{equation}\label{Eq.S4KernelAgreement}
 \widehat{K}^{\mathrm{ext},N}(s,x;t,y)=\widehat{K}^N(s,x;t,y).
\end{equation}
In particular, for $N \geq \mathsf{N}_0$, the random measure $\widehat{M}^N$ from (\ref{Eq.S3ScaledPointProcess}) is a determinantal point process with reference measure given by the counting measure on $\widehat{\mathcal{T}}_N\times\Lambda_N$ and correlation kernel $\widehat{K}^{\mathrm{ext},N}(s,x;t,y)$, for $s,t \in \widehat{\mathcal{T}}_N$ and $x,y \in \mathbb{R}$.
\end{lemma}
\begin{proof}
From the first set of inequalities in (\ref{Eq.Inequalities}), we know that $\Gamma_{a(s,t),N}$, $\Gamma_{-d(s,t),N}$, and $\Gamma_{0,N}$ are all in the domain of $\Psi_{s,N}$ and $\Psi_{t,N}$, which are continuous on these contours. In addition, the second set of inequalities in (\ref{Eq.Inequalities}) shows that for
$Z\in\Gamma_{a(s,t),N}$ and $W\in\Gamma_{-d(s,t),N}$,
$$1< \left|e^{h_NZ}/e^{h_NW}\right| < u_N^{-1}.$$
Since the poles of $\kappa_N(z,w)$ occur when $z = u_N^kw$ for some $k \in \mathbb{Z}$, we conclude that $\kappa_N(e^{h_NZ},e^{h_NW})$ is well-defined and continuous for $Z\in\Gamma_{a(s,t),N}$ and $W\in\Gamma_{-d(s,t),N}$. Therefore, $\widehat{K}^{\mathrm{ext},N}$ is the sum of two integrals of continuous functions over compact contours, implying the first part of the lemma.

\medskip 

The equality in (\ref{Eq.S4KernelAgreement}) follows by setting $r_z=e^{h_Na(s,t)}$, $r_w=e^{-h_Nd(s,t)}$ in the formula for $\widehat{K}^N$ in (\ref{Eq.S3ScaledKernel}), changing variables $z=e^{h_NZ}$, $w=e^{h_NW}$, and using (\ref{Eq.S41GInTermsOfPsi}). We mention that the inequalities in (\ref{Eq.Inequalities}) show that $r_z,r_w\in(\alpha_N,\alpha_N^{-1})$ and $1<r_z/r_w<\min\{\alpha_N^{-1},u_N^{-1}\}$, so these choices satisfy the conditions in (\ref{Eq.S3ScaledKernel}).

The last part of the lemma follows from Proposition \ref{Prop.ScaledKernel}, the kernel equality (\ref{Eq.S4KernelAgreement}), and \cite[(2.13) and Lemma 2.8]{ED24a}; cf.~Remark \ref{Rem.KernelSupport}.
\end{proof}

%----------------------------------------------------------------------
%
% Section 4.2
%
%----------------------------------------------------------------------
\subsection{Analytic estimates}\label{Section4.2} 

In this section, we establish several estimates for the functions $\kappa_N$ and $\Psi_{s,N}$ from (\ref{Eq.S3KappaN}) and (\ref{Eq.S4PsiDefinition}), respectively. We also record several Airy identities used in subsequent sections. Throughout this section, we freely use the notation from Section~\ref{Section4.1}.

For $s\in\mathbb R$ and $Z\in\mathbb C$, set
\begin{equation}\label{Eq.S4CubicLimitPhase}
P_s(Z)=Z^3/3+sZ^2+s^2Z.
\end{equation}
Our first lemma below shows that $\Psi_{s,N}(Z)$ is close to $P_s(Z)$ uniformly over compact sets in $\mathbb{C}$.
\begin{lemma}\label{Lem.S4PhiCloseToP} Fix $\mathsf{A}_1 > 0$. There exist $\mathsf{C}_1 > 0$ and $N_1 \in \mathbb{N}$, depending on $b,R,\mathsf{A}_1$, such that for $N \geq N_1$, $s = t_{k,N} \in \widehat{\mathcal{T}}_N$, and $Z \in \mathbb{C}$ with $|Z| \leq \mathsf{A}_1$, we have $\alpha_N < |e^{h_NZ}| < \alpha_N^{-1}$, and 
\begin{equation}\label{Eq.S4Cubic}
 \left|\Psi_{s,N}(Z)-P_{s}(Z)\right|\le \mathsf{C}_1h_N.
\end{equation}
\end{lemma}
\begin{proof}
Throughout the proof, we write $C$ for a generic positive constant that depends on $b$, $R$, and $\mathsf{A}_1$, and whose value may change from line to line. In addition, all constants in the big $O$ notations depend on $b$, $R$, and $\mathsf{A}_1$. We do not mention this further.

From (\ref{Eq.DefS1Recalled}) and (\ref{Eq.DefS3Recalled}), we have $h_N = c_2^{-1}N^{-1/3} \rightarrow 0$ and $\alpha_N = be^{-1/(2N)} \rightarrow b$. Consequently, there exists $N_1 \in \mathbb{N}$ such that for all $|Z| \leq \mathsf{A}_1$ and $N \geq N_1$, we have $\alpha_N < |e^{h_NZ}| < \alpha_N^{-1}$. This fixes our choice of $N_1$, and we assume $N \geq N_1$ and $|Z| \leq \mathsf{A}_1$ in the remainder of the proof.

\medskip

We record the following trivial bounds, valid for all $z \in \mathbb{C}$:
\begin{equation}\label{Eq.ExponTaylor}
\left|e^z - 1 - z - z^2/2 - z^3/6\right| \leq |z|^4 e^{|z|}, \quad |e^z| \leq e^{|z|},
\end{equation}
which follow by Taylor expanding $e^z$ and $e^{|z|}$ and comparing terms. Applying (\ref{Eq.ExponTaylor}) to the different exponential terms in the definition of $B_{m,k,N}(Z)$ in (\ref{Eq.S41PsiSinh}), we conclude for all $k \in \llbracket -M_N + 1, M_N \rrbracket$
\begin{equation}\label{Eq.ExpandBrackets}
\begin{split}
B_{m,k,N}(Z) = &2mh_NZ-\frac{|k|m^2h_NZ}{N}
+\frac{km^3h_N^2Z^2}{2N}
 \\
 &+\frac{k^2m^3h_NZ}{2N^2}
+\frac{m^3h_N^3Z^3}{3} +O\left(m^4h_N^4e^{Cmh_N}\right).
\end{split}
\end{equation}
We mention that in deriving the above we also used $|k|/N=  O(h_N)$ for $k \in \llbracket -M_N + 1, M_N \rrbracket$.

Since $\left|\frac{1}{2\sinh(v/2)}-\frac{1}{v}\right|\leq C_0v$ for some universal $C_0$ and all $v > 0$, we conclude for all $m \geq 1$
\begin{equation}\label{Eq.S4SinhApproximation}
\left|
\frac{b^m}{2m\sinh(m/(2N))}
-\frac{Nb^m}{m^2}
\right|
\leq \frac{C_0b^m}{N}.
\end{equation}
If $be^{Ch_N}\leq\sqrt{b}$, we get from (\ref{Eq.S4SinhApproximation}) that
\begin{equation}\label{Eq.ErrorApproximation}
\sum_{m\geq1} \frac{b^m}{2m\sinh(m/(2N))} \cdot m^4h_N^4e^{Cmh_N} \leq \frac{C_0h_N^4}{N}\sum_{m\geq1} m^4b^{m/2} + Nh_N^4\sum_{m\geq1} m^2b^{m/2} = O(h_N),
\end{equation}
where the last equality used $h_N = c_2^{-1}N^{-1/3}$.

Substituting (\ref{Eq.ExpandBrackets}) into (\ref{Eq.S41PsiSinh}) and using the bounds in (\ref{Eq.S4SinhApproximation}) and (\ref{Eq.ErrorApproximation}), we conclude, after possibly enlarging $N_1$, that
\begin{equation}\label{Eq.S41PsiWithError}
\begin{split}
\Psi_{s,N}(Z) = &(m_{k,N}-c_1N)h_NZ+\sum_{m\geq1} \frac{Nb^m}{m^2} \\
&\times \left(2mh_NZ-\frac{|k|m^2h_NZ}{N}
+\frac{km^3h_N^2Z^2}{2N}
+\frac{k^2m^3h_NZ}{2N^2}
+\frac{m^3h_N^3Z^3}{3}\right) + O(h_N).
\end{split}
\end{equation}

From (\ref{Eq.DefS1Recalled}) and basic power series, we have
\begin{equation}\label{Eq.S4GeometricSums}
\sum_{m\geq1}\frac{b^m}{m}=\frac{c_1}{2},
\qquad
\sum_{m\geq1}b^m=\frac{b}{1-b},
\qquad
\sum_{m\geq1}mb^m=\frac{b}{(1-b)^2}=c_2^3,
\qquad
Nh_N^3c_2^3=1.
\end{equation}
Substituting (\ref{Eq.S4GeometricSums}) into (\ref{Eq.S41PsiWithError}), we obtain
\begin{equation}\label{Eq.S4PsiExpansion}
\begin{split}
\Psi_{s,N}(Z)
=m_{k,N}h_NZ
-\frac{b|k|}{1-b}h_NZ
+\frac{bk}{2(1-b)^2}h_N^2Z^2+\frac{bk^2}{2(1-b)^2N}h_NZ
+Z^3/3
+O(h_N).
\end{split}
\end{equation}
Equation (\ref{Eq.S4PsiExpansion}) implies (\ref{Eq.S4Cubic}) once we use that $s = t_{k,N}$ and so
$$m_{k,N}h_NZ -\frac{b|k|h_NZ}{1-b} + \frac{bk^2h_NZ}{4(1-b)^2N} = O(h_N), \hspace{2mm} \frac{bkh_N^2}{2(1-b)^2} = s + O(h_N), \hspace{2mm} \frac{bk^2h_N}{4(1-b)^2N} = s^2 + O(h_N),$$
as follows from (\ref{Eq.DefS1Recalled}), (\ref{Eq.DefS3Recalled}), and (\ref{Eq.S4GeometricSums}).
\end{proof}

The following lemma provides bounds for $\left|e^{\pm\Psi_{s,N}(Z)}\right|$ and $\left|e^{\Psi_{s,N}(Z) - \Psi_{t,N}(Z)}\right|$ along vertical contours.
\begin{lemma}\label{Lem.S4DescentContours} Fix $\mathsf{A}_2,\epsilon>0$. There exist $\mathsf{C}_2, \mathsf{c}_2 > 0$ and $N_2 \in \mathbb{N}$, depending on $b,R,\mathsf{A}_2,\epsilon$, such that the following statements hold for $N\geq N_2$, $x\in [-\mathsf{A}_2,\mathsf{A}_2]$, $y\in[-\pi/h_N,\pi/h_N]$, and $s, t \in\widehat{\mathcal{T}}_N$. If $Z = x + \im y$, then $\alpha_N < |e^{h_NZ}| < \alpha_N^{-1}$. In addition, 
\begin{equation}\label{Eq.S4AngularPositive}
 \left|e^{\Psi_{s,N}(x+\mathrm{i}y)}\right|
 \leq \mathsf{C}_2e^{-\mathsf{c}_2y^2}, \mbox{ if $x + s\geq\epsilon$},
\end{equation}
\begin{equation}\label{Eq.S4AngularNegative}
 \left|e^{-\Psi_{s,N}(x+\mathrm{i} y)}\right|
 \leq \mathsf{C}_2e^{-\mathsf{c}_2y^2}, \mbox{ if $x+s\leq-\epsilon$},
\end{equation}
\begin{equation}\label{Eq.S4AngularRatio}
 \left|e^{\Psi_{s,N}(x+\mathrm{i} y) - \Psi_{t,N}(x+\mathrm{i} y)}\right|
 \leq \mathsf{C}_2e^{-\mathsf{c}_2y^2}, \mbox{ if $s - t \geq \epsilon$}.
\end{equation}
\end{lemma}
\begin{proof} Throughout the proof, we write $c,C$ for generic positive constants that depend on $b$, $R$, $\mathsf{A}_2$ and $\epsilon$, and whose values may change from line to line. In addition, all constants in the big $O$ notations depend on $b$, $R$, $\mathsf{A}_2$ and $\epsilon$. We do not mention this further.

By assumption, we have $e^{-h_N \mathsf{A}_2} \leq |e^{h_NZ}| \leq e^{h_N \mathsf{A}_2}$. As in the proof of Lemma \ref{Lem.S4PhiCloseToP}, we can use the facts that $h_N \rightarrow 0$ and $\alpha_N \rightarrow b \in (0,1)$ to conclude that there exists $N_2 \in \mathbb{N}$ such that $e^{h_N \mathsf{A}_2} < \alpha_N^{-1}$ for $N \geq N_2$, which implies $\alpha_N < |e^{h_NZ}| < \alpha_N^{-1}$. This fixes our choice of $N_2$, and we assume $N \geq N_2$ in the remainder of the proof.

Fixing $s = t_{k,N} \in \widehat{\mathcal{T}}_N$ and taking real parts of (\ref{Eq.S41PsiSinh}), we get
\begin{equation}\label{Eq.DiffReal}
\begin{split}
\mathrm{Re}\left[ \Psi_{s,N}(x + \im y) - \Psi_{s,N}(x) \right] = -& 2\sum_{m \geq 1} \frac{b^m e^{-|k|m/(2N)}}{2m \sinh(m/(2N))} \\
& \times \sinh\left(mh_Nx + km/(2N)\right) \cdot (1 - \cos(ymh_N )).
\end{split}
\end{equation}
Notice that all the non-zero terms in the last sum have the same sign as $h_Nx + k/(2N)$.

Suppose first that $x + s \geq \epsilon$. Since $k/(2N) = sh_N + O(h_N^2)$, which follows from (\ref{Eq.DefS1Recalled}) and (\ref{Eq.DefS3Recalled}), we conclude that by possibly making $N_2$ larger, we have $h_Nx + k/(2N) \geq h_N \epsilon/2$. Since $\sinh(x) \geq x$ for $x \geq 0$, we conclude
\begin{equation}\label{Eq.FRT1}
\sinh(h_Nx + k/(2N)) \geq h_N \epsilon/2.
\end{equation}
From (\ref{Eq.S4SinhApproximation}) with $m = 1$ and the fact that $k = O(N^{2/3})$, we have
\begin{equation}\label{Eq.FRT2}
\frac{b e^{-|k|/(2N)}}{2 \sinh(1/(2N))} \geq c N.
\end{equation}
Lastly, using basic properties of the cosine function and the fact that $|yh_N| \leq \pi$, we conclude
\begin{equation}\label{Eq.FRT3}
1 - \cos\left(yh_N\right) \geq cy^2h_N^2.
\end{equation}
Keeping only the $m = 1$ term in (\ref{Eq.DiffReal}) and applying (\ref{Eq.FRT1}), (\ref{Eq.FRT2}), and (\ref{Eq.FRT3}) gives
$$\mathrm{Re}\left[ \Psi_{s,N}(x + \im y) - \Psi_{s,N}(x) \right] \leq - cN h_N^3 y^2,$$
which implies (\ref{Eq.S4AngularPositive}) once we use $h_N = c_2^{-1} N^{-1/3}$ and the fact that $\Psi_{s,N}(x) = O(1)$ by Lemma \ref{Lem.S4PhiCloseToP}.

Suppose next that $x + s \leq -\epsilon$. As before, by possibly making $N_2$ larger we may assume that $h_Nx + k/(2N) \leq -h_N \epsilon/2$. Since $\sinh(-x) \leq -x$ for $x \geq 0$, we conclude
\begin{equation}\label{Eq.FRT4}
\sinh(h_Nx + k/(2N)) \leq -h_N \epsilon/2.
\end{equation}
If we multiply (\ref{Eq.DiffReal}) by $-1$, keep only the $m = 1$ term, and apply the inequalities (\ref{Eq.FRT2}), (\ref{Eq.FRT3}), and (\ref{Eq.FRT4}), we obtain
$$\mathrm{Re}\left[  \Psi_{s,N}(x) - \Psi_{s,N}(x + \im y) \right] \leq - cN h_N^3 y^2.$$
As before, the last inequality implies (\ref{Eq.S4AngularNegative}) once we use that $h_N = c_2^{-1} N^{-1/3}$ and the fact that $\Psi_{s,N}(x) = O(1)$ by Lemma \ref{Lem.S4PhiCloseToP}.

\medskip

Suppose finally that $s-t\geq\epsilon$. Write $s=t_{k,N}$ and
$t=t_{\ell,N}$, and for $m\geq1$ set
\[
F_{m,N}(r;x)
=
e^{-mh_N|r|}\sinh\left(mh_N(x+r)\right).
\]
From \eqref{Eq.DiffReal}, we obtain
\begin{equation}\label{Eq.S4RatioAngularIdentity}
\begin{split}
&\mathrm{Re}\left[
\Psi_{s,N}(x+\im y)-\Psi_{t,N}(x+\im y)
-\Psi_{s,N}(x)+\Psi_{t,N}(x)
\right]\\
&\quad =
-2\sum_{m\geq1}
\frac{b^m \left(1-\cos(ymh_N)\right)}{2m\sinh(m/(2N))}
\left[
F_{m,N}\left(\frac{k}{2Nh_N};x\right)
-
F_{m,N}\left(\frac{\ell}{2Nh_N};x\right)
\right].
\end{split}
\end{equation}
For $r\neq0$, direct differentiation gives
\begin{equation}\label{Eq.S42Derivatives}
\frac{\partial}{\partial r}F_{m,N}(r;x)
= mh_N
\begin{cases}
e^{mh_N(x+2r)},&r<0,\\
e^{-mh_N(x+2r)},&r>0.
\end{cases}
\end{equation}
Thus, $r\mapsto F_{m,N}(r;x)$ is increasing for every $m\geq1$.

By \eqref{Eq.DefS1Recalled} and \eqref{Eq.DefS3Recalled}, $\frac{k}{2Nh_N}=s+O(h_N)$, $\frac{\ell}{2Nh_N}=t+O(h_N)$. In addition, $s,t \in \widehat{\mathcal{T}}_N \subset[-R/2, R/2]$ by assumption. Thus, after possibly increasing $N_2$, we conclude
\begin{equation}\label{Eq.S42DiffST}
\frac{k}{2Nh_N}-\frac{\ell}{2Nh_N}\geq\frac{\epsilon}{2}, \mbox{ and } \left|\frac{k}{2Nh_N}\right|, \left|\frac{\ell}{2Nh_N}\right| \leq R.
\end{equation}

The monotonicity of $F_{m,N}(r;x)$ and (\ref{Eq.S42DiffST}) show that all the summands in the second line of (\ref{Eq.S4RatioAngularIdentity}) are nonnegative, and so keeping only the $m = 1$ term gives
\begin{equation}\label{Eq.S4RatioAngularTower}
\begin{split}
&\mathrm{Re}\left[
\Psi_{s,N}(x+\im y)-\Psi_{t,N}(x+\im y)
-\Psi_{s,N}(x)+\Psi_{t,N}(x)
\right] \\
& \leq -
\frac{b \left(1-\cos(yh_N)\right)}{\sinh(1/(2N))}
\left[
F_{1,N}\left(\frac{k}{2Nh_N};x\right)
-
F_{1,N}\left(\frac{\ell}{2Nh_N};x\right)
\right].
\end{split}
\end{equation}

From (\ref{Eq.S42Derivatives}), we have $\frac{\partial}{\partial r}F_{1,N}(r;x)\geq ch_N$ for $|r|\leq R$, $r \neq 0$, which together with (\ref{Eq.S42DiffST}) shows
\begin{equation}\label{Eq.FRT5}
F_{1,N}\left(\frac{k}{2Nh_N};x\right) - F_{1,N}\left(\frac{\ell}{2Nh_N};x\right) \geq ch_N.
\end{equation}

Combining (\ref{Eq.S4RatioAngularTower}) with (\ref{Eq.FRT2}), (\ref{Eq.FRT3}), and (\ref{Eq.FRT5}), we conclude
\[
\mathrm{Re}\left[
\Psi_{s,N}(x+\im y)-\Psi_{t,N}(x+\im y)
-\Psi_{s,N}(x)+\Psi_{t,N}(x)
\right] \leq-cNh_N^3y^2.
\]
As before, the last inequality implies (\ref{Eq.S4AngularRatio}) once we use that $h_N = c_2^{-1} N^{-1/3}$ and the fact that $\Psi_{s,N}(x) = O(1)$ and $\Psi_{t,N}(x) = O(1)$ by Lemma \ref{Lem.S4PhiCloseToP}.
\end{proof}

The next lemma gives a Laurent expansion for $\kappa_N$ from
(\ref{Eq.S3KappaN}), together with a uniform bound and its scaling limit.
\begin{lemma}\label{Lem.S4KappaEstimates}
Let $\kappa_N$ be as in (\ref{Eq.S3KappaN}), and recall from (\ref{Eq.DefS3Recalled}) that $u_N = q_N^{M_N}$. We have the following absolutely convergent Laurent expansion for $\kappa_N$:
\begin{equation}\label{Eq.S4KappaLaurent}
\kappa_N(z,w)
=
\sum_{n\in\mathbb{Z}}
\frac{(z/w)^n}{1+u_N^{-n-1/2}},
\qquad 1<|z/w|<u_N^{-1}.
\end{equation}
In addition, for every $0<\alpha_-\leq\alpha_+<R$, there exist
$\mathsf{C}_3>0$ and $N_3\in\mathbb{N}$, depending on
$b,R,\alpha_-,\alpha_+$, such that for $N\geq N_3$ and
$Z,W\in\mathbb{C}$ satisfying
\begin{equation}\label{Eq.S4KappaStrip}
\alpha_-\leq\operatorname{Re}(Z-W)\leq\alpha_+,
\end{equation}
we have $1 < |e^{h_N(Z-W)}| < u_N^{-1}$ and
\begin{equation}\label{Eq.S4KappaUniformBound}
h_N\left|\kappa_N(e^{h_NZ},e^{h_NW})\right|
\leq\mathsf{C}_3.
\end{equation}
Moreover, as $N\to\infty$,
\begin{equation}\label{Eq.S4KappaScalingLimit}
h_N\kappa_N(e^{h_NZ},e^{h_NW})
\longrightarrow
\int_{\mathbb{R}}\frac{e^{(Z-W)v}}{1+e^{Rv}}\,dv,
\end{equation}
locally uniformly on
$\{(Z,W)\in\mathbb{C}^2: \alpha_- \leq \operatorname{Re}(Z-W) \leq \alpha_+\}$.
\end{lemma}

\begin{proof}
The Laurent expansion in (\ref{Eq.S4KappaLaurent}) follows directly
from the summation formula in \cite[Remark~2.4]{Bor07}, which in turn is a consequence of Ramanujan's summation formula for the hypergeometric series ${}_1\psi_1$, see \cite[(5.2.2)]{GasperRahman}. Indeed, specializing that formula to Borodin's parameters $t=u_N$ and $z=1$, dividing both sides by $r^{1/2}$, and then writing $m=n+1/2$ gives
$$-\frac{(u_N;u_N)_\infty^3\theta_3(r;u_N)}
{\theta_3(-u_N^{-1/2}r;u_N)\theta_3(1;u_N)}
=
\sum_{n\in\mathbb{Z}}
\frac{r^n}{1+u_N^{-n-1/2}},
\qquad 1<|r|<u_N^{-1}.$$
Taking $r=z/w$ and recognizing the left side as $\kappa_N(z,w)$, we obtain (\ref{Eq.S4KappaLaurent}).

We can use the fact that $r_N := \frac{M_N}{Nh_N} \rightarrow R$, which follows from (\ref{Eq.DefS1Recalled}) and (\ref{Eq.DefS3Recalled}), to conclude that there exists an $\epsilon \in (0, R - \alpha_+)$ and $N_3 \in \mathbb{N}$, depending on $b,R,\alpha_-, \alpha_+$, such that for $N \geq N_3$, we have $r_N > \alpha_+ + \epsilon$. This fixes our choice of $N_3$, and we assume $N \geq N_3$ in the remainder of the proof. By assumption, we have $e^{ \alpha_- h_N} \leq |e^{h_N(Z-W)}| \leq e^{h_N \alpha_+}$, and our choice of $N_3$ ensures $u_N^{-1} = e^{h_Nr_N} > e^{h_N (\alpha_+ + \epsilon)}$. Consequently, $1 < |e^{h_N(Z-W)}| < u_N^{-1}$ for $N \geq N_3$.

\medskip
We next prove (\ref{Eq.S4KappaUniformBound}). Separating the terms with $n\geq0$ and $n<0$ in (\ref{Eq.S4KappaLaurent}), we obtain
\begin{equation}\label{Eq.S4KappaGeometricBound}
\begin{split}
&\left|\kappa_N(e^{h_NZ},e^{h_NW})\right|
\leq
\sum_{n\in\mathbb{Z}}
\frac{e^{h_N\operatorname{Re}(Z-W)n}}
{1+e^{M_N(n+1/2)/N}}\leq
\sum_{n\geq0}
e^{-h_N\left(r_N-\alpha_+\right)n}
+
\sum_{n\geq1}e^{-h_N\alpha_-n}\\
&=
\frac{1}{1-e^{-h_N\left(r_N-\alpha_+\right)}}
+ \frac{1}{e^{h_N\alpha_-}-1} \leq \frac{1}{1-e^{-h_N\epsilon}}
+ \frac{1}{e^{h_N\alpha_-}-1}.
\end{split}
\end{equation}
The inequality in (\ref{Eq.S4KappaGeometricBound}) implies (\ref{Eq.S4KappaUniformBound}) once we recall that $\alpha_- > 0$ and $h_N = c_2^{-1}N^{-1/3}$.  

\medskip

It remains to prove (\ref{Eq.S4KappaScalingLimit}). Fix a compact set $\mathcal{D} \subset \left\{(Z,W)\in\mathbb{C}^2:
\alpha_-\leq\operatorname{Re}(Z-W)\leq\alpha_+\right\}$, and note that for $(Z,W) \in \mathcal{D}$, we have from
(\ref{Eq.S4KappaLaurent}) that
\begin{equation}\label{Eq.S4KappaRiemannSum}
h_N\kappa_N(e^{h_NZ},e^{h_NW})
=
h_N\sum_{n\in\mathbb{Z}}
\frac{e^{(Z-W)h_Nn}}
{1+\exp\left(r_N(h_Nn+h_N/2)\right)}.
\end{equation}

Fix $\eta>0$. For $L>0$, using $r_N-\alpha_+\geq\epsilon$ and arguing as in (\ref{Eq.S4KappaGeometricBound}), we obtain
\begin{equation}\label{Eq.S4KappaDiscreteTails}
\begin{split}
&h_N\sum_{|h_Nn|>L}
\left|
\frac{e^{(Z-W)h_Nn}}
{1+\exp\left(r_Nh_N(n+1/2)\right)}
\right|\\
&\leq
h_N\sum_{ h_Nn>L}e^{-\epsilon h_Nn}
+
h_N\sum_{ h_Nn<-L}e^{-\alpha_-h_N|n|}
\leq
C\left(e^{-\epsilon L}+e^{-\alpha_-L}\right),
\end{split}
\end{equation}
for some $C > 0$, depending on $b$, $\alpha_-$ and $\epsilon$, uniformly for $(Z,W)\in\mathcal{D}$ and $N\geq N_3$. In addition,
\begin{equation}\label{Eq.S4KappaIntegralTails}
\begin{split}
\int_{|v|>L}
\left|\frac{e^{(Z-W)v}}{1+e^{Rv}}\right|\,dv
&\leq
\int_L^\infty e^{-\epsilon v}\,dv
+
\int_{-\infty}^{-L}e^{\alpha_-v}\,dv= 
\frac{e^{-\epsilon L}}{\epsilon}
+
\frac{e^{-\alpha_-L}}{\alpha_-},
\end{split}
\end{equation}
uniformly for $(Z,W)\in\mathcal{D}$. Consequently, we can fix $L>0$ sufficiently large, depending on $\alpha_-$, $\epsilon$, and $\eta$, so that the right-hand sides of both (\ref{Eq.S4KappaDiscreteTails}) and (\ref{Eq.S4KappaIntegralTails}) are less than $\eta/2$.

For this fixed $L$, since $r_N\rightarrow R$ and $h_N\rightarrow0$, we have
\begin{equation}\label{Eq.S4KappaTruncatedLimit}
h_N\sum_{|h_Nn|\leq L}
\frac{e^{(Z-W)h_Nn}}
{1+\exp\left(r_Nh_N(n+1/2)\right)}
\longrightarrow
\int_{-L}^L\frac{e^{(Z-W)v}}{1+e^{Rv}}\,dv
\end{equation}
uniformly for $(Z,W)\in\mathcal{D}$. Combining (\ref{Eq.S4KappaRiemannSum}--\ref{Eq.S4KappaTruncatedLimit}), we conclude that
\[
\limsup_{N\rightarrow\infty}
\sup_{(Z,W)\in\mathcal{D}}
\left|
h_N\kappa_N(e^{h_NZ},e^{h_NW})
-
\int_{\mathbb{R}}\frac{e^{(Z-W)v}}{1+e^{Rv}}\,dv
\right|
\leq \eta.
\]
Since $\eta>0$ was arbitrary, this proves (\ref{Eq.S4KappaScalingLimit}).
\end{proof}

We conclude this section by collecting some well-known statements about the Airy function $\mathrm{Ai}$ in the following proposition.
\begin{proposition}\label{Prop.S4AiryIdentities}
The following statements hold.
\begin{enumerate}
\item[(i)] For $a,d,s,t\in\mathbb R$ satisfying $a+s>0$ and $d-t>0$,
and $x,y,v\in\mathbb R$, we have
\begin{equation}\label{Eq.S4AiryFirst}
 \frac{1}{2\pi\mathrm{i}}
 \int_{a+\mathrm{i}\mathbb R}
 e^{P_s(Z)-(x-v)Z}\,dZ
 =e^{s(x-v)-s^3/3}\Ai(x-v),
\end{equation}
and
\begin{equation}\label{Eq.S4AirySecond}
 \frac{1}{2\pi\mathrm{i}}
 \int_{-d+\mathrm{i}\mathbb R}
 e^{-P_t(W)+(y-v)W}\,dW
 =e^{-t(y-v)+t^3/3}\Ai(y-v).
\end{equation}
\item[(ii)] For $x,y\in\mathbb{R}$ and $r>0$,
\begin{equation}\label{Eq.S4AirySemigroup}
 \int_{\mathbb R}e^{rv}\Ai(x+v)\Ai(y+v)\,dv
 =\frac{1}{\sqrt{4\pi r}}
 \exp\!\left(
 -\frac{(x-y)^2}{4r}-\frac{r(x+y)}2+\frac{r^3}{12}
 \right).
\end{equation}
\item[(iii)] For every $x\in\mathbb R$,
\begin{equation}\label{Eq.S4AirySquareIdentity}
 \int_x^\infty\Ai(z)^2\,dz
 =\Ai'(x)^2-x\Ai(x)^2.
\end{equation}
\item[(iv)] As $u\to\infty$, we have the following asymptotic expansions:
\begin{equation}\label{Eq.S4AiryPositiveAsymptotic}
 \Ai(u)
 =\frac{1}{2\sqrt{\pi}}u^{-1/4}
 e^{-\frac23u^{3/2}}
 \left(1+O(u^{-3/2})\right),
\end{equation}
\begin{equation}\label{Eq.S4AiryPositiveDerivativeAsymptotic}
 \Ai'(u)
 =-\frac{1}{2\sqrt{\pi}}u^{1/4}
 e^{-\frac23u^{3/2}}
 \left(1+O(u^{-3/2})\right),
\end{equation}
\begin{equation}\label{Eq.S4AiryAsymptotic}
 \Ai(-u)
 =\pi^{-1/2}u^{-1/4}
 \left(\sin\!\left(\frac23u^{3/2}+\frac\pi4\right)
 +O(u^{-3/2})\right),
\end{equation}
\begin{equation}\label{Eq.S4AiryDerivativeAsymptotic}
 \Ai'(-u)
 =-\pi^{-1/2}u^{1/4}
 \left(\cos\!\left(\frac23u^{3/2}+\frac\pi4\right)
 +O(u^{-3/2})\right).
\end{equation}
\item[(v)] For every fixed $s,t\in(-R/2,R/2]$, the function
$(x,y)\mapsto K^{\mathrm{per},R}(s,x;t,y)$ from
\eqref{Eq.PeriodicAiryKernel} is continuous on $\mathbb R^2$.
\end{enumerate}
\end{proposition}
\begin{proof}
Part (i). Equations \eqref{Eq.S4AiryFirst} and \eqref{Eq.S4AirySecond} follow from the contour representation of the Airy function in \cite[Chapter~2, Eq.~(8.02)]{Olver} after the changes of variables $Z\mapsto Z-s$ and $W\mapsto -W-t$, respectively. The assumptions $a+s>0$ and $d-t>0$ ensure that the integrals converge absolutely along the given vertical contours.

Part (ii). Equation \eqref{Eq.S4AirySemigroup} is \cite[(44)]{Joh07}.

Part (iii). This is a simple consequence of the Airy equation and integration by parts; see \cite[p.~40]{VS10} for the derivation.

Part (iv). The asymptotic expansions follow directly from \cite[Chapter~11, Eqs.~(1.07)--(1.09)]{Olver}.

Part (v). Using \eqref{Eq.PeriodicAiryKernel}, together with the bounds in \eqref{Eq.S4AiryPositiveAsymptotic} and \eqref{Eq.S4AiryAsymptotic}, the dominated convergence theorem implies that
$$K^{\mathrm{per},R}(s,x_n;t,y_n)\longrightarrow K^{\mathrm{per},R}(s,x;t,y)$$
whenever $x_n\to x$ and $y_n\to y$.
\end{proof}

%----------------------------------------------------------------------
%
% Section 4.3
%
%----------------------------------------------------------------------
\subsection{Convergence of the correlation kernels}\label{Section4.3}

We continue with the same notation as in Section \ref{Section4.1}. The goal of this section is to use the estimates from Section~\ref{Section4.2} to establish the convergence of the extended scaled correlation kernels from (\ref{Eq.S4ExtendedKernel}). The precise statement is contained in the following proposition.

\begin{proposition}\label{Prop.S4KernelLimit}
Suppose $s,t\in(-R/2,R/2]$ and $s_N,t_N\in\widehat{\mathcal T}_N$ satisfy
\begin{equation}\label{Eq.S4TimeSequences}
 \lim_{N \rightarrow \infty} s_N = s \mbox{ and } \lim_{N \rightarrow \infty}t_N =  t.
\end{equation}
If $s=t$, assume in addition that $s_N=t_N$ for every $N$. Then, for every $A>0$,
\begin{equation}\label{Eq.S4KernelLimit}
\lim_{N\to\infty}
\sup_{x,y\in[-A,A]}
\left| h_N^{-1}\widehat{K}^{\mathrm{ext},N}(s_N,x;t_N,y)
-e^{sx-s^3/3-ty+t^3/3}K^{\mathrm{per},R}(s,x;t,y)
\right|=0,
\end{equation}
where $\widehat{K}^{\mathrm{ext},N}$ is as in (\ref{Eq.S4ExtendedKernel}), and $K^{\mathrm{per},R}$ is as in (\ref{Eq.PeriodicAiryKernel}).
\end{proposition}
\begin{proof}
Throughout the proof we fix $A>0$, $s,t \in (-R/2, R/2]$, and sequences $s_N,t_N\in\widehat{\mathcal T}_N$ as in the statement of the proposition. We also write $c,C$ for generic positive constants that depend on $b$, $R$, $s$, $t$, and $A$, and whose values may change from line to line. We recall that $a(s,t)$ and $d(s,t)$ are as in (\ref{Eq.S4ContourParameters}), and we write $a_N=a(s_N,t_N)$ and $d_N=d(s_N,t_N)$, for brevity.

For clarity, we split the proof into three steps. In Step 1, we reduce the statement of the proposition to establishing two asymptotic statements, see (\ref{Eq.S43KernelLimitRed1}) and (\ref{Eq.S43KernelLimitRed2}). These two statements are then established in Steps 2 and 3. 

\medskip 

{\bf \raggedleft Step 1.} Fix $x,y \in [-A,A]$, and sequences $x_N, y_N \in [-A,A]$ such that $\lim_{N \rightarrow \infty} x_N = x$ and $\lim_{N \rightarrow \infty} y_N = y$. We claim that 
\begin{equation}\label{Eq.S43KernelLimitRed1}
\begin{split}
&\lim_{N \rightarrow \infty} \frac{1}{(2\pi\mathrm{i})^2}
\int_{\Gamma_{a_N,N}}dZ
\int_{\Gamma_{-d_N,N}}dW
e^{\Psi_{s_N,N}(Z)-\Psi_{t_N,N}(W)-x_NZ+y_NW}
h_N\kappa_N(e^{h_NZ},e^{h_NW}) \\
& = \frac{1}{(2\pi\mathrm{i})^2}
\int_{a(s,t)+\mathrm{i}\mathbb R}dZ
\int_{-d(s,t)+\mathrm{i}\mathbb R}dW e^{P_s(Z)-P_t(W)-xZ+yW} \int_{\mathbb R}
\frac{e^{(Z-W)v}}{1+e^{Rv}}\,dv,
\end{split}
\end{equation}
\begin{equation}\label{Eq.S43KernelLimitRed2}
\begin{split}
&\lim_{N \rightarrow \infty} -\frac{{\bf 1}\{s_N > t_N\}}{2\pi\mathrm{i}}
\int_{\Gamma_{0,N}}
e^{\Psi_{s_N,N}(W)-\Psi_{t_N,N}(W)+(y_N-x_N)W}dW \\
& = -\frac{{\bf 1}\{s > t\}}{2\pi\mathrm{i}}
\int_{\mathrm{i}\mathbb R} e^{P_s(W)-P_t(W)+(y-x)W}dW,
\end{split}
\end{equation}
where we recall that $P_s(Z)$ is as in (\ref{Eq.S4CubicLimitPhase}), and the right-hand side of (\ref{Eq.S43KernelLimitRed2}) is understood to be $0$ if $s \leq t$. We prove (\ref{Eq.S43KernelLimitRed1}) and (\ref{Eq.S43KernelLimitRed2}) in the steps below. In the remainder of this step, we assume their validity and conclude the proof of the proposition.

\medskip

From (\ref{Eq.S4ContourSum}), we know $a(s,t)+d(s,t) \in (0,R)$ and so
$$\int_{\mathbb R} \frac{e^{(a(s,t)+d(s,t))v}}{1+e^{Rv}}\,dv<\infty.$$
In addition, from (\ref{Eq.S4ContourRelations}), we know $a(s,t)+s=d(s,t)-t=\frac{R+s-t}{4}>0$, and so if $\mathrm{Re}(Z) = a(s,t)$ and $\mathrm{Re}(W) = -d(s,t)$, we have
$$\left|e^{P_s(Z)-P_t(W)}\right| \leq Ce^{-c|Z|^2 - c|W|^2}.$$
The last two bounds show that we can exchange the order of the integrals in the right-hand side of (\ref{Eq.S43KernelLimitRed1}) by Fubini's theorem, and if we subsequently apply Proposition \ref{Prop.S4AiryIdentities}(i), we obtain
\begin{equation}\label{Eq.S43Red11}
\mbox{[RHS of (\ref{Eq.S43KernelLimitRed1})]} = e^{sx-s^3/3-ty+t^3/3}\int_{\mathbb{R}} \frac{e^{v(t-s)}}{1 + e^{Rv}} \mathrm{Ai}(x-v) \mathrm{Ai}(y-v)  dv.
\end{equation}
Since $P_s(W) - P_t(W) = (s-t)W^2 + (s^2 - t^2)W$, we see for $s > t$ that
$$ -\frac{1}{2\pi\mathrm{i}}
\int_{\mathrm{i}\mathbb R}
e^{P_s(W)-P_t(W)+(y-x)W}dW = -\frac{1}{\sqrt{4\pi(s-t)}} \exp\left( -\frac{(s^2-t^2+y-x)^2}{4(s-t)}\right),$$
where we used basic properties of the Fourier transform; see \cite[Chapter 5, Theorem 1.4]{SteinFourier}. Combining the last formula with (\ref{Eq.S4AirySemigroup}), we conclude that 
\begin{equation}\label{Eq.S43Red12}
\mbox{[RHS of (\ref{Eq.S43KernelLimitRed2})]}  = - {\bf 1}\{s>t\} e^{sx-s^3/3-ty+t^3/3} \int_{\mathbb R}e^{(s-t)v}\Ai(x+v)\Ai(y+v)dv,
\end{equation}
where the right-hand side is understood to be $0$ when $s \leq t$.

If we now add (\ref{Eq.S43Red11}) with $v\mapsto -v$ to (\ref{Eq.S43Red12}), and apply the identity 
$$\frac{1}{1 + e^{-Rv}} - 1 = -\frac{1}{1 + e^{Rv}},$$
we conclude, in view of (\ref{Eq.PeriodicAiryKernel}), that
\begin{equation*}
\mbox{[RHS of (\ref{Eq.S43KernelLimitRed1})]} + \mbox{[RHS of (\ref{Eq.S43KernelLimitRed2})]}  = e^{sx-s^3/3-ty+t^3/3}K^{\mathrm{per},R}(s,x;t,y).
\end{equation*}
Combining the latter with (\ref{Eq.S43KernelLimitRed1}), (\ref{Eq.S43KernelLimitRed2}), and the formula for $\widehat{K}^{\mathrm{ext},N}$ from (\ref{Eq.S4ExtendedKernel}) shows that
$$\lim_{N\to\infty} \left| h_N^{-1}\widehat{K}^{\mathrm{ext},N}(s_N,x_N;t_N,y_N)
-e^{sx-s^3/3-ty+t^3/3}K^{\mathrm{per},R}(s,x;t,y)
\right|=0.$$
Using the continuity of $K^{\mathrm{per},R}(s,\cdot;t,\cdot)$, established in Proposition \ref{Prop.S4AiryIdentities}(v), and a standard compactness argument, we see that the last statement implies (\ref{Eq.S4KernelLimit}).

\medskip

{\bf \raggedleft Step 2.} In this step, we prove (\ref{Eq.S43KernelLimitRed1}).

From (\ref{Eq.S4ContourParameters}) and (\ref{Eq.S4TimeSequences}), we see that $\lim_{N \rightarrow \infty} a_N = a(s,t)$ and $\lim_{N \rightarrow \infty} d_N = d(s,t)$. In particular, by (\ref{Eq.S4ContourRelations}) and (\ref{Eq.S4ContourSum}), we conclude
$$a_N + s_N = d_N -t_N \rightarrow \frac{R + s -t}{4}> 0 \mbox{ and } a_N+d_N \rightarrow \frac{R+t-s}{2}\in(0,R).$$
Consequently, we can find $\epsilon_1 \in (0, R)$ such that for all large $N$
\begin{equation}\label{Eq.S43ADBounds}
\epsilon_1 \leq a_N + d_N \leq R - \epsilon_1, \hspace{2mm} a_N + s_N \geq \epsilon_1, \hspace{2mm} -d_N + t_N \leq -\epsilon_1, \hspace{2mm} a_N, d_N \in [-R, R].
\end{equation}

From (\ref{Eq.S4AngularPositive}) with $x = a_N$ and (\ref{Eq.S4AngularNegative}) with $x = -d_N$ in Lemma~\ref{Lem.S4DescentContours}, applied with $\mathsf{A}_2 = R$ and $\epsilon = \epsilon_1$, we have
\begin{equation}\label{Eq.S4DoubleDominatingFunction}
\left|
e^{\Psi_{s_N,N}(a_N+\mathrm{i}z_1)
-\Psi_{t_N,N}(-d_N+\mathrm{i}z_2)
-x_N(a_N+\mathrm{i}z_1)+y_N(-d_N+\mathrm{i}z_2)}
\right|
\leq Ce^{-c(z_1^2+z_2^2)},
\end{equation}
for all sufficiently large $N$ and $z_1,z_2\in[-\pi/h_N,\pi/h_N]$. In addition, (\ref{Eq.S4KappaUniformBound}) in Lemma \ref{Lem.S4KappaEstimates}, applied with $\alpha_- = \epsilon_1$ and $\alpha_+ = R-\epsilon_1$, gives
\begin{equation}\label{Eq.S4DoubleKappaBound}
h_N\left|
\kappa_N(e^{h_N(a_N + \mathrm{i}z_1)},e^{h_N(-d_N + \mathrm{i}z_2)})
\right|\leq \mathsf{C}_3,
\end{equation}
for all sufficiently large $N$ and $z_1,z_2\in[-\pi/h_N,\pi/h_N]$. We mention that in deriving (\ref{Eq.S4DoubleDominatingFunction}) and (\ref{Eq.S4DoubleKappaBound}) we implicitly used that $x_N, y_N \in [-A,A]$ by assumption and the inequalities in (\ref{Eq.S43ADBounds}).

We next have from (\ref{Eq.S4Cubic}), (\ref{Eq.S4KappaScalingLimit}), the continuity of $P_s(Z)$ in $s$, and (\ref{Eq.S4TimeSequences}) that for each $z_1, z_2 \in \mathbb{R}$
\begin{equation*}
\begin{split}
&\lim_{N \rightarrow \infty} e^{\Psi_{s_N,N}(a_N + \mathrm{i} z_1)-\Psi_{t_N,N}(-d_N + \mathrm{i} z_2)-x_N(a_N + \mathrm{i} z_1)+y_N(-d_N + \mathrm{i} z_2)}
h_N\kappa_N(e^{h_N(a_N + \mathrm{i} z_1)},e^{h_N(-d_N + \mathrm{i} z_2)}) \\
& = e^{P_s(a(s,t) + \mathrm{i}z_1)-P_t(-d(s,t) + \mathrm{i}z_2)-x(a(s,t) + \mathrm{i}z_1)+y(-d(s,t) + \mathrm{i}z_2)} \int_{\mathbb R}
\frac{e^{(a(s,t) + d(s,t) + \mathrm{i} z_1 - \mathrm{i}z_2)v}}{1+e^{Rv}}\,dv.
\end{split}
\end{equation*}
Combining the last equation with (\ref{Eq.S4DoubleDominatingFunction}) and (\ref{Eq.S4DoubleKappaBound}) and applying the dominated convergence theorem, we obtain (\ref{Eq.S43KernelLimitRed1}).

\medskip

{\bf \raggedleft Step 3.} In this step, we prove (\ref{Eq.S43KernelLimitRed2}).

If $s \leq t$, then $s_N \leq t_N$ for all large $N$ by our assumptions, and (\ref{Eq.S43KernelLimitRed2}) holds trivially. We may thus assume that $s > t$. Since $s_N \rightarrow s$ and $t_N \rightarrow t$, we can find $\epsilon_2 > 0$ such that for all large $N$
\begin{equation}\label{Eq.S43ADBoundsV2}
s_N -t_N \geq \epsilon_2.
\end{equation}

From (\ref{Eq.S4AngularRatio}) with $x = 0$ in Lemma~\ref{Lem.S4DescentContours}, applied with $\mathsf{A}_2 = 1$ and $\epsilon = \epsilon_2$, we have
\begin{equation}\label{Eq.S4SingleDominatingFunction}
\left| e^{\Psi_{s_N,N}(\mathrm{i} z_1)-\Psi_{t_N,N}( \mathrm{i}z_1)+\mathrm{i}(y_N-x_N)z_1} \right|= \left| e^{\Psi_{s_N,N}(\mathrm{i} z_1)-\Psi_{t_N,N}( \mathrm{i}z_1)} \right|  \leq \mathsf{C}_2e^{-\mathsf{c}_2z_1^2}
\end{equation}
for all sufficiently large $N$ and $z_1 \in [-\pi/h_N,\pi/h_N]$.

We next have from (\ref{Eq.S4Cubic}), the continuity of $P_s(Z)$ in $s$, and (\ref{Eq.S4TimeSequences}) that for each $z_1 \in \mathbb{R}$
\begin{equation*}
\begin{split}
&\lim_{N \rightarrow \infty} e^{\Psi_{s_N,N}(\mathrm{i}z_1)-\Psi_{t_N,N}(\mathrm{i}z_1)+\mathrm{i}(y_N-x_N)z_1} = e^{P_s(\mathrm{i}z_1)-P_t(\mathrm{i}z_1)+\mathrm{i}(y-x)z_1}.
\end{split}
\end{equation*}
Combining the last equation with (\ref{Eq.S4SingleDominatingFunction}) and applying the dominated convergence theorem, we obtain (\ref{Eq.S43KernelLimitRed2}).
\end{proof}

%----------------------------------------------------------------------
%
% Section 4.4
%
%----------------------------------------------------------------------
\subsection{Single-time analysis}\label{Section4.4} 

We continue with the same notation as in Section \ref{Section4.1}. For $s = t_{k,N} \in \widehat{\mathcal{T}}_N$, define the point process $\widehat{M}^{s,N}$ on $\mathbb{R}$ by
\begin{equation}\label{Eq.S44SingleSliceMeasure}
\widehat{M}^{s,N}(A) = \sum_{i \geq 1}{\bf 1}\left\{\widehat{X}_i^{k,N} \in A \right\}, \quad \mbox{ Borel $A \subset \mathbb{R}$,}
\end{equation}
where we recall that $\widehat{X}_i^{k,N}$ is as in (\ref{Eq.S3ScaledParticles}).

The goal of this section is to analyze the asymptotic behavior of the measures $\widehat{M}^{s_N,N}$ for general $s_N \in \widehat{\mathcal{T}}_N \cap [0,R/2]$. The restriction to $[0,R/2]$ is mostly for convenience, and the statements below can be extended to $s_N \in \widehat{\mathcal{T}}_N$. Indeed, if $s,-s \in \widehat{\mathcal{T}}_N$, we have $\widehat{M}^{s,N} \overset{d}{=} \widehat{M}^{-s,N}$, which follows from the distributional equality
\begin{equation}\label{Eq.S44DistrEq}
\left(\lambda^{(-M_N)}, \lambda^{(-M_N+1)}, \ldots,\lambda^{(M_N)}\right) \overset{d}{=} \left(\lambda^{(M_N)}, \lambda^{(M_N-1)},\ldots,\lambda^{(-M_N)}\right)
\end{equation}
in the setup of Definition~\ref{Def.S1Scaling}, which in turn is easily seen from (\ref{Eq.S1ProbDistr}). Consequently, any statement established below for general $s_N \in \widehat{\mathcal{T}}_N \cap [0,R/2]$ can naturally be extended to $s_N \in \widehat{\mathcal{T}}_N$.

The most notable outputs of this section are the tail bounds in Lemmas~\ref{Lem.S4UpperTail} and~\ref{Lem.S4LowerTail}, which will play a key role in verifying Assumption B.3 from Section~\ref{Section2.4}.

\medskip

We begin by introducing some relevant notation that will be used throughout the section.

If $a(s,s)$ and $d(s,s)$ are as in (\ref{Eq.S4ContourParameters}), then for $s \in (-R/2,R/2]$ we have the following improvements over (\ref{Eq.S4ContourRelations}) and (\ref{Eq.S4ContourSum}):
\begin{equation}\label{Eq.S4ContourRelationsImproved}
 a(s,s)+s=d(s,s)-s=R/4>0, \quad a(s,s), d(s,s) \in [-R,R], \quad a(s,s) + d(s,s) = R/2.
\end{equation}
We further set $r_z(s,N) = e^{h_N a(s,s)}$ and $r_w(s,N) = e^{-h_Nd(s,s)}$. When $s$ and $N$ are clear from the context, we drop them from the notation and simply write $r_z, r_w$. We assume throughout that $N \geq \mathsf{N}_0$, where $\mathsf{N}_0$ is as in Definition \ref{Def.Contours}, so that by (\ref{Eq.Inequalities}) we have $r_z, r_w \in (\alpha_N, \alpha_N^{-1})$ and $1 < r_z/r_w < u_N^{-1} < \alpha_N^{-1}$.

From Lemma \ref{Lem.S4KernelAgreement}, \cite[Lemma~2.17]{ED24a}, and \cite[Proposition 2.13(6)]{ED24a}, we conclude that $\widehat{M}^{s,N}$ is a determinantal point process with reference measure $h_N \mu_{\Lambda_N}$ and correlation kernel
\begin{equation}\label{Eq.SingleSliceCorrKernel}
K^{s,N}(x,y) = h_N^{-1}\widehat{K}^{\mathrm{ext},N}(s,x;s,y), \quad x,y \in \mathbb{R},
\end{equation}
where $\widehat{K}^{\mathrm{ext},N}(s,x;s,y)$ is as in (\ref{Eq.S4ExtendedKernel}) and $\mu_{\Lambda_N}$ is the counting measure from Definition \ref{Def.S1Measures}.

\medskip

Our first result shows that for any sequence $s_N \in \widehat{\mathcal{T}}_N \cap [0,R/2]$ the random measures $\widehat{M}^{s_N,N}$ converge weakly.

\begin{lemma}\label{Lem.ConvergencePPSlice} Fix any sequence $s_N \in \widehat{\mathcal{T}}_N \cap [0,R/2]$, and let $\widehat{M}^{s_N,N}$ be as in (\ref{Eq.S44SingleSliceMeasure}). Then $\widehat{M}^{s_N,N}$ converges weakly to a determinantal point process $M^{\mathrm{Ai},R}$ on $\mathbb{R}$ with Lebesgue reference measure and correlation kernel
\begin{equation}\label{Eq.PerAiryKernelSlice}
K^{\mathrm{Ai}, R}(x,y) = \int_{\mathbb R}
 \frac{\Ai(x+v)\Ai(y+v)}{1+e^{-Rv}}\,dv, \quad x,y \in \mathbb{R}.
\end{equation}
\end{lemma}
\begin{remark}\label{Rem.ConvergencePPSlice} From (\ref{Eq.PeriodicAiryKernel}), we have $K^{\mathrm{Ai}, R}(x,y) = K^{\mathrm{per},R}(s,x;s,y)$ for all $s \in (-R/2, R/2]$.
\end{remark}
\begin{proof} From the discussion preceding the lemma, we know that $\widehat{M}^{s_N,N}$ is a determinantal point process with reference measure $h_N\mu_{\Lambda_N}$ and correlation kernel $K^{s_N,N}$ given in (\ref{Eq.SingleSliceCorrKernel}). 

Suppose first that $t_N \in \widehat{\mathcal{T}}_N \cap [0,R/2]$ satisfy $t_N \rightarrow t \in [0,R/2]$. From (\ref{Eq.S4KernelLimit}), we have for any $A > 0$
\begin{equation}\label{Eq.S44KernelLimit}
\lim_{N\to\infty} \sup_{x,y\in[-A,A]} \left| h_N^{-1}\widehat{K}^{\mathrm{ext},N}(t_N,x;t_N,y) -e^{tx-ty}K^{\mathrm{per},R}(t,x;t,y) \right|=0.
\end{equation}
In addition, $h_N\mu_{\Lambda_N}$ converges vaguely to the Lebesgue measure on $\mathbb{R}$. From \cite[Proposition 2.15]{ED24a}, we conclude that $\widehat{M}^{t_N,N}$ converges weakly to a determinantal point process $M^{\mathrm{Ai}, R, t}$ with Lebesgue reference measure and correlation kernel $e^{tx-ty}K^{\mathrm{per},R}(t,x;t,y)$. We mention that in applying \cite[Proposition 2.15]{ED24a}, we implicitly used that $(x,y) \mapsto e^{tx-ty}K^{\mathrm{per},R}(t,x;t,y)$ is continuous, which follows from Proposition \ref{Prop.S4AiryIdentities}(v).

By applying a gauge transformation, see \cite[Proposition 2.13(4)]{ED24a}, and using the identity in Remark \ref{Rem.ConvergencePPSlice}, we conclude that $M^{\mathrm{Ai}, R, t}$ is a determinantal point process with Lebesgue reference measure and correlation kernel $K^{\mathrm{Ai}, R}(x,y)$. In particular, by \cite[Proposition 2.13(3)]{ED24a}, the law of $M^{\mathrm{Ai}, R, t}$ is the same for all $t \in [0,R/2]$, and we denote this common law by $M^{\mathrm{Ai},R}$.

Going back to the original setup, suppose $s_N \in \widehat{\mathcal{T}}_N \cap [0,R/2]$ is arbitrary and $s_{N_k}$ is a subsequence. As $s_{N_k} \in [0, R/2]$ and the latter set is compact, we can find a further subsequence $s_{N_{k_m}}$ that converges to some $s \in [0,R/2]$. Our previous argument shows $\widehat{M}^{s_{N_{k_m}}, N_{k_m}} \Rightarrow M^{\mathrm{Ai},R}$. As the subsequence $s_{N_k}$ was arbitrary, we conclude $\widehat{M}^{s_N,N} \Rightarrow M^{\mathrm{Ai},R}$ as desired. 
\end{proof}

Our next result provides an upper-tail estimate for the location of the rightmost atom $\widehat{X}_1^{k,N}$ of $\widehat{M}^{s,N}$.
\begin{lemma}\label{Lem.S4UpperTail}
Assume the same notation as in the beginning of this section. We can find $\mathsf{C}_4 > 0$ depending on $b$ and $R$ such that for all $A \geq 0$
\begin{equation}\label{Eq.S44UpperTail}
\limsup_{N\to\infty}\max_{0\leq j\leq M_N} \mathbb{P}(\widehat X_1^{j,N} \geq A) \leq \mathsf{C}_4e^{-RA/2}.
\end{equation}
\end{lemma}
\begin{proof}
Fix $s=t_{j,N}\in\widehat{\mathcal T}_N\cap[0,R/2]$. There exists $C>0$, depending on $b$ and $R$, such that for all sufficiently large $N$ and every $x \in \mathbb{R}$, we have the following sequence of inequalities:
\begin{equation*}
\begin{split}
&\left|\widehat{K}^{\mathrm{ext},N}(s,x;s,x)\right| \leq \frac{h_N^2}{(2\pi)^2}
\int_{\Gamma_{a(s,s),N}}|dZ|
\int_{\Gamma_{-d(s,s),N}} \hspace{-4mm} |dW|
\left|e^{ \Psi_{s,N}(Z)-\Psi_{s,N}(W)-xZ+xW}\right| \left|\kappa_N(e^{h_NZ},e^{h_NW})\right| \\
&\leq \frac{h_N}{(2\pi)^2}
\int_{\mathbb{R}} dy_1\int_{\mathbb{R}}dy_2\mathsf{C}_2^2e^{-\mathsf{c}_2 (y_1^2 + y_2^2)} \cdot e^{-x[a(s,s) +d(s,s)]} \cdot \mathsf{C}_3 \leq C  h_N e^{-Rx/2}.
\end{split}
\end{equation*} 
Indeed, the first inequality follows from the formula for $\widehat{K}^{\mathrm{ext},N}(s,x;s,x)$ in (\ref{Eq.S4ExtendedKernel}) by putting absolute values inside the integral. Here, $|dZ|$ and $|dW|$ denote integration with respect to arc length. The second inequality follows from Lemma~\ref{Lem.S4DescentContours}, applied with $\mathsf{A}_2 = R$ and $\epsilon = R/4$, and Lemma~\ref{Lem.S4KappaEstimates}, applied with $\alpha_- = \alpha_+ =  R/2$. The last inequality used (\ref{Eq.S4ContourRelationsImproved}).

Consequently, by (\ref{Eq.S2PairMeanDPP}) and (\ref{Eq.SingleSliceCorrKernel}), we obtain
$$
\mathbb E\left[\widehat M^{s,N}([A,\infty))\right]
=\sum_{x\in\Lambda_N\cap[A,\infty)}\widehat{K}^{\mathrm{ext},N}(s,x;s,x)
\leq Ch_N\sum_{x\in\Lambda_N\cap[A,\infty)}e^{-Rx/2}
\leq \mathsf{C}_4e^{-RA/2},
$$
where $\mathsf{C}_4$ depends on $b$ and $R$, and the last inequality used $\Lambda_N=h_N(\mathbb Z-c_1N)$ from (\ref{Eq.DefS3Recalled}). 

Finally, by (\ref{Eq.S44SingleSliceMeasure}), $\{\widehat X_1^{j,N}\geq A\}\subseteq\{\widehat M^{s,N}([A,\infty))\geq1\}$, and so
$$
\mathbb P(\widehat X_1^{j,N}\geq A) \leq \mathbb E\left[\widehat M^{s,N}([A,\infty))\right] \leq \mathsf{C}_4e^{-RA/2}.
$$
Taking the maximum over $0\leq j\leq M_N$ and then $\limsup_{N\to\infty}$ proves (\ref{Eq.S44UpperTail}).
\end{proof}

We end this section with the following result, which provides lower-tail estimates for the location of the $L$-th atom $\widehat{X}_{L}^{k,N}$ of $\widehat{M}^{s,N}$.
\begin{lemma}\label{Lem.S4LowerTail} Assume the same notation as in the beginning of this section and fix $L \in \mathbb{N}$. We can find universal constants $\mathsf{c}_5, \mathsf{A}_5 > 0$ such that for all $A \geq \mathsf{A}_5$
\begin{equation}\label{Eq.S44LowerTail}
\limsup_{N\to\infty}\max_{0\leq j\leq M_N} \mathbb{P}(\widehat X_L^{j,N} <-A) \leq 2^{L-1}e^{-\mathsf{c}_5A^{3/2}}.
\end{equation}
\end{lemma}
\begin{proof} Let $k(N) \in \llbracket 0, M_N \rrbracket$ be such that 
$$\mathbb{P}(\widehat X_L^{k(N),N} <-A) = \max_{0\leq j\leq M_N} \mathbb{P}(\widehat X_L^{j,N} <-A),$$
and set $s_N = t_{k(N),N}$. From (\ref{Eq.S44SingleSliceMeasure}), we have ${\bf 1}\{\widehat X_L^{k(N),N} <-A\} \leq {\bf 1}\{\widehat{M}^{s_N, N}([-A, 0]) \leq L-1\}$. Consequently, to prove (\ref{Eq.S44LowerTail}), it suffices to show that
\begin{equation}\label{Eq.S44LowerTailRed1}
\limsup_{N\to\infty}  \mathbb{P}\left(\widehat{M}^{s_N, N}([-A, 0]) \leq L-1\right) \leq 2^{L-1}e^{-\mathsf{c}_5A^{3/2}}.
\end{equation}

From Lemma~\ref{Lem.ConvergencePPSlice}, we have $\widehat{M}^{s_N, N} \Rightarrow M^{\mathrm{Ai},R}$. Since the reference measure of $M^{\mathrm{Ai},R}$ is the Lebesgue measure, we have from (\ref{Eq.S2PairMeanDPP}) that $\mathbb{E}[M^{\mathrm{Ai},R}(\{-A,0\})] = 0$, and so $M^{\mathrm{Ai},R}(\{-A,0\}) = 0$ almost surely. By the mapping theorem \cite[Theorem 2.7]{Billing}, we conclude $\widehat{M}^{s_N, N}([-A, 0]) \Rightarrow M^{\mathrm{Ai},R}([-A, 0])$, and as these variables are integer-valued, this implies
\begin{equation}\label{Eq.S44LowerTailRed11}
\lim_{N\to\infty}  \mathbb{P}\left(\widehat{M}^{s_N, N}([-A, 0]) \leq L-1\right) =   \mathbb{P}\left(M^{\mathrm{Ai},R}([-A, 0]) \leq L-1\right).
\end{equation}

From the proof of Proposition~1.1 in \cite{Joh07}, we know that $K^{\mathrm{Ai},R}$ is locally trace-class and positive definite, and then the last displayed equation in the proof of \cite[Theorem 4(a)]{Sos00} gives 
\begin{equation}\label{Eq.S44LowerTailRed12}
\mathbb{P}\left(M^{\mathrm{Ai},R}([-A, 0]) \leq L-1\right) \leq 2^{L-1} \cdot e^{-\tfrac{1}{2} g(A)}, \mbox{ where }g(A) = \int_{-A}^0 K^{\mathrm{Ai},R}(x,x)dx.
\end{equation}

Restricting the integral in (\ref{Eq.PerAiryKernelSlice}) to $v \geq 0$ gives
\begin{equation*}
K^{\mathrm{Ai},R}(x,x) \geq \int_{0}^{\infty}
 \frac{\Ai(x+v)^2}{1+e^{-Rv}}dv \geq \frac{1}{2} \int_{0}^{\infty}
 \Ai(x+v)^2dv = \frac12\left(\Ai'(x)^2-x\Ai(x)^2\right),
\end{equation*}
where the last equality follows from \eqref{Eq.S4AirySquareIdentity}. In addition, by \eqref{Eq.S4AiryAsymptotic}--\eqref{Eq.S4AiryDerivativeAsymptotic}, we have as $x \rightarrow -\infty$
\begin{equation*}
 \Ai'(x)^2-x\Ai(x)^2 =\pi^{-1}\sqrt{-x}+O(|x|^{-1}),
\end{equation*}
and so we can find $\mathsf{A}_5 > 0$ such that
\begin{equation*}
K^{\mathrm{Ai},R}(x,x) \geq (4\pi)^{-1}\sqrt{-x} \mbox{ for } x \leq -\mathsf{A}_5/2.
\end{equation*}
Consequently, for $A\geq \mathsf{A}_5$, we have
$$g(A) =  \int_{-A}^0 K^{\mathrm{Ai},R}(x,x)dx  \geq\frac{1}{4\pi}\int_{A/2}^A\sqrt u\,du =\frac{1-2^{-3/2}}{6\pi}A^{3/2}.$$
Substituting the latter into (\ref{Eq.S44LowerTailRed12}) and applying (\ref{Eq.S44LowerTailRed11}), we conclude that (\ref{Eq.S44LowerTailRed1}) holds with $\mathsf{c}_5 = \frac{1-2^{-3/2}}{12\pi}$, completing the proof of the lemma.
\end{proof}

%----------------------------------------------------------------------
%
% Section 5
%
%----------------------------------------------------------------------
\section{Convergence of the line ensembles}\label{Section5}
The goal of this section is to prove Theorems~\ref{Thm.Convergence} and~\ref{Thm.DeterminantalStructure}. In Section~\ref{Section5.1}, we recast the point processes from Section~\ref{Section3.2} in the framework of the tightness criterion from Section~\ref{Section2.4} and establish several estimates needed for its application. The setup in Section~\ref{Section5.1} differs slightly from that of Section~\ref{Section3.2}, since we restrict the time variables to $[0,R/2]$. The reason is that $\widehat{M}^N$ from (\ref{Eq.S3ScaledPointProcess}) does not include the atoms at $(t_{-M_N,N},\widehat{X}_i^{-M_N,N})$, corresponding to the left endpoint $-R/2$. Consequently, the point process associated with the full array on $[-R/2,R/2]$ is not obtained by restricting $\widehat{M}^N$. By working on $[0,R/2]$ instead, all atoms of the corresponding array are contained in $\widehat{M}^N$, and so the determinantal structure required for applying the tightness criterion follows from the one established for $\widehat{M}^N$ by restriction. In Section~\ref{Section5.2}, we verify Assumptions B.1, B.2, and B.3 from Section~\ref{Section2.4}, obtain tightness on $[0,R/2]$, and use reflection symmetry to deduce tightness on the full interval $[-R/2,R/2]$. Finally, in Section~\ref{Section5.3}, we establish finite-dimensional convergence and complete the proofs of Theorems~\ref{Thm.Convergence} and~\ref{Thm.DeterminantalStructure}.

Throughout this section, we freely use the notation of Sections~\ref{Section3.2} and~\ref{Section4.1}.

%----------------------------------------------------------------------
%
% Section 5.1
%
%----------------------------------------------------------------------
\subsection{General setup}\label{Section5.1} In this section, we use the variables from Section~\ref{Section3.2} to construct random arrays $X^N$ to which we will apply the tightness criterion from Section~\ref{Section2.4}. Subsequently, in Section~\ref{Section5.2}, we show that they satisfy Assumptions B.1, B.2 and B.3.

We set $\alpha = 0$, $\beta = R/2$, $d_N = M_N$, $\mathcal{T}_N = \widehat{\mathcal{T}}_N \cap [0, R/2] = \{t_0^N, \dots, t_{M_N}^{N}\}$ with $t_k^N = t_{k,N} = \frac{Rk}{2M_N}$. In addition, we define
$$X^N=\left(X_i^{j,N}:i\geq 1,\ 0\leq j\leq d_N\right),$$
by setting $X_{i}^{j,N} = \widehat{X}_i^{j,N}$ as in (\ref{Eq.S3ScaledParticles}). 

From (\ref{Eq.S3ScaledParticles}), we readily observe (\ref{Eq.S24Decreasing}). In addition, if $M^N$ is as in (\ref{Eq.S24RandomMeasure}), then we have $M^N(A) = \widehat{M}^N\bigl(A \cap ([0,\infty) \times \mathbb{R}) \bigr)$. By Proposition~\ref{Prop.ScaledKernel} and \cite[Proposition 2.13(2)]{ED24a}, we conclude that $M^N$ is a determinantal point process on $\mathbb{R}^2$ with reference measure $\mu_{\mathcal{T}_N} \times \mu_{\Lambda_N}$ and correlation kernel 
\begin{equation}\label{Eq.S51KNDef}
\begin{split}
&K^N(s,x;t,y) = \widetilde{K}^N_{st}(x,y) - {\bf 1}\{s > t\} \phi_{ts}^N(y,x), \mbox{ for } x,y \in \Lambda_N \mbox{ and } s,t \in \mathcal{T}_N, \mbox{ where }\\
& \widetilde{K}^N_{st}(x,y) = \frac{1}{(2\pi\mathrm{i})^2}
\oint_{C_{r_z}}dz
\oint_{C_{r_w}}dw\,
\frac{G_{s,N}(z)}{G_{t,N}(w)}
\kappa_N(z,w)
z^{-\chi_N(x)-1}w^{\chi_N(y)-1},\\
&\phi_{uv}^N(y,x) = \frac{1}{2\pi\mathrm{i}}
\oint_{C_{1}}dw\frac{G_{v,N}(w)}{G_{u,N}(w)}
w^{\chi_N(y)-\chi_N(x)-1} \mbox{, for }x,y \in \Lambda_N, \hspace{2mm} u, v \in \mathcal{T}_N, \hspace{2mm} u < v.
\end{split} 
\end{equation}
In (\ref{Eq.S51KNDef}), we recall that $G_{s,N}$ is as in (\ref{Eq.GPochSymbol}), $\chi_N(x)=h_N^{-1}x+c_1N$, and $r_z,r_w$ are arbitrary subject to satisfying $r_z, r_w \in (\alpha_N,\alpha_N^{-1})$ and $1< r_z/r_w < \min\left(\alpha_N^{-1},u_N^{-1}\right)$. As usual, we extend $\phi_{uv}^N$, $\widetilde{K}^N_{st}$, and $K^N(s, \cdot; t,\cdot)$ to $\mathbb{R}^2$ by setting them to be equal to zero when $(x,y) \not \in \Lambda_N^2$.

\medskip

Our work so far verifies all the conditions for $X^N$ in Section~\ref{Section2.4}, up to Assumptions B.1, B.2 and B.3. We next derive alternative representations for $\phi_{uv}^N$ and $\widetilde{K}^N_{st}$, and establish several estimates that will be useful in verifying these assumptions.

Let $\mathsf{N}_1 \in \mathbb{N}$ be sufficiently large, depending on $b$ and $R$, so that for $N \geq \mathsf{N}_1$, we have 
\begin{equation}\label{Eq.S51Inequalities}
\alpha_N < e^{h_NR/8}, e^{-3h_NR/4} < \alpha_N^{-1}, \mbox{ and } 1<  e^{7h_NR/8} < e^{h_N R} \leq u_N^{-1} < \alpha_N^{-1}.
\end{equation}
Let us briefly explain why this choice of $\mathsf{N}_1$ is possible. Using $h_N = c_2^{-1}N^{-1/3} \rightarrow 0$ and $\alpha_N = be^{-1/(2N)} \rightarrow b \in (0,1)$, we get the first set of inequalities in (\ref{Eq.S51Inequalities}) for large $N$. The second set of inequalities is verified using $u_N^{-1} =  e^{\lceil c_2^{-1}R N^{2/3} \rceil/N} \geq e^{h_NR}$, and the limits $u_N \rightarrow 1$, $\alpha_N \rightarrow b$.

For $N \geq \mathsf{N}_1$, we take $r_z = e^{h_N R/8}$ and $r_w = e^{-3h_NR/4}$. As in the proof of Lemma~\ref{Lem.S4KernelAgreement}, we can change variables $z=e^{h_NZ}$, $w=e^{h_NW}$, and use (\ref{Eq.S41GInTermsOfPsi}) to conclude that, for $N \geq \mathsf{N}_1$, we have
\begin{equation}\label{Eq.S51DefPhiV2}
\phi_{uv}^N(y,x) = \frac{h_N}{2\pi\mathrm{i}}
\int_{\Gamma_{0,N}} e^{ \Psi_{v,N}(W)-\Psi_{u,N}(W)+(y-x)W}dW \mbox{, for }x,y \in \Lambda_N, \hspace{2mm} u, v \in \mathcal{T}_N, \hspace{2mm} u < v,
\end{equation}
and also that, for $x,y \in \Lambda_N$ and $s,t \in \mathcal{T}_N$,
\begin{equation}\label{Eq.S51DefKTildeV2}
\begin{split}
\widetilde{K}^N_{st}(x,y) = \frac{h_N^2}{(2\pi\mathrm{i})^2}
\int_{\Gamma_{R/8,N}}dZ
\int_{\Gamma_{-3R/4,N}}dW\,
e^{ \Psi_{s,N}(Z)-\Psi_{t,N}(W)-xZ+yW} \cdot \kappa_N(e^{h_NZ},e^{h_NW}).
\end{split}
\end{equation}
We recall that the contour $\Gamma_{a,N}$ was defined in (\ref{Eq.S4VerticalContour}).

The following result provides an exponential bound for the kernel $\widetilde{K}^N_{st}$. 
\begin{lemma}\label{Lem.S51KernelBounds}
Assume the same notation as in the beginning of this section. We can find $N_6 \in \mathbb{N}$ and $\mathsf{C}_6 > 0$, depending on $b$ and $R$, such that for $N \geq N_6$, $s,t \in \mathcal{T}_N$, and $x,y \in \Lambda_N$ we have
\begin{equation}\label{Eq.S51KernelBound}
 |\widetilde{K}^N_{st}(x,y)| \leq \mathsf{C}_{6}h_Ne^{-Rx/8 - 3Ry/4}.
\end{equation}
\end{lemma}
\begin{proof} For all sufficiently large $N$, $s,t \in \mathcal{T}_N$, and $x, y \in \Lambda_N$, we have the following sequence of inequalities:
\begin{equation*}
\begin{split}
&\left|\widetilde{K}^N_{st}(x,y)\right| \leq \frac{h_N^2}{(2\pi)^2}
\int_{\Gamma_{R/8,N}}|dZ|
\int_{\Gamma_{-3R/4,N}}|dW|
\left|e^{ \Psi_{s,N}(Z)-\Psi_{t,N}(W)-xZ+yW}\right| \left|\kappa_N(e^{h_NZ},e^{h_NW})\right| \\
&\leq \frac{h_N}{(2\pi)^2}
\int_{\mathbb{R}} dy_1\int_{\mathbb{R}}dy_2\mathsf{C}_2^2e^{-\mathsf{c}_2 (y_1^2 + y_2^2)} \cdot e^{-Rx/8 - 3Ry/4} \cdot \mathsf{C}_3,
\end{split}
\end{equation*} 
which readily implies (\ref{Eq.S51KernelBound}). Indeed, the first inequality follows from the formula for $\widetilde{K}^N_{st}$ in (\ref{Eq.S51DefKTildeV2}) by putting absolute values inside the integral. Here, $|dZ|$ and $|dW|$ denote integration with respect to arc length. The second inequality follows from Lemma~\ref{Lem.S4DescentContours}, applied with $\mathsf{A}_2 = 3R/4$ and $\epsilon = R/8$, and Lemma~\ref{Lem.S4KappaEstimates}, applied with $\alpha_- = R/8$ and $\alpha_+ =  7R/8$. 
\end{proof}

The following lemma shows that $\phi_{uv}^N(x,y)$ from (\ref{Eq.S51KNDef}) can be viewed as the transition probability of a time-inhomogeneous random walk. This interpretation reveals the probabilistic origin of the contraction identities in Assumption~A.1, and makes their verification in Section~\ref{Section5.2} especially transparent.  
\begin{lemma}\label{Lem.S51ProbabilityLemma}
Assume the same notation as in the beginning of this section. Fix $N \in \mathbb{N}$, $u, v \in \mathcal{T}_N$ with $u=t_{k,N}<v=t_{\ell,N}$, and let $Z_{j,N}$, $k\leq j<\ell$, be independent geometric random variables with
\begin{equation}\label{Eq.S51DefZ}
\mathbb{P}(Z_{j,N}=m)=(1-p_{j,N})p_{j,N}^m, \hspace{2mm} m\in\mathbb{Z}_{\geq0},\mbox{ where }
p_{j,N}=bq_N^{j+1/2}.
\end{equation}
Recalling $m_{j,N}$ from (\ref{Eq.DefS3Recalled}), we define the random variable
\begin{equation}\label{Eq.S51DDefinition}
D_{k,\ell,N}=\sum_{j=k}^{\ell-1}Z_{j,N} + m_{k,N}-m_{\ell,N}.
\end{equation}
Then, for all $w \in \mathbb{C}$ with $|w| \in (\alpha_N, \alpha_N^{-1})$, we have
\begin{equation}\label{Eq.S51MGF}
\mathbb{E}\left[|w|^{-D_{k,\ell,N}}\right] < \infty \mbox{ and }\mathbb{E}\left[w^{-D_{k,\ell,N}}\right] = \frac{G_{v,N}(w)}{G_{u,N}(w)}.
\end{equation}
Moreover, for $x,y \in \Lambda_N$, we have
\begin{equation}\label{Eq.S51PhiProbability}
\phi_{uv}^N(y,x) = \mathbb{P}\left(h_ND_{k,\ell,N}=y-x\right),
\end{equation}
and so in particular $\phi_{uv}^N(y,x) = 0$ whenever $y-x < h_N(m_{k,N} - m_{\ell,N})$, and
\begin{equation}\label{Eq.S51PhiProbabilityIneq}
0\leq \phi_{uv}^N(y,x) \leq 1, \hspace{2mm} \sum_{y' \in \Lambda_N}  \phi_{uv}^N(y',x) = \sum_{x' \in \Lambda_N}  \phi_{uv}^N(y,x') = 1.
\end{equation}
Lastly, for all $\zeta \in \mathbb{C}$ with $|\zeta| \in (\alpha_N, \alpha_N^{-1})$, we have
\begin{equation}\label{Eq.S51DiscreteContraction}
\sum_{y\in\Lambda_N}\phi^N_{uv}(y,x)\zeta^{-\chi_N(y)} =\zeta^{-\chi_N(x)}\frac{G_{v,N}(\zeta)}{G_{u,N}(\zeta)} \mbox{ and } \sum_{x\in\Lambda_N}\phi^N_{uv}(y,x)\zeta^{\chi_N(x)}=\zeta^{\chi_N(y)}\frac{G_{v,N}(\zeta)}{G_{u,N}(\zeta)},
\end{equation}
where the above sums are absolutely convergent.
\end{lemma}
\begin{proof} Notice that 
$$\mathbb{E}\left[w^{-Z_{j,N}}\right] = \sum_{m \geq 0} (1 - p_{j,N})p_{j,N}^m w^{-m} = \frac{1-p_{j,N}}{1 - p_{j,N}/w},$$
where the series is absolutely convergent for $|w| > p_{j,N}$. Using that $p_{j,N} = bq_N^{j+1/2} \leq bq_N^{1/2} = \alpha_N$, we see from the last identity that for $w \in \mathbb{C}$ with $|w| \in (\alpha_N, \alpha_N^{-1})$, we have
\begin{equation}\label{Eq.S51MGFSingle}
\mathbb{E}\left[|w|^{-Z_{j,N}}\right] = \frac{1-p_{j,N}}{1 - p_{j,N}/|w|} < \infty \mbox{ and } \mathbb{E}\left[w^{-Z_{j,N}}\right] = \frac{1-p_{j,N}}{1 - p_{j,N}/w}.
\end{equation}
Equation (\ref{Eq.S51MGF}) follows from the definition of $D_{k,\ell,N}$ in (\ref{Eq.S51DDefinition}), the independence of $Z_{j,N}$, (\ref{Eq.S51MGFSingle}) and the formula for $G_{s,N}$ in (\ref{Eq.GPochSymbol}).

\medskip

By Fubini's theorem, we have for any $n \in \mathbb{Z}$ that
\begin{equation*}
\begin{split}
&\mathbb{P}\left(D_{k,\ell,N} = n \right) = \sum_{m \in \mathbb{Z}} \frac{1}{2\pi \im} \oint_{C_1}w^{n-m-1} \mathbb{P}\left(D_{k,\ell,N} = m\right)dw \\
&=  \frac{1}{2\pi \im} \oint_{C_1} \sum_{m \in \mathbb{Z}} w^{n-m-1} \mathbb{P}\left(D_{k,\ell,N} = m\right)dw = \frac{1}{2\pi \im} \oint_{C_1} w^{n-1}\mathbb{E}\left[w^{-D_{k,\ell,N}}\right]dw.
\end{split}
\end{equation*}
Combining the latter with (\ref{Eq.S51MGF}), we conclude
\begin{equation*}
\begin{split}
&\mathbb{P}\left(h_ND_{k,\ell,N}=y-x\right) =  \frac{1}{2\pi \im} \oint_{C_1} w^{h_N^{-1}(y-x)-1}\frac{G_{v,N}(w)}{G_{u,N}(w)}dw = \phi_{uv}^N(y,x),
\end{split}
\end{equation*}
where the last equality used the formula for $\phi_{uv}^{N}(y,x)$ from (\ref{Eq.S51KNDef}).

The last equation implies (\ref{Eq.S51PhiProbability}), (\ref{Eq.S51PhiProbabilityIneq}) follows from the axioms of probability, and the statement that $\phi_{uv}^N(y,x) = 0$ whenever $y-x < h_N(m_{k,N} - m_{\ell,N})$ follows from the fact that $D_{k,\ell,N} \geq m_{k,N} - m_{\ell,N}$ almost surely.

\medskip

We finally turn to (\ref{Eq.S51DiscreteContraction}). Recalling that $\chi_N(x) = h_N^{-1} x + c_1 N$ and using (\ref{Eq.S51PhiProbability}), we have
\begin{equation*}
\sum_{y\in\Lambda_N}\phi^N_{uv}(y,x)\zeta^{\chi_N(x)-\chi_N(y)} = \sum_{y \in \Lambda_N}\mathbb{P}\left(h_ND_{k,\ell,N}=y-x\right) \zeta^{h_N^{-1}(x-y)} = \mathbb{E}\left[\zeta^{-D_{k,\ell,N}}\right].
\end{equation*}
Since $\mathbb{E}\left[|\zeta|^{-D_{k,\ell,N}}\right] < \infty$ by (\ref{Eq.S51MGF}), the above series is absolutely convergent. Multiplying the last equation by $\zeta^{-\chi_N(x)}$ and using (\ref{Eq.S51MGF}), we obtain the first equality in (\ref{Eq.S51DiscreteContraction}). 

Similarly, using (\ref{Eq.S51PhiProbability}), we have
\begin{equation*}
\sum_{x\in\Lambda_N}\phi^N_{uv}(y,x)\zeta^{\chi_N(x)-\chi_N(y)} = \sum_{x \in \Lambda_N}\mathbb{P}\left(h_ND_{k,\ell,N}=y-x\right) \zeta^{h_N^{-1}(x-y)} = \mathbb{E}\left[\zeta^{-D_{k,\ell,N}}\right].
\end{equation*}
As above, this series is absolutely convergent by (\ref{Eq.S51MGF}). Multiplying the last equation by $\zeta^{\chi_N(y)}$ and using (\ref{Eq.S51MGF}), we obtain the second equality in (\ref{Eq.S51DiscreteContraction}). 
\end{proof}

The last result of this section provides bounds for certain sums involving $\phi_{uv}^N(x,y)$, see (\ref{Eq.S51Moments}), which will play a key role in verifying Assumption B.2 in Section~\ref{Section5.2}. By utilizing the probabilistic interpretation of $\phi_{uv}^N(x,y)$ from Lemma~\ref{Lem.S51ProbabilityLemma}, the inequalities in (\ref{Eq.S51Moments}) can be seen as moment bounds for the increments of the random walks that were introduced in Lemma~\ref{Lem.S51ProbabilityLemma}.  
\begin{lemma}\label{Lem.S51MomentInputs}
Assume the same notation as in the beginning of this section. For any $Q > 0$, we can find $N_7 \in \mathbb{N}$ and $\mathsf{C}_7 > 0$, depending on $b$, $R$, and $Q$, such that for $N \geq N_7$, $u,v \in \mathcal{T}_N$ with $u < v$, $x\in\Lambda_N$, and $r \in \llbracket 0, 4 \rrbracket$, we have
\begin{equation}\label{Eq.S51Moments}
\sum_{y\in\Lambda_N}\left|\phi^N_{uv}(y,x)\right| |y-x|^re^{Q|y-x|}
\leq \mathsf{C}_{7}(v-u)^{r/2}.
\end{equation}
\end{lemma}

\begin{proof} Fix $Q>0$. Throughout the proof, we write $C,c$ for generic positive constants that depend on $b$, $R$, and $Q$, and whose values may change from line to line. Likewise, all constants in the $O$ notations may depend on $b$, $R$, and $Q$. Finally, for two sequences $\{a_N\}$ and $\{b_N\}$, we write $a_N\asymp b_N$ if there exist $C,c > 0$ such that $ca_N\leq b_N\leq Ca_N$. We do not mention this further.

Let $k, \ell \in \llbracket 0, M_N \rrbracket$ be such that $u=t_{k,N}<v=t_{\ell,N}$, and define $Z_{j,N}$ for $k \leq j < \ell$ and $D_{k,\ell,N}$ as in Lemma~\ref{Lem.S51ProbabilityLemma}. From (\ref{Eq.S51PhiProbability}), we have
\begin{equation}\label{Eq.S51MomentsP1}
\sum_{y\in\Lambda_N}\left|\phi^N_{uv}(y,x)\right| |y-x|^re^{Q|y-x|} = \mathbb{E}\left[|h_ND_{k,\ell,N}|^r
e^{Qh_N|D_{k,\ell,N}|}\right],
\end{equation}
and so it suffices to bound the right-hand side of (\ref{Eq.S51MomentsP1}).

\medskip

From (\ref{Eq.DefS1Recalled}) and (\ref{Eq.DefS3Recalled}), we have $h_N \asymp N^{-1/3}$, $(t_{\ell,N}- t_{k,N}) \asymp (\ell - k)N^{-2/3}$, and also
\begin{equation}\label{Eq.S51DeltaComparison}
(v-u) \asymp h_N^2 (\ell -k), \quad h_N^2 \leq C(v-u), \quad u,v \in [0, R/2], \quad  0 \leq \ell, k \leq C N^{2/3}.
\end{equation}
In addition, from (\ref{Eq.DefS1Recalled}), we have $q_N=e^{-1/N}$ and $M_N \asymp N^{2/3}$, which together with (\ref{Eq.S51DefZ}) show for all $j \in \llbracket 0, M_N \rrbracket$ that 
\begin{equation*}
\mathbb{E}\left[Z_{j,N}\right]=\frac{p_{j,N}}{1-p_{j,N}} = \frac{b}{1-b} + O(N^{-1/3}).
\end{equation*}
Consequently, using that $m_{j,N} = \frac{bj}{1-b}-\frac{bj^2}{4(1-b)^2N} + O(1)$ for $j \in \{k, \ell\}$ from (\ref{Eq.DefS3Recalled}), and (\ref{Eq.S51DDefinition}), we conclude
\begin{equation*}
\mathbb{E}\left[D_{k,\ell,N}\right]=\sum_{j = k}^{\ell-1} \mathbb{E}\left[Z_{j,N}\right] + m_{k,N} - m_{\ell,N} = O\left(1 + \frac{\ell-k}{N^{1/3}} + \frac{\ell^2 - k^2}{N} \right).
\end{equation*}
The last equation and the inequalities $h^2_N \leq C(v-u)$, $\ell^2 - k^2 \leq C(\ell-k)N^{2/3}$, and $h_N(\ell-k)N^{-1/3} \asymp (v-u) \leq C\sqrt{v-u}$, which follow from (\ref{Eq.S51DeltaComparison}), together imply
\begin{equation}\label{Eq.S51MeanBound}
h_N\left|\mathbb{E}\left[D_{k,\ell,N}\right]\right| \leq C \sqrt{v -u}.
\end{equation}

Since $p_{j,N}\leq b<1$, the geometric random variables in (\ref{Eq.S51DefZ}) have uniformly bounded moments of every fixed order. Consequently, independence and centering give
$$
\mathbb{E}\left[\left(D_{k,\ell,N}-\mathbb{E}\left[D_{k,\ell,N}\right]\right)^{2r}\right] \leq C(\ell-k)^r, \quad 0 \leq r \leq 4.
$$
Combining the latter with $h_N^2(\ell - k) \asymp (v-u)$ from (\ref{Eq.S51DeltaComparison}), (\ref{Eq.S51MeanBound}), and the trivial inequality $(A+B)^{2r} \leq 2^{2r}A^{2r} + 2^{2r}B^{2r}$, we get
\begin{equation}\label{Eq.S51MomentBound}
\mathbb{E}\left[|h_ND_{k,\ell,N}|^{2r}\right]\leq C(v-u)^r,
\quad 0\leq r\leq4.
\end{equation}

Moreover, since $p_{j,N}\leq b<1$, the explicit geometric moment generating function in (\ref{Eq.S51MGFSingle}) implies  
$$
\log\mathbb{E}\left[\exp\left(\theta\left(Z_{j,N}-\mathbb{E}\left[Z_{j,N}\right]\right)\right)\right]\leq C\theta^2
$$
for all $j\in\llbracket 0,M_N\rrbracket$ and $|\theta| \leq -(1/2) \log(b)$. Increasing $N$ if necessary depending on $Q$ so that $2Qh_N \leq -(1/2) \log(b)$ and using the last inequality and independence, we conclude
\begin{equation*}
\begin{split}
\mathbb{E}\left[\exp\left(\pm 2Qh_N\left(D_{k,\ell,N}-\mathbb{E}\left[D_{k,\ell,N}\right]\right)\right)\right]
&\leq \exp\left(C(\ell-k)h_N^2\right).
\end{split}
\end{equation*}
From the last inequality, (\ref{Eq.S51DeltaComparison}), and the trivial inequality $e^{|x|}\leq e^x+e^{-x}$, it follows that
$$
\mathbb{E}\left[\exp\left(2Qh_N\left|D_{k,\ell,N}-\mathbb{E}\left[D_{k,\ell,N}\right]\right|\right)\right]\leq C.
$$
Combining this with (\ref{Eq.S51DeltaComparison}) and (\ref{Eq.S51MeanBound}) gives
\begin{equation}\label{Eq.S51ExponentialMoment}
\mathbb{E}\left[\exp\left(2Qh_N|D_{k,\ell,N}|\right)\right]\leq C.
\end{equation}

Finally, by Cauchy--Schwarz, (\ref{Eq.S51MomentBound}), and (\ref{Eq.S51ExponentialMoment}), we obtain
\begin{equation*}
\begin{split}
\mathbb{E}\left[|h_ND_{k,\ell,N}|^r e^{Qh_N|D_{k,\ell,N}|}\right]\leq
\mathbb{E}\left[|h_ND_{k,\ell,N}|^{2r}\right]^{1/2}
\mathbb{E}\left[e^{2Qh_N|D_{k,\ell,N}|}\right]^{1/2} \leq C(v-u)^{r/2}.
\end{split}
\end{equation*}
Together with (\ref{Eq.S51MomentsP1}), this proves (\ref{Eq.S51Moments}).
\end{proof}

%----------------------------------------------------------------------
%
% Section 5.2
%
%----------------------------------------------------------------------
\subsection{Verification of the tightness criterion}\label{Section5.2}
The goal of this section is to use the results of Sections~\ref{Section4} and~\ref{Section5.1} to verify that the random arrays $X^N$ introduced in Section~\ref{Section5.1} satisfy Assumptions B.1, B.2, and B.3. Throughout this section, when we say that a statement holds for all sufficiently large $N$, we mean sufficiently large depending only on $b$ and $R$.

%----------------------------------------------------------------------
%
% Section 5.2.1
%
%----------------------------------------------------------------------
\subsubsection{Verifying Assumption B.1}\label{Section5.2.1} Fix $u, v \in \mathcal{T}_N$ with $u < v$, and let $\phi_{uv}^N$ and $\widetilde{K}^N_{st}$ for $s,t \in \{u,v\}$ be as in (\ref{Eq.S51KNDef}).

From (\ref{Eq.S51KNDef}) we know that $\phi_{uv}^N(x,y)$ and $\widetilde{K}^N_{st}(x,y)$ are well-defined and finite for $x,y \in \Lambda_N$, and equal to zero when $\{x,y\} \not \subset \Lambda_N$. This shows that $\phi_{uv}^N$ and $\widetilde{K}^N_{st}$ are locally bounded and measurable, confirming the first condition in Assumption A.1. The required kernel decomposition in (\ref{Eq.FourthMomentKernelDecomposition}) follows from the first line in (\ref{Eq.S51KNDef}).

We next have from (\ref{Eq.S51KNDef}) and the second identity in (\ref{Eq.S51DiscreteContraction}), with $x = y'$ and $\zeta = w$, that
\begin{equation*}
\begin{split}
&\sum_{y' \in \Lambda_N} \widetilde{K}^N_{uv}(x,y')\phi^N_{uv}(y,y') \\
& = \frac{1}{(2\pi\mathrm{i})^2}
\oint_{C_{r_z}}dz
\oint_{C_{r_w}}dw
\sum_{y' \in \Lambda_N}\frac{G_{u,N}(z)}{G_{v,N}(w)}
\kappa_N(z,w)
z^{-\chi_N(x)-1}w^{\chi_N(y')-1}\phi_{uv}^N(y,y')\\
& = \frac{1}{(2\pi\mathrm{i})^2}
\oint_{C_{r_z}}dz
\oint_{C_{r_w}}dw
\frac{G_{u,N}(z)}{G_{u,N}(w)}
\kappa_N(z,w)
z^{-\chi_N(x)-1}w^{\chi_N(y)-1} = \widetilde{K}^N_{uu}(x,y),
\end{split}
\end{equation*}
verifying (\ref{Eq.FourthMomentContractionUU}) for $x,y \in \Lambda_N$. We mention that the absolute convergence in (\ref{Eq.S51DiscreteContraction}) and the fact that $r_w \in (\alpha_N, \alpha_N^{-1})$ by assumption show that the above series is absolutely convergent and justify the exchange of the order of the sum and the integrals. We also mention that when $\{x,y\} \not \subset \Lambda_N$, (\ref{Eq.FourthMomentContractionUU}) holds as both sides equal zero by our convention $\phi_{uv}^N(x,y) = \widetilde{K}^N_{st}(x,y) = 0$.

Similarly, we have from (\ref{Eq.S51KNDef}) and the first identity in (\ref{Eq.S51DiscreteContraction}), with $y = y'$ and $\zeta = z$, that
\begin{equation*}
\begin{split}
&\sum_{y' \in \Lambda_N}\phi^N_{uv}(y',x)\widetilde{K}^N_{uv}(y',y) \\
& = \frac{1}{(2\pi\mathrm{i})^2}
\oint_{C_{r_z}}dz
\oint_{C_{r_w}}dw
\sum_{y' \in \Lambda_N}\frac{G_{u,N}(z)}{G_{v,N}(w)}
\kappa_N(z,w)
z^{-\chi_N(y')-1}w^{\chi_N(y)-1}\phi_{uv}^N(y',x)\\
& = \frac{1}{(2\pi\mathrm{i})^2}
\oint_{C_{r_z}}dz
\oint_{C_{r_w}}dw
\frac{G_{v,N}(z)}{G_{v,N}(w)}
\kappa_N(z,w)
z^{-\chi_N(x)-1}w^{\chi_N(y)-1} = \widetilde{K}^N_{vv}(x,y),
\end{split}
\end{equation*}
verifying (\ref{Eq.FourthMomentContractionVV}) for $x,y \in \Lambda_N$. As before, the absolute convergence in (\ref{Eq.S51DiscreteContraction}) and the fact that $r_z \in (\alpha_N, \alpha_N^{-1})$ by assumption show that the above series is absolutely convergent and justify the exchange of the order of the sum and the integrals. When $\{x,y\} \not \subset \Lambda_N$, (\ref{Eq.FourthMomentContractionVV}) holds as both sides equal zero by our convention $\phi_{uv}^N(x,y) = \widetilde{K}^N_{st}(x,y) = 0$.

Lastly, we have from (\ref{Eq.S51KNDef}), the first identity in (\ref{Eq.S51DiscreteContraction}), with $y = x'$ and $\zeta = z$, and the second identity in (\ref{Eq.S51DiscreteContraction}), with $x = y'$ and $\zeta = w$, that
\begin{equation*}
\begin{split}
&\sum_{x',y' \in \Lambda_N} \phi^N_{uv}(x',x)\widetilde K^N_{uv}(x',y')\phi^N_{uv}(y,y')\\
& = \frac{1}{(2\pi\mathrm{i})^2}
\oint_{C_{r_z}}dz
\oint_{C_{r_w}}dw
\sum_{x',y' \in \Lambda_N}\frac{G_{u,N}(z)}{G_{v,N}(w)}
\kappa_N(z,w)
z^{-\chi_N(x')-1}w^{\chi_N(y')-1}\phi_{uv}^N(x',x)\phi_{uv}^N(y,y')\\
& = \frac{1}{(2\pi\mathrm{i})^2}
\oint_{C_{r_z}}dz
\oint_{C_{r_w}}dw
\frac{G_{v,N}(z)}{G_{u,N}(w)}
\kappa_N(z,w)
z^{-\chi_N(x)-1}w^{\chi_N(y)-1} = \widetilde{K}^N_{vu}(x,y),
\end{split}
\end{equation*}
verifying (\ref{Eq.FourthMomentContractionVU}) for $x,y \in \Lambda_N$. As before, the absolute convergence in (\ref{Eq.S51DiscreteContraction}) and the fact that $r_w, r_z \in (\alpha_N, \alpha_N^{-1})$ by assumption show that the above series is absolutely convergent and justify the exchange of the order of the sum and the integrals. When $\{x,y\} \not \subset \Lambda_N$, (\ref{Eq.FourthMomentContractionVU}) holds as both sides equal zero by our convention $\phi_{uv}^N(x,y) = \widetilde{K}^N_{st}(x,y) = 0$.

The work so far shows that $\phi_{uv}^N$ and $\widetilde{K}^N_{st}$ for $s,t \in \{u,v\}$ satisfy Assumption A.1.

\medskip

We next verify the conditions in Assumption A.2. Fix $f\in\operatorname{Lip}_{b,-}(\mathbb{R})$, $m \in \llbracket 1, 4\rrbracket$, positive integers $k_1, \dots, k_m$ with $k_1 + \cdots + k_m = 4$, and a permutation $\sigma \in \mathfrak{S}_m$. Let $\rho \in \mathbb{R}$ be such that $f(y) = 0$ for $y \leq \rho$. We observe the following tower of inequalities for all large $N$:
\begin{equation*}
\begin{split}
&\sum_{x_1, \dots, x_m, z_1, \dots, z_m \in \Lambda_N}
\prod_{i=1}^m
\Bigl[
|\phi^N_{uv}(z_i,x_i)|
\bigl(|f(z_i)|+|f(x_i)|\bigr)^{k_i}
|\widetilde K^N_{uv}(z_i,x_{\sigma(i)})|
\Bigr]\\
&\leq \mathsf{C}_6^m\prod_{i = 1}^m\left(h_N\sum_{x, z \in \Lambda_N}
|\phi^N_{uv}(z,x)|
\bigl(|f(z)|+|f(x)|\bigr)^{k_i}e^{-\frac{Rz}{8} - \frac{3Rx}{4}}\right)\\
&\leq 2^4\|f\|_{\infty}^4\mathsf{C}_6^m\prod_{i = 1}^m\left(h_N\sum_{x,z \in \Lambda_N}|\phi^N_{uv}(z,x)|
\bigl({\bf 1}\{x > \rho\} + {\bf 1}\{z > \rho\}\bigr)e^{\frac{R|z-x|}{8}}e^{-\frac{7Rx}{8}}\right)\\
& \leq  2^4\|f\|_{\infty}^4\mathsf{C}_6^m\prod_{i = 1}^m\left(h_N\sum_{x \in \Lambda_N}e^{-\frac{7Rx}{8}} 
\bigl(2\cdot {\bf 1}\{x > \rho\} + e^{-\frac{15R(\rho-x)}{8}}{\bf 1}\{x \leq \rho\}\bigr) \sum_{z \in \Lambda_N} |\phi^N_{uv}(z,x)|e^{2R|z-x|}\right) \\
& \leq 2^4\|f\|_{\infty}^4\mathsf{C}_6^m\mathsf{C}_7^m\prod_{i = 1}^m\left(h_N\sum_{x \in \Lambda_N}e^{-\frac{7Rx}{8}} 
\bigl(2\cdot {\bf 1}\{x > \rho\} + e^{-\frac{15R(\rho-x)}{8}}{\bf 1}\{x \leq \rho\}\bigr)\right) \\
&  = 2^4\|f\|_{\infty}^4\mathsf{C}_6^m\mathsf{C}_7^m\prod_{i=1}^m
\left(2h_N \sum_{x\in\Lambda_N: x>\rho} e^{-\frac{7Rx}{8}} +h_Ne^{-\frac{15R\rho}{8}}\sum_{x\in\Lambda_N:x\leq\rho}e^{Rx}\right) <\infty.
\end{split}
\end{equation*}
Let us elaborate on the above inequalities briefly. The first inequality follows from Lemma~\ref{Lem.S51KernelBounds} with $\mathsf{C}_6$ as in that lemma. The second inequality follows from the fact that $k_1 + \cdots + k_m = 4$, $f(y) = 0$ for $y \leq \rho$, and the trivial inequalities
$$\bigl(|f(z)|+|f(x)|\bigr)^{k_i} \leq 2^{k_i}\|f\|_{\infty}^{k_i}\bigl({\bf 1}\{x > \rho\} + {\bf 1}\{z > \rho\}\bigr) \mbox{ and } -3Rx/4 - Rz/8 \leq -7Rx/8 + R|z-x|/8.$$
The third inequality follows from 
\begin{equation*}
\begin{split}
&\bigl({\bf 1}\{x > \rho\} + {\bf 1}\{z > \rho\}\bigr)e^{\frac{R|z-x|}{8}} \leq 2 \cdot {\bf 1}\{x > \rho\}e^{\frac{R|z-x|}{8}} + {\bf 1}\{z > \rho, x \leq \rho\}e^{\frac{R|z-x|}{8}}  \\
& \leq 2 \cdot {\bf 1}\{x > \rho\}e^{2R|z-x|} + {\bf 1}\{z > \rho, x \leq \rho\}e^{2R|z-x|- \frac{15R(\rho-x)}{8}}.
\end{split}
\end{equation*}
The fourth inequality follows from the $r = 0$ case of (\ref{Eq.S51Moments}) in Lemma~\ref{Lem.S51MomentInputs}, where $\mathsf{C}_7$ is as in that lemma for $Q = 2R$.

The last displayed tower of inequalities implies (\ref{Eq.FourthMomentAbsoluteIntegrability}), which verifies Assumption A.2 and hence Assumption B.1.

%----------------------------------------------------------------------
%
% Section 5.2.2
%
%----------------------------------------------------------------------
\subsubsection{Verifying Assumption B.2}\label{Section5.2.2} Fix $u, v \in \mathcal{T}_N$ with $u < v$, and let $\phi_{uv}^N$ and $\widetilde{K}^N_{st}$ for $s,t \in \{u,v\}$ be as in (\ref{Eq.S51KNDef}). In addition, fix $f\in\operatorname{Lip}_{b,-}(\mathbb{R})$, and let $\rho \in \mathbb{R}$ be such that $f(y) = 0$ for $y\leq \rho$, and put 
\begin{equation}\label{Eq.S52LipschitzConst}
\|f\|_{\mathrm{Lip}} = \sup_{x \neq y} \left|\frac{f(x) - f(y)}{x - y} \right|.
\end{equation}
In addition, we let $\mathsf{C}_6$ be as in Lemma~\ref{Lem.S51KernelBounds} and $\mathsf{C}_7$ be as in Lemma~\ref{Lem.S51MomentInputs} for $Q = R$.

With the above data, we proceed to define the functions $U_{u,v}^{N,f}$, $V_{u,v}^{N,f}$, and the constants $S(f)$ and $B_k(f)$ for $k \in \llbracket 1, 4\rrbracket$, for which we will verify Assumption B.2.

\medskip

We define
\begin{equation}\label{Eq.S52UVDef}
\begin{split}
U_{u,v}^{N,f}(x)&=h_Ne^{-\frac{Rx}{8}-R(\rho-x)_+}, \qquad V_{u,v}^{N,f}(x)=e^{-\frac{3Rx}{4}}.
\end{split}
\end{equation}
From (\ref{Eq.S52UVDef}) we can find $\mathsf{C}_8 > 0$, depending on $R$, such that for all large $N$
\begin{equation}\label{Eq.S52C8Def}
\sum_{x \in \Lambda_N}U_{u,v}^{N,f}(x) V_{u,v}^{N,f}(x) = e^{-\frac{7R\rho}{8}}\left(\sum_{\substack{x \in \Lambda_N\\ x \leq \rho}} h_Ne^{-\frac{R(\rho-x)}{8} } + \sum_{\substack{x \in \Lambda_N\\ x > \rho}}h_Ne^{-\frac{7R(x-\rho)}{8}}\right) \leq \mathsf{C}_8 e^{-\frac{7R\rho}{8}}.
\end{equation}
With the constants $\mathsf{C}_6$, $\mathsf{C}_7$, and $\mathsf{C}_8$ specified above, we define
\begin{equation}\label{Eq.S52SBDef}
\begin{split}
S(f) = \mathsf{C}_8 e^{-7R\rho/8} \mbox{ and } B_k(f) = 2\mathsf{C}_6\mathsf{C}_7\|f\|^k_{\mathrm{Lip}} \mbox{ for } 1 \leq k \leq 4.
\end{split}
\end{equation}

For any $k \in \llbracket 1, 4 \rrbracket$, we observe the following tower of inequalities for all large $N$:
\begin{equation}\label{Eq.S52Tower2}
\begin{split}
&\sum_{x' \in \Lambda_N} |\phi^N_{uv}(x',x)| |f(x')-f(x)|^k|\widetilde K^N_{uv}(x',y)|\\
&\leq \mathsf{C}_6h_N\|f\|_{\mathrm{Lip}}^k e^{-\frac{3Ry}{4}}
\sum_{x'\in\Lambda_N} |\phi^N_{uv}(x',x)| |x'-x|^k \bigl({\bf 1}\{x>\rho\}+{\bf 1}\{x'>\rho\}\bigr)e^{-\frac{Rx'}{8}}\\
&\leq 2\mathsf{C}_6h_N\|f\|_{\mathrm{Lip}}^ke^{-\frac{Rx}{8}-R(\rho-x)_+}e^{-\frac{3Ry}{4}} \sum_{x'\in\Lambda_N}|\phi^N_{uv}(x',x)| |x'-x|^k e^{R|x'-x|}\\
&\leq 2\mathsf{C}_6\mathsf{C}_7\|f\|_{\mathrm{Lip}}^k
(v-u)^{k/2} h_Ne^{-\frac{Rx}{8}-R(\rho-x)_+} e^{-\frac{3Ry}{4}} \leq B_k(f)(v-u)^{k/2} U_{u,v}^{N,f}(x)V_{u,v}^{N,f}(y).
\end{split}
\end{equation}
Let us elaborate on the above inequalities briefly. The first inequality follows from Lemma~\ref{Lem.S51KernelBounds} with $\mathsf{C}_6$ as in that lemma, (\ref{Eq.S52LipschitzConst}), and the fact that $f(x)=f(x')=0$ whenever $x,x'\leq\rho$. The second inequality follows from
\begin{equation*}
\begin{split}
&\bigl({\bf 1}\{x>\rho\}+{\bf 1}\{x'>\rho\}\bigr)e^{-\frac{Rx'}{8}}\leq
2\cdot{\bf 1}\{x>\rho\}e^{-\frac{Rx'}{8}}
+
{\bf 1}\{x'>\rho,x\leq\rho\}e^{-\frac{Rx'}{8}}\\
&\leq
2\cdot{\bf 1}\{x>\rho\}e^{-\frac{Rx}{8}}e^{\frac{R|x'-x|}{8}}
+
{\bf 1}\{x'>\rho,x\leq\rho\}
e^{-\frac{Rx}{8}-R(\rho-x)}e^{R|x'-x|}\\
&\leq
2e^{-\frac{Rx}{8}-R(\rho-x)_+}e^{R|x'-x|}.
\end{split}
\end{equation*}
The third inequality follows from the $r=k$ case of (\ref{Eq.S51Moments}) in Lemma~\ref{Lem.S51MomentInputs}, with $Q=R$ and $\mathsf{C}_7$ as in that lemma. The fourth inequality follows from (\ref{Eq.S52UVDef}) and (\ref{Eq.S52SBDef}).

Combining (\ref{Eq.S52UVDef}), (\ref{Eq.S52C8Def}), and (\ref{Eq.S52SBDef}), we conclude (\ref{Eq.FourthMomentEnvelopeIntegrability}), while (\ref{Eq.S52Tower2}) implies (\ref{Eq.FourthMomentDifferenceKernelBounds}). This verifies the conditions of Assumption B.2.

\medskip

%----------------------------------------------------------------------
%
% Section 5.2.3
%
%----------------------------------------------------------------------
\subsubsection{Verifying Assumption B.3}\label{Section5.2.3} Continue with the same notation as in Section~\ref{Section5.2.2}. From (\ref{Eq.S52SBDef}), we can find $\mathsf{C}_9 >0$, depending on $b$ and $R$, such that 
\begin{equation*}
\mathfrak{C}(f) \leq \mathsf{C}_9 (1 + \|f\|_{\mathrm{Lip}}^4)\left(1 + e^{-7R\rho/2} \right),
\end{equation*}
where we recall that $\mathfrak{C}(f)$ is as in (\ref{Eq.FourthMomentConstant}). Consequently, if $g_\varsigma^+$ and $g_\varsigma^-$ are as in (\ref{Eq.TranslatedCutoffFunctions}), we have
\begin{equation*}
\mathfrak{C}(g_\varsigma^+) \leq 4\mathsf{C}_9\mbox{ and } \mathfrak{C}(g_\varsigma^-) \leq 2\mathsf{C}_9 \left(1 + e^{7R(\varsigma + 1)/2} \right).
\end{equation*}
On the other hand, if $p_+(\varsigma)$ and $p_{k,-}(\varsigma)$ are as in (\ref{Eq.UpperTailProbability}) and (\ref{Eq.LowerTailProbability}), respectively, we have from Lemmas~\ref{Lem.S4UpperTail} and~\ref{Lem.S4LowerTail} that, for all sufficiently large $\varsigma$, 
\begin{equation*}
p_+(\varsigma) \leq \mathsf{C}_4e^{-R\varsigma/2} \mbox{ and } p_{k,-}(\varsigma) \leq 2^{k-1}e^{-\mathsf{c}_5\varsigma^{3/2}},
\end{equation*}
where the constants $\mathsf{C}_4$ and $\mathsf{c}_5$ are as in those lemmas.

The last two displayed inequalities imply
$$p_+(\varsigma)\mathfrak{C}(g_\varsigma^+) \rightarrow 0 \mbox{ and }p_{k,-}(\varsigma)\mathfrak{C}(g_\varsigma^-) \rightarrow 0, \qquad\text{as }\varsigma\to\infty,$$
verifying (\ref{Eq.UpperTailCompatibility}) and (\ref{Eq.LowerTailCompatibility}), and hence Assumption B.3.

%----------------------------------------------------------------------
%
% Section 5.2.4
%
%----------------------------------------------------------------------
\subsubsection{Applying the tightness criterion}\label{Section5.2.4}

In this section, we use the conclusions obtained above to deduce the following tightness statement. In what follows, we freely use the notation from Sections~\ref{Section2.2}, \ref{Section2.4}, and~\ref{Section5.1}.

\begin{proposition}\label{Prop.S52Tightness} Assume the same notation as in Definition~\ref{Def.S1Scaling} and (\ref{Eq.S3Parameters}), and let $(\widehat{X}_i^{k,N}: i\geq 1, k \in \llbracket -M_N, M_N \rrbracket)$ be as in (\ref{Eq.S3ScaledParticles}). For each $N$, define the $\mathbb{N}$-indexed line ensemble $\widehat{\mathcal{L}}^N=\{\widehat{\mathcal{L}}_i^N\}_{i\geq 1}$ on $[-R/2,R/2]$ by linearly interpolating the points $\{(t_{k,N},\widehat{X}_i^{k,N})\}_{k= -M_N}^{M_N}$. Then the sequence $\widehat{\mathcal{L}}^N$ is tight in $C(\mathbb{N} \times [-R/2, R/2])$. 
\end{proposition}
\begin{proof} Using the notation from Section~\ref{Section2.2}, we see that the restricted line ensemble $\widehat{\mathcal{L}}^N|_{\mathbb{N}\times [0,R/2]}$ is precisely the line ensemble $\mathcal{L}^N$ from Section~\ref{Section2.4} corresponding to the array $X^N$ from Section~\ref{Section5.1}.

By Sections~\ref{Section5.1}, \ref{Section5.2.1}, \ref{Section5.2.2}, and~\ref{Section5.2.3}, we know that the sequence of random arrays $X^N$ satisfies Assumptions B.1, B.2, and B.3. From Lemma~\ref{Lem.DPPTightnessCriterion}, we conclude that the sequence $\widehat{\mathcal{L}}^N|_{\mathbb{N}\times [0,R/2]}$ is tight in $C(\mathbb{N} \times [0, R/2])$.

Using the distributional equality in (\ref{Eq.S44DistrEq}) and (\ref{Eq.S3ScaledParticles}), we conclude
\begin{equation*}
 \bigl(\widehat{\mathcal L}_i^N(t):i\geq1,\,-R/2\leq t\leq R/2\bigr)
 \overset{d}{=}
 \bigl(\widehat{\mathcal L}_i^N(-t):i\geq1,\,-R/2\leq t\leq R/2\bigr).
\end{equation*}
The established tightness of $\widehat{\mathcal{L}}^N|_{\mathbb{N}\times [0,R/2]}$ then readily implies that $\widehat{\mathcal{L}}^N|_{\mathbb{N}\times [-R/2,0]}$ is tight in $C(\mathbb{N} \times [-R/2,0])$, and hence that $\widehat{\mathcal{L}}^N$ is tight in $C(\mathbb{N} \times [-R/2, R/2])$, for instance by the tightness criterion in \cite[Lemma 2.4]{DEA21}.
\end{proof}

%----------------------------------------------------------------------
%
% Section 5.3
%
%----------------------------------------------------------------------
\subsection{Proofs of Theorems~\ref{Thm.Convergence} and~\ref{Thm.DeterminantalStructure}}\label{Section5.3}

In this section, we prove Theorems~\ref{Thm.Convergence} and~\ref{Thm.DeterminantalStructure} simultaneously. In what follows, we assume the notation from Section~\ref{Section3.2} and we let $\widehat{\mathcal{L}}^N=\{\widehat{\mathcal{L}}_i^N\}_{i\geq 1}$ denote the $\mathbb{N}$-indexed line ensemble on $[-R/2,R/2]$ from Proposition~\ref{Prop.S52Tightness}. 

For clarity, we split the proof into two steps. In Step 1, we reduce the statements of the theorems to a finite-dimensional convergence statement for the line ensembles $\widehat{\mathcal{L}}^N$, see (\ref{Eq.S53FiniteDim}). In Step 2, we establish (\ref{Eq.S53FiniteDim}) using the determinantal point process convergence results from \cite[Section~2.3]{ED24a}. The required conditions for the application of these results will be verified using the kernel convergence from Proposition~\ref{Prop.S4KernelLimit} and the tightness of $\widehat{\mathcal{L}}^N$ from Proposition~\ref{Prop.S52Tightness}.

\medskip

{\bf \raggedleft Step 1.} Fix $m \in \mathbb{N}$, $-R/2 < s_1 < \cdots < s_m \leq R/2$ and set $\mathsf{S} = \{s_1, \dots, s_m\}$. In addition, let $s_{j,N} \in \widehat{\mathcal{T}}_N$ be sequences such that $\lim_{N \rightarrow \infty} s_{j,N} = s_j$ for $j \in \llbracket 1, m \rrbracket$. We claim that there exists a random element $Y = \left(Y_{i}^{j}: i \geq 1, j \in \llbracket 1, m \rrbracket  \right) \in \mathbb{R}^{\infty}$, see \cite[Example 1.2]{Billing}, such that
\begin{equation}\label{Eq.S53FiniteDim}
\left(\widehat{\mathcal{L}}_{i}^N(s_{j,N}): i \geq 1, j \in \llbracket 1, m \rrbracket  \right) \Rightarrow \left(Y_{i}^{j}: i \geq 1, j \in \llbracket 1, m \rrbracket  \right).
\end{equation}
Moreover, the random measure 
\begin{equation}\label{Eq.S53FormedY}
M^{\mathsf{S};Y}(A) = \sum_{i=1}^{\infty}\sum_{j=1}^{m} {\bf 1} \left\{ (s_j,Y^j_i)\in A \right\},
\end{equation}
is a determinantal point process on $\mathbb{R}^2$ with reference measure $\mu_{\mathsf{S}}\times\mathrm{Leb}$, as in Definition~\ref{Def.S1Measures}, and correlation kernel $K^{\mathrm{per},R}$ as in Definition~\ref{Def.S1PeriodicAiryKernel}.

We establish the above claims in the next step. Here, we assume their validity and conclude the proof of the two theorems.

\medskip

From Proposition~\ref{Prop.S52Tightness}, we know that $\widehat{\mathcal{L}}^N$ is tight. Let $\widehat{\mathcal{L}}^{\infty,1}$ and $\widehat{\mathcal{L}}^{\infty,2}$ denote any two subsequential limits. From (\ref{Eq.S53FiniteDim}) we have the following equality in the sense of finite-dimensional distributions
$$\left(\widehat{\mathcal{L}}_{i}^{\infty,1}(s): i \geq 1, s \in (-R/2, R/2]  \right) \overset{f.d}{=}\left(\widehat{\mathcal{L}}_{i}^{\infty,2}(s): i \geq 1, s \in (-R/2, R/2]  \right).$$
By the almost sure continuity of $\widehat{\mathcal{L}}_{i}^{\infty,1}$ and $\widehat{\mathcal{L}}_{i}^{\infty,2}$, we conclude that the above finite-dimensional equality extends to $[-R/2,R/2]$, and so by \cite[Lemma 3.1]{DM21}, we conclude $\widehat{\mathcal{L}}^{\infty,1} \overset{d}{=}\widehat{\mathcal{L}}^{\infty,2}$ as random elements in $C(\mathbb{N} \times [-R/2, R/2])$. Since $\widehat{\mathcal{L}}^N$ is tight, and all subsequential limits agree, we conclude that $\widehat{\mathcal{L}}^N$ converges weakly to some random element in $C(\mathbb{N} \times [-R/2, R/2])$, which we denote by $\mathcal{A}^{\mathrm{per},R} = \{\mathcal{A}^{\mathrm{per},R}_i\}_{i\geq1}$. Combining $\widehat{\mathcal{L}}^N \Rightarrow \mathcal{A}^{\mathrm{per},R}$ with (\ref{Eq.S53FiniteDim}), we conclude 
\begin{equation}\label{Eq.S53MatchingFD}
\left(\mathcal{A}^{\mathrm{per},R}_i(s_j): i \geq 1, j \in \llbracket 1, m \rrbracket \right) \overset{d}{=} \left(Y_{i}^{j}: i \geq 1, j \in \llbracket 1, m \rrbracket  \right)
\end{equation}
as random elements in $\mathbb{R}^{\infty}$.

From (\ref{Eq.S3ParticleComparisonBound}), (\ref{Eq.S53SmallSN}), and linear interpolation, we have for every fixed $i \geq 1$ that
\begin{equation*}
\sup_{t \in [-R/2, R/2]}\left|\widehat{\mathcal{L}}_i^{N}(t) - \mathcal{L}_i^N(t)\right| \leq h_N\left(|S_N|+i+1\right) \overset{\mathbb{P}}{\rightarrow} 0.
\end{equation*}
Consequently, by the convergence together lemma \cite[Theorem 3.1]{Billing}, we conclude that $\mathcal{L}^N \Rightarrow \mathcal{A}^{\mathrm{per},R}$. From Definition~\ref{Def.S1Scaling} we have almost surely for each $N$ that
\begin{equation*}
\mathcal{L}^N_i(t) \geq \mathcal{L}^N_{i+1}(t) \mbox{ for all } i \geq 1 \mbox{ and } t \in [-R/2,R/2],
\end{equation*}
and
\begin{equation*}
\mathcal{L}^N_i(-R/2) = \mathcal{L}^N_i(R/2) \mbox{ for all } i \geq 1.
\end{equation*}
Together with the weak convergence above, these relations imply (\ref{Eq.S1OrderedPerAiry}) and (\ref{Eq.S1PeriodicPerAiry}), completing the proof of Theorem~\ref{Thm.Convergence}.

From (\ref{Eq.S53MatchingFD}) we conclude that $M^{\mathsf{S};R}$ from Theorem~\ref{Thm.DeterminantalStructure} has the same distribution as $M^{\mathsf{S};Y}$ from (\ref{Eq.S53FormedY}), and so $M^{\mathsf{S};R}$ is a determinantal point process on $\mathbb{R}^2$ with reference measure $\mu_{\mathsf{S}}\times\mathrm{Leb}$ and correlation kernel $K^{\mathrm{per},R}$. Since $m\in\mathbb{N}$ and $-R/2<s_1<\cdots<s_m\leq R/2$ were arbitrary, this proves Theorem~\ref{Thm.DeterminantalStructure}.

\medskip

{\bf \raggedleft Step 2.} Define the random measures
\begin{equation}\label{Eq.S53FormedLN}
M^{\mathsf{S};N}(A) = \sum_{i=1}^{\infty}\sum_{j=1}^{m} {\bf 1} \left\{ (s_j,\widehat{\mathcal{L}}^N_i(s_{j,N}))\in A \right\}.
\end{equation}
We claim that 
\begin{equation}\label{Eq.S53PPConvergence}
M^{\mathsf{S};N} \Rightarrow M^{\mathsf{S};\infty},
\end{equation}
where $M^{\mathsf{S};\infty}$ is a determinantal point process on $\mathbb{R}^2$ with reference measure $\mu_{\mathsf{S}}\times\mathrm{Leb}$ and correlation kernel $K^{\mathrm{per},R}$.

From (\ref{Eq.S53PPConvergence}) and the tightness of the sequence of random variables $\{\widehat{\mathcal{L}}^N_i(s_{j,N})\}_{N \geq 1}$ for each fixed $i \in \mathbb{N}$, $j \in \llbracket 1, m\rrbracket$, which follows from Proposition~\ref{Prop.S52Tightness}, we see that the conditions of \cite[Proposition~2.19]{ED24a} are satisfied. The latter proposition implies the two claims in Step 1, and so it suffices to establish (\ref{Eq.S53PPConvergence}).

\medskip 

In the remainder of this step, we assume that $N$ is sufficiently large so that $N \geq \mathsf{N}_0$ as in Lemma~\ref{Lem.S4KernelAgreement} and $s_{1,N} < \cdots < s_{m,N}$.

Let $f_N:\mathbb{R}\rightarrow\mathbb{R}$ be an increasing piecewise-linear bijection such that
$f_N(s_{j,N})=s_j$ for $j\in\llbracket1,m\rrbracket$ and $f_N$ is linear on $(-\infty, s_{1,N}), (s_{1,N}, s_{2,N}), \dots, (s_{m-1,N},s_{m,N}), (s_{m,N}, \infty)$, and define $\psi_N(s,x) = (f_N(s),x)$. Notice that $\psi_N: \mathbb{R}^2 \rightarrow \mathbb{R}^2$ is a locally bounded bijection with a locally bounded inverse and that
\begin{equation}\label{Eq.S53Pullback}
M^{\mathsf{S};N}(A) = \widehat{M}^N (\psi_N^{-1}(A\cap D_{\mathsf{S}})),\mbox{ where } D_{\mathsf{S}} = \{(s,x) \in \mathbb{R}^2: s \in \mathsf{S}\}.
\end{equation}

From Lemma~\ref{Lem.S4KernelAgreement}, the change of variables formula for determinantal point processes, see \cite[Proposition 2.13(5)]{ED24a}, the fact that restrictions of determinantal point processes are still determinantal, see \cite[Proposition 2.13(2)]{ED24a}, and a rescaling of the kernel and reference measure, see \cite[Proposition 2.13(6)]{ED24a}, we conclude that $M^{\mathsf{S};N}$ is a determinantal point process on $\mathbb{R}^2$ with reference measure given by $\mu_{\mathsf{S}} \times (h_N \mu_{\Lambda_N})$ and correlation kernel given by 
\begin{equation}\label{Eq.S53NewKernel}
K^{\mathsf{S};N}(s,x;t,y) = h_N^{-1}\widehat{K}^{\mathrm{ext},N}(f_N^{-1}(s),x; f_N^{-1}(t), y) \mbox{ for }s,t \in \mathsf{S} \mbox{ and }x,y \in \mathbb{R}.
\end{equation}
Here, we recall from Definition~\ref{Def.S1Measures} that $\mu_{\mathsf{S}}$ and $\mu_{\Lambda_N}$ are counting measures, and that $\widehat{K}^{\mathrm{ext},N}$ is as in Lemma~\ref{Lem.S4KernelAgreement}.

Finally, we seek to apply \cite[Proposition~2.18]{ED24a} to $M^{\mathsf{S};N}$. Our measures are of the form considered in that proposition, once we set $r= m$, $K_N = K^{\mathsf{S};N}$, $t_i = s_i$, and $\nu_{t_i}^N = h_N \mu_{\Lambda_N}$. Note that $h_N \mu_{\Lambda_N}$ converges vaguely to the Lebesgue measure. In addition, from Proposition~\ref{Prop.S4KernelLimit} and (\ref{Eq.S53NewKernel}), for every $A>0$ and $s,t\in\mathsf{S}$, we have
\begin{equation}\label{Eq.S53KernelLimit}
\lim_{N\to\infty}
\sup_{x,y\in[-A,A]}
\left| K^{\mathsf{S};N}(s,x;t,y) -e^{sx-s^3/3-ty+t^3/3}K^{\mathrm{per},R}(s,x;t,y) \right|=0.
\end{equation}
Lastly, from Proposition~\ref{Prop.S4AiryIdentities}(v), we know that for every fixed $s,t\in(-R/2,R/2]$, the function
$(x,y)\mapsto K^{\mathrm{per},R}(s,x;t,y)$ is continuous on $\mathbb R^2$. 

The last few observations verify the conditions from \cite[Proposition~2.18]{ED24a}, which shows that $M^{\mathsf{S};N}$ converges weakly to a determinantal point process on $\mathbb{R}^2$ with reference measure $\mu_{\mathsf{S}}\times\mathrm{Leb}$, and correlation kernel $(s,x;t,y) \mapsto e^{sx-s^3/3-ty+t^3/3}K^{\mathrm{per},R}(s,x;t,y)$. By a gauge transformation, see \cite[Proposition 2.13(4)]{ED24a}, the limiting point process also has correlation kernel $K^{\mathrm{per},R}$. This proves (\ref{Eq.S53PPConvergence}).

%----------------------------------------------------------------------
%
% Section 6
%
%----------------------------------------------------------------------
\section{The periodic Brownian Gibbs property}\label{Section6}
In Section~\ref{Section6.1}, we introduce Brownian line ensembles with area tilts, define the Brownian Gibbs and periodic Brownian Gibbs properties with area tilts, and establish a characterization relating the two. In Section~\ref{Section6.2}, we introduce random loops that arise naturally from the random cylindric plane partition models of Section~\ref{Section1.2} and prove that, under an appropriate scaling, they converge to Brownian bridges. Finally, in Section~\ref{Section6.3}, we use these results to prove the periodic Brownian Gibbs property of the periodic Airy line ensemble stated in Theorem~\ref{Thm.GibbsProperty}. Throughout this section, we freely use the notation from Sections~\ref{Section1.2} and~\ref{Section2.2}.

%----------------------------------------------------------------------
%
% Section 6.1
%
%----------------------------------------------------------------------
\subsection{Line ensembles and area tilts}\label{Section6.1} 

In this section, we introduce the Brownian Gibbs property with area tilts. Our notation and presentation largely follow \cite[Section~2]{DimitrovSerio}. Unlike in that work, where different curves are allowed to have different area-tilt coefficients, we will only require the case in which all curves have the same coefficient.

Let $\widetilde{B}$ be the standard Brownian bridge on $[0,1]$, given by
$$
\widetilde{B}(t)=W_t-tW_1,
\qquad 0\leq t\leq1,
$$
where $W$ is a standard Brownian motion. Given $a,b,x,y\in\mathbb{R}$ with $a<b$, and $\sigma > 0$, the process
\begin{equation}\label{Eq.BrownianBridge}
B(t)
=
(b-a)^{1/2}
\sigma \widetilde{B}\left(\frac{t-a}{b-a}\right)
+
\frac{b-t}{b-a} \cdot x
+
\frac{t-a}{b-a} \cdot y,
\qquad t\in[a,b],
\end{equation}
is called a {\em Brownian bridge from $x$ to $y$ on $[a,b]$ with diffusion parameter $\sigma$}.

Fix integers $k_1\leq k_2$, set $k=k_2-k_1+1$, and let $\vec{x},\vec{y}\in\mathbb{R}^k$. We denote by $\mathbb{P}_{\mathrm{free}}^{k_1,k_2,a,b,\vec{x},\vec{y}}$ the law of $k$ independent Brownian bridges $\{B_i:[a,b]\rightarrow\mathbb{R}\}_{i=k_1}^{k_2}$, all with diffusion parameter $1$, such that $B_i(a)=x_{i-k_1+1}$ and $B_i(b)=y_{i-k_1+1}$ for $k_1\leq i\leq k_2$. We write $\mathbb{E}_{\mathrm{free}}^{k_1,k_2,a,b,\vec{x},\vec{y}}$ for the expectation with respect to this measure.

We next introduce Brownian line ensembles conditioned to avoid one another and remain between two boundary curves.

\begin{definition}\label{Def.AvoidingBrownianLE}
Fix integers $k_1\leq k_2$ and set $k=k_2-k_1+1$. Let $\weyl_k$ denote the open Weyl chamber in $\mathbb{R}^k$, i.e.,
$$
\weyl_k=\left\{\vec{z}=(z_1,\ldots,z_k)\in\mathbb{R}^k:z_1>z_2>\cdots>z_k\right\}.
$$
Let $\vec{x},\vec{y}\in\weyl_k$, let $a,b\in\mathbb{R}$ with $a<b$, and let $f:[a,b]\rightarrow(-\infty,\infty]$ and $g:[a,b]\rightarrow[-\infty,\infty)$ be continuous functions. Here, the latter condition means that either $f:[a,b]\rightarrow\mathbb{R}$ is continuous or $f=\infty$ everywhere, and similarly for $g$. We assume that $f(t)>g(t)$ for every $t\in[a,b]$, as well as $f(a)>x_1$, $f(b)>y_1$, $g(a)<x_k$, and $g(b)<y_k$.

With the above data, we define the {\em $(f,g)$-avoiding Brownian line ensemble on $[a,b]$ with entrance data $\vec{x}$ and exit data $\vec{y}$} to be the $\llbracket k_1,k_2\rrbracket$-indexed line ensemble $\mathcal{Q}$ whose law is $\mathbb{P}_{\mathrm{free}}^{k_1,k_2,a,b,\vec{x},\vec{y}}$, conditioned on the event
\begin{equation}\label{Eq.Event}
\left\{
f(t)>\mathcal{Q}_{k_1}(t)>\mathcal{Q}_{k_1+1}(t)>
\cdots>\mathcal{Q}_{k_2}(t)>g(t)
\text{ for every }t\in[a,b]
\right\}.
\end{equation}
We denote its law by $\mathbb{P}_{\mathrm{avoid}}^{k_1,k_2,a,b,\vec{x},\vec{y},f,g}$ and the corresponding expectation by $\mathbb{E}_{\mathrm{avoid}}^{k_1,k_2,a,b,\vec{x},\vec{y},f,g}$. 
\end{definition}
\begin{remark}\label{Rem.Event} The event in (\ref{Eq.Event}) has positive probability and hence the conditioning in Definition~\ref{Def.AvoidingBrownianLE} is well-defined; see \cite[Definition~2.4]{DM21} for details.
\end{remark}

We next introduce the area-tilted version of the preceding measure. 

\begin{definition}\label{Def.AreaTiltedBrownianLE}
Assume the same notation as in Definition~\ref{Def.AvoidingBrownianLE}, suppose that $g\neq-\infty$, and fix $A\in[0,\infty)$. We define the {\em $(f,g)$-avoiding Brownian line ensemble with $A$-area tilt on $[a,b]$ with entrance data $\vec{x}$ and exit data $\vec{y}$} to be the $\llbracket k_1,k_2\rrbracket$-indexed line ensemble $\mathcal{Q}$ whose law $\mathbb{P}_{\mathrm{avoid},A}^{k_1,k_2,a,b,\vec{x},\vec{y},f,g}$ is defined through the Radon--Nikodym derivative
\begin{equation}\label{Eq.AreaTiltedBrownianRN}
\frac{
\mathrm{d}\mathbb{P}_{\mathrm{avoid},A}^{k_1,k_2,a,b,\vec{x},\vec{y},f,g}
}{
\mathrm{d}\mathbb{P}_{\mathrm{avoid}}^{k_1,k_2,a,b,\vec{x},\vec{y},f,g}
}(\mathcal{Q})
=
\frac{1}{
Z_{\mathrm{avoid},A}^{k_1,k_2,a,b,\vec{x},\vec{y},f,g}
}
\exp\left(
-A\sum_{i=k_1}^{k_2}
\int_a^b\bigl(\mathcal{Q}_i(u)-g(u)\bigr)du
\right),
\end{equation}
where the normalization constant is given by
\begin{equation}\label{Eq.AreaTiltedNormalization}
Z_{\mathrm{avoid},A}^{k_1,k_2,a,b,\vec{x},\vec{y},f,g}
=
\mathbb{E}_{\mathrm{avoid}}^{k_1,k_2,a,b,\vec{x},\vec{y},f,g}
\left[
\exp\left(
-A\sum_{i=k_1}^{k_2}
\int_a^b\bigl(\mathcal{Q}_i(u)-g(u)\bigr)du
\right)
\right].
\end{equation}
We write $\mathbb{E}_{\mathrm{avoid},A}^{k_1,k_2,a,b,\vec{x},\vec{y},f,g}$ for expectation with respect to $\mathbb{P}_{\mathrm{avoid},A}^{k_1,k_2,a,b,\vec{x},\vec{y},f,g}$.

Since $A\geq0$ and $\mathcal{Q}_i(u)>g(u)$, the exponential appearing in \eqref{Eq.AreaTiltedBrownianRN} takes values in $(0,1]$. Consequently, $Z_{\mathrm{avoid},A}^{k_1,k_2,a,b,\vec{x},\vec{y},f,g}\in(0,1]$, and the preceding probability measure is well-defined.
\end{definition}

We next define the Brownian Gibbs property with a common area-tilt coefficient. In what follows, we use the convention that $\llbracket1,N\rrbracket=\mathbb{N}$ when $N=\infty$.

\begin{definition}\label{Def.AreaTiltedBGP}
Fix $N\in\mathbb{N}\cup\{\infty\}$, an interval $\Lambda\subset\mathbb{R}$, and $A\in[0,\infty)$, and set $\Sigma=\llbracket1,N\rrbracket$. A $\Sigma$-indexed line ensemble $\mathcal{L}$ on $\Lambda$ is said to satisfy the {\em partial Brownian Gibbs property with $A$-area tilt} if it is non-intersecting and the following holds.

Fix a finite interval of indices $K=\llbracket k_1,k_2\rrbracket\subseteq\Sigma$ with $k_2\leq N-1$, an interval $[a,b]\subset\Lambda$, and a bounded Borel-measurable function $F:C(K\times[a,b])\rightarrow\mathbb{R}$. Set $D_{K,a,b}=K\times(a,b)$ and define
$$
\mathcal{F}_{\mathrm{ext}}\bigl(K\times(a,b)\bigr)
=
\sigma\left\{
\mathcal{L}_i(t):
(i,t)\in(\Sigma\times\Lambda)\setminus D_{K,a,b}
\right\}.
$$
Let $\vec{x}=(\mathcal{L}_{k_1}(a),\ldots,\mathcal{L}_{k_2}(a))$ and $\vec{y}=(\mathcal{L}_{k_1}(b),\ldots,\mathcal{L}_{k_2}(b))$, and set $f=\mathcal{L}_{k_1-1}|_{[a,b]}$ and $g=\mathcal{L}_{k_2+1}|_{[a,b]}$, with the convention that $f=\infty$ if $k_1=1$. Then, almost surely,
\begin{equation}\label{Eq.AreaTiltedBGP}
\mathbb{E}\left[
F\left(\mathcal{L}|_{K\times[a,b]}\right)
\middle|
\mathcal{F}_{\mathrm{ext}}\bigl(K\times(a,b)\bigr)
\right]
=
\mathbb{E}_{\mathrm{avoid},A}^{k_1,k_2,a,b,\vec{x},\vec{y},f,g}
\left[F(\mathcal{Q})\right].
\end{equation}

When $N=\infty$, we say that $\mathcal{L}$ satisfies the {\em Brownian Gibbs property with $A$-area tilt}, dropping the word ``partial.'' When $A=0$, these definitions reduce to the usual {\em partial Brownian Gibbs property} and {\em Brownian Gibbs property}; see \cite[Definition~2.7]{DM21} and \cite[Definition~2.2]{CorHamA}.
\end{definition}

We next define the periodic Brownian Gibbs property with a common area-tilt coefficient.
\begin{definition}\label{Def.PeriodicAreaTiltedBGP}
Fix $N\in\mathbb{N}\cup\{\infty\}$, $\alpha,\beta\in\mathbb{R}$ with $\alpha<\beta$, and $A\in[0,\infty)$, and set $\Sigma=\llbracket1,N\rrbracket$. A $\Sigma$-indexed line ensemble $\mathcal{L}$ on $[\alpha, \beta]$ is said to satisfy the {\em partial periodic Brownian Gibbs property with $A$-area tilt} if it is non-intersecting, $\mathcal{L}_i(\alpha)=\mathcal{L}_i(\beta)$ for every $i\in\Sigma$ almost surely, and the following holds.

Fix a finite interval of indices $K=\llbracket1,k_2\rrbracket\subseteq\Sigma$ with $k_2\leq N-1$, and a bounded Borel-measurable function $F:C(K\times[\alpha,\beta])\rightarrow\mathbb{R}$. Set $D_{K,\alpha,\beta}=K\times(\alpha,\beta)$ and define
$$
\mathcal{F}_{\mathrm{ext}}\bigl(K\times(\alpha,\beta)\bigr)
=
\sigma\left\{
\mathcal{L}_i(t):
(i,t)\in(\Sigma\times[\alpha,\beta])\setminus D_{K,\alpha,\beta}
\right\}.
$$
Let $\vec{x}=(\mathcal{L}_1(\alpha),\ldots,\mathcal{L}_{k_2}(\alpha))$ and set $g=\mathcal{L}_{k_2+1}$. Then, almost surely,
\begin{equation}\label{Eq.PeriodicAreaTiltedBGP}
\mathbb{E}\left[
F\left(\mathcal{L}|_{K\times[\alpha,\beta]}\right)
\middle|
\mathcal{F}_{\mathrm{ext}}\bigl(K\times(\alpha,\beta)\bigr)
\right]
=
\mathbb{E}_{\mathrm{avoid},A}^{1,k_2,\alpha,\beta,\vec{x},\vec{x},\infty,g}
\left[F(\mathcal{Q})\right].
\end{equation}

When $N=\infty$, we say that $\mathcal{L}$ satisfies the {\em periodic Brownian Gibbs property with $A$-area tilt}, dropping the word ``partial.'' When $A=0$, we further drop ``with $A$-area tilt''.
\end{definition}

We end this section with the following lemma that explains the relationship between Definition \ref{Def.AreaTiltedBGP} and Definition \ref{Def.PeriodicAreaTiltedBGP}.

\begin{lemma}\label{Lem.PeriodicBGPCharacterization}
Fix $N\in\mathbb{N}\cup\{\infty\}$, $\alpha,\beta\in\mathbb{R}$ with
$\alpha<\beta$, and $A\in[0,\infty)$, and set
$\Sigma=\llbracket1,N\rrbracket$. Suppose that $\mathcal{L}$ is a $\Sigma$-indexed line ensemble on $[\alpha,\beta]$. Then the following statements are equivalent:
\begin{enumerate}
\item $\mathcal{L}$ satisfies the partial periodic Brownian Gibbs property with $A$-area tilt from Definition~\ref{Def.PeriodicAreaTiltedBGP}.
\item $\mathcal{L}$ satisfies the partial Brownian Gibbs property with $A$-area tilt from Definition~\ref{Def.AreaTiltedBGP} and is periodic, in the sense that $\mathcal{L}_i(\alpha)=\mathcal{L}_i(\beta)$ for every $i\in\Sigma$ almost surely.
\end{enumerate}
\end{lemma}
\begin{proof}
We first assume (2) and prove (1). Fix a finite interval of indices $K=\llbracket1,k_2\rrbracket\subseteq\Sigma$ with $k_2 \leq N-1$, and let $F:C(K\times[\alpha,\beta])\rightarrow\mathbb{R}$ be bounded and Borel-measurable. Applying (\ref{Eq.AreaTiltedBGP}) with $k_1=1$, $a=\alpha$, and $b=\beta$ and using that $\mathcal{L}_i(\alpha)=\mathcal{L}_i(\beta)$ for every $i\in\Sigma$ almost surely, we obtain \eqref{Eq.PeriodicAreaTiltedBGP}. Thus, $\mathcal{L}$ satisfies the partial periodic Brownian Gibbs property with $A$-area tilt.

We next assume (1) and prove (2). By Definition~\ref{Def.PeriodicAreaTiltedBGP}, $\mathcal{L}$ is non-intersecting and periodic. Consequently, it remains only to verify \eqref{Eq.AreaTiltedBGP}. The proof below is similar to that of \cite[Lemma~2.9]{DY25}.

Fix a finite interval of indices $K=\llbracket k_1,k_2\rrbracket\subseteq\Sigma$ with $k_2\leq N-1$, and an interval $[a,b]\subseteq[\alpha,\beta]$. Let $\vec{x} = \bigl( \mathcal{L}_{k_1}(a),\ldots,\mathcal{L}_{k_2}(a) \bigr)$, $\vec{y} = \bigl( \mathcal{L}_{k_1}(b),\ldots,\mathcal{L}_{k_2}(b) \bigr)$, and set $f=\mathcal{L}_{k_1-1}|_{[a,b]}$ and $g=\mathcal{L}_{k_2+1}|_{[a,b]}$, with the convention that $f=\infty$ when $k_1=1$. By the defining properties of conditional expectation, to show \eqref{Eq.AreaTiltedBGP}, it suffices to prove that for every bounded Borel-measurable function $F:C(K\times[a,b])\rightarrow\mathbb{R}$ and every bounded $\mathcal{F}_{\mathrm{ext}}(K\times(a,b))$-measurable random variable $Y$, 
\begin{equation}\label{Eq.PeriodicBGPTestIdentity} \mathbb{E}\left[ F\left(\mathcal{L}|_{K\times[a,b]}\right)Y \right] = \mathbb{E}\left[ \mathbb{E}_{\mathrm{avoid},A}^{k_1,k_2,a,b, \vec{x},\vec{y},f,g} \left[F(\mathcal{Q})\right]Y \right]. 
\end{equation} 

Fix $m\in\mathbb{N}$, $n_1,\ldots,n_m\in\Sigma$, $t_1,\ldots,t_m\in[\alpha,\beta]$, and bounded continuous functions $h_1,\ldots,h_m:\mathbb{R}\rightarrow\mathbb{R}$. Set 
$$ I=\llbracket1,k_2\rrbracket, \hspace{2mm} \mathsf{R}_0 = \left\{ j\in\llbracket1,m\rrbracket:n_j\in I \right\}, \hspace{2mm} \mathsf{R} = \left\{ j\in\llbracket1,m\rrbracket: n_j\in K,\ t_j\in[a,b] \right\}, $$ 
and note that $\mathsf{R}\subseteq\mathsf{R}_0$. We claim that 
\begin{equation}\label{Eq.PeriodicBGPCylinderIdentity} \begin{split} \mathbb{E}\left[ \prod_{j=1}^m h_j\bigl(\mathcal{L}_{n_j}(t_j)\bigr) \right] ={}& \mathbb{E}\left[ \prod_{j\notin\mathsf{R}} h_j\bigl(\mathcal{L}_{n_j}(t_j)\bigr) \cdot \mathbb{E}_{\mathrm{avoid},A}^{k_1,k_2,a,b, \vec{x},\vec{y},f,g} \left[ \prod_{j\in\mathsf{R}} h_j\bigl(\mathcal{Q}_{n_j}(t_j)\bigr) \right] \right]. \end{split} 
\end{equation} 
By the monotone class argument in \cite[Lemma~A.1]{DY25}, we see that (\ref{Eq.PeriodicBGPCylinderIdentity}) implies \eqref{Eq.PeriodicBGPTestIdentity}, and so it suffices to establish the former.

Set $\vec{z} = \bigl( \mathcal{L}_1(\alpha),\ldots,\mathcal{L}_{k_2}(\alpha) \bigr)$ and $h=\mathcal{L}_{k_2+1}|_{[\alpha,\beta]}$. By the tower property of conditional expectation and (\ref{Eq.PeriodicAreaTiltedBGP}), we have
\begin{equation}\label{Eq.PeriodicBGPFirstApplication} \begin{split} \mathbb{E}\left[ \prod_{j=1}^m h_j\bigl(\mathcal{L}_{n_j}(t_j)\bigr) \right] = \mathbb{E}\left[ \prod_{j\notin\mathsf{R}_0} h_j\bigl(\mathcal{L}_{n_j}(t_j)\bigr) \cdot \mathbb{E}_{\mathrm{avoid},A}^{1,k_2,\alpha,\beta, \vec{z},\vec{z},\infty,h} \left[ \prod_{j\in\mathsf{R}_0} h_j\bigl(\mathcal{Q}_{n_j}(t_j)\bigr) \right] \right]. 
\end{split} 
 \end{equation} 
For the remainder of the argument inside the expectation $\mathbb{E}_{\mathrm{avoid},A}^{1,k_2,\alpha,\beta, \vec{z},\vec{z},\infty,h}$, we set $\mathcal{Q}_{k_2+1}=h$ and define $\vec{x}_{\mathcal{Q}} = \bigl( \mathcal{Q}_{k_1}(a),\ldots,\mathcal{Q}_{k_2}(a) \bigr)$, $\vec{y}_{\mathcal{Q}} = \bigl( \mathcal{Q}_{k_1}(b),\ldots,\mathcal{Q}_{k_2}(b) \bigr)$, as well as $f_{\mathcal{Q}} = \mathcal{Q}_{k_1-1}|_{[a,b]}$, $g_{\mathcal{Q}} = \mathcal{Q}_{k_2+1}|_{[a,b]}$, with the convention that $f_{\mathcal{Q}}=\infty$ when $k_1=1$. 

Applying \cite[Lemma~2.6]{DimitrovSerio} with
$A_i=A$ for every $i\in\llbracket1,k_2\rrbracket$, we see that the
line ensemble consisting of $\mathcal{Q}_1,\ldots,\mathcal{Q}_{k_2}$
together with the lower boundary $\mathcal{Q}_{k_2+1}=h$ satisfies the
partial Brownian Gibbs property with $A$-area tilt. Applying (\ref{Eq.AreaTiltedBGP}) to this ensemble and the tower property gives
\begin{equation}\label{Eq.PeriodicBGPFiniteEnsemble} 
\begin{split} & \mathbb{E}_{\mathrm{avoid},A}^{1,k_2,\alpha,\beta, \vec{z},\vec{z},\infty,h} \left[ \prod_{j\in\mathsf{R}_0} h_j\bigl(\mathcal{Q}_{n_j}(t_j)\bigr) \right]\\ &\quad= \mathbb{E}_{\mathrm{avoid},A}^{1,k_2,\alpha,\beta, \vec{z},\vec{z},\infty,h} \left[ \prod_{j\in\mathsf{R}_0\setminus\mathsf{R}} h_j\bigl(\mathcal{Q}_{n_j}(t_j)\bigr) \cdot \mathbb{E}_{\mathrm{avoid},A}^{k_1,k_2,a,b, \vec{x}_{\mathcal{Q}},\vec{y}_{\mathcal{Q}}, f_{\mathcal{Q}},g_{\mathcal{Q}}} \left[ \prod_{j\in\mathsf{R}} h_j\bigl(\widetilde{\mathcal{Q}}_{n_j}(t_j)\bigr) \right] \right], 
\end{split} 
\end{equation} 
where $\widetilde{\mathcal{Q}}$ denotes the line ensemble appearing in the innermost expectation. 

Substituting \eqref{Eq.PeriodicBGPFiniteEnsemble} into \eqref{Eq.PeriodicBGPFirstApplication} and applying (\ref{Eq.PeriodicAreaTiltedBGP}) to the right side, we obtain (\ref{Eq.PeriodicBGPCylinderIdentity}). This proves (2) and completes the proof.
\end{proof}

%----------------------------------------------------------------------
%
% Section 6.2
%
%----------------------------------------------------------------------
\subsection{Random loops}\label{Section6.2} In this section, we introduce certain random periodic discrete paths that naturally arise within the random cylindric plane partition models of Section \ref{Section1.2}. The properties of these paths will be used in Section \ref{Section6.3} to prove the Gibbs property of the periodic Airy line ensemble in Theorem \ref{Thm.GibbsProperty}. 

Fix $M\in\mathbb{N}$ and $x\in\mathbb{Z}$. We call a function
$L:\llbracket -M,M\rrbracket\rightarrow\mathbb{Z}$
a {\em loop} with endpoint $x$ if
\begin{equation}\label{Eq.GeometricLoopConditions}
L(-M)=L(M)=x,
\qquad
L(t)-L(t-1)\in \begin{cases} \mathbb{Z}_{\geq0} & \text{ for }t\in\llbracket-M+1,0\rrbracket, \\ \mathbb{Z}_{\leq 0}
& \text{ for }t\in\llbracket1,M\rrbracket. \end{cases}
\end{equation}
We denote the set of all such loops by $\Omega_{x,M}^{\mathrm{loop}}$. Thus, a loop is weakly increasing on $\llbracket-M,0\rrbracket$, weakly decreasing on $\llbracket0,M\rrbracket$, and returns to its starting value at time $M$. We also let $\Omega_{M}^{\mathrm{loop}} = \bigcup_{x \in \mathbb{Z}} \Omega_{x,M}^{\mathrm{loop}}$ denote the set of all loops on $[-M, M]$.

If $L\in\Omega_{x,M}^{\mathrm{loop}}$ and $L'\in\Omega_{y,M}^{\mathrm{loop}}$, we say that $L$ and $L'$ {\em interlace}, and write $L\preceq L'$, if
\begin{equation}\label{Eq.GeometricLoopInterlacing}
L(t)\leq L'(t-1)
\quad\text{for }t\in\llbracket-M+1,0\rrbracket,
\qquad
L(t)\leq L'(t+1)
\quad\text{for }t\in\llbracket0,M-1\rrbracket.
\end{equation}

We next introduce the probability measure on loops that we will consider.
\begin{definition}\label{Def.GeometricLoopMeasure}
Fix $M\in\mathbb{N}$, $x\in\mathbb{Z}$, and $b\in(0,1)$. We denote by $\mathbb{P}_{\mathrm{loop}}^{b,M,x}$ the probability measure on $\Omega_{x,M}^{\mathrm{loop}}$ given by
\begin{equation}\label{Eq.GeometricLoopMeasure}
\mathbb{P}_{\mathrm{loop}}^{b,M,x}(L)
=\frac{b^{2L(0)}}{Z_{\mathrm{loop}}^{b,M,x}}, \mbox{ where } Z_{\mathrm{loop}}^{b,M,x} = \sum_{L\in\Omega_{x,M}^{\mathrm{loop}}}b^{2L(0)}.
\end{equation}
We write $\mathbb{E}_{\mathrm{loop}}^{b,M,x}$ for expectation with respect to $\mathbb{P}_{\mathrm{loop}}^{b,M,x}$.
\end{definition}
\begin{remark}\label{Rem.ZLoop} Let us briefly explain why $Z_{\mathrm{loop}}^{b,M,x} \in (0, \infty)$, and hence $\mathbb{P}_{\mathrm{loop}}^{b,M,x}$ is a well-defined probability measure. If $L(0)=x+n$ for some $n\in\mathbb{Z}_{\geq0}$, then the $M$ nonnegative increments of $L$ on $\llbracket-M,0\rrbracket$ have sum $n$. By the usual stars-and-bars argument, there are
$$
\binom{n+M-1}{M-1}
$$
possible choices for these increments. Arguing the same way for $\llbracket 0, M \rrbracket$, we get 
\begin{equation}\label{Eq.GeometricLoopMidpointCount}
\#\left\{L\in\Omega_{x,M}^{\mathrm{loop}}:L(0)=x+n\right\}
=
\binom{n+M-1}{M-1}^{2}.
\end{equation}
It follows that
\begin{equation}\label{Eq.GeometricLoopNormalization}
\begin{split}
Z_{\mathrm{loop}}^{b,M,x}
= b^{2x} \sum_{n=0}^{\infty}
\binom{n+M-1}{M-1}^{2}b^{2n}.
\end{split}
\end{equation}
From (\ref{Eq.GeometricLoopNormalization}), we see that $Z_{\mathrm{loop}}^{b,M,x} \in (0, \infty)$ as the series converges for $b \in (0,1)$ by the ratio test. 
\end{remark}

As in Section~\ref{Section2.2}, we identify a loop $L$ with the continuous function on $[-M,M]$ obtained by linearly interpolating the points $(t,L(t))$, $t\in\llbracket-M,M\rrbracket$. In this way, a random loop with law $\mathbb{P}_{\mathrm{loop}}^{b,M,x}$ may be viewed as a random element of $C([-M,M])$. Similarly, collections of loops defined on the same probability space will be viewed as line ensembles using the notation introduced in Section~\ref{Section2.2}, without further comment.

\medskip

Our goal in the remainder of this section is to show that, under an appropriate scaling, a sequence of loops $L^n$ with laws $\mathbb{P}_{\mathrm{loop}}^{b,M_n,x_n}$ converges to a Brownian bridge. The precise statement appears in Lemma~\ref{Lem.ConvBM}. Before stating it, we introduce some notation from \cite{D24b} and recall a few results that will be used in its proof.

\begin{definition}\label{Def.IncreasingPath}
Fix $t_1, t_2 \in \mathbb{Z}$ with $t_1 < t_2$. We call a function $f: \llbracket t_1, t_2 \rrbracket \rightarrow \mathbb{Z}$ such that $f(j+1) - f(j) \in \mathbb{Z}_{\geq 0}$ when $j \in\llbracket t_1, t_2 -1 \rrbracket$  an {\em increasing path}. As usual, by linearly interpolating the points $(t,f(t))$, we can view increasing paths as random elements of $C([t_1,t_2])$, as in Section~\ref{Section2.2}.

 Suppose next that $z_1, z_2 \in \mathbb{Z}$ satisfy $z_1 \leq z_2$. We denote by $\Omega(t_1,t_2,z_1,z_2)$ the collection of increasing paths that start from $(t_1,z_1)$ and end at $(t_2,z_2)$, by $\mathbb{P}_{\operatorname{Geom}}^{t_1,t_2, z_1, z_2}$ the uniform distribution on $\Omega(t_1,t_2,z_1,z_2)$ and write $\mathbb{E}^{t_1,t_2,z_1,z_2}_{\operatorname{Geom}}$ for the expectation with respect to this measure. 
\end{definition}
\begin{remark} One thinks of the distribution $\mathbb{P}_{\operatorname{Geom}}^{t_1,t_2, z_1, z_2}$ as the law of a random walk with i.i.d.\ geometric increments with parameter $b \in (0,1)$ that starts from $z_1$ at time $t_1$ and is conditioned to end at $z_2$ at time $t_2$---this interpretation does not depend on the choice of $b \in (0,1)$. Notice that by our assumptions on the parameters the state space $\Omega(t_1,t_2,z_1,z_2)$ is non-empty.
\end{remark}

We next state a strong coupling between geometric random walk bridges and Brownian bridges, which follows from the results in \cite{DW19}. 

\begin{proposition}\label{Prop.KMT} Fix $p \in (0, \infty)$, and set $\sigma = \sqrt{p(p+1)}$. There exist constants $0< C, a, \alpha' < \infty$, depending on $p$ alone, such that, for any $n \in \mathbb{N}$ and any $z \in \mathbb{Z}_{\geq 0}$, we can find a Brownian bridge $B^{\sigma}$ from $0$ to $0$ on $[0,1]$ with diffusion parameter $\sigma$, and a random increasing path $L$ with law $\mathbb{P}_{\operatorname{Geom}}^{0, n, 0,z}$, defined on the same probability space, such that
\begin{equation}\label{Eq.KMT}
\mathbb{E}\left[e^{a \Delta(n,z)} \right] \leq C n^{\alpha'} e^{(z- pn)^2/n}, \mbox{ where $\Delta(n,z):=  \sup_{0 \leq t \leq n} \left| \sqrt{n} B^\sigma_{t/n} + \frac{t}{n} \cdot z - L(t) \right|$.}
\end{equation}
\end{proposition}

\begin{proof} We fix $b \in (0,1)$ such that $p = \frac{b}{1-b}$. From \cite[Example 2, page 719]{DW19} we have that the geometric distribution with parameter $b$ satisfies the conditions of \cite[Theorem 4]{DW19}, which implies the statement of the proposition.
\end{proof}

We also require the following simple decomposition of a Brownian bridge, which is related to the L{\'e}vy--Ciesielski construction of Brownian motion discussed in \cite[Section 1.2]{McK69}.
\begin{lemma}\label{Lem.LevyBM} Fix $x \in \mathbb{R}$, $\mathfrak{m} > 0$, and let $Z$ be a Gaussian random variable with $\mathbb{E}[Z] = 0$ and $\mathrm{Var}(Z) = \mathfrak{m}/2$. In addition, suppose that $B_1, B_2$ are two Brownian bridges from $0$ to $0$ on $[0, \mathfrak{m}]$ with diffusion parameter $1$, such that $B_1$, $B_2$, and $Z$ are all independent. Define the random element $\mathcal{L}$ of $C([-\mathfrak{m}, \mathfrak{m}])$ by 
\begin{equation}\label{Eq.LevyBM}
\mathcal{L}(t) = x + \begin{cases} B_1(t + \mathfrak{m}) + \frac{t +\mathfrak{m}}{\mathfrak{m}} \cdot Z & \mbox{ if $t \in [-\mathfrak{m}, 0]$, } \\ B_2(t ) + \frac{\mathfrak{m} - t}{\mathfrak{m}} \cdot Z & \mbox{ if $t \in [0,\mathfrak{m}]$.}\end{cases}
\end{equation}
Then $\mathcal{L} \overset{d}{=} B$, where $B$ is a Brownian bridge from $x$ to $x$ on $[-\mathfrak{m}, \mathfrak{m}]$ with diffusion parameter $1$.
\end{lemma}
\begin{proof} 
The statement is quite standard and analogous to \cite[Lemma 2.8]{CorHamA} with $j = 2$. It is proved by first showing that $\mathcal{L}$ and $B$ are Gaussian processes with the same mean and covariance. This implies that $\mathcal{L}$ and $B$ have the same finite-dimensional distributions. The equality of laws as random elements in $C([-\mathfrak{m}, \mathfrak{m}])$ then follows from the fact that finite-dimensional sets form a separating class in $C([-\mathfrak{m}, \mathfrak{m}])$, see \cite[Example 1.3]{Billing}.
\end{proof}

With the above notation in place, we can now state the main result of this section.
\begin{lemma}\label{Lem.ConvBM} Assume the same notation as in Definitions \ref{Def.GeometricLoopMeasure} and \ref{Def.IncreasingPath}. Fix parameters $b \in (0,1)$, $\mathfrak{m} > 0$, $x \in \mathbb{R}$, and set $p = \frac{b}{1-b}$ and $\sigma = \sqrt{p(p+1)} = \frac{b^{1/2}}{1-b}$. Let $M_n \in \mathbb{N}$, $d_n \in (0,\infty)$, and $x_n \in \mathbb{Z}$ be sequences such that 
\begin{equation}\label{Eq.ConvBMScaling}
d_n \rightarrow \infty, \hspace{2mm} M_n/d_n \rightarrow \mathfrak{m}, \hspace{2mm} \sigma^{-1}d_n^{-1/2}(x_n + p M_n) \rightarrow x.
\end{equation}
Suppose that $L^n$ has law $\mathbb{P}_{\mathrm{loop}}^{b,M_n,x_n}$, and define the following random element in $C([-\mathfrak{m},\mathfrak{m}])$
\begin{equation}\label{Eq.ScaledPaths}
\mathcal{L}^n(t) = \sigma^{-1}d_n^{-1/2} \cdot \left(L^n\left(\frac{tM_n}{\mathfrak{m}}\right) + \frac{pM_n|t|}{\mathfrak{m}} \right), \hspace{2mm} t\in[-\mathfrak{m}, \mathfrak{m}].
\end{equation} 
Then $\mathcal{L}^n \Rightarrow B$, where $B$ is a Brownian bridge from $x$ to $x$ on $[-\mathfrak{m}, \mathfrak{m}]$ with diffusion parameter $1$.
\end{lemma}
\begin{proof} For clarity, we split the proof into two steps. In the first step, we prove that $\mathcal{L}^n(0)$ converges to $B(0)$. In the second step, we use Proposition~\ref{Prop.KMT} to couple the restrictions of $\mathcal{L}^n$ to $[-\mathfrak{m},0]$ and $[0,\mathfrak{m}]$ with two independent Brownian bridges that are spliced together at $\mathcal{L}^n(0)$. Since $\mathcal{L}^n(0)$ converges to $B(0)$ by Step 1, Lemma~\ref{Lem.LevyBM} shows that the resulting spliced Brownian bridges converge to $B$. The convergence together lemma then implies that $\mathcal{L}^n$ converges to $B$ as well. 

\medskip

{\bf \raggedleft Step 1.} In this step, we prove that 
\begin{equation}\label{Eq.ConvToGsn}
\mathcal{L}^n(0) \Rightarrow x + Z,
\end{equation}
where $Z$ is a Gaussian random variable with $\mathbb{E}[Z] = 0$ and $\mathrm{Var}(Z) = \mathfrak{m}/2$.

From (\ref{Eq.GeometricLoopMeasure}) and (\ref{Eq.GeometricLoopMidpointCount}), we have, for $y\in\mathbb{Z}_{\geq0}$,
\begin{equation}\label{Eq.GeometricLoopMidpointDistribution}
\mathbb{P}\bigl(L^n(0)=x_n+y\bigr)
= \frac{ \binom{y+M_n-1}{M_n-1}^{2}b^{2y}
}{\sum_{m \geq 0} \binom{m+M_n-1}{M_n-1}^{2}b^{2m}}.
\end{equation}
From \cite[(15.2.1)]{DLMF}, we have, for $M\in\mathbb{N}$ and $|z|<1$,
\begin{equation}\label{Eq.GaussSum} F(M,M;1;z) = \sum_{m=0}^{\infty}
\binom{m+M-1}{M-1}^{2}z^m,
\end{equation}
where $F(M,M;1;z)$ is the Gauss hypergeometric function. Combining (\ref{Eq.ScaledPaths}), (\ref{Eq.GeometricLoopMidpointDistribution}), and (\ref{Eq.GaussSum}), we conclude that if $\xi\in\mathbb{R}$ and $u_n=\frac{\xi}{\sigma\sqrt{d_n}}$, then
\begin{equation}\label{Eq.MidpointCF}
\mathbb{E}\left[e^{\mathrm{i}\xi\mathcal{L}^n(0)}\right]
=
\exp\left(\frac{\mathrm{i}\xi x_n}{\sigma\sqrt{d_n}}\right)
\frac{F(M_n,M_n;1;b^2e^{\mathrm{i}u_n})}
{F(M_n,M_n;1;b^2)}.
\end{equation}

We next analyze the ratio on the right-hand side of (\ref{Eq.MidpointCF}). From \cite[(15.8.1) and (15.9.7)]{DLMF}, we have
\begin{equation}\label{Eq.HyperLegendre}
F(M,M;1;z)
= (1-z)^{-M} F\left(1-M,M;1;\frac{z}{z-1}\right) = (1-z)^{-M} P_{M-1}\left(\frac{1+z}{1-z}\right),
\end{equation}
where $P_j$ denotes the Legendre polynomial of degree $j$. From \cite[Theorem 8.21.1]{Sze75}, we have, for any compact set 
$K \subset \mathbb{C} \setminus [-1,1]$, as $M \rightarrow \infty$,
\begin{equation}\label{Eq.LegendreAsymptotics}
P_{M-1}(w)=
\frac{\left(w+(w^2-1)^{1/2}\right)^{M-1/2}(w^2-1)^{-1/4}}
{\sqrt{2\pi(M-1)}}
\left(1+o(1)\right),
\end{equation}
uniformly for $w \in K$, where $(w^2-1)^{-1/4}$, $(w^2-1)^{1/2}$ and $\left(w+(w^2-1)^{1/2}\right)^{M-1/2}$ are real and positive if $w$ is real and greater than $1$.

For $s$ in a sufficiently small complex neighborhood of $b$, set
$$
w(s)=\frac{1+s^2}{1-s^2}.
$$
This neighborhood can be chosen so that $w(s)$ remains in a compact subset of $\mathbb{C}\setminus[-1,1]$, and, with the principal branches,
$$
\sqrt{w(s)^2-1}
=
\frac{2s}{1-s^2},
\qquad
w(s)+\sqrt{w(s)^2-1}
=
\frac{1+s}{1-s}.
$$
Combining these identities with (\ref{Eq.HyperLegendre}) and (\ref{Eq.LegendreAsymptotics}), we get
\begin{equation}\label{Eq.HyperRatioAsymptotics}
\frac{F(M_n,M_n;1;b^2e^{\mathrm{i}u_n})}
{F(M_n,M_n;1;b^2)}
=
\left(\frac{1-be^{\mathrm{i}u_n/2}}{1-b}\right)^{1-2M_n}
e^{-\mathrm{i}u_n/4}
\left(1+o(1)\right).
\end{equation}

For $u$ sufficiently small,
\begin{equation}\label{Eq.ExpandLog}
\log(1-be^{\mathrm{i}u/2})
=
\log(1-b)-\frac{\mathrm{i} p}{2}u+\frac{\sigma^2}{8}u^2+O(u^3),
\end{equation}
where we recall that $p=\frac{b}{1-b}$ and $\sigma^2=p(p+1)=\frac{b}{(1-b)^2}$. Using (\ref{Eq.ExpandLog}) in (\ref{Eq.HyperRatioAsymptotics}), together with
$u_n=\xi/(\sigma\sqrt{d_n})$ and the assumptions in
(\ref{Eq.ConvBMScaling}), we obtain
\begin{equation}\label{Eq.HyperRatioLog}
\frac{F(M_n,M_n;1;b^2e^{\mathrm{i}u_n})}
{F(M_n,M_n;1;b^2)}
=
\exp\left(
\mathrm{i}pM_nu_n-\frac{\sigma^2M_n}{4}u_n^2+o(1)
\right).
\end{equation}
Substituting (\ref{Eq.HyperRatioLog}) into (\ref{Eq.MidpointCF}) gives
\begin{equation*}
\mathbb{E}\left[e^{\mathrm{i}\xi\mathcal{L}^n(0)}\right]
=
\exp\left(
\frac{\mathrm{i}\xi(x_n+pM_n)}{\sigma\sqrt{d_n}}
-\frac{\xi^2M_n}{4d_n}
+o(1)
\right)\rightarrow
\exp\left(
\mathrm{i}\xi x-\frac{\mathfrak{m}\xi^2}{4}
\right).
\end{equation*}
The expression on the right is the characteristic function of $x+Z$, which implies (\ref{Eq.ConvToGsn}).

\medskip 

{\bf \raggedleft Step 2.} In this step, we finish the proof of the lemma. 

For each
$n\in\mathbb{N}$, let $(\Omega_n,\mathcal{F}_n,\mathbb{P}_n)$ be a probability
space that supports, for every $z\in\mathbb{Z}_{\geq0}$ and $j\in\{1,2\}$, a pair $(L_j^{n,z},B_j^{\sigma,n,z})$ such that the following holds. Under $\mathbb{P}_n$, $L_j^{n,z}$ has law $\mathbb{P}_{\operatorname{Geom}}^{0,M_n,0,z}$, $B_j^{\sigma,n,z}$ is a Brownian bridge from $0$ to $0$ on $[0,1]$ with diffusion parameter $\sigma$, and, upon setting
\begin{equation}\label{Eq.KMTErrorNZ}
\Delta_j^{n,z}
:=
\sup_{0\leq t\leq M_n}
\left|
\sqrt{M_n}B^{\sigma,n,z}_{j,t/M_n}
+\frac{t}{M_n}z-L_j^{n,z}(t)
\right|,
\end{equation}
we have
\begin{equation}\label{Eq.KMTErrorBoundNZ}
\mathbb{E}_n\left[e^{a\Delta_j^{n,z}}\right]
\leq
CM_n^{\alpha'}e^{(z-pM_n)^2/M_n}.
\end{equation}
We assume that all pairs $(L_j^{n,z},B_j^{\sigma,n,z})$ for $z \geq 0$ are mutually independent. The existence of this space is ensured by Proposition \ref{Prop.KMT}.

By possibly enlarging the probability space, we may also assume that it
supports a random variable $Z_n$, independent of all the above pairs, with
distribution
\begin{equation}\label{Eq.ZnDistribution}
\mathbb{P}_n(Z_n=y)
=
\mathbb{P}\bigl(L^n(0)=x_n+y\bigr),
\qquad y\in\mathbb{Z}_{\geq0}.
\end{equation}
Define a random loop $\widetilde L^n$ on $[-M_n,M_n]$ by
\begin{equation}\label{Eq.TildeLoop}
\widetilde L^n(t)
=
x_n+
\begin{cases}
L_1^{n,Z_n}(t+M_n), & t\in[-M_n,0],\\
L_2^{n,Z_n}(M_n-t), & t\in[0,M_n].
\end{cases}
\end{equation}
Conditional on $Z_n=z$, the paths $L_1^{n,z}$ and $L_2^{n,z}$ are
independent and each is uniformly distributed on $\Omega(0,M_n,0,z)$.
On the other hand, using Definition~\ref{Def.GeometricLoopMeasure}, we see
that conditional on $L^n(0)=x_n+z$, the first half of $L^n$ and the
time-reversal of its second half have precisely this distribution. Consequently,
\begin{equation}\label{Eq.TildeLoopLaw}
\widetilde L^n\overset{d}{=}L^n.
\end{equation}

Let $\widetilde{\mathcal L}^n$ denote the rescaling of $\widetilde L^n$ as
in (\ref{Eq.ScaledPaths}), and set
\begin{equation}\label{Eq.anYn}
a_n=\frac{x_n+pM_n}{\sigma\sqrt{d_n}},
\qquad
Y_n= \frac{x_n+Z_n}{\sigma\sqrt{d_n}}.
\end{equation}
In view of (\ref{Eq.TildeLoopLaw}), we have $\widetilde{\mathcal L}^n\overset{d}{=}\mathcal{L}^n$, and so to prove the lemma, it suffices to show that
\begin{equation}\label{Eq.ConvBMRed1}
\widetilde{\mathcal L}^n\Rightarrow B.
\end{equation}

\medskip

Our next task is to define a Brownian approximation to
$\widetilde{\mathcal L}^n$ and show that it converges to $B$. For $t\in[0,\mathfrak{m}]$, define
\begin{equation}\label{Eq.BnDefinition}
B_1^n(t)
=
\frac{\sqrt{M_n}}{\sigma\sqrt{d_n}}
B_{1,t/\mathfrak{m}}^{\sigma,n,Z_n},
\qquad
B_2^n(t)
=
\frac{\sqrt{M_n}}{\sigma\sqrt{d_n}}
B_{2,1-t/\mathfrak{m}}^{\sigma,n,Z_n}.
\end{equation}
Since the law of $B_j^{\sigma,n,z}$ does not depend on $z$, the random
elements $Y_n,B_1^n,B_2^n$ are independent. Moreover, $B_1^n$ and $B_2^n$
are Brownian bridges from $0$ to $0$ on $[0,\mathfrak{m}]$ with diffusion
parameter $\sqrt{\frac{M_n}{\mathfrak{m}d_n}}$, which converges to $1$ by (\ref{Eq.ConvBMScaling}). In addition, by (\ref{Eq.ConvBMScaling}), (\ref{Eq.ConvToGsn}), (\ref{Eq.ZnDistribution}), and (\ref{Eq.anYn}), we have
\begin{equation}\label{Eq.YnConvergence}
Y_n\Rightarrow x+Z, \hspace{2mm} \frac{Z_n - pM_n}{\sigma \sqrt{d_n}} \Rightarrow Z, \mbox{ and } a_n \rightarrow x.
\end{equation}
Consequently, if
$B_1,B_2$ are independent Brownian bridges from $0$ to $0$ on
$[0,\mathfrak{m}]$ with diffusion parameter $1$, also independent of $Z$,
then
\begin{equation}\label{Eq.JointBridgeConvergence}
(a_n,Y_n,B_1^n,B_2^n)
\Rightarrow
(x, x+Z,B_1,B_2) \mbox{ in $\mathbb{R}^2 \times
C_0([0,\mathfrak{m}])^2$,}
\end{equation}
see \cite[Theorem 2.8(ii)]{Billing}, where
$$
C_0([0,\mathfrak{m}])
=
\left\{
f\in C([0,\mathfrak{m}]):f(0)=f(\mathfrak{m})=0
\right\}.
$$

Define
$\Phi:\mathbb{R}^2\times C_0([0,\mathfrak{m}])^2\rightarrow
C([-\mathfrak{m},\mathfrak{m}])$ by
\begin{equation}\label{Eq.PhiDefinition}
\Phi(a,y,f_1,f_2)(t)
=
\begin{cases}
f_1(t+\mathfrak{m})
-\dfrac{t}{\mathfrak{m}}a
+\dfrac{t+\mathfrak{m}}{\mathfrak{m}}y,
& t\in[-\mathfrak{m},0],\\[3mm]
f_2(t)
+\dfrac{t}{\mathfrak{m}}a
+\dfrac{\mathfrak{m}-t}{\mathfrak{m}}y,
& t\in[0,\mathfrak{m}].
\end{cases}
\end{equation}
The map $\Phi$ is continuous with respect to the usual Euclidean and uniform
metrics. Set
\begin{equation}\label{Eq.GnDefinition}
\mathcal{G}^n
=
\Phi(a_n,Y_n,B_1^n,B_2^n).
\end{equation}
From (\ref{Eq.JointBridgeConvergence}) and the mapping theorem \cite[Theorem 2.7]{Billing}, we conclude
\begin{equation}\label{Eq.GnConvergence}
\mathcal{G}^n
\Rightarrow
\Phi(x,x+Z,B_1,B_2).
\end{equation}
From (\ref{Eq.PhiDefinition}), the process on the right-hand side of
(\ref{Eq.GnConvergence}) is given by
$$
\Phi(x,x+Z,B_1,B_2)(t)
=
x+
\begin{cases}
B_1(t+\mathfrak{m})
+\dfrac{t+\mathfrak{m}}{\mathfrak{m}}Z,
& t\in[-\mathfrak{m},0],\\[3mm]
B_2(t)
+\dfrac{\mathfrak{m}-t}{\mathfrak{m}}Z,
& t\in[0,\mathfrak{m}].
\end{cases}
$$
Thus, by Lemma~\ref{Lem.LevyBM},
\begin{equation}\label{Eq.GnToB}
\mathcal{G}^n\Rightarrow B.
\end{equation}

It remains to compare $\mathcal{G}^n$ and
$\widetilde{\mathcal L}^n$. From (\ref{Eq.TildeLoop}),
(\ref{Eq.BnDefinition}), and (\ref{Eq.GnDefinition}), we have
\begin{equation}\label{Eq.KMTPathDistance}
\left\|
\widetilde{\mathcal L}^n-\mathcal{G}^n
\right\|_{\infty} = \sup_{t \in [-\mathfrak{m}, \mathfrak{m}]} \left|\widetilde{\mathcal L}^n(t)-\mathcal{G}^n(t)\right|
\leq
\frac{1}{\sigma\sqrt{d_n}}
\max_{j\in\{1,2\}}\Delta_j^{n,Z_n}.
\end{equation}
Moreover, from (\ref{Eq.YnConvergence}) we have
\begin{equation}\label{Eq.ZnTypical}
\lim_{n \rightarrow \infty} \mathbb{P}_n\left(
A_n^c\right)= 0, \mbox{ where }A_n = \{ |Z_n-pM_n| \leq d_n^{1/10} \sqrt{M_n}\}.
\end{equation}

Using (\ref{Eq.KMTErrorBoundNZ}) conditionally on $Z_n$, followed by
Markov's inequality, for $j\in\{1,2\}$ we obtain 
\begin{align*}
&\mathbb{P}_n\left(
\Delta_j^{n,Z_n}> d_n^{-1/10}\sigma\sqrt{d_n}, A_{n}
\right)\leq CM_n^{\alpha'}e^{d_n^{2/10}-ad_n^{-1/10} \sigma\sqrt{d_n}}
\longrightarrow0.
\end{align*}
Together with (\ref{Eq.KMTPathDistance}) and (\ref{Eq.ZnTypical}), this shows that
\begin{equation}\label{Eq.ConvergenceTogetherDistance}
\left\|
\widetilde{\mathcal L}^n-\mathcal{G}^n
\right\|_{\infty}
\longrightarrow0
\qquad\mbox{in probability}.
\end{equation}
The convergence together lemma \cite[Theorem 3.1]{Billing}, (\ref{Eq.GnToB}), and (\ref{Eq.ConvergenceTogetherDistance}) imply (\ref{Eq.ConvBMRed1}), completing the proof of the lemma.
\end{proof}

%----------------------------------------------------------------------
%
% Section 6.3
%
%----------------------------------------------------------------------
\subsection{Proof of Theorem \ref{Thm.GibbsProperty}}\label{Section6.3}

Throughout this section, we freely use the notation and results from Sections~\ref{Section6.1} and~\ref{Section6.2}, in addition to those from Sections~\ref{Section1.2} and~\ref{Section2.2}. We start by introducing some notation and collecting some consequences of the results we have already proved. 

Assuming the same notation as in Definition \ref{Def.S1Scaling}, we define 
\begin{equation}\label{Eq.S63LIN}
L_i^N(t) = \lambda_i^{(t)} - \lfloor c_1 N \rfloor \mbox{ for } i \geq 1, t \in \llbracket -M_N, M_N \rrbracket.
\end{equation}
As usual, by linear interpolation we view $L_i^N$ as random elements of $C([-M_N, M_N])$; see Section \ref{Section2.2}. We next define 
\begin{equation}\label{Eq.S63NewScalingParameters}
\tilde{\sigma} = \frac{b^{1/2}}{1-b}, \hspace{2mm} d_N = 2c_2^{-1} N^{2/3}, 
\end{equation}
and also
\begin{equation}\label{Eq.S63ScaledLIN}
\widetilde{\mathcal{L}}_i^N(t) = \tilde{\sigma}^{-1} d_N^{-1/2} \left( L^N_i\left(\frac{2tM_N}{R}\right)+\frac{b}{1-b} \cdot \frac{2M_N|t|}{R} \right) \mbox{ for } i\geq 1, t \in [-R/2, R/2].
\end{equation}
From Definition \ref{Def.S1Scaling}, we observe that 
\begin{equation}\label{Eq.S63ReconcileScaling}
d_N \rightarrow \infty, \hspace{2mm} M_N/d_N \rightarrow R/2, \hspace{2mm} \widetilde{\mathcal{L}}_i^N(t) = 2^{-1/2}  \left(\mathcal{L}_i^N(t) + t^2\right)  + O(N^{-1/3}),
\end{equation}
where the constant in the big $O$ notation depends on $b$ alone. Finally, we define the line ensemble $\widetilde{\mathcal{L}}^{\mathrm{per},R}
=
\{\widetilde{\mathcal{L}}^{\mathrm{per},R}_i\}_{i\geq1}$ on $[-R/2,R/2]$
through
\begin{equation}\label{Eq.S63PlusParabolaLE}
\widetilde{\mathcal{L}}^{\mathrm{per},R}_i(t)
= 2^{-1/2}\left(\mathcal{A}^{\mathrm{per},R}_i(t)+t^2\right),
\qquad
i\geq1,\quad t\in[-R/2,R/2].
\end{equation}

Combining Theorem \ref{Thm.Convergence}, the third statement in
(\ref{Eq.S63ReconcileScaling}), and (\ref{Eq.S63PlusParabolaLE}) with the
mapping theorem \cite[Theorem 2.7]{Billing} and the convergence together
lemma \cite[Theorem 3.1]{Billing}, we obtain
\begin{equation}\label{Eq.S6ConvToPALE}
\widetilde{\mathcal{L}}^N \Rightarrow \widetilde{\mathcal{L}}^{\mathrm{per},R} \mbox{ in } C(\mathbb{N} \times [-R/2,R/2]).
\end{equation}
From (\ref{Eq.S1PeriodicPerAiry}), we have almost surely
\begin{equation}\label{Eq.S6PALEIsPeriodic}
\widetilde{\mathcal{L}}^{\mathrm{per},R}_i(-R/2)
=
\widetilde{\mathcal{L}}^{\mathrm{per},R}_i(R/2)
\qquad\text{for all $i\geq1$}.
\end{equation}
In addition, from Theorem \ref{Thm.DeterminantalStructure} we know that the measure $M^{\mathsf{S};R}$ with $\mathsf{S} = \{s\}$ is a determinantal point process on $\mathbb{R}^2$ for each $s \in (-R/2, R/2]$. As determinantal point processes are almost surely simple, see \cite[Proposition 2.13(1)]{ED24a}, and the line ensemble $\mathcal{A}^{\mathrm{per},R}$ is ordered, we conclude that for each $s \in (-R/2, R/2]$, almost surely,
\begin{equation*}
\mathcal{A}^{\mathrm{per},R}_1(s) > \mathcal{A}^{\mathrm{per},R}_2(s) > \cdots.
\end{equation*}
The latter, (\ref{Eq.S63PlusParabolaLE}), and (\ref{Eq.S6PALEIsPeriodic}) imply that for each $s \in [-R/2, R/2]$, almost surely,
\begin{equation}\label{Eq.S6PALEIsSimpleOrdered}
\widetilde{\mathcal{L}}^{\mathrm{per},R}_1(s) > \widetilde{\mathcal{L}}^{\mathrm{per},R}_2(s) > \cdots.
\end{equation}

For clarity, we split the remainder of the proof into three steps. In Step 1, we reduce the statement of the theorem to establishing a certain equality of expectations involving $\widetilde{\mathcal{L}}^{\mathrm{per},R}_{n_j}(t_j)$ for finitely many $j$'s, see (\ref{Eq.S63ExpectationEquality}). In Step 2, we recast the definition of our cylindric plane partition model from Section \ref{Section1.2} in terms of loops, and establish a discrete analogue of the Brownian Gibbs property, see (\ref{Eq.S63DiscreteExpectationEqualityV1}). In Step 3, we deduce (\ref{Eq.S63ExpectationEquality}) by utilizing (\ref{Eq.S63DiscreteExpectationEqualityV1}) from Step 2 as well as Lemma \ref{Lem.ConvBM}.

\medskip

{\bf \raggedleft Step 1.} Fix $k_2, m\in\mathbb{N}$, $n_1,\ldots,n_m\in \mathbb{N}$, $t_1,\ldots,t_m\in[-R/2,R/2]$, and bounded continuous functions $h_1,\ldots,h_m:\mathbb{R}\rightarrow\mathbb{R}$. Set $\mathsf{R} = \left\{ j\in\llbracket1,m\rrbracket:n_j\in  \llbracket 1, k_2 \rrbracket \right\}$. 

We claim that 
\begin{equation}\label{Eq.S63ExpectationEquality} 
\begin{split} \mathbb{E}\left[ \prod_{j=1}^m h_j\bigl(\widetilde{\mathcal{L}}^{\mathrm{per},R}_{n_j}(t_j)\bigr) \right] ={}& \mathbb{E}\left[ \prod_{j\notin\mathsf{R}} h_j\bigl(\widetilde{\mathcal{L}}^{\mathrm{per},R}_{n_j}(t_j)\bigr) \cdot \mathbb{E}_{\mathrm{avoid},2^{3/2}}^{1,k_2,-R/2,R/2, \vec{x},\vec{x},\infty,g} \left[ \prod_{j\in\mathsf{R}} h_j\bigl(\mathcal{Q}_{n_j}(t_j)\bigr) \right] \right], 
\end{split} 
\end{equation} 
where $\vec{x} = (\widetilde{\mathcal{L}}^{\mathrm{per},R}_1(-R/2), \dots, \widetilde{\mathcal{L}}^{\mathrm{per},R}_{k_2}(-R/2))$ and $g = \widetilde{\mathcal{L}}^{\mathrm{per},R}_{k_2 +1}$. On the right-hand side of (\ref{Eq.S63ExpectationEquality}), we have that $\mathcal{Q} = \{\mathcal{Q}_i\}_{i = 1}^{k_2}$ is the $\llbracket 1, k_2 \rrbracket$-indexed line ensemble on $[-R/2,R/2]$ with law given by $\mathbb{P}_{\mathrm{avoid},2^{3/2}}^{1,k_2,-R/2,R/2, \vec{x},\vec{x},\infty,g}$ as in Definition \ref{Def.AreaTiltedBrownianLE}. We mention that in view of (\ref{Eq.S6PALEIsSimpleOrdered}) with $s = -R/2$, we have that the expectation over $\mathcal{Q}$, and hence the right-hand side of (\ref{Eq.S63ExpectationEquality}), is well-defined and finite. 

We establish (\ref{Eq.S63ExpectationEquality}) in the steps below. In the remainder of this step, we assume its validity and conclude the proof of the theorem.

\medskip
By the monotone class argument in \cite[Lemma~A.1]{DY25},
(\ref{Eq.S63ExpectationEquality}) implies that, for $\alpha=-R/2$,
$\beta=R/2$, every bounded Borel-measurable function
$F:C(\llbracket 1, k_2 \rrbracket\times[\alpha,\beta])\rightarrow\mathbb{R}$,
and every bounded
$\mathcal{F}_{\mathrm{ext}}(\llbracket 1, k_2 \rrbracket \times(\alpha,\beta))$-measurable
random variable $Y$, we have
\begin{equation}\label{Eq.S63PeriodicGibbs}
\mathbb{E}\left[ F\left(\widetilde{\mathcal{L}}^{\mathrm{per},R}|_{\llbracket 1, k_2 \rrbracket\times[\alpha,\beta]}\right)Y \right] = \mathbb{E}\left[ \mathbb{E}_{\mathrm{avoid},2^{3/2}}^{1,k_2,\alpha,\beta, \vec{x},\vec{x},\infty,g} \left[F(\mathcal{Q})\right]Y \right]. 
\end{equation} 
Here, $\mathcal{F}_{\mathrm{ext}}(\llbracket 1, k_2 \rrbracket \times(\alpha,\beta))$ is as in Definition \ref{Def.PeriodicAreaTiltedBGP} with $\mathcal{L}$ replaced by $\widetilde{\mathcal{L}}^{\mathrm{per},R}$.

Setting in (\ref{Eq.S63PeriodicGibbs}) $Y = 1$ and 
$$F(f_1, \dots, f_{k_2}) = {\bf 1}\{f_1(t) > \cdots > f_{k_2}(t) \mbox{ for all } t \in [\alpha, \beta]\},$$
we conclude that $\widetilde{\mathcal{L}}^{\mathrm{per},R}$ is non-intersecting almost surely. By (\ref{Eq.S6PALEIsPeriodic}), it is also periodic. The last two observations, together with (\ref{Eq.S63PeriodicGibbs}) and the defining properties of conditional expectation, show that $\widetilde{\mathcal{L}}^{\mathrm{per},R}$ satisfies the periodic Brownian Gibbs property with $2^{3/2}$-area tilt from Definition \ref{Def.PeriodicAreaTiltedBGP}.

From (\ref{Eq.ParabolicPeriodicAiry}) and (\ref{Eq.S63PlusParabolaLE}), we conclude that 
$$\widetilde{\mathcal{L}}^{\mathrm{per},R}_i(t) = \mathcal{L}^{\mathrm{per},R}_i(t) + 2^{1/2} t^2 \mbox{ for } i \geq 1, t\in [-R/2,R/2],$$
and so \cite[Lemma 2.8]{DimitrovSerio} implies that $\mathcal{L}^{\mathrm{per},R}$ satisfies the partial Brownian Gibbs property from Definition \ref{Def.AreaTiltedBGP}. Combining the latter with the periodicity of $\mathcal{L}^{\mathrm{per},R}$ and Lemma \ref{Lem.PeriodicBGPCharacterization}, we conclude the statement of the theorem.

\medskip

{\bf \raggedleft Step 2.} Fix $b,q \in (0,1)$, $M \geq 1$, and let $(\lambda^{(-M)}, \dots, \lambda^{(M)}) \in \Seq$ be distributed as in (\ref{Eq.S1ProbDistr}), where we recall that $\Seq$ is as in (\ref{Eq.DefLambda}). Given $(\lambda^{(-M)}, \dots, \lambda^{(M)}) \in \Seq$, we define 
\begin{equation}\label{Eq.S6FromLambdaToLoop}
L_i(t) = \lambda_i^{(t)},
\end{equation}
and note that (\ref{Eq.S6FromLambdaToLoop}) defines a bijection between $\Seq$ and 
\begin{equation}\label{Eq.InterlacingLoopsSpace}
S^{\mathrm{loop}} = \left\{\mathfrak{L} = \{L_i\}_{i \geq 1} \in (\Omega_{M}^{\mathrm{loop}})^{\mathbb{N}}: L_1 \succeq L_2 \succeq \cdots \mbox{ and } \lim_{n \rightarrow \infty} L_n(s) = 0 \hspace{2mm}  \forall s \in \llbracket -M, M \rrbracket \right\}.
\end{equation}
In view of (\ref{Eq.S1ProbDistr}), we see that the induced measure on $S^{\mathrm{loop}}$ under the above correspondence satisfies
\begin{equation}\label{Eq.S6LoopMeasure}
\mathbb{P}\left(\mathfrak{L} \right) = \frac{1}{Z_{M,b,q}} \cdot b^{ 2\sum_{i \geq 1} \left(L_i(0) - L_i(M)\right)} \cdot q^{-(M+1)\sum_{i \geq 1}L_i(M)} \cdot q^{\sum_{i \geq 1} \|L_i\|_1},
\end{equation}
where $\|L_i\|_1 = \sum_{s \in \llbracket -M, M \rrbracket} L_i(s)$.

Fix $k_2 \in \mathbb{N}$ and set $\mathcal{F}^{k_2}_{\mathrm{ext}} = \sigma \left(\{L_i(s): s \in \llbracket -M, M \rrbracket, i \geq k_2 + 1\} \cup \{L_i(-M): i \in \llbracket 1, k_2 \rrbracket\}\right)$. Our goal in this step is to show that for any bounded $F: (\Omega_{M}^{\mathrm{loop}})^{\llbracket 1, k_2\rrbracket} \rightarrow \mathbb{R}$, we have 
\begin{equation}\label{Eq.S63DiscreteExpectationEqualityV1}
\mathbb{E}\left[F(L_1, \dots, L_{k_2}) \vert \mathcal{F}^{k_2}_{\mathrm{ext}}\right] = \frac{\mathbb{E}\left[F(Q_1, \dots, Q_{k_2}) {\bf 1}\{Q_1 \succeq \cdots \succeq Q_{k_2} \succeq L_{k_2+1} \} q^{\sum_{i = 1}^{k_2} \|Q_i\|_1}\right]}{\mathbb{E}\left[ {\bf 1}\{Q_1 \succeq \cdots \succeq Q_{k_2} \succeq L_{k_2+1} \} q^{\sum_{i = 1}^{k_2} \|Q_i\|_1}\right]},
\end{equation}
where $Q_1, \dots, Q_{k_2}$ are independent and $Q_i$ has law $\mathbb{P}_{\mathrm{loop}}^{b,M,L_i(-M)}$ as in Definition \ref{Def.GeometricLoopMeasure}. 

Fix $\vec{x}=(x_1,\ldots,x_{k_2})\in\mathbb{Z}^{k_2}$ and loops
$\ell_i\in\Omega_M^{\mathrm{loop}}$ for $i\geq k_2+1$, such that the event
\[
E=
\left\{
L_i(-M)=x_i \mbox{ for } i \in \llbracket 1, k_2 \rrbracket,
\quad
L_i=\ell_i \mbox{ for } i\geq k_2+1
\right\}
\]
has positive probability. Fix additional loops $\ell_i\in\Omega_{x_i,M}^{\mathrm{loop}}$ for $i\in\llbracket1,k_2\rrbracket$. From (\ref{Eq.S6LoopMeasure}), we have
\begin{equation*}
\mathbb{P}\left(
L_i=\ell_i \text{ for } i\in\llbracket1,k_2\rrbracket
\middle| E
\right) =
\frac{
{\bf 1}\{\ell_1\succeq\cdots\succeq \ell_{k_2}
\succeq \ell_{k_2+1}\}
\prod_{i=1}^{k_2}
b^{2\ell_i(0)}
q^{\|\ell_i\|_1}
}{
\sum_{(\widetilde{\ell}_1,\ldots,\widetilde{\ell}_{k_2})
\in \prod_{i=1}^{k_2}\Omega_{x_i,M}^{\mathrm{loop}}}
{\bf 1}\{
\widetilde{\ell}_1\succeq\cdots\succeq
\widetilde{\ell}_{k_2}\succeq\ell_{k_2+1}
\}
\prod_{i=1}^{k_2}
b^{2\widetilde{\ell}_i(0)}
q^{\|\widetilde{\ell}_i\|_1}
}.
\end{equation*}
Using the last identity and Definition~\ref{Def.GeometricLoopMeasure}, we conclude that if $Q_1,\ldots,Q_{k_2}$ are independent and $Q_i$ has law $\mathbb{P}_{\mathrm{loop}}^{b,M,x_i}$, then
\begin{equation*}
\mathbb{P}\left(
L_i=\ell_i \text{ for } i\in\llbracket1,k_2\rrbracket
\middle| E
\right) =
\frac{
\mathbb{P}\left(
Q_i=\ell_i \text{ for } i\in\llbracket1,k_2\rrbracket
\right)
{\bf 1}\{\ell_1\succeq\cdots\succeq \ell_{k_2}
\succeq \ell_{k_2+1}\}
q^{\sum_{i=1}^{k_2}\|\ell_i\|_1}
}{
\mathbb{E}\left[
{\bf 1}\{Q_1\succeq\cdots\succeq Q_{k_2}
\succeq\ell_{k_2+1}\}
q^{\sum_{i=1}^{k_2}\|Q_i\|_1}
\right]
}.
\end{equation*}
Multiplying the last identity by $F(\ell_1,\ldots,\ell_{k_2})$ and summing over
$\ell_i\in\Omega_{x_i,M}^{\mathrm{loop}}$,
$i\in\llbracket1,k_2\rrbracket$ gives
\begin{equation*}
\mathbb{E}\left[
F(L_1,\ldots,L_{k_2})\middle|E
\right]  =
\frac{
\mathbb{E}\left[
F(Q_1,\ldots,Q_{k_2})
{\bf 1}\{Q_1\succeq\cdots\succeq Q_{k_2}
\succeq\ell_{k_2+1}\}
q^{\sum_{i=1}^{k_2}\|Q_i\|_1}
\right]
}{
\mathbb{E}\left[
{\bf 1}\{Q_1\succeq\cdots\succeq Q_{k_2}
\succeq\ell_{k_2+1}\}
q^{\sum_{i=1}^{k_2}\|Q_i\|_1}
\right]
}.
\end{equation*}
Since this holds for every atom $E$ of
$\mathcal{F}_{\mathrm{ext}}^{k_2}$ having positive probability,
(\ref{Eq.S63DiscreteExpectationEqualityV1}) follows.

\medskip

{\bf \raggedleft Step 3.} In this step, we prove (\ref{Eq.S63ExpectationEquality}).

Using the tower property for conditional expectation, we obtain
\begin{equation}\label{Eq.S63DiscreteExpectationEquality}
\begin{split}
&\mathbb{E}\left[ \prod_{j=1}^m
h_j\bigl(\widetilde{\mathcal{L}}^N_{n_j}(t_j)\bigr)
\right] = \mathbb{E}\left[ \prod_{j\notin\mathsf{R}}
h_j\bigl(\widetilde{\mathcal{L}}^N_{n_j}(t_j)\bigr) \mathbb{E}\left[ \prod_{j\in\mathsf{R}}
h_j\bigl(\widetilde{\mathcal{L}}^N_{n_j}(t_j)\bigr)\vert \mathcal{F}_{\mathrm{ext}}^{k_2}  \right]   \right] \\
&=
\mathbb{E}\left[
\prod_{j\notin\mathsf{R}}
h_j\bigl(\widetilde{\mathcal{L}}^N_{n_j}(t_j)\bigr)
\cdot
\frac{
\mathbb{E}\left[
\prod_{j\in\mathsf{R}}
h_j\bigl(\widetilde{\mathcal{Q}}^N_{n_j}(t_j)\bigr)
{\bf 1}_{\mathsf{A}_N}
q_N^{\sum_{i=1}^{k_2}\|Q_i^N - L_{k_2+1}^N\|_1}
\right]
}{
\mathbb{E}\left[
{\bf 1}_{\mathsf{A}_N}
q_N^{\sum_{i=1}^{k_2}\|Q_i^N - L_{k_2+1}^N\|_1}
\right]}
\right],
\end{split}
\end{equation}
where $Q_1^N,\ldots,Q_{k_2}^N$ are independent and $Q_i^N$ has law $\mathbb{P}_{\mathrm{loop}}^{b,M_N,L_i^N(-M_N)}$ for $i \in \llbracket 1, k_2 \rrbracket$, the event $\mathsf{A}_N$ is given by
\begin{equation}\label{Eq.S63EventInterlace}
\mathsf{A}_N = \left\{ Q_1^N\succeq\cdots\succeq Q_{k_2}^N \succeq L_{k_2+1}^N\right\},
\end{equation}
we have set 
\begin{equation}\label{Eq.S63ScaledQN}
\widetilde{\mathcal{Q}}_i^N(t) = \tilde{\sigma}^{-1}d_N^{-1/2}
\left( Q_i^N\left(\frac{2tM_N}{R}\right)
+\frac{b}{1-b}\cdot\frac{2M_N|t|}{R} \right), \mbox{ for } i \in \llbracket 1, k_2 \rrbracket, t\in[-R/2,R/2],
\end{equation}
and for two loops $Q,G\in\Omega_{M_N}^{\mathrm{loop}}$, we set
$$ \|Q-G\|_1 = \sum_{s\in\llbracket-M_N,M_N\rrbracket}(Q(s)-G(s)).$$
We mention that in obtaining the second equality in (\ref{Eq.S63DiscreteExpectationEquality}), we used (\ref{Eq.S63DiscreteExpectationEqualityV1}) with $M=M_N$ and $q=q_N$, with both numerator and denominator in the right-hand side of (\ref{Eq.S63DiscreteExpectationEqualityV1}) multiplied by
$$q_N^{ - k_2\|L_{k_2+1}\|_1},$$
and we have set $Q_i = Q_i^N + \lfloor c_1 N \rfloor$ for $i \in \llbracket 1, k_2 \rrbracket$, $L_i = L_i^N + \lfloor c_1 N \rfloor$ for $i \geq 1$.

\medskip

We first analyze the expectations on the second line of (\ref{Eq.S63DiscreteExpectationEquality}). Suppose that
$z_1^N,\ldots,z_{k_2}^N\in\mathbb{Z}$ and
$G^N\in\Omega_{M_N}^{\mathrm{loop}}$ are deterministic and satisfy
\begin{equation}\label{Eq.S63BoundaryConvergence}
\tilde{\sigma}^{-1}d_N^{-1/2}(z_i^N+pM_N)
\longrightarrow x_i,
\qquad
g^N\longrightarrow g
\quad\text{uniformly on }[-R/2,R/2],
\end{equation}
where $p=b/(1-b)$,
\[
g^N(t)
=
\tilde{\sigma}^{-1}d_N^{-1/2}
\left(
G^N\left(\frac{2tM_N}{R}\right)
+
p\frac{2M_N|t|}{R}
\right),
\]
and $\vec{x} = (x_1, \dots, x_{k_2})$ is such that
\[
x_1>\cdots>x_{k_2}>g(-R/2)=g(R/2).
\]
Let $Q_i^N$ be independent with law
$\mathbb{P}_{\mathrm{loop}}^{b,M_N,z_i^N}$, and define
$\widetilde{\mathcal{Q}}_i^N$ as in
(\ref{Eq.S63ScaledQN}), and $\mathsf{A}_N$ as in (\ref{Eq.S63EventInterlace}) with $L_{k_2+1}^N$ replaced by $G^N$. In view of
(\ref{Eq.S63ReconcileScaling}), Lemma~\ref{Lem.ConvBM}, and \cite[Theorem 2.8(ii)]{Billing}, we have
\begin{equation}\label{Eq.S63FreeBridgeConvergence}
(\widetilde{\mathcal{Q}}_1^N,\ldots,
\widetilde{\mathcal{Q}}_{k_2}^N)
\Rightarrow
(\mathcal{Q}_1,\ldots,\mathcal{Q}_{k_2}),
\end{equation}
where $\mathcal{Q}_1,\ldots,\mathcal{Q}_{k_2}$ are independent Brownian
bridges from $x_i$ to $x_i$ on $[-R/2,R/2]$, all with diffusion parameter
$1$. By the Skorohod representation theorem
\cite[Theorem~6.7]{Billing}, we may assume that the convergence in
(\ref{Eq.S63FreeBridgeConvergence}) holds $\omega$-wise uniformly.
We note that the Skorohod representation theorem is applicable since $C(\llbracket 1, k_2 \rrbracket \times [-R/2, R/2])$ is Polish,
as discussed in Section~\ref{Section2.2}.

Set
$$ \mathsf{A}
=
\left\{
\mathcal{Q}_1(t)>\cdots>\mathcal{Q}_{k_2}(t)>g(t)
\text{ for all }t\in[-R/2,R/2]
\right\},$$
and 
$$ \mathsf{B}
=\left\{
\mathcal{Q}_i(t) < \mathcal{Q}_{i+1}(t) \mbox{ for some } i \in \llbracket 1, k_2 \rrbracket \mbox{ and } t\in[-R/2,R/2]
\right\},$$
where $\mathcal{Q}_{k_2 + 1} = g$ by convention. From \cite[Corollary~2.9]{CorHamA} we know 
$$\mathbb{P}(\mathsf{A} \cup \mathsf{B}) = 1.$$
By the assumed uniform convergence, we also have
$$\lim_{N \rightarrow \infty} {\bf 1}_{\mathsf{A}_N} = 1 \mbox{ on } \mathsf{A}, \mbox{ and }\lim_{N \rightarrow \infty} {\bf 1}_{\mathsf{A}_N} = 0 \mbox{ on } \mathsf{B}.$$
Consequently,
\begin{equation}\label{Eq.S63IndicatorConvergence}
\lim_{N \rightarrow \infty} {\bf 1}_{\mathsf{A}_N} = {\bf 1}_{\mathsf{A}}
\qquad\text{almost surely}.
\end{equation}

Setting $t_{s,N}=\frac{Rs}{2M_N}$ for $s\in\llbracket-M_N,M_N\rrbracket$ and recalling that $q_N=e^{-1/N}$, we obtain
\begin{equation*}
q_N^{\sum_{i=1}^{k_2}\|Q_i^N - G^N\|_1} =
\exp\left(
-a_N
\sum_{i=1}^{k_2}
\frac{R}{2M_N}
\sum_{s=-M_N}^{M_N}
\left(
\widetilde{\mathcal{Q}}_i^N(t_{s,N})
-g^N(t_{s,N})
\right)
\right),
\end{equation*}
where $a_N = \frac{2M_N\tilde{\sigma}d_N^{1/2}}{RN}$. Using (\ref{Eq.S63NewScalingParameters}), the fact that $M_N = \left\lceil c_2^{-1}RN^{2/3}\right\rceil$, and
$\tilde{\sigma}=c_2^{3/2}$, we have
\begin{equation}\label{Eq.S63AreaCoefficient}
\lim_{N \rightarrow \infty} a_N = 2^{3/2}.
\end{equation}
Thus, the almost sure uniform convergence above and the convergence of
Riemann sums imply
\begin{equation}\label{Eq.S63AreaWeightConvergence}
\lim_{N \rightarrow \infty} q_N^{\sum_{i=1}^{k_2}\|Q_i^N - G^N\|_1}
=
\exp\left(
-2^{3/2}\sum_{i=1}^{k_2}
\int_{-R/2}^{R/2}
\bigl(\mathcal{Q}_i(t)-g(t)\bigr) dt
\right)
\end{equation}
almost surely. In addition, we observe that on $\mathsf{A}_N$, we have $Q_i^N(s) \geq G^N(s)$, and so
\begin{equation}\label{Eq.S63AreaWeightBdd}
0 \leq {\bf 1}_{\mathsf{A}_N} q_N^{\sum_{i=1}^{k_2}\|Q_i^N - G^N\|_1} \leq 1.
\end{equation}

Combining (\ref{Eq.S63IndicatorConvergence}), (\ref{Eq.S63AreaWeightConvergence}), and (\ref{Eq.S63AreaWeightBdd}) and using bounded convergence and the boundedness and continuity of the functions $h_j$, we conclude that
\begin{equation}\label{Eq.S63RatioConvergence}
\begin{split}
&\lim_{N \rightarrow \infty} \frac{
\mathbb{E}\left[
\prod_{j\in\mathsf{R}}
h_j\bigl(\widetilde{\mathcal{Q}}^N_{n_j}(t_j)\bigr)
{\bf 1}_{\mathsf{A}_N}
q_N^{\sum_{i=1}^{k_2}\|Q_i^N - G^N\|_1}
\right]
}{
\mathbb{E}\left[
{\bf 1}_{\mathsf{A}_N}
q_N^{\sum_{i=1}^{k_2}\|Q_i^N - G^N\|_1}
\right]}\\
&= 
\frac{
\mathbb{E}_{\mathrm{free}}^{1,k_2,-R/2,R/2,\vec{x},\vec{x}}
\left[
\prod_{j\in\mathsf{R}}h_j(\mathcal{Q}_{n_j}(t_j))
{\bf 1}_{\mathsf{A}}
\exp\left(
-2^{3/2}\sum_{i=1}^{k_2}
\int_{-R/2}^{R/2}
(\mathcal{Q}_i(t)-g(t))dt
\right)
\right]
}{
\mathbb{E}_{\mathrm{free}}^{1,k_2,-R/2,R/2,\vec{x},\vec{x}}
\left[
{\bf 1}_{\mathsf{A}}
\exp\left(
-2^{3/2}\sum_{i=1}^{k_2}
\int_{-R/2}^{R/2}
(\mathcal{Q}_i(t)-g(t))dt
\right)
\right]
}
\\
&=
\mathbb{E}_{\mathrm{avoid},2^{3/2}}^{1,k_2,-R/2,R/2,
\vec{x},\vec{x},\infty,g}
\left[
\prod_{j\in\mathsf{R}}
h_j\bigl(\mathcal{Q}_{n_j}(t_j)\bigr)
\right].
\end{split}
\end{equation}
We mention that, as discussed in Remark \ref{Rem.Event}, the denominator on the second line of (\ref{Eq.S63RatioConvergence}) is strictly positive. In addition, the last equality in (\ref{Eq.S63RatioConvergence}) follows from
Definitions~\ref{Def.AvoidingBrownianLE} and~\ref{Def.AreaTiltedBrownianLE}.

\medskip

We now return to (\ref{Eq.S63DiscreteExpectationEquality}). From
(\ref{Eq.S6ConvToPALE}) and another application of the Skorohod
representation theorem \cite[Theorem~6.7]{Billing}, we may assume that
\begin{equation}\label{Eq.S63SkorohodOuter}
\widetilde{\mathcal{L}}^N
\longrightarrow
\widetilde{\mathcal{L}}^{\mathrm{per},R}
\qquad\text{almost surely in }
C(\mathbb{N}\times[-R/2,R/2]).
\end{equation}
As above, this application is justified by the fact that $C(\mathbb{N}\times[-R/2,R/2])$ is Polish.

On the probability-one event on which
(\ref{Eq.S63SkorohodOuter}) holds and
(\ref{Eq.S6PALEIsSimpleOrdered}) holds at $s=-R/2$, the deterministic
convergence statement (\ref{Eq.S63RatioConvergence}) applies with
$$ z_i^N=L_i^N(-M_N), \hspace{2mm} G^N=L_{k_2+1}^N,\hspace{2mm} x_i=\widetilde{\mathcal{L}}^{\mathrm{per},R}_i(-R/2), \hspace{2mm} g=\widetilde{\mathcal{L}}^{\mathrm{per},R}_{k_2+1}.$$
Consequently, the ratio appearing on the right-hand side of
(\ref{Eq.S63DiscreteExpectationEquality}) converges almost surely to
\[
\mathbb{E}_{\mathrm{avoid},2^{3/2}}^{1,k_2,-R/2,R/2,
\vec{x},\vec{x},\infty,g}
\left[
\prod_{j\in\mathsf{R}}
h_j\bigl(\mathcal{Q}_{n_j}(t_j)\bigr)
\right].
\]
Moreover, by the continuity of the $h_j$,
\[
\prod_{j\notin\mathsf{R}}
h_j\bigl(\widetilde{\mathcal{L}}^N_{n_j}(t_j)\bigr)
\longrightarrow
\prod_{j\notin\mathsf{R}}
h_j\bigl(\widetilde{\mathcal{L}}^{\mathrm{per},R}_{n_j}(t_j)\bigr)
\]
almost surely. The ratio in
(\ref{Eq.S63DiscreteExpectationEquality}) is bounded in absolute value by
$\prod_{j\in\mathsf{R}}\|h_j\|_{\infty}$. Consequently, (\ref{Eq.S63ExpectationEquality}) follows by taking the $N \rightarrow \infty$ limit on both sides of (\ref{Eq.S63DiscreteExpectationEquality}) and applying the bounded convergence theorem.

%----------------------------------------------------------------------
%
% Section 7
%
%----------------------------------------------------------------------
\section{Limiting regimes of the periodic Airy line ensemble}\label{Section7}
In this section, we prove the asymptotic results for the periodic Airy line
ensemble stated in Section~\ref{Section1.2}. The proofs of Theorems~\ref{Thm.TranslationInvariance}, \ref{Thm.ConvergenceToAiryLE} and \ref{Thm.ConvergenceToGumbel} are given in Sections~\ref{Section7.1}, \ref{Section7.2}, and \ref{Section7.3}, respectively. Throughout this section, we freely use the notation from Sections~\ref{Section1.2}, \ref{Section2.1} and \ref{Section2.2}.

%----------------------------------------------------------------------
%
% Section 7.1
%
%----------------------------------------------------------------------
\subsection{Proof of Theorem \ref{Thm.TranslationInvariance}}\label{Section7.1}
We continue with the notation from the statement of the theorem. The argument proceeds by reducing the claim to showing that two point processes on $\mathbb{R}^2$ have the same distribution, see (\ref{Eq.TransRed2}). Using Theorem~\ref{Thm.DeterminantalStructure}, we show that both are determinantal point processes whose correlation kernels are related by a gauge transformation. Their equality in distribution then follows from \cite[Proposition~2.13(3)--(4)]{ED24a}. We now turn to the details.

\medskip

By \cite[Lemma~3.1]{DM21}, it suffices to prove that, for every $\theta\in\R$, the line ensembles $\{\widehat{\mathcal A}^{\mathrm{per},R}_i(t+\theta):i\geq1,t\in\R\}$ and $\{\widehat{\mathcal A}^{\mathrm{per},R}_i(t):i\geq1,t\in\R\}$ have the same finite-dimensional distributions. By $R$-periodicity, it is enough to show that
\begin{equation}\label{Eq.TransRed1}
\bigl(\widehat{\mathcal A}^{\mathrm{per},R}_i(s_j+\theta):1\leq i\leq k,\ 1\leq j\leq m\bigr)
\overset{d}{=}
\bigl(\mathcal A^{\mathrm{per},R}_i(s_j):1\leq i\leq k,\ 1\leq j\leq m\bigr)
\end{equation}
for every $\theta\in[0,R)$, $k,m\in\N$, and $-R/2<s_1<\cdots<s_m\leq R/2$.

Fix $\theta\in[0,R)$ and $\mathsf{S}=\{s_1,\ldots,s_m\}\subset(-R/2,R/2]$, and define the random measure
$$
M^{\mathsf{S}}_{\theta}(A)=\sum_{i=1}^{\infty}\sum_{j=1}^{m}{\bf 1}\left\{\left(s_j,\widehat{\mathcal A}^{\mathrm{per},R}_i(s_j+\theta)\right)\in A\right\},\qquad A\subseteq\mathbb{R}^2\text{ Borel}.
$$
From Theorem~\ref{Thm.DeterminantalStructure}, we know that $M^{\mathsf{S}}_{\theta}$ and $M^{\mathsf{S};R}$ are point processes on $\mathbb{R}^2$. Combining this with the fact that $\mathcal A^{\mathrm{per},R}$ is ordered, cf.~Theorem~\ref{Thm.Convergence}, and \cite[Corollary~2.20]{ED24a}, we see that it suffices to show that
\begin{equation}\label{Eq.TransRed2}
M^{\mathsf{S}}_{\theta}\overset{d}{=}M^{\mathsf{S};R}.
\end{equation}

\medskip

Define $f:\mathbb R\to\mathbb R$ by
$$
f(t)=
\begin{cases}
t+\theta, & -R/2<t\leq R/2-\theta,\\
t+\theta-R, & R/2-\theta<t\leq R/2,\\
t, & \text{otherwise},
\end{cases}
$$
and put $\phi(t,x)=(f^{-1}(t),x)$. Note that $\phi:\mathbb{R}^2\to\mathbb{R}^2$ is a locally bounded measurable bijection with a locally bounded inverse, and that
$$
M_{\theta}^{\mathsf{S}}=M^{f(\mathsf{S});R}\phi^{-1}.
$$
By Theorem~\ref{Thm.DeterminantalStructure}, $M^{f(\mathsf{S});R}$ is a determinantal point process with reference measure $\mu_{f(\mathsf{S})}\times\mathrm{Leb}$ and correlation kernel $K^{\mathrm{per},R}$. Consequently, by \cite[Proposition~2.13(5)]{ED24a}, $M_{\theta}^{\mathsf{S}}$ is a determinantal point process with reference measure
$$
(\mu_{f(\mathsf{S})}\times\mathrm{Leb})\phi^{-1}=\mu_{\mathsf{S}}\times\mathrm{Leb}
$$
and correlation kernel
$$
K_{\theta}(s,x;t,y)=K^{\mathrm{per},R}(f(s),x;f(t),y),\qquad s,t\in(-R/2,R/2].
$$
We claim that, for $x,y\in\mathbb{R}$ and $s,t\in(-R/2,R/2]$,
\begin{equation}\label{Eq.TransRed3}
K^{\mathrm{per},R}(f(s),x;f(t),y)
=
\frac{g(s)}{g(t)}K^{\mathrm{per},R}(s,x;t,y),
\end{equation}
where
$$
g(s)=
\begin{cases}
1, & s\in(-R/2,R/2-\theta],\\
-1, & s\in(R/2-\theta,R/2].
\end{cases}
$$
Assuming (\ref{Eq.TransRed3}) for the moment, \cite[Proposition~2.13(4)]{ED24a} shows that $M_{\theta}^{\mathsf{S}}$ is also a determinantal point process with reference measure $\mu_{\mathsf{S}}\times\mathrm{Leb}$ and correlation kernel $K^{\mathrm{per},R}$. Theorem~\ref{Thm.DeterminantalStructure} gives the same determinantal description for $M^{\mathsf{S};R}$, and hence (\ref{Eq.TransRed2}) follows from \cite[Proposition~2.13(3)]{ED24a}.

\medskip

It remains to establish (\ref{Eq.TransRed3}). Fix $s,t\in(-R/2,R/2]$. If $s,t\leq R/2-\theta$ or $s,t>R/2-\theta$, then $f(t)-f(s)=t-s$, and the relative order of $s$ and $t$ is preserved under $f$. It follows from (\ref{Eq.PeriodicAiryKernel}) that
$$
K^{\mathrm{per},R}(f(s),x;f(t),y)=K^{\mathrm{per},R}(s,x;t,y).
$$
Since $g(s)=g(t)$, this proves (\ref{Eq.TransRed3}) in this case.

Suppose next that $s\leq R/2-\theta<t$. Then $s<t$, while $f(s)>f(t)$, and
$$
f(t)-f(s)=t-s-R.
$$
Consequently, by (\ref{Eq.PeriodicAiryKernel}),
\begin{equation*}
\begin{split}
&K^{\mathrm{per},R}(f(s),x;f(t),y) =-\int_{-\infty}^{\infty}\frac{e^{-(t-s-R)v}}{1+e^{Rv}}\Ai(x+v)\Ai(y+v)\,dv\\
&=-\int_{-\infty}^{\infty}\frac{e^{-(t-s)v}}{1+e^{-Rv}}\Ai(x+v)\Ai(y+v)\,dv =-K^{\mathrm{per},R}(s,x;t,y).
\end{split}
\end{equation*}
Since $g(s)=1$ and $g(t)=-1$, this proves (\ref{Eq.TransRed3}) in this case.

Finally, suppose that $t\leq R/2-\theta<s$. Then $s>t$, while $f(s)<f(t)$, and
$$
f(t)-f(s)=t-s+R.
$$
It follows that
\begin{equation*}
\begin{split}
&K^{\mathrm{per},R}(f(s),x;f(t),y)=\int_{-\infty}^{\infty}\frac{e^{-(t-s+R)v}}{1+e^{-Rv}}\Ai(x+v)\Ai(y+v)\,dv\\
&=\int_{-\infty}^{\infty}\frac{e^{-(t-s)v}}{1+e^{Rv}}\Ai(x+v)\Ai(y+v)\,dv=-K^{\mathrm{per},R}(s,x;t,y).
\end{split}
\end{equation*}
Since $g(s)=-1$ and $g(t)=1$, (\ref{Eq.TransRed3}) follows in this case. This completes the proof.

%----------------------------------------------------------------------
%
% Section 7.2
%
%----------------------------------------------------------------------
\subsection{Proof of Theorem \ref{Thm.ConvergenceToAiryLE}}\label{Section7.2}

We continue with the notation from the statement of the theorem and divide the proof into three steps. In Step 1, we fix a sequence $R_N\uparrow\infty$ and prove convergence of the point processes obtained by restricting the periodic Airy line ensembles to finitely many time slices. In Step 2, we combine this point process convergence with one-time upper-tail bounds to obtain finite-dimensional convergence of the ordered curves. In Step 3, we use the periodic Brownian Gibbs property established in Theorem \ref{Thm.GibbsProperty} and the tightness argument of Corwin and Hammond \cite[Section~3]{CorHamA} to upgrade the finite-dimensional convergence to weak convergence in $C(\mathbb N\times\mathbb R)$. The last step is analogous to \cite[Section~4.2]{DimitrovSerio}.

\medskip

{\bf \raggedleft Step 1.} Fix an arbitrary sequence $R_N\uparrow\infty$, an integer $m\in\mathbb N$, and a set $\mathsf{S}=\{s_1<\cdots<s_m\}\subset\mathbb R$. For all sufficiently large $N$, we have $\mathsf{S}\subset(-R_N/2,R_N/2)$. In particular, for such $N$ we have
$$
M^{\mathsf{S};R_N}(A)
= \sum_{i=1}^{\infty}\sum_{j=1}^{m}
{\bf 1}\left\{
\left(s_j,\widehat{\mathcal A}^{\mathrm{per},R_N}_i(s_j)\right)\in A
\right\},
\qquad A\subseteq\mathbb R^2\text{ Borel},
$$
where we recall that $M^{\mathsf{S};R_N}$ is the point process from (\ref{Eq.PeriodicAiryPointProcess}).
In what follows, we assume that $N$ is sufficiently large so that $\mathsf{S}\subset(-R_N/2,R_N/2)$, and do not mention this further. 

Define the random measure $M^{\mathsf{S};\mathrm{Ai}}$ on $\mathbb R^2$ by
$$
M^{\mathsf{S};\mathrm{Ai}}(A)
=
\sum_{i=1}^{\infty}\sum_{j=1}^{m}
{\bf 1}\left\{
\left(s_j,\mathcal A_i(s_j)\right)\in A
\right\},
\qquad A\subseteq\mathbb R^2\text{ Borel}.
$$

The goal of this step is to prove that 
\begin{equation}\label{Eq.LargePeriodPointProcessConvergence}
M^{\mathsf{S};R_N}
\Rightarrow
M^{\mathsf{S};\mathrm{Ai}}.
\end{equation}

By \cite[Definition~1.3 and Theorem~1.4]{CS14}, $M^{\mathsf{S};\mathrm{Ai}}$ is a determinantal point process with reference measure $\mu_{\mathsf{S}}\times\mathrm{Leb}$ and correlation kernel
\begin{equation}\label{Eq.AiryKernel}
K^{\mathrm{Ai}}(s,x;t,y)=
\begin{cases}
\displaystyle\int_0^\infty e^{-(t-s)v}
\Ai(x+v)\Ai(y+v)\,dv, & s\leq t,\\[3ex]
\displaystyle-\int_{-\infty}^0 e^{-(t-s)v}
\Ai(x+v)\Ai(y+v)\,dv, & s>t.
\end{cases}
\end{equation}
We mention that the kernel in (\ref{Eq.AiryKernel}) is the transpose of the one in \cite[Definition~1.3]{CS14}; by \cite[Proposition~2.13(4)]{ED24a}, the transposed kernel is also a correlation kernel for the same determinantal point process.

From the Airy function asymptotics in Proposition~\ref{Prop.S4AiryIdentities}(iv), we can find $C,c>0$ such that
\begin{equation}\label{Eq.LargePeriodAiryBounds}
|\Ai(x)|\leq Ce^{-cx^{3/2}}\quad\mbox{for }x\geq0,
\qquad
|\Ai(x)|\leq C\quad\mbox{for }x<0.
\end{equation}
From (\ref{Eq.AiryKernel}), (\ref{Eq.LargePeriodAiryBounds}), and the dominated convergence theorem, we know that $K^{\mathrm{Ai}}(s,\cdot;t,\cdot)$ is continuous on $\mathbb R^2$ for every $s,t\in\mathsf{S}$. In addition, by Theorem~\ref{Thm.DeterminantalStructure}, $M^{\mathsf{S};R_N}$ is a determinantal point process with reference measure $\mu_{\mathsf{S}}\times\mathrm{Leb}$ and correlation kernel $K^{\mathrm{per},R_N}$ as in \eqref{Eq.PeriodicAiryKernel}. Consequently, by \cite[Proposition~2.18]{ED24a}, \cite[Proposition~2.13(3)]{ED24a}, the continuity of $K^{\mathrm{Ai}}(s,\cdot;t,\cdot)$, and a standard compactness argument, to prove (\ref{Eq.LargePeriodPointProcessConvergence}) it suffices to show that 
\begin{equation}\label{Eq.LargePeriodKernelConvergence}
\lim_{N\to\infty}
K^{\mathrm{per},R_N}(s,x_N;t,y_N)
=
K^{\mathrm{Ai}}(s,x;t,y)
\end{equation}
whenever $s,t\in\mathsf{S}$ and $x_N,y_N\in\mathbb R$ satisfy $x_N\to x$ and $y_N\to y$.

\medskip

To prove (\ref{Eq.LargePeriodKernelConvergence}), choose $B>0$ such that $x_N,y_N,x,y\in[-B,B]$ for all $N$. We will use the elementary bound
\begin{equation}\label{Eq.LargePeriodScalarBound}
0<
\frac{1}{1+e^{-Rv}}
\leq
\min\{1,e^{Rv}\},
\qquad R> 0, v\in\mathbb R.
\end{equation}
Using (\ref{Eq.LargePeriodScalarBound}) together with (\ref{Eq.LargePeriodAiryBounds}), we can find $C_B,c_B>0$ such that for $x,y \in [-B,B]$, $s,t \in \mathsf{S}$, and $R \geq s_m - s_1 + 1$, we have
\begin{equation}\label{Eq.ConvALEBounds}
\begin{split}
&\left|\frac{e^{-(t-s)v}} {1+e^{-Rv}}
\Ai(x+v)\Ai(y+v)\right| \leq C_B\left(
e^{v}{\bf 1}\{v<0\} + e^{-c_Bv^{3/2}}{\bf 1}\{v\geq0\} \right) \mbox{ if } s \leq t, \\
&\left|\frac{e^{-(t-s)v}}{1+e^{Rv}} \Ai(x+v)\Ai(y+v)\right| \leq C_B\left(
e^{(s-t)v}{\bf 1}\{v<0\} + e^{-c_Bv^{3/2}}{\bf 1}\{v\geq0\} \right) \mbox{ if } s > t.
\end{split}
\end{equation}
Here, the constants $C_B$ and $c_B$ depend on $B$ alone.

We also have the following pointwise limits for $v \neq 0$:
\begin{equation}\label{Eq.ConvALELimits}
\begin{split}
&\lim_{N \rightarrow \infty} \frac{e^{-(t-s)v}} {1+e^{-R_Nv}}
\Ai(x_N+v)\Ai(y_N+v) = e^{-(t-s)v}\Ai(x+v)\Ai(y+v){\bf 1}\{v>0\} \mbox{ if } s \leq t, \\
&\lim_{N \rightarrow \infty} \frac{e^{-(t-s)v}}{1+e^{R_Nv}} \Ai(x_N+v)\Ai(y_N+v) =  e^{-(t-s)v}\Ai(x+v)\Ai(y+v) {\bf 1}\{v < 0\} \mbox{ if } s > t.
\end{split}
\end{equation}

Applying the pointwise limits in (\ref{Eq.ConvALELimits}) to the integrands in the formula for $K^{\mathrm{per},R_N}(s,x_N;t,y_N)$ in (\ref{Eq.PeriodicAiryKernel}), we obtain (\ref{Eq.LargePeriodKernelConvergence}) from the dominated convergence theorem, with dominating functions given by the right-hand sides of (\ref{Eq.ConvALEBounds}).

\medskip

{\bf \raggedleft Step 2.} In this step, we use (\ref{Eq.LargePeriodPointProcessConvergence}) to prove that, for every $k,m\in\mathbb N$ and $\mathsf{S}=\{s_1<\cdots<s_m\}\subset\mathbb R$,
\begin{equation}\label{Eq.LargePeriodFDD}
\left(
\widehat{\mathcal A}^{\mathrm{per},R_N}_i(s_j):
1\leq i\leq k,\ 1\leq j\leq m
\right)
\Rightarrow
\left(
\mathcal A_i(s_j):
1\leq i\leq k,\ 1\leq j\leq m
\right).
\end{equation}

By (\ref{Eq.LargePeriodPointProcessConvergence}), \cite[Proposition~2.19]{ED24a}, and \cite[Corollary~2.20]{ED24a}, to prove (\ref{Eq.LargePeriodFDD}) it suffices to show that
\begin{equation}\label{Eq.LargePeriodCoordinateTightness}
\left\{\widehat{\mathcal A}^{\mathrm{per},R_N}_i(s_j)
\right\}_{N\geq1}\quad\mbox{is tight}
\end{equation}
for every $i\geq1$ and $1\leq j\leq m$. 

\medskip

In the remainder of this step, we fix $1 \leq j \leq m$ and prove (\ref{Eq.LargePeriodCoordinateTightness}). Discarding finitely many terms, we may assume that $s_j\in[-R_N/2,R_N/2]$ for every $N$. Define the random measure $M^{j,N}$ on $\mathbb R$ by
$$
M^{j,N}(A)
=
\sum_{i=1}^{\infty}
{\bf 1}\left\{
\mathcal A_i^{\mathrm{per},R_N}(s_j)\in A
\right\},
\qquad A\subseteq\mathbb R\text{ Borel}.
$$
By Theorem~\ref{Thm.DeterminantalStructure} and \cite[Lemma~2.17]{ED24a}, $M^{j,N}$ is a determinantal point process with Lebesgue reference measure and correlation kernel
\begin{equation}\label{Eq.LargePeriodOneTimeKernel}
K^{j,N}(x,y)
:= K^{\mathrm{per},R_N}(s_j,x;s_j,y) =
\int_{\mathbb R}
\frac{\Ai(x+v)\Ai(y+v)}
{1+e^{-R_Nv}}\,dv.
\end{equation}
From (\ref{Eq.LargePeriodKernelConvergence}) we conclude that $K^{j,N}(x,y)$ converges uniformly over compact sets to the Airy kernel $K^{\mathrm{Ai}}(s_j,x;s_j,y) = K^{\mathrm{Ai}}(0,x;0,y)$. From \cite[Proposition~2.15]{ED24a} and \cite[Proposition~2.13(3)]{ED24a}, we conclude that $M^{j,N}$ converges weakly to the Airy point process.

The last sentence and the fact that the Airy point process a.s.\ has infinitely many atoms, see \cite[(7.11)]{ED24a}, imply that $\left\{\left(\mathcal{A}^{\mathrm{per},R_N}_i(s_j): i \geq 1\right)\right\}_{N \geq 1}$ satisfies conditions~(1) and~(2) in \cite[Proposition~2.21]{ED24a}. Condition (3) in that proposition requires that 
\begin{equation}\label{Eq.LargePeriodRem1}
\lim_{a \rightarrow \infty} \limsup_{N \rightarrow \infty} \mathbb P\left(\mathcal{A}^{\mathrm{per},R_N}_1(s_j)\geq a
\right) = 0.
\end{equation}
Assuming (\ref{Eq.LargePeriodRem1}) for the moment, we see that all conditions of \cite[Proposition~2.21]{ED24a} are satisfied, and hence $\left\{\left(\mathcal{A}^{\mathrm{per},R_N}_i(s_j): i \geq 1\right)\right\}_{N \geq 1}$ is tight in $\mathbb{R}^{\infty}$. As $\widehat{\mathcal A}^{\mathrm{per},R_N}_i(s_j) = \mathcal{A}^{\mathrm{per},R_N}_i(s_j)$ for all large $N$, this implies (\ref{Eq.LargePeriodCoordinateTightness}).

\medskip

It remains to establish (\ref{Eq.LargePeriodRem1}). For all sufficiently large $N$, we have $R_N\geq1$, and thus, for $v<0$,
$$
\frac{1}{1+e^{-R_Nv}}
\leq e^{R_Nv}
\leq e^v.
$$
Using the latter, (\ref{Eq.LargePeriodAiryBounds}), (\ref{Eq.LargePeriodOneTimeKernel}) and splitting the integral at $v=-x/2$, for $x\geq0$ we obtain for some universal constants $D,d > 0$
\begin{equation}\label{Eq.LargePeriodDiagonalBound}
\begin{split}
K^{j,N}(x,x)
&\leq
\int_{-x/2}^{\infty}\Ai(x+v)^2dv
+C^2\int_{-\infty}^{-x/2}e^vdv =
\int_{x/2}^{\infty}\Ai(u)^2du
+ C^2e^{-x/2} \leq De^{-dx}.
\end{split}
\end{equation}

For $a\geq 0$, taking expectations on both sides of the trivial inequality
$$
{\bf 1}\{\mathcal A^{\mathrm{per},R_N}_1(s_j) \geq a\}\leq M^{j,N}([a,\infty)),
$$
and using (\ref{Eq.S2PairMeanDPP}) and (\ref{Eq.LargePeriodDiagonalBound}), we conclude that
$$
\mathbb P\left(
\mathcal A^{\mathrm{per},R_N}_1(s_j)\geq a
\right)
\leq
\int_a^\infty K^{j,N}(x,x)\,dx
\leq \int_a^{\infty} De^{-dx}\,dx = (D/d)e^{-da}.
$$
The last inequality proves (\ref{Eq.LargePeriodRem1}), completing our work in this step.

\medskip

{\bf \raggedleft Step 3.} In this step, we conclude the proof of the theorem. The main idea is to apply the tightness argument of Corwin and Hammond \cite[Section~3]{CorHamA}. As similar arguments have already appeared earlier, see \cite[Section~4.2]{DimitrovSerio}, our discussion will be brief.

In view of (\ref{Eq.LargePeriodFDD}), it suffices to prove that the sequence $\{\widehat{\mathcal A}^{\mathrm{per},R_N}\}_{N\geq1}$ is tight in $C(\mathbb N\times\mathbb R)$. Indeed, if $\widehat{\mathcal A}^{\mathrm{per},R_{N_j}}$ converges weakly along a subsequence to some line ensemble $\mathcal A'$, then (\ref{Eq.LargePeriodFDD}) implies that $\mathcal A'$ and $\mathcal A$ have the same finite-dimensional distributions. By \cite[Lemma~3.1]{DM21}, this implies that $\mathcal A'$ and $\mathcal A$ have the same distribution. Consequently, once tightness is established, every subsequential limit has the same distribution as $\mathcal A$, and hence $\widehat{\mathcal A}^{\mathrm{per},R_N}\Rightarrow\mathcal A$.
It therefore remains only to prove tightness.

Define the parabolic Airy line ensemble $\mathcal L^{\mathrm{Ai}}=\{\mathcal L_i^{\mathrm{Ai}}\}_{i\geq1}$ by
\begin{equation}\label{Eq.LargePeriodParabolicAiry}
\mathcal L_i^{\mathrm{Ai}}(t)
=
2^{-1/2}
\left(
\mathcal A_i(t)-t^2
\right),
\qquad i\geq1,\quad t\in\mathbb R.
\end{equation}
Also, let $\mathcal L^N=\{\mathcal L_i^N\}_{i=1}^{N+1}$ be the $\llbracket 1,N+1\rrbracket$-indexed line ensemble on $[-R_N/2,R_N/2]$ defined by
\begin{equation}\label{Eq.LargePeriodParabolicEnsemble}
\mathcal L_i^N(t)
=
2^{-1/2}
\left(
\widehat{\mathcal A}_i^{\mathrm{per},R_N}(t)-t^2
\right)
=
2^{-1/2}
\left(
\mathcal A_i^{\mathrm{per},R_N}(t)-t^2
\right),
\end{equation}
for $(i,t)\in\llbracket 1,N+1\rrbracket\times[-R_N/2,R_N/2]$, where the last equality follows from the definition of $\widehat{\mathcal A}^{\mathrm{per},R_N}$.

Fix $k\geq1$ and $T>0$, and choose $N_0\geq k$ so that
$$
R_N\geq2T + 6
\qquad\mbox{for all }N\geq N_0.
$$
We verify the hypotheses of \cite[Definition~3.3]{CorHamA} for $\mathcal L^N$, with $N\geq N_0$. By (\ref{Eq.LargePeriodFDD}), (\ref{Eq.LargePeriodParabolicAiry}), and (\ref{Eq.LargePeriodParabolicEnsemble}), the ensembles $\mathcal L^N$ satisfy hypothesis $(H2)_{k,T+2}$. Since $\mathcal L^{\mathrm{Ai}}$ is non-intersecting almost surely, they also satisfy hypothesis $(H3)_{k,T+2}$. Finally, Theorem~\ref{Thm.GibbsProperty} and Lemma~\ref{Lem.PeriodicBGPCharacterization} imply that $\mathcal L^N$ satisfies hypothesis $(H1)_{k,T+2}$ with ``Brownian Gibbs property'' replaced by ``partial Brownian Gibbs property.'' As observed in \cite[Section~4.2]{DimitrovSerio}, the results of \cite[Section~3]{CorHamA} remain valid with this modification, since their arguments do not use properties of the bottom curves in the ensembles.

Consequently, \cite[Proposition~3.6]{CorHamA} implies that
$$
\left\{
\left.
\mathcal L^N
\right|_{\llbracket 1,k\rrbracket\times[-T,T]}
\right\}_{N\geq N_0}
$$
is tight in $C(\llbracket 1,k\rrbracket\times[-T,T])$. By (\ref{Eq.LargePeriodParabolicEnsemble}), it follows that
$$
\left\{
\left.
\widehat{\mathcal A}^{\mathrm{per},R_N}
\right|_{\llbracket 1,k\rrbracket\times[-T,T]}
\right\}_{N\geq N_0}
$$
is tight in the same space. Since $k\geq1$ and $T>0$ were arbitrary, \cite[Lemma~2.4]{DEA21} implies that $\{\widehat{\mathcal A}^{\mathrm{per},R_N}\}_{N\geq1}$ is tight in $C(\mathbb N\times\mathbb R)$, completing the proof of the theorem.

%----------------------------------------------------------------------
%
% Section 7.3
%
%----------------------------------------------------------------------
\subsection{Proof of Theorem \ref{Thm.ConvergenceToGumbel}}\label{Section7.3}

In this section, we deal with random elements $Y = (Y_i: i\geq 1)$ of $\mathbb{R}^{\infty}$, cf.~\cite[Example 1.2]{Billing}. The random elements considered below will almost surely lie in the subspace
$$
\mathbb{R}_{\geq}^{\infty} = \{ x= (x_1,x_2,\dots) \in \mathbb{R}^{\infty}: x_1 \geq x_2 \geq \cdots \mbox{ and } \lim_{n \rightarrow \infty}x_n = -\infty  \}.
$$
For a random element $Y$ of $\mathbb{R}_{\geq}^{\infty}$, we also define its associated point process $M^Y$ by
\begin{equation}\label{Eq.DefMY}
M^Y(A)=\sum_{i=1}^{\infty}{\bf 1}\{Y_i\in A\},
\qquad A\subseteq\mathbb R\text{ Borel}.
\end{equation}

Before turning to the proof of Theorem \ref{Thm.ConvergenceToGumbel}, we record the following one-dimensional consequence of \cite[Proposition~2.19 and Corollary~2.20]{ED24a}.

\begin{lemma}\label{Lem.OneDimensionalPointProcess}
For each $N\in\mathbb N$, let $X^N=(X_i^N:i\geq1)$ be a random element of $\mathbb R_{\geq}^\infty$. Suppose that $M^{X^N}\Rightarrow M$ for some point process $M$ and that $\{X_i^N\}_{N\geq1}$ is tight for every $i\geq1$. Then $X^N\Rightarrow X$ for some random element $X=(X_i:i\geq1)$ of $\mathbb R^\infty_{\geq}$, whose associated point process $M^X$ has the same distribution as $M$. Moreover, if $Y=(Y_i:i\geq1)$ is a random element of $\mathbb R^\infty_{\geq}$ whose associated point process has the same distribution as $M$, then $X^N\Rightarrow Y$.
\end{lemma}

\begin{proof}
Identify $\mathbb R$ with $\{0\}\times\mathbb R$ through $x\mapsto(0,x)$. The first statement follows from \cite[Proposition~2.19]{ED24a} applied with $r=1$ and $t_1=0$. The second statement then follows from \cite[Corollary~2.20]{ED24a}.
\end{proof}

\begin{proof}[Proof of Theorem~\ref{Thm.ConvergenceToGumbel}]
We continue with the notation from the statement of the theorem and divide the proof into three steps. In Step 1, we construct the limiting point process, identify its determinantal structure, and reduce the proof of the theorem to establishing a certain weak convergence statement, see (\ref{Eq.GumbelRed1}). In Step 2, we introduce rescaled random elements built from $\{\mathcal{A}^{\mathrm{per},R}_i(0)\}_{i \geq 1}$ and prove convergence of their associated point processes. In Step 3, we establish (\ref{Eq.GumbelRed1}) using the point process convergence from Step 2 and Lemma~\ref{Lem.OneDimensionalPointProcess}. Throughout this subsection, we freely use the notation for Poisson random measures from \cite[Chapter VI]{Cin11}.

\medskip
{\bf \raggedleft Step 1.} Let $\mathsf{P}$ be a Poisson random measure on $[0,\infty)$ with mean measure given by Lebesgue measure. From \cite[Propositions~VI.1.11 and~VI.2.13, and Theorem~VI.2.17]{Cin11}, we can find random variables $\{X_i\}_{i\geq1}$ that form $\mathsf{P}$ and satisfy $0<X_1<X_2<\cdots$ and $\lim_{n \rightarrow \infty} X_n = \infty$ almost surely.

Define $Y_i=-\log X_i$ for $i \geq 1$, and let $M^Y$ be as in (\ref{Eq.DefMY}). By \cite[Remark~VI.2.4(b)]{Cin11}, $M^Y$ is a Poisson random measure with mean measure $\mu$, defined by $\mu(A)=\int_A e^{-x}\,dx$. In particular, $\{\mathcal{G}_i\}_{i\geq1}$ in the statement of the theorem has the same distribution as $\{Y_i\}_{i\geq1}$. Consequently, to prove (\ref{Eq.ConvToGumbelK}), it suffices to show that for any sequence $R_N \downarrow 0$
\begin{equation}\label{Eq.GumbelRed1}
\left(
R_N\mathcal{A}^{\mathrm{per},R_N}_i(0)
+\frac{1}{2}\log(4\pi R_N^3): i \geq 1
\right)
\Rightarrow
\left(Y_i: i \geq 1\right),
\end{equation}
where the convergence is in $\mathbb{R}^{\infty}$. Note that
$$
\mathbb{P}(Y_1 \leq s) = \mathbb{P}\bigl(\mathsf{P}([0,e^{-s})) = 0 \bigr) = e^{-e^{-s}},
$$
so that (\ref{Eq.GumbelRed1}) also implies (\ref{Eq.ConvToGumbel}).

\medskip

In the remainder of this step, we show that $M^Y$ is a determinantal point process on $\mathbb{R}$ with Lebesgue reference measure and correlation kernel
\begin{equation}\label{Eq.KerExponPP}
K(x,y) = e^{-x}{\bf 1}\{x=y\},\qquad x,y\in\mathbb R.
\end{equation}
By \cite[Lemma~2.8]{ED24a}, it suffices to show that for every collection of pairwise disjoint bounded Borel sets $B_1,\ldots,B_m\subseteq\mathbb R$ and $n_1,\ldots,n_m\in\mathbb{N}$, with $n=n_1+\cdots+n_m$, we have
\begin{equation}\label{Eq.ExponPPFactorial}
\mathbb E\left[\prod_{j=1}^m\frac{M^Y(B_j)!}{(M^Y(B_j)-n_j)!}\right]
=
\int_{B_1^{n_1}\times\cdots\times B_m^{n_m}}
\det[K(x_i,x_j)]_{i,j=1}^n\,dx_1\cdots dx_n,
\end{equation}
where we use the usual convention $\frac{1}{k!} = 0$ for $k \leq -1$.

We know that the random variables $M^Y(B_1),\ldots,M^Y(B_m)$ are independent Poisson random variables with respective means $\mu(B_1),\ldots,\mu(B_m)$. Since a Poisson random variable $Z$ with mean $\lambda$ satisfies
$$
\mathbb E\left[\frac{Z!}{(Z-k)!}\right]=\lambda^k,
$$
we obtain
\begin{equation}\label{Eq.MYFact}
\mathbb E\left[\prod_{j=1}^m\frac{M^Y(B_j)!}{(M^Y(B_j)-n_j)!}\right]
=
\prod_{j=1}^m\left(\int_{B_j}e^{-x}\,dx\right)^{n_j}.
\end{equation}
The latter agrees with the right-hand side of (\ref{Eq.ExponPPFactorial}) once we use (\ref{Eq.KerExponPP}), hence verifying (\ref{Eq.ExponPPFactorial}).

\medskip

{\bf \raggedleft Step 2.} Fix an arbitrary sequence $R_N\downarrow0$ and set
$$
c_N=\frac12\log(4\pi R_N^3),\qquad
Y_i^N=R_N\mathcal A_i^{\mathrm{per},R_N}(0)+c_N,\qquad i\geq1.
$$
Set $Y^N=(Y_i^N:i\geq1)$. By Theorems~\ref{Thm.Convergence} and~\ref{Thm.DeterminantalStructure}, $Y^N$ lies in $\mathbb{R}^{\infty}_{\geq}$ almost surely, and we let $M^{Y^N}$ be its associated point process as in (\ref{Eq.DefMY}). The goal of this step is to prove that
\begin{equation}\label{Eq.GumbelPPConvergence}
M^{Y^N}\Rightarrow M^Y.
\end{equation}

By Theorem~\ref{Thm.DeterminantalStructure}, \cite[Lemma~2.17]{ED24a}, and \cite[Proposition~2.13(5)--(6)]{ED24a}, $M^{Y^N}$ is a determinantal point process on $\mathbb R$ with Lebesgue reference measure and correlation kernel
\begin{equation}\label{Eq.GumbelScaledKernel}
K_N(x,y)=
\frac{1}{R_N}\int_{\mathbb R}
\frac{\Ai((x-c_N)/R_N+v)\Ai((y-c_N)/R_N+v)}
{1+e^{-R_Nv}}\,dv.
\end{equation}
With $\alpha=R_N$, the kernel in (\ref{Eq.GumbelScaledKernel}) is precisely the rescaled kernel on the left-hand side of \cite[(15)]{Joh07}. Consequently, \cite[Theorem~1.2]{Joh07} gives
\begin{equation}\label{Eq.GumbelKernelLimit}
\lim_{N\to\infty}K_N(x,y)=K(x,y),\qquad x,y\in\mathbb R,
\end{equation}
where $K$ is as in (\ref{Eq.KerExponPP}). In addition, using
$$
\frac{1}{1+e^{-R_Nv}}\leq e^{R_Nv}
$$
and \cite[(44)]{Joh07}, we obtain
\begin{equation}\label{Eq.GumbelDiagonalBound}
0\leq K_N(x,x)\leq e^{R_N^3/12-x},\qquad x\in\mathbb R.
\end{equation}

For $x_1,\ldots,x_n\in\mathbb R$, the integral representation (\ref{Eq.GumbelScaledKernel}) shows that $[K_N(x_i,x_j)]_{i,j=1}^n$ is positive semidefinite. By Hadamard's inequality \cite[Corollary~33.2.1.1]{Prasolov} and (\ref{Eq.GumbelDiagonalBound}), we obtain
\begin{equation}\label{Eq.GumbelDetBound}
0\leq
\det[K_N(x_i,x_j)]_{i,j=1}^n
\leq
\prod_{i=1}^n K_N(x_i,x_i)
\leq
e^{nR_N^3/12}\prod_{i=1}^ne^{-x_i}.
\end{equation}
On the other hand, if $x_1,\ldots,x_n$ are pairwise distinct, (\ref{Eq.GumbelKernelLimit}) gives
\begin{equation}\label{Eq.GumbelDetLimit}
\lim_{N\to\infty}
\det[K_N(x_i,x_j)]_{i,j=1}^n
=
\prod_{i=1}^ne^{-x_i}.
\end{equation}

Let $B_1,\ldots,B_m$ be pairwise disjoint bounded Borel sets and let $n_1,\ldots,n_m\in\mathbb N$, with $n=n_1+\cdots+n_m$. Since the tuples with coincident coordinates form a Lebesgue null set, (\ref{Eq.GumbelDetBound}), (\ref{Eq.GumbelDetLimit}), dominated convergence, and \cite[(2.13)]{ED24a} give
\begin{equation}\label{Eq.GumbelFactorialConvergence}
\lim_{N\to\infty}
\mathbb E\left[
\prod_{j=1}^m
\frac{M^{Y^N}(B_j)!}{(M^{Y^N}(B_j)-n_j)!}
\right]
=
\prod_{j=1}^m
\left(\int_{B_j}e^{-x}\,dx\right)^{n_j}
=
\mathbb E\left[
\prod_{j=1}^m
\frac{M^Y(B_j)!}{(M^Y(B_j)-n_j)!}
\right],
\end{equation}
where the last equality follows from (\ref{Eq.MYFact}). Moreover, for every bounded Borel set $B$ and $\epsilon>0$,
\begin{equation}\label{Eq.SummableEpsilon}
\sum_{n=1}^{\infty}\frac{\epsilon^n}{n!}
\left(\int_Be^{-x}\,dx\right)^n<\infty.
\end{equation}
From (\ref{Eq.GumbelFactorialConvergence}), (\ref{Eq.SummableEpsilon}), and \cite[Proposition~2.2 and Corollary~2.4]{ED24a}, we see that (\ref{Eq.GumbelPPConvergence}) holds.

\medskip

{\bf \raggedleft Step 3.} In this step, we prove (\ref{Eq.GumbelRed1}). In view of (\ref{Eq.GumbelPPConvergence}) and Lemma~\ref{Lem.OneDimensionalPointProcess}, it suffices to prove that the sequence $\{Y^N\}_{N\geq1}$ is tight in $\mathbb{R}^{\infty}$.

We seek to apply \cite[Proposition~2.21]{ED24a} to $Y^N$. Condition~(1) is precisely (\ref{Eq.GumbelPPConvergence}), while condition~(2) follows from the construction in Step~1, since $M^Y(\mathbb R)=\infty$ almost surely. It remains to verify condition~(3).

Since $Y_1^N$ is the rightmost atom of $M^{Y^N}$, for every $a\in\mathbb R$,
$$
{\bf 1}\{Y_1^N\geq a\}\leq M^{Y^N}([a,\infty)).
$$
Taking expectations on both sides and using (\ref{Eq.S2PairMeanDPP}) and (\ref{Eq.GumbelDiagonalBound}), we obtain
\begin{equation}\label{Eq.GumbelUpperTail}
\mathbb P(Y_1^N\geq a)
\leq
\mathbb E[M^{Y^N}([a,\infty))]
=
\int_a^\infty K_N(x,x)\,dx
\leq
e^{R_N^3/12-a}.
\end{equation}
Consequently,
\begin{equation*}
\lim_{a\to\infty}\limsup_{N\to\infty}
\mathbb P(Y_1^N\geq a)=0.
\end{equation*}
Thus, condition~(3) of \cite[Proposition~2.21]{ED24a} holds, and hence $\{Y^N\}_{N\geq1}$ is tight in $\mathbb R^\infty$, completing the proof.
\end{proof}

\bibliographystyle{alpha}
\bibliography{P}

@article{DLMS15,
  author = {Dean, D. S. and Le Doussal, P. and
            Majumdar, S. N. and Schehr, G.},
  title = {Finite-temperature free fermions and the
           {K}ardar--{P}arisi--{Z}hang equation at finite time},
  journal = {Phys. Rev. Lett.},
  volume = {114},
  number = {11},
  pages = {110402},
  year = {2015},
  doi = {10.1103/PhysRevLett.114.110402}
}

@article{DLMS16,
  author = {Dean, D. S. and Le Doussal, P. and
            Majumdar, S. N. and Schehr, G.},
  title = {Noninteracting fermions at finite temperature in a
           {$d$}-dimensional trap: Universal correlations},
  journal = {Phys. Rev. A},
  volume = {94},
  number = {6},
  pages = {063622},
  year = {2016},
  doi = {10.1103/PhysRevA.94.063622}
}

@article{AH21,
author = {Aggarwal, A. and Huang, J.},
title = {Edge statistics for lozenge tilings of polygons, {I}{I}: {A}iry line ensemble},
journal = {Forum Math. Pi},
volume = {13},
year = {2025},
pages = {e2, 1--60},
}

@article{LW20,
  author = {Liechty, K. and Wang, D.},
  title = {Asymptotics of free fermions in a quadratic well at
           finite temperature and the
           {M}oshe--{N}euberger--{S}hapiro random matrix model},
  journal = {Ann. Inst. Henri Poincar{\'e} Probab. Stat.},
  volume = {56},
  number = {2},
  pages = {1072--1098},
  year = {2020},
  doi = {10.1214/19-AIHP994}
}

@article{AH23,
author = {Aggarwal, A. and Huang, J.},
title = {Strong characterization for the {A}iry line ensemble},
journal = {Invent. Math.},
volume = {243},
number = {2},
pages = {293--605},
year = {2026},
}

@article{FNH99,
author = {Forrester, P. J. and Nagao, T. and Honner, G.},
title = {Correlations for the orthogonalunitary and symplectic-unitary transitions at the hard and soft edges},
journal = {Nuclear Phys. B},
year = 1999,
volume = {53},
number = {3},
pages = {601-643}
}

@article{Mac94,
author = {Mac{\^e}do, A.M.S.},
title = {Universal parametric correlations at the soft edge of the spectrum of random matrix ensembles},
journal = {Europhys. Lett.},
year = 1994,
volume = {26},
issue = {9},
pages = {641}
}

@article{Ser23,
author = {Serio, C.},
title = {Scaling limit for line ensembles of random walks with geometric area tilts},
journal = {Electron. J. Probab.},
year = {2023},
volume = {28},
pages = {1--14},
note = {Paper No. 1026}
}

@article{HKS25,
author = {Hegde, M. and Kim, Y.H. and Serio, C.},
title = {Scaling limit and tail bounds for a random walk model of {SOS} level lines},
year = {2025},
note = {Preprint: arXiv:2502.10384}
}

@article{BO22,
  author = {Betea, D. and Occelli, A.},
  title = {Peaks of cylindric plane partitions},
  journal = {S{\'e}m. Lothar. Combin.},
  volume = {86B},
  note = {Art. 49, 12 pp.},
  year = {2022},
  url = {https://www.mat.univie.ac.at/~slc/wpapers/FPSAC2022/49.html}
}

@article{Zha24,
author = {Zhang, X.},
title = {{TASEP} in half-space},
year = {2024},
note = {Preprint: arXiv:2409.09974},
}

@article{BLL26,
author = {Baik, J. and Liao, Y. and Liu, Z.},
title = {Periodic {KPZ} fixed point with general initial conditions},
year = {2026},
note = {Preprint: arXiv:2603.01964},
}

@article{ACS26,
author = {Aggarwal, A. and Corwin, I. and Schmid, D.},
title = {Periodic directed landscape},
year = {2026},
note = {Preprint: arXiv:2607.18104},
}

@article{DZ26,
author = {Dauvergne, D. and Zhang, L.},
title = {The directed landscape in half-space},
year = {2026},
note = {Preprint: arXiv:2604.10020},
}

@article{DDY26,
author = {Das, S. and Dimitrov, E. and Yang, Z.},
title = {Pinning in non-critical half-space geometric last passage percolation},
year = {2026},
note = {Preprint: arXiv:2603.24616},
}

@article{DSY26,
author = {Dimitrov, E. and Serio, C. and Yang, Z.},
title = {The pinned half-space {A}iry line ensemble},
year = {2026},
note = {Preprint: arXiv:2601.04546},
}

@book{GasperRahman,
author = {Gasper, G. and Rahman, M.},
title = {Basic {H}ypergeometric {S}eries},
edition = {2},
year = {2004},
publisher = {Cambridge University Press},
address = {Cambridge}
}

@book{Prasolov,
author = {Prasolov, V.},
title = {Problems and {T}heorems in {L}inear {A}lgebra},
year = {1994},
publisher = {Amer. Math. Soc.},
address = {Providence, RI}
}

@book{VS10,
author = {Vall{\'e}e, O. and Soares, M.},
title = {Airy Functions and Applications to Physics},
edition = {2},
publisher = {Imperial College Press},
address = {London},
year = {2010},
}

@book{Cin11,
author = {\c{C}{\i}nlar, E.},
title = {Probability and Stochastics},
series = {Graduate Texts in Mathematics},
volume = {261},
publisher = {Springer},
address = {New York},
year = {2011},
}

@article{Joh07,
author = {Johansson, K.},
title = {From {G}umbel to {T}racy--{W}idom},
journal = {Probab. Theory Related Fields},
volume = {138},
number = {1--2},
pages = {75--112},
year = {2007},
doi = {10.1007/s00440-006-0012-7},
}

@article{GK97,
  title={Cylindric partitions},
  author={Gessel, I. M. and Krattenthaler, C.},
  journal={Trans. Amer. Math. Soc.},
  volume={349},
  number={2},
  pages={429--479},
  year={1997},
}

@article{DimitrovSerio,
  title={Uniform convergence of {D}yson {F}errari--{S}pohn diffusions to the {A}iry line ensemble},
  author={Dimitrov, E. and Serio, C.},
  journal = {Ann. Inst. Henri Poincar{\'e} Probab. Stat.},
  volume={61},
  number={1},
  pages={385--402},
  year={2025},
}

@book{SteinFourier,
  title={Fourier analysis: an introduction},
  author={Stein, E. M. and Shakarchi, R.},
  volume={1},
  year={2011},
  publisher={Princeton University Press}
}

@book{McK69,
  author    = {McKean, H. P.},
  title     = {Stochastic integrals},
  publisher = {Academic Press},
  address = {New York},
  year      = 1969,
}

@book{Sze75,
  author    = {Szeg{\H{o}}, G.},
  title     = {Orthogonal polynomials},
  edition   = {4},
  series    = {American Mathematical Society Colloquium Publications},
  volume    = {23},
publisher = {American Mathematical Society},
address = {Providence, RI},
  year      = {1975},
}

@article{DY25,
author = {Dimitrov, E. and Yang, Z.},
title = {Half-space {A}iry line ensembles},
year = {2026},
journal = {Probab. Theory Related Fields},
note = {DOI: 10.1007/s00440-026-01469-x},
}

@article{D24b,
  title={Tightness for interlacing geometric random walk bridges},
  author={Dimitrov, E.},
  note={Preprint: arXiv:2410.23899},
  year={2024},
}

@article{Pet14,
author = {Petrov, L.},
title = {Asymptotics of random lozenge tilings via {G}elfand--{T}setlin schemes},
journal = {Probab. Theory Related Fields},
volume = {160},
number = {3-4},
pages = {429--487},
year = {2014},
doi = {10.1007/s00440-013-0532-x},
}

@article{BBCCR17,
author = {Boutillier, C. and Bouttier, J. and Chapuy, G. and Corteel, S. and Ramassamy, S.},
title = {Dimers on rail yard graphs},
journal = {Ann. Inst. Henri Poincar{\'e} D},
volume = {4},
number = {4},
year = {2017},
pages = {479--539},
doi = {10.4171/AIHPD/46},
}

@article{Spohn,
author = {Pr{\"a}hofer, M. and Spohn, H.},
title = {Scale invariance of the {P}{N}{G} droplet and the {A}iry process},
journal = {J. Stat. Phys.},
year = 2002,
volume = {108},
number = {5--6},
pages = {1071--1106}
}

@book{EMOT,
author = {Erd{\'e}lyi, A. and Magnus, W. and Oberhettinger, F. and Tricomi, F. G.},
title = {Higher Transcendental Functions},
publisher = {McGraw--Hill},
address = {New York},
year = {1953},
volume = {2},
}

@article{DEA21,
author = {Dimitrov, E. and Fang, X. and Fesser, L. and Serio, C. and Teitler, C. and Wang, A. and Zhu, W.},
title = {Tightness of {B}ernoulli {G}ibbsian line ensembles},
journal = {Electron. J. Probab.},
year = 2021,
volume = {26},
pages = {1--93},
}

@article{BR05,
  author    = {Borodin, A. and Rains, E. M.}, 
  title     = {Eynard-{M}ehta theorem, {S}chur process, and their {P}faffian analogs},
  journal = {J. Stat. Phys.},
  year      = 2005,
  number = {3--4},
  volume = {121}, 
  pages = {291--317},
}

@article{CS14,
author = {Corwin, I. and Sun, X.},
title = {Ergodicity of the {A}iry line ensemble},
journal = {Electron. Commun. Probab.},
volume = {19},
number = {49},
pages = {1--11},
year = {2014},
}

@article{CU2,
author = { Corwin, I.},
title = {The {K}ardar-{P}arisi-{Z}hang equation and universality class},
journal = {Random Matrices: Theory Appl.},
volume = {1},
year = 2012,
}

@article{Sod15,
  author    = {Sodin, S.}, 
  title     = {A limit theorem at the spectral edge for corners of time-dependent {W}igner matrices},
  journal = {Int. Math. Res. Not.},
  year   = 2015,
  volume = {2015}, 
  pages = {7575-7607},
}

@article{DOV22,
author = {Dauvergne, D. and Ortmann, J. and Vir{\'a}g, B.},
title = {The directed landscape},
journal = {Acta Math.},
year = 2022,
volume = {229},
number = {2},
pages = {201-285}
}

@misc{DLMF,
  key          = "{\relax DLMF}",
  title        = "{\it NIST Digital Library of Mathematical Functions}",
  howpublished = "\url{https://dlmf.nist.gov/}, Release 1.2.8 of 2026-09-15",
  url          = "https://dlmf.nist.gov/",
  note         = "F.~W.~J. Olver, A.~B. {Olde Daalhuis}, D.~W. Lozier,
                  B.~I. Schneider, R.~F. Boisvert, C.~W. Clark,
                  B.~R. Miller, B.~V. Saunders, H.~S. Cohl,
                  and M.~A. McClain, eds."
}

@article{ED24a,
author = {Dimitrov, E.},
title = {Airy wanderer line ensembles},
year = {2026},
journal = {J. Lond. Math. Soc.},
volume = {113},
number = {3},
pages = {e70482},
}

@article{DNV23,
  author    = {Dauvergne, D. and Nica, M. and Vir{\'a}g, B.}, 
  title     = {Uniform convergence to the {A}iry line ensemble},
  journal = {Ann. Inst. Henri Poincar{\'e} Probab. Stat.},
  year      = 2023,
  volume = {59}, 
  number = {4},
  pages = {2220--2256},
}

@article{DM21,
  author    = {Dimitrov, E. and Matetski, K.}, 
  title     = {Characterization of {B}rownian {G}ibbsian line ensembles},
  journal = {Ann. Probab.},
  year      = 2021,
  volume = {49}, 
  number = {5},
  pages = {2477--2529},
}

@book{Munkres,
  author    = {Munkres, J.}, 
  title     = {Topology},
  edition = {2},
  publisher = {Prentice Hall, Inc.},
  address = {Upper Saddle River, NJ},
  year      = 2003,
}

@book{kallenberg2017random,
  author    = {Kallenberg, O.}, 
  title     = {Random Measures, Theory and Applications},
  publisher = {Springer International Publishing Switzerland},
  year      = 2017,
}

@article{DW19,
author = {Dimitrov, E. and Wu, X.},
title = {K{M}{T} coupling for random walk bridges},
journal = {Probab. Theory Related Fields},
year = 2021,
number = {3--4},
volume = {179},
pages = {649--732}
}

@article{CorHamA,
author = {Corwin, I. and Hammond, A.},
title = {Brownian {G}ibbs property for {A}iry line ensembles},
journal = {Invent. Math.},
number = {2},
year = 2014,
volume = {195},
pages = {441--508}
}

@book{Billing,
  author    = {Billingsley, P.}, 
  title     = {Convergence of probability measures},
  edition = {2},
  publisher = {John Wiley \& Sons},
  address = {New York},
  year      = 1999,
}

@article{OR03,
  author    = {Okounkov, A. and Reshetikhin, N.}, 
  title     = {Correlation function of the {S}chur process with application to local geometry of a random 3-dimensional {Y}oung diagram},
  journal = {J. Amer. Math. Soc.},
  year      = 2003,
  number = {3},
  volume = {16}, 
  pages = {581--603},
}

@article{Bor07,
  title={Periodic {S}chur process and cylindric partitions},
  author={Borodin, A.},
  journal={Duke Math. J.},
  volume={140},
  number={3},
  pages={391--468},
  year={2007}
}

@article{BB19,
  title={The periodic {S}chur process and free fermions at finite temperature},
  author={Betea, D. and Bouttier, J.},
  journal={Math. Phys. Anal. Geom.},
  volume={22},
  number = {1},
  year={2019},
  note={Paper No. 3, 47 pp.},
}

@article{le2017periodic,
  title={Periodic {A}iry process and equilibrium dynamics of edge fermions in a trap},
  author={Le Doussal, P. and Majumdar, S. N. and Schehr, G.},
  journal={Ann. Physics},
  volume={383},
  pages={312--345},
  year={2017},
}

@article{Sos00,
  title={Determinantal random point fields},
  author={Soshnikov, A.},
  journal={Russian Math. Surveys},
  volume={55},
  number={5},
  pages={923--975},
  year={2000}
}

@article {Kurt03,
    AUTHOR = {Johansson, K.},
     TITLE = {Discrete polynuclear growth and determinantal processes},
   JOURNAL = {Comm. Math. Phys.},
    VOLUME = {242},
      YEAR = {2003},
    number = {1--2},
     PAGES = {277--329},
}

@article {BBCS18,
    AUTHOR = {Baik, J. and Barraquand, G. and Corwin, I. and Suidan, T.},
     TITLE = {Pfaffian {S}chur processes and last passage percolation in a half-quadrant},
   JOURNAL = {Ann. Probab.},
    VOLUME = {46},
      YEAR = {2018},
    number = {6},
     PAGES = {3015--3089},
}

@article {BCD25,
    AUTHOR = {Barraquand, G. and Corwin, I. and Das, S.},
     TITLE = {{KPZ} exponents for the half-space log-gamma polymer},
   JOURNAL = {Probab. Theory Related Fields},
    VOLUME = {192},
      YEAR = {2025},
    number = {3--4},
     PAGES = {1113--1243},
}

@article {MQR21,
    AUTHOR = {Matetski, K. and Quastel, J. and Remenik, D.},
     TITLE = {The {KPZ} fixed point},
   JOURNAL = {Acta Math.},
    VOLUME = {227},
      YEAR = {2021},
    number = {1},
     PAGES = {115--203},
}

@book{Olver,
author = {Olver, F. W. J.},
title = {Asymptotics and Special Functions},
year = {1997},
publisher = {A K Peters},
address = {Wellesley, MA},
note = {Originally published by Academic Press, New York, 1974}
}

\end{document}